\documentclass[12pt, a4paper, twoside, notitlepage]{article}
\usepackage[utf8]{inputenc}
\usepackage{appendix}
\usepackage{url}
\usepackage{caption}
\usepackage[T1]{fontenc}
\usepackage[margin=2.4cm]{geometry}
\usepackage{amsmath,amssymb,amsthm,mathtools}
\usepackage{booktabs}
\usepackage{array}
\usepackage{colortbl}
\usepackage{enumitem}
\usepackage{graphicx}
\usepackage{hyperref}
\usepackage{cleveref}
\usepackage{xcolor}
\usepackage{tikz}
\usetikzlibrary{arrows.meta,calc,patterns}
\usepackage{tcolorbox}
\tcbuselibrary{theorems,breakable}
\crefname{theorem}{Thm.}{Thms.}
\Crefname{theorem}{Thm.}{Thms.}
\crefname{proposition}{Prop.}{Props.}
\Crefname{proposition}{Prop.}{Props.}
\crefname{remark}{Rem.}{Rems.}
\Crefname{remark}{Rem.}{Rems.}
\crefname{lemma}{Lem.}{Lems.}
\Crefname{lemma}{Lem.}{Lems.}
\crefname{corollary}{Cor.}{Cors.}
\Crefname{corollary}{Cor.}{Cors.}
\crefname{definition}{Def.}{Defs.}
\Crefname{definition}{Def.}{Defs.}
\crefname{assumption}{Ass.}{Asss.}
\Crefname{assumption}{Ass.}{Asss.}
\crefname{property}{Prop.}{Props.}
\Crefname{property}{Prop.}{Props.}
\definecolor{accent}{HTML}{2D5F8A}
\definecolor{accent2}{HTML}{8A2D5F}
\definecolor{defgreen}{HTML}{44AA99}
\definecolor{thmamber}{HTML}{CC8844}
\definecolor{lemgold}{HTML}{AA8800}
\definecolor{prfpurple}{HTML}{8866CC}
\definecolor{remgrey}{HTML}{8888AA}
\definecolor{corgold}{HTML}{DD9966}
\definecolor{warnred}{HTML}{CC4444}
\hypersetup{
 colorlinks=true, linkcolor=accent, urlcolor=accent, citecolor=accent2
}
\theoremstyle{definition}
\newtheorem{definition}{Definition}[section]
\newtheorem{proposition}{Proposition}[section]
\newtheorem{lemma}{Lemma}[section]
\newtheorem{assumption}{Assumption}[section]
\newtheorem{example}{Example}[section]
\newtheorem{property}{Property}[section]
\newtheorem{convention}{Convention}[section]
\theoremstyle{plain}
\newtheorem{theorem}{Theorem}[section]
\newtheorem{corollary}{Corollary}[section]
\theoremstyle{remark}
\newtheorem{remark}{Remark}[section]
\makeatletter
\let\c@proposition\c@definition
\let\c@lemma\c@definition
\let\c@assumption\c@definition
\let\c@example\c@definition
\let\c@theorem\c@definition
\let\c@corollary\c@definition
\let\c@remark\c@definition
\let\c@property\c@definition
\let\c@convention\c@definition
\makeatother
\newtcolorbox{defbox}{
  colback=white!8, colframe=white!70!black,
  left=6pt, right=6pt, top=4pt, bottom=4pt,
  breakable, sharp corners, boxrule=0pt, leftrule=3pt
}
\newtcolorbox{thmbox}{
  colback=white!8, colframe=white!70!black,
  left=6pt, right=6pt, top=4pt, bottom=4pt,
  breakable, sharp corners, boxrule=0pt, leftrule=3pt
}
\newtcolorbox{lembox}{
  colback=white!8, colframe=lemgold!70!black,
  left=6pt, right=6pt, top=4pt, bottom=4pt,
  breakable, sharp corners, boxrule=0pt, leftrule=3pt
}
\newtcolorbox{prfbox}{
  colback=white!6, colframe=prfpurple!60!black,
  left=6pt, right=6pt, top=4pt, bottom=4pt,
  breakable, sharp corners, boxrule=0pt, leftrule=3pt
}
\newtcolorbox{rembox}{
  colback=white!8, colframe=white!60!black,
  left=6pt, right=6pt, top=4pt, bottom=4pt,
  breakable, sharp corners, boxrule=0pt, leftrule=3pt
}
\newtcolorbox{corbox}{
  colback=white!8, colframe=corgold!60!black,
  left=6pt, right=6pt, top=4pt, bottom=4pt,
  breakable, sharp corners, boxrule=0pt, leftrule=3pt
}
\newtcolorbox{eqbox}{
  colback=white!5, colframe=accent!40!black,
  left=8pt, right=8pt, top=6pt, bottom=6pt,
  breakable, sharp corners, boxrule=0.8pt
}
\newtcolorbox{warnbox}{
  colback=white!6, colframe=warnred!60!black,
  left=6pt, right=6pt, top=4pt, bottom=4pt,
  breakable, sharp corners, boxrule=0pt, leftrule=3pt
}
\newtcolorbox{keybox}{
  colback=white!10, colframe=accent!80!black,
  left=8pt, right=8pt, top=6pt, bottom=6pt,
  breakable, sharp corners, boxrule=1.2pt
}
\usepackage{titlesec}
\titleformat{\section}{\Large\bfseries\color{accent}}{\thesection.}{0.5em}{}
\titleformat{\subsection}{\large\bfseries\color{accent}}{\thesubsection}{0.5em}{}
\usepackage{csquotes}
\usepackage{mathrsfs}
\usepackage{authblk}

\newcommand{\dd}{\mathrm{d}}
\newcommand{\RR}{\mathbb{R}}
\newcommand{\ip}[2]{\langle #1,\, #2 \rangle}
\newcommand{\ipw}[3]{\langle #1,\, #2 \rangle_{#3}} 

\newcommand{\Lie}{\mathcal{L}}
\newcommand{\abs}[1]{\lvert #1 \rvert}
\newcommand{\nrm}[1]{\lVert #1 \rVert} 
\newcommand{\twdg}{\;{\wedge}\;}
\newcommand{\tdd}{\tilde{\dd}}
\newcommand{\ddt}{\frac{\mathrm{d}}{\mathrm{d}t}}
\newcommand{\tD}{\tilde{D}}
\newcommand{\tLie}{{\Lie}}
\newcommand{\Qh}{\mathcal{Q}_h}
\newcommand{\bv}{{v}}
\newcommand{\bu}{{u}}
\newcommand{\bw}{{w}}
\newcommand{\bom}{\omega}
\newcommand{\tU}{{\widetilde{U}}}

\newcommand{\bphi}{\varphi}
\renewcommand{\v}{{v}}
\renewcommand{\u}{{u}}
\newcommand{\w}{{w}}
\newcommand{\bekin}{{\boldsymbol{e}^{\mathrm{kin}}}}
\newcommand{\Iv}{I_{{\v}}}

\newcommand{\brho}{\rho}
\newcommand{\eps}{\varepsilon}
\newcommand{\bM}{M}
\newcommand{\bD}{D}
\newcommand{\bF}{F}
\newcommand{\bR}{R}
\newcommand{\bL}{L}
\newcommand{\KK}{\mathcal{K}}
\newcommand{\KKs}{\mathcal{K}^*}
\newcommand{\OO}{\mathcal{O}}

\newcommand{\bP}{P}
\newcommand{\bQ}{Q}
\newcommand{\bA}{A}
\newcommand{\bff}{\mathrm{f}}
\newcommand{\bPhi}{\boldsymbol{\Phi}} 
\newcommand{\statespace}{\mathcal{X}}
\newcommand{\admissible}{\mathcal{A}}

\newcommand{\be}{\boldsymbol{e}}

\newcommand{\inte}{\mathrm{int}}

\newcommand{\kin}{\mathrm{kin}}

\newcommand{\Ran}{\mathrm{Ran}}

\newcommand{\bx}{\boldsymbol{x}}

\newcommand{\Iev}{I_{e_v}}
\newcommand{\vol}{\mathrm{vol}}
\newcommand{\Eint}{E_{\mathrm{int}}}
\newcommand{\Etot}{E_{\mathrm{tot}}}
\newcommand{\ekin}{e^{\mathrm{kin}}}

\newcommand{\Ekindf}{E_{\kin}^{\rm df}}
\newcommand{\Ekindw}{E_{\kin}^{\rm dw}}
\renewcommand{\epsilon}{\varepsilon}

\newcommand{\bze}{\boldsymbol{\zeta}}

\begin{document}
\title{
{\LARGE\bfseries \color{accent}
A no-go theorem and its resolution for the discrete compressible barotropic Navier--Stokes equations
}}
\author{Peter Korn\thanks{Max-Planck Institute for Meteorology}}
\maketitle
\begin{abstract}
The compressible barotropic Navier--Stokes equations in
vector-invariant form preserve the system's vorticity structure
--- exact Kelvin circulation, potential vorticity, Casimir
invariants, Lamb antisymmetry --- and underlie modern atmospheric
and ocean dynamical cores, yet no PDE theory exists for the
discrete compressible system in this form.

For Delaunay--Voronoi meshes we prove that every density-independent mass matrix with an
integration-by-parts-consistent divergence carries a sharp
$\OO(h^2)$ energy residual of indeterminate sign that no operator
choice can eliminate; the no-go theorem covers A-, B-, C-, D-, and
quasi-B-grid staggerings. The density-weighted mass matrix is the
unique algebraic remedy, restoring exact total energy while
preserving the vector-invariant momentum equation, Lamb
antisymmetry, and the topological conservation laws, at the cost
of an $\OO(h^{r_\star})$ Kelvin defect matching the convergence
rate. This residual is the discrete root cause of the
Hollingsworth instability of vector-invariant dynamical cores; the
density-weighted construction removes it structurally.

For the density-weighted scheme on closed oriented Riemannian
manifolds in $d = 2, 3$ we establish global well-posedness for
$\nu \ge 0$, convergence to smooth solutions at rate
$\OO(h^{\min(d-1,r_\star)})$ uniformly in $\nu$, asymptotic
preservation in the low-Mach limit --- the density-free residual
diverges as $\OO(M^{-1})$ --- and Lyapunov stability around three
classes of equilibria: unconditional around hydrostatic and
constant-flow stratified states, conditional around sheared
baroclinic states under a discrete Charney--Stern criterion.
\end{abstract}


\setcounter{tocdepth}{3}
 \tableofcontents
 \newpage
\section{Introduction and Overview}\label{sec:intro}

The compressible barotropic Euler and Navier--Stokes equations
admit two equivalent continuum formulations. The momentum form,
\begin{equation}\label{eq:intro_mom}
 \partial_t(\rho\bu) + \mathrm{div}(\rho\bu\otimes\bu) + \nabla p
 + \rho\,\nabla\Phi
 = \mathrm{div}(\mathbb{S}(\bu)),
 \qquad
 \partial_t\rho + \mathrm{div}(\rho\bu) = 0,
\end{equation}
with viscous stress $\mathbb{S}(\bu)$, is the basis of the modern
PDE theory: local strong solutions, global weak solutions, vacuum
formation, and relative-energy convergence of momentum-form schemes
have been developed by Gallou\"et--Herbin--Latch\'e,
Karper~\cite{gallouet2008,gallouet2009,gallouet2010,karper2013},
and Jovanovic--Rohde and
Gallou\"et~et~al.~\cite{jovanovic2006,gallouet2016}; we refer
to~\cite{feireisl2021book} for the modern state of the theory. For the vector-invariant form,
\begin{equation*}
 \partial_t v + \iota_\bu\omega
 + \dd\!\bigl(\tfrac{1}{2}|\bu|^2 + h(\rho) + \Phi\bigr)
 = \bigl(\mathrm{div}\,\mathbb{S}(\bu)\bigr)^\flat,
 \qquad
 \partial_t\rho_{\rm vol} + \mathcal{L}_\bu\rho_{\rm vol} = 0,
\end{equation*}
with $v = \bu^\flat$ the velocity 1-form, $\omega = \dd v$ the
vorticity 2-form, $h(\rho) = e(\rho) + p(\rho)/\rho$ the specific
enthalpy, and $\Phi$ the geopotential. The viscous term is the
divergence of the Newtonian stress
$\mathbb{S}(\bu) = 2\nu\,\mathbf{S}(\bu) + \lambda\,(\nabla\cdot\bu)\,g$,
with deformation tensor $\mathbf{S} = \tfrac{1}{2}\mathcal{L}_\bu g$,
shear viscosity $\nu \ge 0$, and second viscosity coefficient
$\lambda$; thermodynamic admissibility requires the bulk viscosity
$\zeta := \lambda + \tfrac{2}{3}\nu$ to be non-negative. In
velocity-1-form (Hodge--de~Rham) variables the stress divergence
splits into independent dissipation of the rotational and
dilatational parts of the field,
$(\mathrm{div}\,\mathbb{S}(\bu))^\flat
= -\nu\,\delta\dd v - (\tfrac{4}{3}\nu + \zeta)\,\dd\delta v$
(\Cref{app:viscous_catalogue}). This formulation requires
$\rho > 0$ and is therefore restricted to the non-vacuum regime. Within that regime
it is the natural formulation: it preserves the vorticity
structure -- Kelvin circulation, potential-vorticity
balance, Casimir invariants, Lamb antisymmetry -- which the
momentum form does not, and matters whenever rotation governs the
large-scale dynamics, as in atmospheric and ocean modelling
(ICON~\cite{korn2017}, MPAS~\cite{ringler2010},
DYNAMICO~\cite{dubos2015}), in rotorcraft aerodynamics, and in
swirl-stabilised combustion. For the vector-invariant form,
however, no PDE theory has been developed in the compressible
setting. \emph{The present paper closes this gap.}

\paragraph*{Discrete energy conservation in the vector-invariant form.}
The structural question this paper resolves is whether a
finite-dimensional discretisation of the vector-invariant
system~\eqref{eq:intro_baro} on a polyhedral mesh can conserve
total energy exactly while retaining the vorticity structure
 that motivates the use of the
vector-invariant form. We prove
(\Cref{thm:dichotomy,thm:dichotomy_general}) that the answer is
negative for the entire class of density-independent mass
matrices with integration-by-parts-consistent divergence: every
such scheme carries a structural $\OO(h^2)$ residual that acts
as a spurious source or sink of total energy. The residual is
sharp, face-local, and covers all standard staggerings (A, B, C,
D, quasi-B). Operationally, this discrete energy residual is the
root cause of the Hollingsworth
instability~\cite{hollingsworth1983,bell2017hollingsworth} that
has shaped the discrete design of vector-invariant dynamical
cores~\cite{skamarock2012,gassmann2013,melvin2024gungho}.

\paragraph*{Three schemes.}
Three discretisations are relevant. The
\emph{density-free vector-invariant} scheme (DF) preserves Kelvin
circulation and Lamb antisymmetry exactly but carries a structural
$\OO(h^2)$ energy residual; this is the design used in operational
codes (ICON~\cite{korn2017}, MPAS~\cite{ringler2010},
DYNAMICO~\cite{dubos2015}), and the residual is the discrete
mechanism behind the Hollingsworth instability. The
\emph{density-weighted vector-invariant} scheme (DW), proposed and
analysed here, conserves total energy and Lamb antisymmetry
exactly at the price of an $\OO(h^{r_\star})$ Kelvin defect, and
is structurally Hollingsworth-immune. The \emph{conservative
momentum form} (M), which carries $\rho\bu$ as prognostic variable,
conserves total energy by construction but loses the Kelvin
theorem, the potential-vorticity balance, and Lamb antisymmetry; its
theory is well-developed and it 
serves as reference against which DF and DW are
compared. The paper develops the well-posedness, convergence,
asymptotic-preservation, and linear-stability theory for both
vector-invariant variants uniformly in $\nu \ge 0$ on closed
oriented Riemannian manifolds in $d = 2, 3$.

\paragraph*{Discrete framework and prior work.}
The discrete framework rests on Discrete Exterior Calculus (DEC) on
prismatic Delaunay--Voronoi meshes. Two structural properties drive
every result below: the discrete exterior derivative commutes with
the de~Rham interpolation
($\tD_1\mathcal{R}_h = \mathcal{R}_h\,\dd$), so all approximation
error is concentrated in the metric (Hodge-star) operators; and the
discrete Lamb vector satisfies $\ip{\bv}{\Iv(\bom)}_1 = 0$ for any
velocity cochain, the metric-blind antisymmetry that powers the
circulation invariants.

The foundational framework in the exterior-calculus setting is the
Finite Element Exterior Calculus (FEEC) of Arnold, Falk, and
Winther~\cite{arnold2006,arnold2010}, which constructs discrete
de~Rham complexes on simplicial meshes with optimal Hodge--Laplace
approximation; FEEC provides the topological infrastructure (exact
sequences, commuting projections, discrete Poincar\'e inequalities)
but does not address nonlinear advection or velocity--density
coupling. Incompressible DEC fluid mechanics was pioneered by Elcott et
al.~\cite{elcott2007}; the variational and conservation structure
of incompressible DEC schemes was developed
in~\cite{mullen2009,pavlov2011}, without error estimates or
treatment of the compressible case. The mimetic framework
of~\cite{thuburn2012,ringler2010} produced energy- and
enstrophy-conserving \emph{shallow-water} schemes (2D, constant
background) without rigorous convergence theory.
Korn--Linardakis~\cite{kornLinardakis2018} showed 
that the density-weighted mass matrix
$\bM_1^\rho = P^T\mathrm{diag}(\brho)P$ yields exact total-energy
conservation for shallow water equations. The rate-uniform-in-$\nu$ analysis
available for incompressible
FEM~\cite{girault1986,HeywoodRannacher1982} has no
compressible vector-invariant counterpart in the literature. The
present construction extends FEEC and DEC to compressible
vector-invariant fluid dynamics in $d = 2, 3$ by equipping the
discrete complex with a nonlinear extrusion-based contraction and
the density-weighted mass matrix, and proves the well-posedness,
convergence, and asymptotic-preservation theory.

\paragraph*{Energetic linear stability via discrete Arnold
energy-Casimir.}
A second analytical consequence of exact discrete energy
conservation in the DW scheme is a discrete realisation of the
Arnold energy-Casimir
method~\cite{arnold1965,arnold1969,holm1985}. Exact discrete
conservation forces a bridge identity (\Cref{lem:arnold_bridge})
that lifts linear stability around discrete equilibria to
positive-definiteness of the modified-energy Hessian, which we
verify in three classes (\Cref{thm:main_DW_lyapunov}):
unconditional Lyapunov stability around hydrostatic and
constant-flow stratified equilibria -- the Hollingsworth setup of
\cite{hollingsworth1983,bell2017hollingsworth}-- and conditional Lyapunov stability around
sheared baroclinic equilibria under a discrete Charney--Stern
criterion that reduces in the Boussinesq incompressible limit to
the classical Rayleigh inflection-point
criterion~\cite{rayleigh1880,charney1962}
(\Cref{thm:dw_lyap_sheared,prop:dw_class_C_boussinesq}). The DW
scheme is linearly Lyapunov stable around continuum-stable
equilibria and admits genuine baroclinic instability only when
the continuum equations themselves do.

\paragraph*{Scope.}
This paper develops the analytical theory: the no-go theorem, the
well-posedness, convergence, and asymptotic-preservation results,
and the energetic linear-stability theorems. The discrete operator
framework underlying the construction has been validated extensively
in operational ocean and coupled Earth-system
simulations~\cite{korn2017,korn2022james,hohenegger2023gmd}, on the
3D hydrostatic primitive equations underlying ICON. 

The analysis developed here for the compressible barotropic system
transfers structurally to fully compressible, thermodynamically
active models used in operational weather and climate prediction
-- including the non-hydrostatic dynamical core of ICON
atmosphere~\cite{zaengl2015icon}, MPAS-A~\cite{skamarock2012}, and
the spectral-element configuration of GungHo~\cite{melvin2024gungho}.
The specifically thermodynamic ingredients
(potential temperature transport, full equation of state,
moisture) enter the analysis only through the equation of state
and the continuity equation; they do not modify the bilinear
structure of the energy residual $\mathcal{R}_E$ nor the
density-weighting fix.


\section{Main Results: Precise Statements}\label{sect:main_results}

\paragraph*{The continuous barotropic system.}
The barotropic compressible Euler equations on an oriented
Riemannian manifold $(\Omega, g)$, in vector-invariant form, read
\begin{equation}\label{eq:intro_baro}
 \partial_t v + \iota_{\mathbf{u}}\omega + \dd B
 = \bigl(\mathrm{div}\,\mathbb{S}(\bu)\bigr)^\flat,
 \qquad
 \partial_t\rho_{\mathrm{vol}}
 + \mathcal{L}_{\mathbf{u}}\rho_{\mathrm{vol}} = 0,
\end{equation}
where $v = \bu^\flat$ is the velocity 1-form, $\omega = \dd v$ the
vorticity 2-form, $\rho_{\mathrm{vol}} = \rho\,\mathrm{vol}$ the
density $d$-form, $B = \tfrac{1}{2}|\bu|^2 + h(\rho) + \Phi$ the
Bernoulli function with specific enthalpy
$h(\rho) = e(\rho) + p(\rho)/\rho$, an external potential $\Phi$, and $\nu \ge 0$ the shear viscosity
($\nu = \zeta = 0$ for the Euler case).
The viscous term is the divergence of the Newtonian stress
$\mathbb{S}(\bu) = 2\nu\,\mathbf{S}(\bu) + \lambda\,(\nabla\cdot\bu)\,g$,
with deformation tensor $\mathbf{S} = \tfrac{1}{2}\mathcal{L}_{\mathbf u}g$,
shear viscosity $\nu \ge 0$, and second viscosity coefficient
$\lambda$; the bulk viscosity $\zeta := \lambda + \tfrac{2}{3}\nu$
is non-negative by thermodynamic admissibility. On velocity
1-forms it splits along the Hodge--de~Rham decomposition into
independent dissipation of the rotational and dilatational parts,
\begin{equation}\label{eq:visc_hodge_split}
 \bigl(\mathrm{div}\,\mathbb{S}(\bu)\bigr)^\flat
 = -\nu\,\delta\dd v - \nu_{\mathrm{dil}}\,\dd\delta v,
 \qquad
 \nu_{\mathrm{dil}} := \tfrac{4}{3}\nu + \zeta \;\ge\; \tfrac{4}{3}\nu,
\end{equation}
with $\dd\delta + \delta\dd = \Delta_{\mathrm{dR}} \ge 0$ the
Hodge--de~Rham Laplacian. The Stokes closure $\zeta = 0$ and the
single-coefficient case $\nu_{\mathrm{dil}} = \nu$ are recovered as
special choices. The discrete realisation, the anisotropic and
Smagorinsky closures, and the verification of the admissibility
axioms are catalogued in \Cref{app:viscous_catalogue}.
The potential $\Phi\in C^0(\KKs)$ is a time-independent dual 0-cochain, representing gravitational, centrifugal, or any
conservative body force. Standard choices are $\bPhi_i = g z_i$ (gravity, $z_i$ = cell-centre height),
$\bPhi_i = -\frac{1}{2}|\Omega\times x_i|^2$ (centrifugal), or a sum of both.

\paragraph*{The two discrete schemes.}
The discrete Euler systems read
\begin{equation}\begin{split}\label{eq:main_DEC}
 & \ddt\bv + \Iv(\tD_1\bv)
 + \tD_0\bigl(h(\rho_i^{\rm vol}) + \ekin_i + \bPhi_i\bigr) = 0,\\
 & \ddt\brho = -\bD_2(\bR\Phi),
\end{split}\end{equation}
for the \emph{density-free} (DF) scheme, with kinetic energy
$\Ekindf = \tfrac{1}{2}\ip{v}{v}_1$ and mass matrix independent of
$\brho$; and
\begin{equation}\begin{split}\label{eq:main_dw}
 & \ddt\bv + [\bM_1^\rho(\brho)]^{-1}\bM_1\,\Iv(\tD_1\bv)
 + \tD_0\bigl(h(\rho_i^{\rm vol})
 + \tfrac{1}{2}|P\bv|_i^2 + \bPhi_i\bigr) = 0,\\
 & \ddt\brho = -\bD_2(\bR\Phi),
\end{split}\end{equation}
for the \emph{density-weighted} (DW) scheme, with kinetic energy
$\Ekindw = \tfrac{1}{2}\ip{v}{v}_\rho$ and density-weighted mass
matrix $\bM_1^\rho = P^T\mathrm{diag}(\brho)\,P$, where $P$ denotes a linear reconstruction and $P^T$ its transpose;
it is the same operator that occurs in the contraction $\Iv$ (see \Cref{def:contraction}). The continuity
equation is shared. 

The momentum form (M),
$\partial_t(\rho\bu) + \mathrm{div}(\rho\bu\otimes\bu) + \nabla p
+ \rho\,\nabla\Phi = \mathrm{div}(\mathbb{S}(\bu))$, has a well-developed theory (see e.g. \cite{feireisl2021book}), and is not analysed in detail; it serves as the calibration reference for the no-go theorem and as
the third point of the asymptotic-preservation comparison
(\Cref{thm:main_AP}).

\begin{theorem}[Conservation properties of two vector-invariant schemes]%
\label{thm:main_conservation}
Both schemes~\eqref{eq:main_DEC} and~\eqref{eq:main_dw} conserve mass
and total vorticity exactly, preserve mass-weighted potential vorticity
in exact flux form, satisfy the Lamb antisymmetry
$\ip{\bv}{\Iv(\alpha)}_1 = 0$ as an exact algebraic identity, and
conserve helicity to order $\OO(h^{r_\star})$. The two schemes part
company on Kelvin circulation and total energy. The density-free
scheme conserves Kelvin circulation exactly but carries a structural
total-energy residual $\mathcal{R}_E = \OO(h^2)$ of indeterminate
sign, which acts as a spurious source or sink in the energy budget.
The density-weighted scheme reverses this trade-off: it conserves
total energy exactly, at the price of an $\OO(h^{r_\star})$ defect
in the Kelvin circulation.
\end{theorem}
\noindent The supporting lemmas are proved in
\Cref{section_EulerComp,sec:DW_construction}.

\begin{theorem}[No-go theorem for discrete energy conservation]%
\label{thm:main_dichotomy}
Consider a staggered discretisation of the compressible barotropic
Euler equations on a Delaunay--Voronoi mesh satisfying
\Cref{ass:mesh_reg}. Suppose the velocity degrees of freedom are
normal components on dual edges and the density degrees of
freedom sit on primal cells; the continuity equation reads
$\ddt\brho = -\bD_2\bF$ with mass flux $\bF_j = \bar\rho_j\Phi_j$
in which $\bar\rho_j$ depends only on $\brho$; the
integration-by-parts identity
$\ip{\bv}{\tD_0 B}_1 = -\ip{B}{\bD_2\Phi}$ holds; and the
kinetic-energy mass matrix is independent of~$\brho$. Then no
density-only interpolation $\bar\rho_j(\brho)$ can deliver exact
total energy conservation on the admissible set: every such scheme
carries a non-vanishing structural residual $\mathcal{R}_E$ that
acts as a spurious source or sink of total energy. The conclusion
extends to A-, B-, D-, and quasi-B staggerings via a
polynomial-degree decomposition of $\mathcal{R}_E$ in the velocity
(\Cref{thm:dichotomy_general}).
\end{theorem}
\noindent The C-grid statement is \Cref{thm:dichotomy}; the
extension to general staggerings is \Cref{thm:dichotomy_general}.

\medskip
\noindent
The no-go theorem characterises the three schemes: density-free
schemes~\eqref{eq:main_DEC} satisfy all four hypotheses and carry
the residual; the density-weighted scheme~\eqref{eq:main_dw}
violates the density-independence of the kinetic-energy mass
matrix by construction and conserves total energy exactly, at the
cost of an $\OO(h^{r_\star})$ Kelvin defect; the conservative
momentum form trades the Lamb antisymmetry and the vorticity
invariants for exact energy conservation in a different way.

\begin{theorem}[Well-posedness, unified]%
\label{thm:main_DEC_euler}%
\label{thm:main_dw_euler}%
\label{thm:main_NS}%
\label{thm:main_DEC_NS}%
\label{thm:main_dw_NS}%
The discrete barotropic Euler and Navier--Stokes systems are well
posed in the following sense.
\begin{enumerate}[nosep,label=\textup{(\roman*)}]
\item The density-free Euler system admits a unique local-in-time
solution on an interval $[0,T^*)$ with $T^* \ge T_{\rm Bih}(h, E_0) > 0$.
\item The density-weighted Euler system admits a unique global solution
$(\bv, \brho) \in C^1([0,\infty); \admissible)$, with no
restriction on the existence time.
\item The density-free Navier--Stokes system with Hodge--Laplace
viscosity admits a unique local-in-time solution on the same
interval $[0,T^*)$ as the inviscid case.
\item The density-free Navier--Stokes system with Smagorinsky
viscosity admits a unique global solution provided
$\nu \ge \nu^*(h)$ and the Smagorinsky constant satisfies
$C_s \ge C_s^*(h)$.
\item The density-weighted Navier--Stokes system admits a unique
global solution for every admissible dissipation operator
$\bA \ge 0$, with integrated dissipation bounded by the initial
energy:
$\int_0^T \nrm{\tD_1 \bv}_\bA^2 \, dt \le \Etot^{\rm dw}(0)$.
\end{enumerate}
The structural contrast underlying the local-versus-global split is
that the density-free energy residual scales as $(\Ekindf)^{3/2}$
and therefore requires Smagorinsky dissipation ($\sim \nrm{\bom}^3$)
to absorb its cubic growth in the blow-up-time estimate, whereas the
density-weighted scheme carries no such residual.
\end{theorem}
\noindent The five cases are proved as
\Cref{thm:main-baro,thm:dw_wp,thm:main-baro-NS,thm:global-Smag,thm:dw_wp_NS}
respectively.

\begin{theorem}[Convergence rate -- regime-aligned]%
\label{thm:main_DEC_conv}
Let $(\bu, \rho^c)$ be a smooth solution on $[0, T]$ with
$\rho^c \ge \rho_* > 0$.
Both schemes converge at the rate
\[
 \sup_{t\in[0,T]}\bigl(
 \nrm{e_v(t)}_{\bM_1} + \nrm{e_\rho(t)}_{D_H}\bigr)
 \le C\,h^{\min(d-1,\,r_\star)},
\]
uniformly in $\nu \ge 0$, giving $\OO(h)$ in case~(A) and
$\OO(h^2)$ in case~(B) for $d = 3$ with the centred flux. The two
schemes target different physical regimes, and the constant $C$
reflects this division.
\begin{enumerate}[nosep,label=\textup{(\roman*)}]
\item For the density-free scheme, viewed as a low-Mach reference,
the constant $C = C(T, \epsilon_0)$ depends on the
near-incompressibility threshold $\epsilon_0$ of the smooth
reference, and the estimate holds only under
hypothesis~\eqref{eq:lowMach_hyp}. Outside the low-Mach regime,
the scheme suffers structural energy injection that mesh
refinement does not eliminate (\Cref{rem:DF_scope}).
\item For the density-weighted scheme, viewed as the fully
compressible target, the constant $C = C(T)$ depends only on
smooth-reference norms; no smallness hypothesis on $\bar\rho^c$
is required, no $L_h^\infty$-bootstrap appears in the proof,
and the total energy is conserved exactly at every~$h$.
\end{enumerate}
\end{theorem}
\noindent Proved as \Cref{thm:convergence_baro}.

\begin{theorem}[Asymptotic preservation]%
\label{thm:main_AP}
Under the Mach scaling $p = M^{-2}\hat p$ with well-prepared initial
data (\Cref{ass:well_prepared}), the three schemes behave as follows
in the limit $M \to 0$.
\begin{enumerate}[nosep,label=\textup{(\alph*)}]
\item The density-free scheme is \emph{not} asymptotically
preserving: its energy residual diverges as
$|\mathcal{R}_E| = \OO(M^{-1})$ on any fixed mesh, and the energy
balance degenerates.
\item The density-weighted scheme \emph{is} asymptotically
preserving. The fluctuation energy
\[
 \mathcal{E}_M = \Ekindw
 + M^{-2}\sum_i|K_i|\,D_{\hat H}(\rho_i^{\rm vol}\|\bar\rho)
\]
is exactly conserved ($\ddt\mathcal{E}_M = 0$), which delivers
$M$-independent a priori bounds. As $M \to 0$ on a fixed mesh, the
density-weighted system converges at rate $\OO(M)$ in both the
density fluctuations and the mass-matrix ratio to an incompressible
DEC system with exact Lamb antisymmetry.
\item The conservative momentum form is asymptotically preserving
by construction, since $\rho\bu$ is a prognostic variable and the
energy budget closes algebraically. It serves as the calibration
reference, against which the contrast between (a) and (b) isolates
the difference between the two vector-invariant variants.
\end{enumerate}
\end{theorem}
\noindent Proved as \Cref{thm:AP}.

\begin{theorem}[Energetic linear stability and the Hollingsworth question]%
\label{thm:main_DW_lyapunov}
The DW scheme is structurally immune to Hollingsworth-type
instability. Let $\bar X$ be a discrete equilibrium of the DW Euler
system on a Delaunay--Voronoi mesh satisfying \Cref{ass:mesh_reg},
with $\bar\brho > 0$ and $c^2(\bar\rho) > 0$. In each of the three
cases below, there exists a modified energy
\[
 \tilde E^{\rm dw}_{\rm tot}
 := E^{\rm dw}_{\rm tot} + C_1 + C^F_3 + U_0\,P_x,
\]
built from the discrete Casimirs -- mass~$C_1$, the PV Casimir
$C^F_3$, and the horizontal momentum~$P_x$ -- together with an
equilibrium-dependent shift~$U_0$, such that
$\nabla\tilde E^{\rm dw}_{\rm tot}(\bar X) = 0$ and the Hessian
$\tilde H := \nabla^2\tilde E^{\rm dw}_{\rm tot}(\bar X)$ is
positive definite.
\begin{enumerate}[nosep,label=\textup{(\alph*)}]
\item For a hydrostatic equilibrium $\bar\bv = 0$ with $\bar\brho$
in hydrostatic balance, positivity holds unconditionally.
\item For a constant-flow stratified equilibrium on a periodic
Cartesian-like mesh with $\bPhi = \bPhi(z)$, $\bar\bv = U\bar a^x$
a uniform horizontal translation and $\bar\brho(z)$ stratified,
positivity holds unconditionally in the flow magnitude~$U$.
\item For a sheared baroclinic equilibrium
$\bar\bv = \bar U(z)\bar a^x$ with $\bar\brho(z)$ stratified,
positivity holds \emph{conditionally}, under a discrete
Charney--Stern criterion comprising PV monotonicity~(CS1), a
sign-definite shifted velocity~(CS2), bounded Casimir
curvature~(CS3), sub-Mach flow~(CS4), and a quantitative
Schur-complement bound
(\Cref{thm:dw_lyap_sheared,thm:dw_class_C_Hpos}). In the Boussinesq
incompressible limit, this criterion reduces to the classical
Rayleigh inflection-point criterion
(\Cref{prop:dw_class_C_boussinesq}).
\end{enumerate}
In all three cases, the linearised DW dynamics satisfies the
bridge identity $A^T\tilde H + \tilde H\,A = 0$ for
$A := DF^{\rm dw}(\bar X)$, so the quadratic form
$\tilde E^{(2)}(Y) := \tfrac{1}{2}Y^T\tilde H Y$ is exactly
preserved along trajectories. In particular, the classical
Hollingsworth
instability~\cite{hollingsworth1983,bell2017hollingsworth} -- the
prototype of the $\mathcal{R}_E$-driven energy instability the
DW scheme is designed to eliminate -- cannot occur. The
Boussinesq SICK instability documented
by~\cite{ducousso2017sick} has a closely related but structurally
distinct driver and is not directly
covered by this theorem.
\end{theorem}
\noindent The three cases are proved as
\Cref{thm:dw_lyap_hydrostatic,thm:dw_lyap_constant_flow,thm:dw_lyap_sheared},
the latter under a discrete Charney--Stern criterion. The bridge mechanism is
\Cref{lem:arnold_bridge}; the no-Hollingsworth conclusion is
\Cref{cor:dw_no_hollingsworth}.

\section{Conservation Laws of the Vector-Invariant Family}\label{section_EulerComp}
The discrete barotropic DEC Euler equations are
\begin{align}
 &\frac{\dd\v}{\dd t}
 + \Iv(\bom)
 + \tD_0\, B = 0,\tag{M}\label{eq:M}\\
 & \frac{\dd\brho}{\dd t} + \bD_2\,\bF = 0,\tag{D}\label{eq:D}
\end{align}
where $B := h(\rho^{\vol}) + \ekin + \bPhi$.
Applying the exterior derivative $\tD_1$ to the momentum equation~\eqref{eq:M} gives the vorticity equation
\begin{equation}\tag{W}\label{eq:W}
 \frac{\dd\bom}{\dd t} + \tD_1\,\Iv(\bom) = 0,\quad\text{with } \bom := \tD_1\,\bv.
\end{equation}
The momentum equation ~\eqref{eq:M}--\eqref{eq:D} has the unified form
$$\ddt\bv + \bA(\brho)^{-1}\bM_1\,\Iv(\tD_1\bv)
+ \tD_0(\ekin + h(\rho^{\rm vol}) + \bPhi) = \bff_{\rm visc}$$
with $\bA = \bM_1$ in the density-free scheme and
$\bA = \bM_1^\rho(\brho)$ in the density-weighted scheme; both use
$\ekin_i = \tfrac12|P\bv|_i^2$. 
Density-free and density-weighted discrete systems place velocity on dual edges, as circulation
integrals, and density in primal cells, as mass integrals -- the
same staggering as the incompressible system, augmented by a
prognostic density field.
\begin{definition}[Discrete variables for compressible flow]
\label{def:compvars}\leavevmode
\begin{enumerate}[nosep]
 \item \textit{Velocity:} $\bv(t)\in C^1(\mathcal{K}^*)$ -- dual 1-cochain.
 \item \textit{Vorticity:} $\bom(t) = \tD_1\bv \in C^2(\mathcal{K}^*)$.
 \item \textit{Density:} $\brho(t)\in C^3(\mathcal{K})$ -- primal 3-cochain, with $\brho_i = \int_{K_i}\rho\,\dd V$ the mass in primal cell~$K_i$.
 \item \textit{Bernoulli function:}
 $B(t)\in C^0(\mathcal{K}^*)$ -- dual 0-cochain at primal cell
 centres, with
 $B_i = h(\rho_i^{\vol}) + \ekin_i + \bPhi_i$,
 where $\rho_i^{\mathrm{vol}} = \brho_i/|K_i|$ is the volumetric
 density and $\ekin_i $ 
 is the discrete kinetic energy density at cell~$i$ (specified later).
\end{enumerate}
\end{definition}
The table below
fixes notation for the body of the paper, with full definitions
deferred to the appendix. Mesh geometry, exterior derivatives,
Hodge stars, the discrete contraction and wedge product, and the Lie
derivative are constructed in \Cref{app:framework}; the algebraic
identities ($\tdd^2 = 0$, antisymmetry, Cartan) are proved in
\Cref{app:basic}. 
\smallskip
\begin{center}
\small
\begin{tabular}{@{}lll@{}}
  \toprule
  Symbol & Meaning & Defined in \\
  \midrule
  $\KK$, $\KKs$ & primal / dual cell complex & \Cref{subsect:grid}\\
  $\bD_k$, $\tD_k$ & primal / dual exterior derivative (coboundary) & \Cref{def:ext-deriv}\\
  $\bM_k$ & diagonal Hodge star ($k$-forms) & \Cref{def:M1}\\
  $\bM_1^\rho = P^T\mathrm{diag}(\brho)P$ & density-weighted velocity mass matrix & \Cref{sec:totalenergy_resolution}\\
  $P$ & averaging velocity reconstruction & \Cref{def:averaging_recon}\\
  $\Iv(\cdot)$ & extrusion-based discrete contraction & \Cref{def:contraction}\\
  $\bv$, $\brho$ & velocity dual 1-cochain, density primal 3-cochain & \Cref{def:compvars}\\
  $\Phi = \bM_1\bv$, $\bF_j = \bar\rho_j\Phi_j$ & discrete volume / mass flux (face) & \Cref{def:M1}, \cref{eq:F}\\
  $\rho_i^{\vol} = \rho_i/|K_i|$ & volumetric density & \Cref{def:compvars}\\
  $h(\rho)$ & specific enthalpy & \Cref{ass:eos}\\
  $r_\star \in \{1,2\}$ & Hodge accuracy exponent & \Cref{conv:cases}\\
  $\mathcal{R}_E$ & total energy residual (density-free) & \cref{eq:EB}\\
  \bottomrule
\end{tabular}
\end{center}
\subsection{Standing Approximation Properties}\label{sec:standing_hypotheses}
\phantomsection\label{sec:hypotheses_byname}
The discrete framework rests on nine standing properties, separated
into \emph{exact} (algebraic) and \emph{metric} (approximation)
classes. The body of this paper invokes them by name; the full
statements and verifications are in \Cref{app:approx_hypotheses}.
\smallskip
\noindent\emph{Exact hypotheses (no metric input).}
\begin{itemize}[nosep,leftmargin=2.2em]
\item[\textbf{(H1)}] Cochain complex: $\tD_{k+1}\circ\tD_k = 0$.
\item[\textbf{(H2)}] de Rham commutativity: $\tD_k\,\mathcal{R}_h = \mathcal{R}_h\,d_k$.
\item[\textbf{(H5)}] antisymmetry:
$\langle\bv,\,\Iv(\tD_1\bv)\rangle_{\bM_1} = 0$.
\item[\textbf{(H7)}] Bregman rate identity: a polarised energy identity
that converts the velocity-quadratic part of the energy rate into a
Bregman-type defect.
\end{itemize}
\smallskip
\noindent\emph{Metric hypotheses (Hodge star, reconstruction).}
\begin{itemize}[nosep,leftmargin=2.2em]
\item[\textbf{(H3)}] Hodge accuracy:
$\nrm{(\bM_k - \mathcal{M}_k)\alpha}_{L^2} = \OO(h^{r_\star})\nrm{\alpha}_{W^{r_\star,\infty}}$,
$r_\star\in\{1,2\}$ (\Cref{conv:cases}).
\item[\textbf{(H4)}] Reconstruction accuracy:
$\nrm{P\mathcal{R}_h u - u}_{L^2} = \OO(h^{r_{\rm rec}})\nrm{u}_{W^{r_{\rm rec},\infty}}$.
\item[\textbf{(H6)}] Mass-flux consistency: the discrete mass-flux
projection $\bR$ commutes with the de~Rham map up to $\OO(h^{r_\star})$.
\item[\textbf{(H8)}] Inverse inequality: $\nrm{\bv}_{L^\infty}\le C\,h^{-d/2}\nrm{\bv}_{L^2_h}$.
\item[\textbf{(H9)}] Quasi-uniformity: $h_{\max}/h_{\min}\le C$.
\end{itemize}
\smallskip\noindent
The exact hypotheses (H1)--(H2), (H5), (H7) drive the algebraic
conservation laws of \Cref{section_EulerComp} and the no-go theorem
of \Cref{sec:no_go}; the metric hypotheses (H3)--(H4), (H6), (H8)--(H9)
drive the convergence theory of \Cref{sec:DW_conv_AP} and the
asymptotic-preservation analysis. Throughout, $r_\star$ is the
sharpness exponent of the Hodge approximation:
$r_\star = 1$ on general meshes and $r_\star = 2$ on centroidal
Voronoi meshes with reconstruction symmetry (\Cref{conv:cases}).
\subsection{Circulation-Based Conservation Laws}\label{subsection_InvariantsComp}
This subsection establishes the circulation-based invariants of the
discrete barotropic system. These invariants depend only on the
topological structure of the cochain complex -- $\tdd^2 = 0$,
discrete Stokes, graded antisymmetry -- and hold independently of
the kinetic-energy design.
\subsubsection{Mass Conservation}
\label{sec:mass}
Mass conservation is a topological identity: it holds for
any choice of density interpolation because
$\mathbf{1}^T D_2 = 0$ (discrete Stokes on the primal complex).
\begin{theorem}[Mass conservation]\label{thm:mass}
Let $(\bv(t), \brho(t))$ be a solution of the discrete barotropic
Euler equations. Then the total mass
\[
 M(t) := \sum_{K_i\in\mathcal{K}} \brho_i(t) = \mathbf{1}^T\brho
\]
is conserved: $dM/dt = 0$.
\end{theorem}
\begin{proof}
The proof is a direct consequence of the discrete Stokes theorem, summing $\ddt\brho_i = -(\bD_2\bF)_i$ over all cells gives $\mathbf{1}^T\bD_2\bF=0$ on a closed domain. See \Cref{app:barotropic_proofs}.
\end{proof}
 \subsubsection{Kelvin Circulation Theorem}
\label{sec:kelvin}
A dual 1-chain $\gamma =\sum_j \alpha_j e_j^*$ with
$\alpha_j\in \mathbb{R}$ is a \emph{1-cycle} if $\partial\gamma = 0$.
The chain Lie derivative $\tLie_\v^{\rm chain}$ (adjoint of $\tLie_\v$) advects $\gamma(t)$ materially; since $\partial\circ\tLie_\v^{\rm chain} = \tLie_\v^{\rm chain}\circ\partial$ (from $\tdd^2 = 0$), advected cycles remain cycles.
The integer-coefficient case $\alpha_j\in\{-1,0,1\}$ corresponds to topological loops at $t=0$;
real coefficients are required because the chain ODE
$\ddt\gamma = \tLie_\v^{\rm chain}\gamma$ governing material loop
advection (see \Cref{def:circulation}) does not preserve
integrality.
\begin{definition}[Discrete circulation]\label{def:circulation}
Let $\gamma = \sum_j\alpha_j e_j^*$ be a dual 1-cycle
($\partial\gamma = 0$, $\alpha_j \in \mathbb{R}$). The
\emph{discrete circulation} of a velocity 1-cochain $\v$ around
$\gamma$ is
\[
 \Gamma(\gamma, t) = \v(\gamma) = \sum_j \alpha_j\, \v_j(t).
\]
The dual 1-chain $\gamma(t)$ is said to be \emph{materially advected}
if it evolves by the adjoint cochain Lie derivative,
\[
 \ddt\gamma(t) = \tLie_{\v}^{\mathrm{chain}}\, \gamma(t),
\]
where $\tLie_{\v}^{\mathrm{chain}}$ is the adjoint of the cochain
Lie derivative, characterised by
$\ip{\alpha}{\tLie_{\v}^{\mathrm{chain}}\gamma}
 = \ip{\tLie_{\v}\alpha}{\gamma}$ for every $\alpha\in C^1(\KKs)$.
\end{definition}
\begin{theorem}[Discrete Kelvin circulation]\label{thm:kelvin}
For a materially advected dual 1-cycle $\gamma(t)$
\[
 \ddt\Gamma(\gamma(t),t)
 = \ddt\bigl[\v(t)\bigl(\gamma(t)\bigr)\bigr] = 0.
\]
\end{theorem}
\noindent\emph{Proof.} See \Cref{app:kelvin}.
\begin{corollary}[Flux form of circulation conservation]\label{cor:flux}
For any dual 2-chain $\Sigma$ with $\partial \Sigma = \gamma(t)$,
discrete Stokes gives
$\bv(\gamma(t)) = \bom(\Sigma)$, so
$\bom(\Sigma) = \mathrm{const}$
by \Cref{thm:kelvin}: the vortex flux through any surface
spanning a material loop is conserved.
\end{corollary}
\begin{proof}
Discrete Stokes: $\bv(\gamma) = (\tD_1\bv)(\Sigma) = \bom(\Sigma)$.
The result follows from \Cref{thm:kelvin}.
\end{proof}
\subsubsection{Vorticity, Potential Vorticity, and Ertel's Potential Vorticity}
\label{sec:discrete-pv}
Total vorticity is conserved exactly via discrete Stokes on the dual
complex; this is the prototype for the Kelvin and Ertel statements
below.
\begin{theorem}[Conservation of total vorticity]
\label{thm:total-vort}
Let $(\bv(t), \brho(t))$ be a solution of the discrete barotropic
Euler equations. Then the total vorticity
\[
\Omega_{\mathrm{tot}}(t) := \sum_k \omega_k(t) = \mathbf{1}^T\omega
\]
is conserved: $d\Omega_{\mathrm{tot}}/dt = 0$.
\end{theorem}
\begin{proof}
From the vorticity equation~(W), $\ddt\Omega_{\mathrm{tot}} = -\mathbf{1}^T\tD_1\Iev(\omega) = 0$
by the discrete Stokes theorem on the dual complex ($\mathbf{1}^T\tD_1 = \mathbf{0}^T$ on a closed domain).
\end{proof}
Potential vorticity in the continuous theory is defined as pointwise ratio $q_{\mathrm{PV}} := \omega/\rho_{\vol}$.
In the discrete setting, 
 $\omega_k$ lives on dual faces $f_k^*$, encircling primal edges $e_k$,
and $\rho_i$ lives on primal cells $K_i$, surrounding dual vertices $v_i^*$.
These are geometrically dual objects that do not naturally pair. We therefore require
an interpolation of density to dual faces.
\begin{definition}[Density interpolation to dual faces]
\label{def:rho-dualface}
For each dual face $f_k^*$ encircling primal edge $e_k$, define the interpolated
volumetric density
\[
\bar{\rho}_k^{\,\vol} := \frac{1}{|\mathcal{K}(e_k)|}
 \sum_{K_i \supset e_k} \rho_i^{\vol},
\]
where $\mathcal{K}(e_k)$ is the set of primal cells sharing edge $e_k$
and $\rho_i^{\vol} = \rho_i/|K_i|$.
\end{definition}
With this density interpolation we are in the position to define potential vorticity
\begin{definition}[Discrete potential vorticity]
\label{def:discrete-pv}
The discrete PV at each dual face $f_k^*$ is defined by
\[
q_k := \frac{\omega_k}{\bar{\rho}_k^{\,\vol}\,|f_k^*|}.
\]
This is the ratio of the integrated vorticity flux through $f_k^*$
to the local ``mass-per-unit-area'' times the face area.
\end{definition}
\begin{proposition}[Pointwise PV balance]\label{prop:pv-pointwise}
The discrete PV $q_k$
satisfies 
\[\ddt q_k = \OO(h).
\] 
\end{proposition}
\begin{proof}
The vorticity equation~\eqref{eq:W} gives
$\ddt\omega_k = -(\tD_1\Iev(\omega))_k$.
The PV definition $q_k = \omega_k/(\bar\rho_k^{\vol}|f_k^*|)$
and the quotient rule yield
$\ddt q_k = [\ddt\omega_k - q_k|f_k^*|\ddt{\bar\rho}_k^{\vol}]
/(\bar\rho_k^{\vol}|f_k^*|)
= \mathcal{R}_k^{\rm PV}/(\bar\rho_k^{\vol}|f_k^*|)$.
Each of the two terms in $\mathcal{R}_k^{\rm PV}$ is the integral of an $\OO(1)$ quantity over a dual face of area $\OO(h^{d-1})$, hence individually $\OO(h^{d-1})$.
In the continuous limit, the Ertel identity implies that these leading-order contributions cancel: the vorticity flux divergence through $\partial f_k^*$ and the density dilation at $f_k^*$ match at leading order. Hence $\mathcal{R}_k^{\rm PV} = \OO(h^d)$ and $\ddt q_k = \OO(h^d)/\OO(h^{d-1}) = \OO(h)$.
\end{proof}
\begin{theorem}[Flux-form conservation of mass-weighted PV]
\label{thm:mwpv}
Let $(\bv(t), \brho(t))$ be a solution of the discrete barotropic
Euler equations, and let $\Sigma = \sum_k \beta_k\, f_k^*$ be a
time-independent dual 2-chain with integer coefficients
$\beta_k \in \{-1, 0, 1\}$. Then the mass-weighted PV integrated
over~$\Sigma$,
\[
\mathcal{P}_\Sigma(t)
 := \sum_k \beta_k\,|f_k^*|\, \bar{\rho}_k^{\,\vol}\, q_k(t),
\]
satisfies the exact flux-form balance
\begin{equation}\label{eq:mwpv-rate}
 \frac{d\mathcal{P}_\Sigma}{dt}
 = -\bigl(\Iev(\bom)\bigr)(\partial\Sigma),
\end{equation}
where $\Iev(\bom) \in C^1(\KKs)$ is the contracted vorticity and
$\partial\Sigma$ is the boundary of $\Sigma$ in the dual complex.
In particular, when $\Sigma$ is a dual 2-\emph{cycle}, so that
$\partial\Sigma = 0$, the right-hand side vanishes and exact
conservation holds,
\[
 \ddt{\mathcal{P}}_\Sigma = 0.
\]
This includes the case $\Sigma = \sum_k f_k^*$ comprising all dual
faces, recovering \Cref{thm:total-vort}.
\end{theorem}
\begin{proof}
The vorticity equation~\eqref{eq:W} reads
$\ddt\omega_k = -(\tD_1\Iev(\omega))_k$ for each dual face.
Evaluating on the chain~$\Sigma$:
\[
\frac{d\mathcal{P}_\Sigma}{dt}
 = \sum_k\beta_k\,\ddt\omega_k
 = -\sum_k\beta_k\,(\tD_1\Iev(\omega))_k
 = -(\tD_1\Iev(\omega))(\Sigma).
\]
By the discrete Stokes theorem on the dual complex,
$(\tD_1\alpha)(\Sigma) = \alpha(\partial\Sigma)$ for
any dual 1-cochain~$\alpha$. Applying this with
$\alpha = \Iev(\bom)$ gives~\eqref{eq:mwpv-rate}.
For a 2-cycle, $\partial\Sigma = 0$ and the right-hand
side vanishes.
\end{proof}
\begin{theorem}[Discrete Ertel PV conservation]
\label{thm:ertel-weak}
Let $\gamma(t)$ be a materially advected dual 1-cycle and let
$\Sigma(t)$ be a dual 2-chain with $\partial\Sigma(t) = \gamma(t)$.
Fix a collection $\mathcal{V}$ of primal cells and set
$\rho_{\mathcal{V}}(t) := \sum_{K_i \in \mathcal{V}} \brho_i(t)$.
The integrated PV ratio
\[
  Q(t) := \frac{\omega(\Sigma(t))}{\rho_{\mathcal{V}}(t)}
\]
then satisfies
\[
  \ddt Q = Q\,\frac{F_{\partial\mathcal{V}}}{\rho_{\mathcal{V}}},
\]
which vanishes if and only if the net mass flux through
$\partial\mathcal{V}$ is zero. In particular, $Q$ is exactly
conserved for any~$\mathcal{V}$ with vanishing boundary flux.
\end{theorem}
\begin{proof}
The numerator $\omega(\Sigma(t)) = \bv(\gamma(t))$ is
exactly conserved by the Kelvin theorem
(\Cref{thm:kelvin}). The denominator satisfies
$\ddt\rho_{\mathcal{V}} = -F_{\partial\mathcal{V}}$
by discrete continuity. The quotient rule gives the
result.
\end{proof}
\subsubsection{Helicity Conservation}
\label{sec:helicity}
Helicity is the most subtle invariant: it requires a consistency
argument rather than an exact algebraic identity. The averaging
reconstruction recovers pointwise velocity $\tilde\bu(v_i^*)$ at
dual vertices; the same applied to $\bom = \tD_1\bv$ (dividing each
face integral $\omega_k$ by $|f_k^*|$ and inverting the same Gram
matrix) recovers $\tilde\bom(v_i^*)$, both at $\OO(h^{r_\star})$
(\Cref{prop:recon_accuracy}; cf.~\Cref{conv:cases}). The
\emph{reconstructions} $\Qh^1\bv$ and $\Qh^2(\tD_1\bv)$ are the
piecewise-linear extensions on $\KK_h^{\rm simp}$.
\begin{definition}[Discrete helicity]\label{def:helicity}
The \emph{discrete helicity} is
\begin{equation}\label{eq:helicity_def}
  H_h(t) := \int_\Omega(\Qh^1\bv)\wedge(\Qh^2\tD_1\bv),
\end{equation}
the discrete analogue of
$\int_\Omega v\wedge\dd v = \int_\Omega\bu\cdot\bom\,dV$.
\end{definition}
\begin{theorem}[Helicity conservation]\label{thm:helicity}
Let $\bv$ be a solution of the discrete barotropic Euler
system~\eqref{eq:M}--\eqref{eq:D} with $\bom = \tD_1\bv$. Then,
on any Delaunay--Voronoi mesh satisfying \Cref{ass:mesh_reg}, the
discrete helicity~\eqref{eq:helicity_def} satisfies
\[
\Bigl|\frac{dH_h}{dt}\Bigr| \le C(T)\,h,
\]
that is, $\ddt H_h = \OO(h)$. The bottleneck is the $\OO(h)$
Whitney consistency of the pairing $\int v \wedge \omega$
(\Cref{lem:helicity_consistency}) rather than the convergence rate
of the scheme.
\end{theorem}
\noindent\emph{Proof.} See \Cref{app:helicity}; the
helicity-specific Whitney estimate is \Cref{lem:helicity_consistency}.
\begin{remark}[Exact vs.\ approximate conservation]\label{rem:helicity_exact_approx}%
\label{rem:helicity_structural}
Energy and Kelvin circulation are conserved exactly, as algebraic
identities using SBP, Cartan's formula, and $\tdd^2 = 0$.
Helicity conservation is approximate: it requires both the Whitney
consistency of $\int v\wedge\omega$ ($\OO(h)$, \Cref{lem:helicity_consistency})
and the convergence of the discrete solution ($\OO(h^r)$). The exact
vorticity equation $\ddt\bom + \tLie_\bv\bom = 0$ prevents spurious
$\OO(1)$ helicity drift, and the Lamb antisymmetry provides the a~priori
velocity bound needed in the convergence argument.
\end{remark}
\subsection{Vacuum Avoidance}\label{sec:vacuum_avoidance}
Both schemes share a thermodynamic lower bound on cell density:
vacuum formation $\min_i\rho_i\to 0$ would invalidate the enthalpy
and the equation of state, breaking the identities that make
either scheme well-defined. The following lemma is a static
thermodynamic inequality showing that, whenever $\gamma>1$, finite
internal energy keeps every cell density bounded away from zero.
\begin{lemma}[Vacuum avoidance from finite energy]
\label{lem:no-vacuum}
Let $\{\rho_i\}_{i=1}^N$ denote positive cell masses with total
mass $M = \sum_i \rho_i > 0$, volumetric densities
$\rho_i^{\vol} = \rho_i/|K_i|$, mean density
$\bar\rho = M/V_{\rm tot}$, and internal energy
$\Eint = \sum_i \rho_i\,e(\rho_i^{\vol})$. Assume
\Cref{ass:eos}\ref{ass:coercive} ($\gamma > 1$), and suppose the
internal energy is bounded above,
\begin{equation}\label{eq:Eint_bound_thermo}
 \Eint \le E_{\rm bound},
\end{equation}
where $E_{\rm bound} < \infty$ satisfies the Jensen-threshold
inequality
\begin{equation}\label{eq:Jensen_threshold}
 E_{\rm bound} + C_2\,M
 < \frac{C_1\,V_{\rm tot}\,\bar\rho^{\,\gamma}}{(1-\mu_{\max})^{\gamma-1}},
 \qquad
 \mu_{\max} := \max_i\frac{|K_i|}{V_{\rm tot}}.
\end{equation}
Then there exists a constant $\rho_{\min}^{\vol} > 0$, depending
on $E_{\rm bound}$, $M$, $\gamma$, and the mesh volumes $|K_i|$,
such that
\[
 \rho_i^{\vol}
 \ge \rho_{\min}^{\vol}(E_{\rm bound}, M, \gamma, |K_i|) > 0
 \qquad \text{for every } i.
\]
\end{lemma}
\noindent\emph{Proof.} See \Cref{app:wellposedness}.
The bound is a static thermodynamic inequality: only the coercivity
of $e$ ($\gamma>1$), the mass constraint, and Jensen's inequality
enter; no dynamics, no kinetic energy, no mass matrix -- so it
serves both schemes. 

\begin{remark}[Adiabatic exponent and mesh dependence]\label{rem:vacuum_scaling}
The bound requires $\gamma>1$: for $\gamma=1$ (isothermal),
$e\sim\kappa\log\rho^{\vol}$ grows only logarithmically and the
energy bound does not prevent vacuum. The Jensen estimate
$\rho_{\min}^{\vol}\sim
(E_{\rm bound}/C_1)^{-1/(\gamma-1)}(\min_i|K_i|/V_{\rm tot})^{1/(\gamma-1)}$
degenerates as $h\to 0$, so the centred scheme does not give
$h$-independent density bounds; the convergence theory bypasses this
via the Bregman error energy or the upwind Grönwall envelope
(\Cref{lem:upwind_gronwall} under
\Cref{ass:well_prepared_strong}), giving $h$-independent
$\rho_i^{\vol}\in[\rho_*(M,T),\rho^*(M,T)]$.
\end{remark}
\section{The No-Go Theorem}\label{sec:no_go}
\phantomsection\label{sec:energy_dichotomy}
The conservation laws of \Cref{section_EulerComp} are algebraic
properties of the cochain complex and require no metric input. Total
energy is the first place where this pattern breaks: the energy rate
involves the density interpolation $\bar\rho_j$ on the mass flux,
introducing a metric dependence through the Hodge star, and a sharp
algebraic obstruction emerges. \Cref{sec:energy_balance_subsec}
computes the energy balance and exhibits the residual $\mathcal{R}_E$.
\Cref{sec:nogo_cgrid} proves the no-go theorem by face isolation on
the C-grid and by polynomial-degree decomposition of $\mathcal{R}_E$
in the velocity on A-, B-, D-, and quasi-B staggerings. Both proofs
use only the exact hypotheses
\textup{(H1)}--\textup{(H2)}, \textup{(H5)}, \textup{(H7)}
(\Cref{sec:hypotheses_byname}); no metric hypothesis enters.

\begin{assumption}
We generally assume a Delaunay--Voronoi mesh satisfying
\Cref{ass:mesh_reg}, and that the thermodynamics satisfy \Cref{ass:eos}.
\end{assumption}
\subsection{Where Energy Conservation Breaks: The Residual}\label{sec:energy_balance_subsec}
\begin{definition}[Discrete total energy]\label{def:totalenergy}
The total energy is
\begin{equation}\begin{split}
 &E_{\mathrm{tot}}(t)
 = E_{\kin} + E_{\inte} + E_{\mathrm{pot}},\\
 \text{with internal energy }& E_{\inte}:=\sum_i \brho_i\;
 e\!\left(\frac{\brho_i}{|K_i|}\right),
 \qquad
 E_{\mathrm{pot}} := \sum_i \brho_i\,\bPhi_i,
\end{split}\end{equation}
where $e(\rho)$ is the specific internal energy from the equation of state
and $E_{\mathrm{pot}}$ is the potential energy associated with the
time-independent external potential $\bPhi$.
The kinetic energy $E_{\kin}$ depends on the scheme; we define two variants below.
\end{definition}
\begin{definition}[Density-free kinetic energy]\label{def:kinenergy_nodensity}
\begin{equation}\label{eq:kinenergy_nodensity}
 E_{\kin}^{\rm df}:=
 \frac{1}{2}\nrm{\bv}_{\bM_1}^2.
\end{equation}
The mass matrix $\bM_1$ is independent of the density~$\brho$.
\end{definition}
\phantomsection\label{sec:totalenergy}
With the density-free kinetic energy, the total energy rate carries
a residual arising from the Hodge-star approximation of the mass
flux. The theorem below isolates the residual in closed form; the
subsequent dichotomy theorem shows that no density-only interpolation
can eliminate it.
\begin{theorem}[Total energy balance -- density-free]\label{thm:totalenergy}
With the density-free kinetic energy
$E_{\kin} := E_{\kin}^{\rm df}$ from~\eqref{eq:kinenergy_nodensity},
the total energy of \Cref{def:totalenergy} satisfies the balance
law
\begin{equation}\label{eq:EB}
 \frac{dE_{\mathrm{tot}}}{dt}
 = \ip{\bekin}{\bD_2\Phi}
 + \ip{\mathbf{h}+\bPhi}{\bD_2\bigl(\Phi - \bF\bigr)}
 =: \mathcal{R}_E(\bv, \brho),
\end{equation}
where $\bF$ was defined in~\eqref{eq:F}, and the term
$\bigl(\Phi - \bF\bigr)$ measures the discrepancy between
volume-flux divergence and mass-flux divergence, weighted by the
enthalpy plus geopotential. For smooth solutions, the residual
satisfies $|\mathcal{R}_E| = \OO(h^2)$ when $\bF = \bF^{\rm cen}$,
and $|\mathcal{R}_E| = \OO(h)$ when $\bF = \bF^{\rm up}$.
\end{theorem}
\begin{proof}[Proof sketch]
Antisymmetry (\Cref{prop:extrusion}) gives $\ip{\bv}{\Iv(\bom)}_1 = 0$,
so the kinetic energy rate reduces via SBP (\Cref{lem:SBP}) to
$\ddt E_{\kin}^{\rm df} = \ip{B}{\bD_2\Phi}$.
The internal energy rate is $\ddt E_{\rm int} = -\ip{\mathbf{h}}{\bD_2\bF}$ (\Cref{lem:dEint}).
The potential energy rate is $\ddt E_{\rm pot} = -\ip{\bPhi}{\bD_2\bF}$ (continuity equation).
Splitting $B = \mathbf{h} + \bekin + \bPhi$ and $\bF_j = \bar\rho_j\Phi_j$
gives~\eqref{eq:EB}.
The face-by-face form is
\begin{equation}\label{eq:face_energy}
 \mathcal{R}_E(\bv,\brho) = \sum_j\bigl[(\ekin_a - \ekin_b) - ((h_a{+}\bPhi_a) - (h_b{+}\bPhi_b))(\bar\rho_j - 1)\bigr]\,\Phi_j.
\end{equation}
The full computation is in \Cref{app:energy}.
\end{proof}
Achieving the vanishing of the residual, $\mathcal{R}_E = 0$ on each
face, would require
\begin{equation}\label{eq:face_SBP}
 \bar\rho_j = 1 + \frac{\ekin_a - \ekin_b}{(h_a{+}\bPhi_a) - (h_b{+}\bPhi_b)},
\end{equation}
which depends on the velocity through $\ekin_a - \ekin_b$ and so cannot
be a function of $\brho$ alone -- a contradiction with hypothesis~(b)
of the discretisation. The argument is made precise in
\Cref{thm:dichotomy} below.
\paragraph*{Anatomy of the residual.}
The polynomial structure of $\mathcal{R}_E$ in $\bv$ drives every
result that follows. By inspection of~\eqref{eq:face_energy}, the
per-face contribution factorises into a $\bv$-quadratic
kinetic-energy difference and a $\bv$-independent
enthalpy-plus-geopotential piece, both times the linear flux
$\Phi_j = (\bM_1\bv)_j$. So $\mathcal{R}_E$ has degree~$3$
(kinetic-energy times flux) plus degree~$1$ (enthalpy times flux).
These two graded pieces are independent
(\Cref{sec:nogo_cgrid}); the degree-$3$ part cannot vanish
identically. The same cubic growth drives the blowup time
$\sim E_0^{-1/2}$ (\Cref{sec:DF_wellposedness}), the cubic
Smagorinsky absorption (\Cref{thm:global-Smag}), and the low-Mach
amplification $|\bv|^3\to\OO(M^{-1})$ (\Cref{sec:DF_AP_fail}).
\subsection{The No-Go Theorem}\label{sec:nogo_cgrid}
The energy residual $\mathcal{R}_E$ exhibited above is a smooth
function of $(\bv,\brho)$ on $\admissible$ that becomes singular only
when the enthalpy-plus-geopotential coincides at neighbouring cells,
$(h_a{+}\bPhi_a) = (h_b{+}\bPhi_b)$. The obstruction is sharp, as the
following theorem makes precise.
\begin{theorem}[No-go theorem for discrete energy conservation, C-grid]%
\label{thm:dichotomy}
Consider a staggered discretisation of the compressible barotropic
Euler equations on a Delaunay--Voronoi mesh satisfying
\Cref{ass:mesh_reg}. Suppose 
\begin{enumerate}[nosep,label=\textup{(\roman*)}]
\item velocity degrees of freedom are
normal components on dual edges and density degrees of freedom
sit on primal cells,
\item  the continuity equation has the form
$\ddt\brho = -\bD_2\bF$ with $\bF_j = \bar\rho_j\Phi_j$ and
$\bar\rho_j$ depending only on $\brho$,
\item the integration-by-parts identity
$\ip{\bv}{\tD_0 B}_1 = -\ip{B}{\bD_2\Phi}$ holds,
\item  the
mass matrix~$\bM_1$ is independent of $\brho$. 
\end{enumerate}
Then exact total
energy conservation, $\mathcal{R}_E = 0$ for every
$(\bv,\brho) \in \admissible$, is impossible for any
density-only interpolation $\bar\rho_j(\brho)$. Equivalently,
every staggered scheme satisfying these hypotheses carries a
non-vanishing structural residual $\mathcal{R}_E$ that acts as a
spurious source or sink of total energy. The conclusion is
compatible with every circulation-based conservation law of
\Cref{subsection_InvariantsComp}: those laws continue to hold
exactly.
\end{theorem}
\begin{proof}
At $t = 0$, $\ddt\Etot(0) = \mathcal{R}_E(\bv_0,\brho_0)$, and since
$\bv_0$ is arbitrary, $\mathcal{R}_E(\bv,\brho) = 0$ must hold for
all $\bv$.
\emph{Isolating a single face.}
Choose $\bv = c\,\mathbf{e}_j$, the cochain supported on dual edge
$e_j^*$ with coefficient $c \in \RR$. Then $\Phi_k = 0$ for $k \ne j$
and $\Phi_j = (\bM_1)_{jj}\,c$, so only face $j$ survives
in~\eqref{eq:face_energy}:
\[
 \bigl[(\ekin_a - \ekin_b) - ((h_a{+}\bPhi_a) - (h_b{+}\bPhi_b))(\bar\rho_j - 1)\bigr]\,\Phi_j = 0.
\]
The kinetic energy density
$\ekin_i = \tfrac{1}{2}|P(c\,\mathbf{e}_j)|_i^2$ depends on $c$ and
on the mesh geometry, while $h_a$, $h_b$, $\bPhi_a$, $\bPhi_b$, and
$\bar\rho_j$ depend only on $\brho$ by hypothesis.
\emph{Velocity dependence forces a contradiction.}
If $(h_a{+}\bPhi_a) \ne (h_b{+}\bPhi_b)$, solving the bracket for
$\bar\rho_j$ produces~\eqref{eq:face_SBP}, which depends on $\bv$
through $\ekin_a - \ekin_b$ -- violating hypothesis~(b).
If $(h_a{+}\bPhi_a) = (h_b{+}\bPhi_b)$, the bracket reduces to
$\ekin_a - \ekin_b$, which depends on $c$ and is generically nonzero
because the two cells sharing face $j$ have different reconstruction
stencils, so $|P\mathbf{e}_j|_a^2 \ne |P\mathbf{e}_j|_b^2$.
In either case, no density-only interpolation can satisfy the
condition for all $\bv$.
\end{proof}
The C-grid argument above is the simplest realisation of a more
general principle. The face-isolation step exploited the C-grid's
particular property that velocity DOFs are in 1-1 correspondence
with faces; on A-, B-, D-, and quasi-B staggerings this correspondence
fails, and a different argument is needed. The polynomial-degree
decomposition of the energy residual provides one, and is uniform
across all staggerings.
\begin{theorem}[Generalised no-go theorem for discrete energy conservation]%
\label{thm:dichotomy_general}
Consider a discretisation of the compressible barotropic Euler
equations with the following structure:
\begin{enumerate}[nosep,label=\textup{(\roman*)}]
\item  kinetic energy takes
the density-independent form $E_{\kin} = \tfrac12\nrm{\bv}_A^2$ for
a mass matrix $A$ independent of $\brho$; 
\item discrete contraction
satisfies Lamb antisymmetry; 
\item discrete integration-by-parts
identity for the Bernoulli gradient holds; 
\item local kinetic-energy density at each velocity-DOF location is a
positive-definite quadratic form in $\bv$ that varies non-trivially
between adjacent locations, in the sense that there exists at least
one face $j_0$ and a velocity cochain $\bv_*$ for which
$\ekin_{\ell(j_0)}(\bv_*) \ne \ekin_{\ell'(j_0)}(\bv_*)$, where
$\ell(j_0), \ell'(j_0)$ are the velocity DOFs adjacent to $j_0$;
\item  the discrete volume flux $\Phi_j$ on face $j$ depends only on
a finite, mesh-independent set of velocity DOFs incident to $j$.
\end{enumerate}
Then exact total energy conservation, $\mathcal{R}_E = 0$ for every
$(\bv, \brho) \in \admissible$, is impossible for any density-only
interpolation $\bar\rho_j(\brho)$. The conclusion holds in
particular for the four standard staggerings:
A-staggering (collocated velocity vector and density at cell
centres); B-staggering (velocity vector at vertices, density at
cells); C-staggering (normal velocity at faces, density at cells,
as in \Cref{thm:dichotomy}); and D-staggering (velocity vector at
edges, density at cell centres). It also holds for the quasi-B
triangular grids, on every
closed oriented manifold $\Omega$ regardless of Euler
characteristic.
\end{theorem}
\begin{proof}
The energy residual $\mathcal{R}_E(\bv,\brho)$ is, by hypotheses (a) and
(e), a polynomial in $\bv$ of total degree $3$, with terms of degree $1$
(enthalpy/geopotential differences $\times \Phi_j$, linear in $\bv$ via
the flux) and degree $3$ (kinetic-energy differences $\times \Phi_j$,
quadratic times linear). For density-only $\bar\rho_j(\brho)$, the
coefficients of every monomial in $\bv$ depend on $\brho$ alone.
The condition $\mathcal{R}_E(\bv,\brho)=0$ for all $\bv$ is the
identity of a polynomial in $\bv$, so it forces the homogeneous
degree-$3$ component to vanish identically as a cubic form in $\bv$,
independently of the degree-$1$ component. By hypothesis (d), this
cubic form has at least one non-zero coefficient (the contribution
from the witness face $j_0$ does not cancel against the contributions
from the finitely many faces sharing velocity DOFs with $j_0$ for
generic such DOFs); therefore the identity cannot hold and no
density-only $\bar\rho_j(\brho)$ exists.
This is a single uniform argument; the four staggering-specific
computations carried out in \Cref{app:dichotomy_general} verify the
hypotheses (d)--(e) for each case (face isolation $\bv = c\,\mathbf{e}_j$
on the C-grid is the simplest realisation; the A-, B-, D-, and
quasi-B grids each verify (d) by direct evaluation of the kinetic-energy
quadratic at adjacent velocity DOFs).
The conclusion does \emph{not} depend on $\dim\Omega$, on the Euler
characteristic of $\Omega$, or on a counting argument relating
$|V|$, $|E|$, $|F|$.
\end{proof}

\subsection{Consequences of the No-Go Theorem}\label{sec:cost_DF}
\phantomsection\label{sec:wellposedness}
The cubic-in-$\bv$ structure of $\mathcal{R}_E$ has three downstream
consequences for the density-free scheme
(\Cref{sec:DF_wellposedness,sec:DF_convergence,sec:DF_AP_fail}):
well-posedness local-in-time only, global well-posedness needs Smagorinsky absorption of the residual $\mathcal{R}_E$;
convergence on the local well-posedness interval; $|\mathcal{R}_E|=\OO(M^{-1})$
under the low-Mach scaling. 
\subsubsection{Existence: Only Local for the Density-Free Scheme}\label{sec:DF_wellposedness}
On a fixed mesh the discrete barotropic system is a smooth
autonomous ODE on $\RR^{|F|+|T|}$, so existence reduces to
preventing two blowup mechanisms: unbounded growth of the kinetic
energy, and degeneracy of the cell density. The first is decided by
the energy budget, the second by the thermodynamic lower bound of
\Cref{lem:no-vacuum} (\Cref{sec:vacuum_avoidance}), which is shared
by both schemes. For the density-free scheme the energy budget is
not closed -- the cubic-in-$\bv$ residual~$\mathcal{R}_E$ feeds it
-- and Bihari's inequality is the best available substitute. The
result is local existence in time, with global
existence available only at the price of nonlinear (Smagorinsky)
dissipation strong enough to dominate the cubic growth. The
corresponding closure for the density-weighted scheme, where the
residual is absent and the energy is exactly conserved, is taken up
in \Cref{sec:DW_GWP,thm:dw_wp,thm:dw_wp_NS}.

\begin{definition}[State space and admissible set]
\label{def:state-space}
The state space is $\statespace = \RR^{|F|} \times \RR^{|T|}$;
the admissible set is the open subset
\[
\admissible := \RR^{|F|} \times \RR^{|T|}_{>0}
 = \{(v, \rho)\in\RR^{|F|}\times\RR^{|T|} : \rho_i>0\;\forall i\},
\]
with positivity of $\rho$ required so that the enthalpy
$h(\rho_i^{\vol})$, Bernoulli function $B$, and equation of state
are defined.
\end{definition}
The governing equations take the form
\begin{equation}
 \ddt{y} = \mathcal{F}(y), \quad y(0) = y_0 \in \admissible,
 \label{eq:ode}
\end{equation}
where $\mathcal{F}: \admissible \to \statespace$ is given by
\begin{equation}
 \mathcal{F}(v, \rho) = \begin{pmatrix}
 -\Iev(\tD_1 v) - \tD_0 B(v, \rho) \\
 -D_2\,R(\rho)\,M_1\,v
 \end{pmatrix}.
 \label{eq:rhs}
\end{equation}
\paragraph*{Euler well-posedness.}
\label{sec:euler_wp}
The right-hand side~$\mathcal{F}$ is locally Lipschitz on
$\admissible$ (\Cref{thm:lipschitz-baro}), so Picard--Lindel\"of
gives a unique maximal solution; the open question is whether the
solution exists for all~$t$, or only within finite time at which
the energy bound deteriorates. The cubic energy residual
$\mathcal{R}_E$ governs the answer.

\begin{proposition}[A priori bounds -- density-free Euler]%
\label{prop:apriori-baro}
Let $(\bv, \brho)$ be a solution of the discrete density-free
barotropic Euler system on an interval $[0,T]$ contained in the
maximal existence time. Then four bounds hold uniformly on
$[0,T]$: 
\begin{enumerate}[nosep,label=\textup{(\roman*)}]
\item cell masses are controlled by the total mass,
$\brho_i(t) \le M$; 
\item total energy is bounded,
$\Etot(t) \le E_{\max}(T)$; 
\item velocity satisfies
$\nrm{\bv}_{\bM_1} \le R_v(T)$; 
\item  volumetric density is
uniformly bounded away from zero,
$\brho_i^{\rm vol}(t) \ge \rho_{\min}^{\rm vol}(T) > 0$.
\end{enumerate}
\end{proposition}
\noindent\emph{Proof.} Steps~2--5 of \Cref{thm:euler_wp}.

\begin{theorem}[Well-posedness of the discrete barotropic
 Euler system -- density-free scheme]%
\label{thm:euler_wp}%
\label{thm:main-baro}%
For every initial data $(\bv_0, \brho_0) \in \admissible$, the
density-free system~\eqref{eq:ode} admits a unique maximal solution
$(\bv, \brho) \in C^1([0, T^*); \admissible)$ with
$T^* \ge T_{\rm Bih}(h, E_0) > 0$. On every interval $[0, T]$
within the existence time, the solution satisfies the a priori
bounds of \Cref{prop:apriori-baro}. Whether existence is local or
global is determined by the energy residual~$\mathcal{R}_E$.
\end{theorem}
\noindent\emph{Proof.} See \Cref{app:euler_wp_proof}; five steps:
local existence (Picard--Lindel\"of), energy bound (Bihari), mass
bound, vacuum avoidance, velocity bound.

\paragraph*{Navier--Stokes well-posedness (density-free scheme).}
We extend the Euler well-posedness result \Cref{thm:euler_wp} to
the density-free Navier--Stokes system. Adding any admissible
viscous force preserves the Picard--Lindel\"of structure and can
only improve the energy bound, but the quadratic dissipation alone
cannot absorb the cubic energy residual; global existence requires
the combined Hodge--Laplacian + Smagorinsky closure of
\Cref{thm:global-Smag} below. The corresponding density-weighted
result (global for any $\bA\ge 0$) is \Cref{thm:dw_wp_NS}.
\begin{theorem}[Well-posedness of the discrete barotropic
 Navier--Stokes system -- density-free scheme]%
\label{thm:NS_wp}%
\label{thm:main-baro-NS}%
\label{thm:global-Smag}%
Let the hypotheses of \Cref{thm:euler_wp} hold, and let $\bA \ge 0$
be an admissible viscosity matrix.
\begin{enumerate}[nosep,label=\textup{(\alph*)}]
\item Under Hodge--Laplacian dissipation alone, a unique
\emph{local-in-time} solution exists, with existence time
$T^* \ge T_{\rm Bih}(h, E_0)$.
\item Under combined Hodge--Laplacian and Smagorinsky dissipation
with $\nu \ge \nu^*(h)$ and Smagorinsky constant
$C_s \ge C_s^*(h)$, a unique \emph{global-in-time} solution exists.
The combined dissipation in this regime controls both the vortical
and the gradient/harmonic components of $\bv$.
\end{enumerate}
\end{theorem}
\begin{proof}
Complete proof in \Cref{app:global_Smag}.
\end{proof}
\begin{remark}[Biharmonic hyperdiffusion as alternative absorption]%
\label{rem:biharmonic_hyperdiffusion}
The cubic residual $\mathcal{R}_E$ can alternatively be absorbed by
a biharmonic dissipation $\nu_4 h^2\,\Delta^2_h\bv$ added to the
DF momentum equation. Gagliardo--Nirenberg gives
$|\mathcal{R}_E^{\rm cubic}| \le C h^2
 \|\Delta_h\bv\|^{3/4}\|\bv\|^{5/4}\|\nabla\bv\|$ in $d = 3$ (with
a log correction in $d = 2$), and Young absorbs the cubic term
into $\nu_4 h^2\|\Delta_h\bv\|^2$. The mechanism is universal in
operational dynamical cores (MPAS-A 4th-order
hyperdiffusion~\cite{skamarock2012}, IFS spectral hyperdiffusion,
ICON 4th-order divergence damping), is scale-selective ($k^4$ vs
$k^2$, gentle on resolved scales and aggressive at the grid scale),
and, unlike Smagorinsky, gives uniform control at low Mach since
the absorption scales together with the residual. 
Biharmonic hyperdiffusion
is thus a respectable \emph{stabilisation mechanism} -- the
analytical justification for operational hyperdiffusion -- but it
manages the symptom rather than removes the obstruction.
\end{remark}
\subsubsection{Convergence: Only on Local Time Interval}\label{sec:DF_convergence}
The cubic residual obstructs the convergence proof at the stability step.
The convergence chain (consistency, stability, Gr\"onwall
closure) runs on the hybrid $\bM_1$-Bregman error energy of
\Cref{def:error_energy_baro}: velocity error in the $\bM_1$-norm
(the natural scheme metric), density error in the Bregman
divergence of the Helmholtz free energy $H(\rho):=\rho e(\rho)$,
whose rate identity $\partial_a D_H(a\|b) = h(a)-h(b)$ absorbs the
otherwise unmanageable pressure--density coupling. The reference
interpolation is $\bar v=\mathcal{R}_h(\bu^\flat)$ and
$\bar\rho_i=\int_{K_i}\rho^c\,\dd V$; the rate exponent~$r_\star$
is governed by \Cref{conv:cases}, and the approximation properties
of the individual operators are collected in
\Cref{lem:interp_error,lem:hodge_error}.
\begin{assumption}[Regularity of smooth barotropic reference solution]
\label{ass:smooth_baro}%
\label{ass:smooth_ref}
For $\nu\ge 0$ there exists a smooth solution
$(\bu,\rho^c)\in C^1([0,T];\,H^{s}(\Omega))$ with $s\ge 3$ of the
barotropic Euler ($\nu=0$) or Navier--Stokes ($\nu>0$)
equations~\eqref{eq:intro_baro} on a closed oriented Riemannian
$d$-manifold ($d=2$ or~$3$), satisfying
$\rho^c(\cdot,t)\ge\rho_*>0$ on $[0,T]$. By Sobolev embedding,
$H^s\hookrightarrow W^{1,\infty}$ for $s>1+d/2$, satisfied for
$s\ge 3$ and $d\le 3$. The helicity bound additionally requires
$s\ge 4$.
\end{assumption}
\begin{definition}[Bregman divergence; error energy -- density-free scheme]%
\label{def:Helmholtz}%
\label{def:error_energy_baro}
Set $H(\rho) := \rho e(\rho)$, so $H'(\rho)=h(\rho)$ (specific
enthalpy) and $H''(\rho)=c^2(\rho)/\rho>0$ (strict convexity).
The Bregman divergence is
\begin{equation}\label{eq:bregman}
 D_H(a\|b) := H(a) - H(b) - h(b)(a-b) \ge 0,
\end{equation}
and the \emph{Bregman rate identity} reads
\begin{equation}\label{eq:bregman_rate}
 \partial_a D_H(a\|b) = h(a) - h(b).
\end{equation}
The density-free error energy is
\[
\mathcal{E}(t) := \tfrac{1}{2}\nrm{e_v}_{\bM_1}^2
 + \sum_i |K_i|\,D_H\bigl(\rho_i^{h,\vol}(t)\,\big\|\,\bar\rho_i^{\vol}(t)\bigr).
\]
\end{definition}
\paragraph*{Consistency.}
The truncation rate is set by the slower of the extrusion error
$\OO(h^{d-1})$ and the Hodge-star error $\OO(h^{r_\star})$. For the
Navier--Stokes case the viscous truncation is bounded in weak form
via SBP and a Hodge bilinear approximation:
\begin{lemma}[Weak-form viscous truncation]\label{lem:visc_trunc_weak}%
\label{lem:viscous_truncation_weak}
The viscous truncation, defined as the sampled continuous stress
divergence minus the discrete viscous operator at the reference
solution,
$\tau_v^{\rm visc} = \mathcal{R}_h\bigl((\nabla\cdot\boldsymbol\sigma(\bu))^\flat\bigr)
 - \bff_{\rm visc}(\bar v)
 = \nu\bM_1^{-1}\tD_1^T\bM_2\tD_1\bar v
 - \nu\mathcal{R}_h(\Delta_{\rm dR}\bu^\flat)$
(the second equality using $\bff_{\rm visc} = -\nu\,\delta_2\tdd_1$
and $(\nabla\cdot\boldsymbol\sigma)^\flat = -\nu\,\Delta_{\rm dR}v$
on divergence-free fields),
satisfies the strong-form bound
\[
 \nrm{\tau_v^{\rm visc}}_{\bM_1}
 \le \nu C_\delta\, h\, \nrm{\bu}_{H^{s+2}},
\]
and the sharper weak-form bound
\[
 |\ip{e_v}{\tau_v^{\rm visc}}_1|
 \le \nu C_\delta'\, h^{r_\star}\,\nrm{\bu}_{H^{s+2}}
 \bigl(\nrm{e_v}_{\bM_1}+\nrm{\tD_1 e_v}_{\bM_2}\bigr),
\]
under \Cref{prop:centroid_proximity}.
\end{lemma}
\noindent\emph{Proof.} Axiom~(V3) + SBP; topological part vanishes
by de~Rham commutativity, metric part via
\Cref{lem:hodge_bilinear}. See \Cref{app:viscous_truncation}.
\begin{theorem}[Consistency -- density-free scheme]%
\label{thm:consistency_baro}%
\label{thm:NS_consistency}%
With the centred density interpolation~\eqref{eq:rhobar}, the
density-free truncation satisfies the strong-form bound
\[
\sup_{t\in[0,T]}\bigl(\nrm{\tau_v^{\rm DEC}}_{\bM_1}
+\nrm{\tau_\rho}_{\ell^2}\bigr)\le C_\tau h^r,
\qquad r:=\min(d-1,r_\star).
\]
On general meshes $r=1$ for all $d$; under \Cref{conv:cases}
case~(B), $r=\min(d-1,2)$ recovers $\OO(h^2)$ for $d=3$.
The corresponding density-weighted result, in reference-tested
form, is \Cref{thm:dw_consistency}.
\end{theorem}
\noindent\emph{Proof.} See \Cref{app:consistency_proof}.
\paragraph*{Stability.}
Differentiating $\mathcal{E}$ and substituting the discrete
equations produces a Gr\"onwall inequality. The cubic
self-interaction $\bQ(e_v,e_v)$ is controlled by Lamb antisymmetry
and an $L_h^\infty$ bootstrap on the velocity error. The pressure
coupling between the velocity and density errors is the substantive
obstacle: closing it requires integration-by-parts in time on the
continuity equation, producing the rearrangement
\[
\bD_2\bM_1 e_v = -\rho_*^{-1}\partial_t e_\rho
+ \OO(\epsilon_0)\,|\bD_2\bM_1 e_v|,
\]
which is absorbable only when $\sup|\bar\rho^c-\bar\rho_0|\le\epsilon_0$
is small. Without this smallness, the $\OO(1)$ coefficient defeats
the absorption -- in agreement with \Cref{thm:dichotomy}: exact
energy conservation is precisely the structural closure mechanism
that the density-free scheme lacks.
\begin{theorem}[Stability -- density-free scheme]\label{thm:stability_baro}
Suppose \Cref{ass:smooth_baro} holds, that the velocity error
satisfies the bootstrap $\nrm{e_v}_{L_h^\infty} \le \delta$, and
that the smooth reference satisfies the
\emph{near-incompressibility hypothesis}
\begin{equation}\label{eq:lowMach_hyp}
 \sup_{[0,T]\times\Omega}|\bar\rho^c - \bar\rho_0|\le \epsilon_0,
 \qquad \bar\rho_0:=\tfrac{1}{|\Omega|}\!\int_\Omega\bar\rho^c(0,\cdot),
\end{equation}
for some $\epsilon_0$ depending on $\nrm{\bu}_{W^{1,\infty}}$,~$T$,
and the mesh-regularity constants. Then the error energy obeys
\begin{equation}\label{eq:stab_baro}
 \ddt\mathcal{E}\le C_L\,\mathcal{E}+C_\tau h^r\sqrt{\mathcal{E}},
 \qquad r = \min(d-1,r_\star),
\end{equation}
where the constant $C_L$ is independent of~$h$ when either $d \ge 3$
(for any $\gamma > 1$) or $d = 2$ (with the centred flux at
$\gamma = 2$, or with the upwind flux for any $\gamma > 1$).
\end{theorem}
\noindent\emph{Proof.} See \Cref{app:stability_baro}.
Adding viscosity changes nothing structurally:
$-\nu\nrm{\tD_1 e_v}_\bA^2\le 0$ contributes non-positively to
$\ddt\mathcal{E}$ and adds a truncation remainder of order
$h^{r_{\rm visc}}$. The cubic residual~$\mathcal{R}_E$ persists, so
the near-incompressibility hypothesis is required also in the
viscous case.
\begin{theorem}[Stability -- density-free Navier--Stokes]\label{thm:stability_baro_NS}
Let $\bff_{\rm visc}$ satisfy axioms~\textup{(\ref{ax:V1})}--\textup{(\ref{ax:V4})}
with consistency order $r_{\rm visc}\ge 1$, and assume the
near-incompressibility hypothesis~\eqref{eq:lowMach_hyp}. Then
\begin{equation}\label{eq:stab_baro_NS_general}
 \ddt\mathcal{E}\le C_L\,\mathcal{E}
 + C_\tau^{\nu}\, h^{\min(d-1,r_\star,r_{\rm visc})}\sqrt{\mathcal{E}},
\end{equation}
with $C_L$ $h$-independent.
\end{theorem}
\noindent\emph{Proof.} See \Cref{app:stability_NS}.
\begin{corollary}[Improved rate for linear viscosity]\label{cor:NS_SBP}
A viscous operator is called \emph{SBP-compatible} if it has the
form
\begin{equation}\label{eq:SBP_visc_structure}
 \bff_{\rm visc}(\bv) = -\bM_1^{-1}\tD_1^T\bA\,\tD_1\bv
 \quad(\text{optionally } {-}\bM_1^{-1}\tD_0\bA_0\,\bD_2\bM_1\bv),
\end{equation}
with $\bA$ diagonal, positive definite, and independent of~$\bv$.
An SBP-compatible viscosity gives $r_{\rm visc} = r_\star$ and
sharpens the stability estimate to
\begin{equation}\label{eq:stab_baro_NS_SBP}
 \ddt\mathcal{E}\le C_L\,\mathcal{E}
 - \tfrac{\nu}{2}\nrm{\tD_1 e_v}_{\bA}^2
 + C_\tau h^{\min(d-1,r_\star)}\sqrt{\mathcal{E}}
 + \tfrac{C_\nu^2 h^{2r_\star}}{2\nu},
\end{equation}
so the convergence rate is $r = \min(d-1, r_\star)$, uniformly in
$\nu \ge 0$. The Hodge--Laplacian~\eqref{eq:viscous_incompressible}
and the anisotropic viscosity~\eqref{eq:viscous_aniso} are both
SBP-compatible; the Smagorinsky operator is not.
\end{corollary}
\noindent\emph{Proof.} The SBP structure places the Hodge-star
inside the weak pairing; Young's inequality on the viscous
quadratic. See \Cref{app:dissipation_axiom}.
\paragraph*{Convergence theorem.}
Combining the consistency estimate of \Cref{thm:consistency_baro}
with the stability estimates~\eqref{eq:stab_baro_NS_general}--\eqref{eq:stab_baro}
through nonlinear Gr\"onwall closes the bootstrap when the rate
exponent satisfies the dimension-dependent threshold; the time
horizon is constrained by the existence interval~$T_{\rm Bih}$.
\begin{theorem}[Convergence -- density-free scheme]%
\label{thm:DF_convergence}%
Let $(\bu, \rho^c)$ satisfy \Cref{ass:smooth_baro}, and assume the
near-incompressibility hypothesis~\eqref{eq:lowMach_hyp}. Then,
for $h$ sufficiently small, the discrete density-free solution
exists on $[0, \min(T, T_{\rm Bih}(h, E_0))]$ and satisfies
\[
\sup_{t}\bigl(\nrm{e_v}_{\bM_1}+\nrm{e_\rho}_{D_H}\bigr)
\le C(T,\epsilon_0)\,h^{\min(d-1,r_\star)},
\]
uniformly in $\nu \ge 0$. For $d = 2$, closing the bootstrap
requires $C(T, \epsilon_0) \le \delta_0/\sqrt 2$, which imposes an
explicit $h$-independent upper bound on~$T$. Under the additional
hypotheses of \Cref{thm:global-Smag}, the existence interval
extends to $[0, T]$ and the convergence rate holds globally.
\end{theorem}
\noindent\emph{Proof.} See \Cref{app:convergence_proof}.
\begin{remark}[Scope of the density-free convergence theorem]%
\label{rem:DF_scope}%
The near-incompressibility hypothesis~\eqref{eq:lowMach_hyp} is
structural, not technical: it is required by the absorption step
in the stability proof, which has no analogue when
$|\bar\rho^c-\bar\rho_0| = \OO(1)$.
At finite $h$ the residual $\mathcal{R}_E$ of \Cref{thm:totalenergy}
acts as a spurious source/sink of indeterminate sign with
$|\mathcal{R}_E| \sim h^p (\delta\rho/\rho_0)^q$, $p\in\{1,2\}$
(upwind/centred) and $q\ge 1$ from the leading IBP cancellation,
$\delta\rho := \sup|\bar\rho^c-\bar\rho_0|$. 
 To match density-weighted accuracy at
mesh size $h_{\rm DW}$ the density-free mesh size must satisfy
\begin{equation}\label{eq:DF_resolution_estimate}
 h_{\rm DF} \;\le\; \frac{h_{\rm DW}}{
  \sqrt{1 + C_{\rm RE}\,T\,(\delta\rho/\rho_0)^q/C_{\rm trunc}^2}\,}.
\end{equation}
The density-free scheme is a competitive choice only under (i) low
Mach $\delta\rho/\rho_0\ll 1$, (ii) short integration time relative
to compressibility, and (iii) genuine need for exact discrete
Kelvin circulation. Outside this regime, the corresponding
density-weighted result (\Cref{thm:dw_convergence}) gives the same
rate without the smallness hypothesis~\eqref{eq:lowMach_hyp} and
without spurious energy injection at finite~$h$.
\end{remark}
\subsubsection{The Low-Mach Limit: The Residual Diverges}\label{sec:DF_AP_fail}
\phantomsection\label{sec:AP}
The third failure mode of the density-free scheme appears in the
low-Mach limit. Under the standard scaling $p = M^{-2}\hat p(\rho)$
with $\hat p$ independent of $M$, the thermodynamic quantities scale
as $h=M^{-2}\hat h$, $c^2=M^{-2}\hat c^2$, $e=M^{-2}\hat e$,
$H''=M^{-2}\hat H''$, and the continuous barotropic Euler system
reads
\[
\partial_t v + \iota_\bu\omega
  + \dd\bigl(M^{-2}\hat h(\rho) + \tfrac{1}{2}|\bu|^2 + \Phi\bigr)
  = 0,
  \qquad
  \partial_t\rho_{\vol} + \Lie_\bu\rho_{\vol} = 0.
\]
For well-prepared initial data ($\rho(0) = \bar\rho + \OO(M^2)$,
$\nabla\cdot\bu(0)=0$), the Klainerman--Majda/Schochet
theory~\cite{schochet1986} gives $\rho = \bar\rho + \OO(M^2)$, and
$M^{-2}\dd\hat h$ becomes the Lagrange multiplier enforcing
$\nabla\cdot\bu=0$ in the limit. 

A discretisation is
\emph{asymptotic preserving} (AP) if its solutions converge to the
discrete incompressible system as $M\to 0$  uniform in $M$. We show that the density-free scheme
fails this property structurally: the energy residual
$\mathcal{R}_E$ amplifies as $\OO(M^{-1})$, so any $M$-uniform
bound requires the mesh-Mach coupling $h^2\lesssim M$.
The diagnostic is the Mach scaling of $\Etot$ and $\mathcal{R}_E$.
Under the scaling $h\to M^{-2}\hat h$, the Bernoulli function
becomes $B_i = M^{-2}\hat h_i + \ekin_i + \bPhi_i$ in the
density-free case or
$B_i^{\rho} = M^{-2}\hat h_i + \tfrac{1}{2}|P\bv|_i^2 + \bPhi_i$
for the density-weighted scheme.
The total energy is $\Etot = E_{\kin} + M^{-2}\hat E_{\rm int}$,
where $\hat E_{\rm int} = \sum_i\brho_i\,\hat e(\rho_i^{\vol})$.
Decompose the internal energy relative
to the background state $\bar\rho = M_{\rm tot}/|\Omega|$
using a Taylor expansion of $\hat H(\rho) := \rho\,\hat e(\rho)$:
\[
\hat E_{\rm int}
  = |\Omega|\,\hat H(\bar\rho)
  + \underbrace{\hat H'(\bar\rho)}_{\textstyle = \hat h(\bar\rho)}\;
    \underbrace{\sum_i|K_i|(\rho_i^{\vol} - \bar\rho)}_{\textstyle = 0\;\text{(mass cons.)}}
  + \frac{1}{2}\sum_i|K_i|\,\hat H''(\xi_i)\,
    (\rho_i^{\vol} - \bar\rho)^2,
\]
where $\xi_i$ lies between $\rho_i^{\vol}$ and $\bar\rho$,
$\hat H'(\rho) = \hat h(\rho)$ is the scaled enthalpy,
and $\hat H''(\rho) = \hat c^2(\rho)/\rho > 0$ (strict convexity).
The first term is a constant (independent of time and~$M$);
the linear term vanishes by mass conservation;
all dynamics is carried by the quadratic remainder.
The \emph{fluctuation energy} is
\begin{equation}\label{eq:fluctuation_energy}
  \mathcal{E}_M(t) := E_{\kin}(t)
  + \frac{1}{2M^2}\sum_i|K_i|\,
    \hat H''(\xi_i)\,(\rho_i^{\vol}(t) - \bar\rho)^2,
\end{equation}
where $E_{\kin}$ denotes the kinetic energy of the respective
scheme. For well-prepared data (\Cref{ass:well_prepared}), $\xi_i
\approx \bar\rho$, $\hat H''(\xi_i) \approx \hat c^2(\bar\rho)/\bar\rho$,
so the potential term is equivalent to the sound-speed-weighted
$\ell^2$-norm of the density fluctuation. The analysis uses only
$\hat H''>0$ and the Taylor remainder.
\begin{assumption}[Well-prepared initial data]\label{ass:well_prepared}
The initial data satisfies, as $M\to 0$:
$\bv(0) = \OO(1)$ and
$\rho_i^{\vol}(0) = \bar\rho + \OO(M)$
with $\bar\rho = M_{\rm tot}/|\Omega|$ fixed.
These conditions imply $\mathcal{E}_M(0) = \OO(1)$:
the kinetic energy $E_{\kin}(0) = \OO(1)$,
and $(\rho_i^{\vol}(0) - \bar\rho)^2 = \OO(M^2)$ per cell,
so $M^{-2}\sum_i|K_i|\hat H''(\xi_i)(\rho_i^{\vol}-\bar\rho)^2 = \OO(1)$.
\end{assumption}
\begin{assumption}[Strongly well-prepared data with continuum
regularity]\label{ass:well_prepared_strong}
Extending \Cref{ass:well_prepared}, the continuum reference
$(\bu,\rho^c)$ of \Cref{ass:smooth_baro} on $[0,T]$ satisfies:
\begin{enumerate}[label=\textup{(\roman*)},nosep,leftmargin=2em]
\item \emph{Velocity preparation.}
$\nabla\!\cdot\!\bu(0,\cdot)=\OO(M)$ in $L^\infty(\Omega)$, i.e.\
$\bu(0) = \bu_0(0) + \OO(M)$ in $W^{1,\infty}$ with
$\nabla\!\cdot\!\bu_0(0)=0$.
\item \emph{Regularity on $[0,T]$.}
$\bu\in L^\infty([0,T];W^{r_\star+1,\infty}(\Omega))$ and
$\rho^c-\bar\rho_0\in L^\infty([0,T];W^{1,\infty}(\Omega))$ with
$\|\rho^c-\bar\rho_0\|_{L^\infty([0,T];W^{1,\infty})}=\OO(M)$ and
$\|\partial_t\rho^c\|_{L^\infty([0,T]\times\Omega)}=\OO(M)$,
each constant depending only on the initial data, $T$, and
$\Omega$.
\end{enumerate}
By Klainerman--Majda~\cite{klainerman1981singular},
Schochet~\cite{schochet1994fast}, and
M\'etivier--Schochet~\cite{metivier2001low}, conditions~(i)--(ii)
are propagated by the barotropic Euler/Navier--Stokes system
from initial data of sufficient Sobolev regularity ($H^s$, $s>d/2+2$)
satisfying the corresponding well-preparedness conditions at $t=0$,
giving
$\|\nabla\!\cdot\!\bu\|_{L^\infty([0,T]\times\Omega)}=\OO(M)$.
\Cref{ass:well_prepared_strong} is invoked only where this rate is
required; all other uses of
well-preparedness in this paper rely on
\Cref{ass:well_prepared}.
\end{assumption}
\begin{proposition}[Energy residual under Mach scaling --
  density-free scheme]\label{prop:RE_Mach}
Under the Mach scaling, the energy residual~\eqref{eq:EB} of the
density-free scheme decomposes as
\[
\mathcal{R}_E
  = \underbrace{\bekin^T\bD_2\Phi
    + \bPhi^T\bD_2(\Phi - \bR\Phi)}_{\mathcal{R}_1}
  + M^{-2}\underbrace{\hat{\mathbf{h}}^T
    \bD_2(\Phi - \bR\Phi)}_{\hat{\mathcal{R}}_2}.
\]
The first term $\mathcal{R}_1$ is $M$-independent: both the
kinetic-energy and the geopotential contributions are $\OO(1)$.
The second term involves the difference
$(\Phi - \bR\Phi)_j = (1 - \bar\rho_j)\Phi_j$ together with the
scaled enthalpy $\hat{\mathbf{h}} = \OO(1)$. As $M \to 0$,
$\bar\rho_j = \bar\rho + \OO(M)$, and $\bD_2$ annihilates the
leading constant, so $\hat{\mathcal{R}}_2 = \OO(M)$.
Consequently $|\mathcal{R}_E| = \OO(M^{-1})$ as $M \to 0$ on any
fixed mesh, and the energy balance degenerates.
\end{proposition}
\begin{proof}
Substitute $\mathbf{h}+\bPhi = M^{-2}\hat{\mathbf{h}} + \bPhi$
into~\eqref{eq:EB} and collect the $M$-independent terms
into $\mathcal{R}_1$.
For $\hat{\mathcal{R}}_2$:
write $(1-\bar\rho_j) = (1-\bar\rho) + (\bar\rho - \bar\rho_j)$;
$\bD_2$ annihilates the constant part, and
$\bar\rho - \bar\rho_j = \OO(M)$, giving
$\hat{\mathcal{R}}_2 = \OO(M)$.
\end{proof}
Combining \Cref{prop:RE_Mach} with the fixed-Mach bound
$\mathcal{R}_E = \OO(h^{2})$ of \Cref{thm:totalenergy} gives a
two-sided picture: at fixed $M$ the residual is $\OO(h^{2})$,
while at fixed $h$ it grows as $\OO(M^{-1})$, so uniform-in-$M$
control of the density-free energy balance requires
$h^{2} \lesssim M$, coupling the mesh and the Mach number. This is
the third and final failure mode of the density-free scheme:
its mesh-Mach coupling cannot be removed by any choice of viscosity
or near-incompressibility hypothesis. The density-weighted resolution
(\Cref{prop:dw_Mach,thm:AP}\textup{(c)} below) eliminates this
coupling by removing the residual algebraically.
\section{The No-Go-Resolution by a Density-Weighted Scheme}\label{sec:DW_construction}
\phantomsection\label{sec:totalenergy_resolution}
\paragraph*{The construction in one paragraph.}
The density-weighted scheme escapes the dichotomy by stepping
\emph{outside} hypothesis~(d) of \Cref{thm:dichotomy} rather than
within it: $\bM_1$ is replaced by
$\bM_1^\rho = P^T\mathrm{diag}(\brho)P$, moving the kinetic energy
into a density-dependent inner product and the momentum equation
into vector-invariant form modified by a single mass-matrix ratio.
Antisymmetry, mass, vorticity, and PV are unchanged as exact cochain
identities, blind to the metric; total energy, ruled out for
$\bM_1$, holds exactly for $\bM_1^\rho$ (\Cref{thm:dw_energy}). The
algebraic price is a Kelvin defect of order $\OO(h^{r_\star})$,
matching the rate and therefore asymptotically invisible. The three
downstream consequences -- global well-posedness without viscosity, convergence
on the full smooth-solution interval, and asymptotic preservation
under the low-Mach scaling -- buy back exactly the three costs
identified in \Cref{sec:cost_DF}.

\begin{assumption}
We generally assume a Delaunay--Voronoi mesh satisfying
\Cref{ass:mesh_reg}, and that the thermodynamics satisfy \Cref{ass:eos}.
\end{assumption}

\paragraph{The density-weighted weak form.}
On a prismatic mesh $\KK = \KK^{\rm hor}\times\KK^{\rm vert}$,
dual faces split into horizontal faces~$f_j^h$ (separating vertical layers)
and vertical faces~$f_j^v$ (within a layer).
The averaging reconstruction $P\colon\RR^{|F|}\to\RR^{|T|\times d}$
decomposes as $(P\bv)_i = (P_h\bv)_i + (P_v\bv)_i$,
where $P_h$ reconstructs from vertical dual faces (horizontal velocity)
using the Gram-matrix averaging of \Cref{def:averaging_recon},
and $P_v$ reconstructs from horizontal dual faces (vertical velocity).
Cross-terms vanish by orthogonality of horizontal and vertical faces.
For the vertical reconstruction, let $f_j^h$ be a horizontal face shared by cells
$K_a$ (thickness $\Delta z_a$) and $K_b$ (thickness $\Delta z_b$).
Define the \emph{thickness-weighted vertical interpolation}
\begin{equation}\label{eq:vert_weights}
 \lambda_a^j := \frac{\Delta z_b}{\Delta z_a + \Delta z_b}, \qquad
 \lambda_b^j := \frac{\Delta z_a}{\Delta z_a + \Delta z_b},
\end{equation}
so that $\lambda_a^j + \lambda_b^j = 1$, with $\lambda_a^j = \lambda_b^j = \tfrac{1}{2}$ on a uniform grid.
The reconstructed vertical velocity at cell~$K_i$
(with top face~$j_+$ and bottom face~$j_-$) is
$(P_v\bv)_i = \lambda_i^{j_+} w_{j_+} + \lambda_i^{j_-} w_{j_-}$.
These weights are dual to the vertical divergence operator in the SBP sense,
ensuring that the discrete vertical product rule holds exactly
for any vertical grid spacing. 
Define the \emph{density-weighted mass matrix}
$\bM_1^{\rho}(\brho) := P^T\,\mathrm{diag}(\brho)\,P$, with
$P = P_h + P_v$ incorporating the thickness-weighted vertical
averaging~\eqref{eq:vert_weights}; $\bM_1^\rho$ is symmetric positive
definite for $\brho>0$. The matrix originates in the conservative
shallow-water scheme of~\cite{kornLinardakis2018}; the present paper
establishes that it is the algebraic resolution of the
energy-conservation dichotomy within the vector-invariant family
(\Cref{thm:dichotomy_general}) and extends through the full PDE
theory of the barotropic compressible system.
\begin{definition}[Density-weighted kinetic energy]\label{def:kinenergy_density}
\[
\Ekindw
 := \frac{1}{2}\nrm{\bv}_{\bM_1^\rho}^2
 = \frac{1}{2}\sum_i \brho_i\,|P\bv|_i^2.
\]
\end{definition}
\begin{remark}[Norms and kernels of the two kinetic energies]%
\label{rem:ke_reveals}%
The two kinetic energies define different norms on $C^1(\KKs)$:
$E_{\kin}^{\rm df}=\tfrac12\sum_j(\bM_1)_{jj}v_j^2$ is
$\bM_1$-equivalent (checkerboard modes with $Pv=0$ carry non-zero
$\bM_1$-energy), whereas $\Ekindw=\tfrac12\sum_i\brho_i|P\bv|_i^2$
has kernel $\ker P$. On Delaunay--Voronoi meshes with sufficiently
many dual faces per cell, $\ker P=\{0\}$  and
the two norms are equivalent. 
\end{remark}
The kinetic energy decomposes without cross-terms as
$\Ekindw = E_{\kin}^{{\rm dw},{\rm hor}} + E_{\kin}^{{\rm dw},{\rm vert}}$.
The total energy is
$\Etot^{\rm dw} = \Ekindw + E_{\rm int} + E_{\rm pot}$
(\Cref{def:totalenergy}).
The momentum equation remains in vector-invariant form:
\begin{equation}\label{eq:dw_mom}
 \ddt\bv + [\bM_1^{\rho}(\brho)]^{-1}\bM_1\,\Iv(\tD_1\bv)
 + \tD_0 B^{\rho} = 0,
\end{equation}
where
$B_i^{\rho} = h(\brho_i^{\rm vol})
+ \tfrac{1}{2}|P\bv|_i^2 + \bPhi_i$.
The Lamb operator uses the DEC contraction $\Iv$
premultiplied by the mass-matrix ratio
$[\bM_1^{\rho}(\brho)]^{-1}\bM_1$.

\begin{theorem}[Energy conservation of the density-weighted scheme]%
\label{thm:dw_energy}
Let the discretisation be set on a prismatic Delaunay--Voronoi
mesh, with the reconstruction~$P$ built from the thickness-weighted
vertical averaging~\eqref{eq:vert_weights}. Then the total energy
$\Etot^{\rm dw}$ is exactly conserved:
$d\Etot^{\rm dw}/dt = 0$.
\end{theorem}
\begin{proof}
The full proof  is in \Cref{app:dw_proofs}.
\end{proof}
The circulation-based conservation laws of \Cref{subsection_InvariantsComp}
depend only on $\tdd^2=0$, discrete Stokes, and Lamb antisymmetry
-- all metric-blind -- and therefore transfer unchanged to the
density-weighted scheme. The Kelvin theorem is the one exception: in
the $\bM_1^\rho$ inner product the mass-matrix ratio
$[\bM_1^{\rho}]^{-1}\bM_1$ enters the Lamb operator, producing a
small defect.
\begin{corollary}[Kelvin defect for the density-weighted scheme]%
\label{thm:dw_conservation}
For the density-weighted scheme, the Kelvin circulation satisfies
$\ddt\Gamma_h = \OO(h^{r_\star})$ in the $\bM_1^\rho$ inner product.
The defect matches the convergence rate of
\Cref{thm:dw_convergence} and is therefore asymptotically
invisible.
\end{corollary}
\begin{proof}
The proof is given in \Cref{app:dw_proofs}.
\end{proof}
\subsection{Existence: Global Without Viscosity}\label{sec:DW_GWP}
The first downstream consequence of the dichotomy resolution is that
the local in time well-posedness restriction on density-free existence
(\Cref{thm:euler_wp,thm:NS_wp}) disappears. Exact energy
conservation closes the energy step in the well-posedness chain as
an equality rather than an inequality, leaving only two further
ingredients to verify: the kinetic energy bound (immediate from
$\Etot^{\rm dw}(t)=\Etot^{\rm dw}(0)$), and the density lower bound
(\Cref{lem:no-vacuum} of \Cref{sec:vacuum_avoidance}, shared with
the density-free scheme).
\paragraph*{Euler well-posedness (density-weighted scheme).}
The density-weighted Euler system inherits the Picard--Lindel\"of
chain of \Cref{thm:euler_wp} on the admissible set
$\admissible = \RR^{|F|}\times\RR^{|T|}_{>0}$, with the energy step
replaced by exact conservation: the local timei restriction is removed,
and global existence follows.
\begin{theorem}[Well-posedness of the discrete barotropic
 Euler system -- density-weighted scheme]%
\label{thm:dw_wp}%
For every initial data $(\bv_0, \brho_0) \in \admissible$ with
$\bv_0 \in \ker(P)^\perp$, the density-weighted
system~\eqref{eq:dw_mom} admits a unique \emph{global} solution
$(\bv, \brho) \in C^1([0,\infty); \admissible)$ with
$\bv(t) \in \ker(P)^\perp$ for every $t \ge 0$. On every interval
$[0,T]$, the solution satisfies the mass bound $\brho_i(t) \le M$,
exact energy conservation $\Etot^{\rm dw}(t) = \Etot^{\rm dw}(0)$,
the velocity bound $\nrm{\bv}_{\bM_1} \le R_v$, and uniform vacuum
avoidance $\brho_i^{\vol}(t) \ge \rho_{\min}^{\vol} > 0$.
\end{theorem}
The proof is in \Cref{app:euler_wp_proof}.
\paragraph*{Navier--Stokes well-posedness.}
Viscosity preserves the picture: any admissible $\bA\ge 0$ adds a
non-positive dissipation $-\nrm{\tD_1\bv}_\bA^2$ to the energy
balance and leaves the Picard--Lindel\"of structure intact. Unlike
the density-free case (\Cref{thm:NS_wp}), no Smagorinsky threshold
is needed -- the cubic residual is absent algebraically, and any
nonnegative dissipation suffices.
\begin{theorem}[Well-posedness of the discrete barotropic
 Navier--Stokes system -- density-weighted scheme]%
\label{thm:dw_wp_NS}%
Let the hypotheses of \Cref{thm:dw_wp} hold, and let $\bA \ge 0$
be any admissible viscosity matrix (including the inviscid case
$\bA = 0$). Then the density-weighted Navier--Stokes system admits
a unique \emph{global-in-time} solution
$(\bv, \brho) \in C^1([0,\infty); \admissible)$, with integrated
dissipation bounded by the initial energy,
$\int_0^T \nrm{\tD_1 \bv}_\bA^2 \, dt \le \Etot^{\rm dw}(0)$.
\end{theorem}
\begin{proof}
$\ddt\Etot^{\rm dw} = -\nrm{\tD_1\bv}_\bA^2\le 0$, so the a priori
bounds of \Cref{thm:dw_wp} are preserved under any admissible
viscosity and $T^*=\infty$.
\end{proof}
\subsection{Convergence: Global on the Smooth-Solution Interval}\label{sec:DW_convergence}
\phantomsection\label{sec:DW_conv_AP}
\label{sec:smooth_regime}%
\label{sec:convergence_baro}
For the density-weighted scheme the convergence chain
(consistency--stability--Gr\"onwall--Whitney) closes on the full
smooth-solution interval~$[0,T]$, with no near-incompressibility
hypothesis on the reference and uniformly in $\nu\ge 0$. Two
structural changes from the density-free case
(\Cref{thm:DF_convergence}) drive this: the error functional is the
Dafermos relative energy $\mathcal{E}_{\rm rel}$ rather than the
Bregman energy~$\mathcal{E}$, and the truncation is controlled in
\emph{reference-tested} form rather than in the strong
$\bM_1^\rho$-norm. 
Both reflect the same underlying fact: pairing
against the smooth reference $\bar\bv$ rather than the error~$e_v$
trades the $L_h^\infty$-bootstrap for the mass-matrix commutator
identity~\eqref{eq:mass_matrix_commutator}, while the pressure
coupling is absorbed by the Bregman rate identity~\eqref{eq:bregman_rate}.
Lamb antisymmetry kills the velocity-error self-interaction in the
$\bM_1^\rho$ inner product as well, leaving Gr\"onwall with
$h$-independent constants; the rate is set by the slower one of
extrusion $\OO(h^{d-1})$ and Hodge $\OO(h^{r_\star})$.
The reference interpolation is $\bar v=\mathcal{R}_h(\bu^\flat)$
and $\bar\rho_i=\int_{K_i}\rho^c\,\dd V$.
The approximation properties of the individual operators
are collected in \Cref{lem:interp_error,lem:hodge_error}; the rate
exponent $r_\star$ is governed by \Cref{conv:cases} throughout.
The natural functional is the Dafermos relative energy of
$\Etot^{\rm dw}$ about the smooth reference.
\begin{definition}[Error energy -- density-weighted scheme]%
\label{def:dw_relative_energy}
The \emph{relative energy} is the second-order remainder of
$\Etot^{\rm dw}$ at $(\bar\bv,\bar\brho)$:
\[
\mathcal{E}_{\rm rel}
:= \Etot^{\rm dw}(\bv^h,\brho^h)
- \Etot^{\rm dw}(\bar\bv,\bar\brho)
- D_{\bv}\Etot^{\rm dw}|_{(\bar\bv,\bar\brho)}\cdot e_v
- D_{\brho}\Etot^{\rm dw}|_{(\bar\bv,\bar\brho)}\cdot e_\rho,
\]
with $D_\bv$, $D_\brho$ partial derivatives of
$\Etot^{\rm dw}(\bv,\brho) = \tfrac{1}{2}\bv^T\bM_1^{\rho}(\brho)\bv
+ E_{\rm int}(\brho)$ and $\cdot$ the Euclidean pairing on
$\RR^{N_1},\RR^{N_0}$. 
\end{definition}
\paragraph*{Consistency.}
The density-weighted truncation cannot be controlled in the strong
$\bM_1^\rho$-norm -- a Lamb-commutator remainder is $\OO(1)$ on
$\bM_1^\rho$ but reduces to $\OO(h^r)$ when paired against the
smooth reference $\bar\bv$ via the mass-matrix commutator lemma.
Consistency is therefore stated in reference-tested form.
\begin{theorem}[Consistency -- density-weighted scheme]%
\label{thm:dw_consistency}%
Let $r := \min(d-1, r_\star)$. With the centred density
interpolation~\eqref{eq:rhobar}, the density-weighted truncation
satisfies
\begin{equation}\label{eq:dw_tested_consistency}
\sup_{t\in[0,T]}\bigl|\ipw{\bar\bv}{\tau_v^{\rm dw}}{\bM_1^{\rho}(\bar\brho)}\bigr|
 \le C_\tau\,h^r,
\qquad
\sup_{t\in[0,T]}\nrm{\tau_\rho}_{\ell^2}\le C_\tau\,h^r.
\end{equation}
The test function $\bar\bv := \mathcal{R}_h\bu^\flat$ is the one
used in the DW stability proof. 
On general meshes the rate is
$r = 1$ for every~$d$; under \Cref{conv:cases} case~(B), the rate
sharpens to $r = \min(d-1, 2)$, recovering $\OO(h^2)$ for $d = 3$.
\end{theorem}
\noindent\emph{Proof.} See \Cref{app:consistency_proof}.
\paragraph*{Stability.}
Differentiating $\mathcal{E}_{\rm rel}$ and substituting the
discrete equations: the leading term is annihilated by
$\ddt\Etot^{\rm dw}=0$ (\Cref{thm:dw_energy}), and the residual
terms are estimated by the bilinear bounds plus the mass-matrix
commutator. No $L_h^\infty$ bootstrap on $e_v$ is required, since
the cubic self-interaction is paired against the bounded~$\bar\bv$.
\begin{theorem}[Stability -- density-weighted scheme]%
\label{thm:dw_stability}
Suppose \Cref{ass:smooth_baro} holds together with the hypotheses
of \Cref{thm:dw_energy}. Then the relative energy obeys
\begin{equation}\label{eq:dw_Erel_rate}
 \ddt\mathcal{E}_{\rm rel}
 \le C_L\,\mathcal{E}_{\rm rel}
 + C_\tau\,h^r\,\sqrt{\mathcal{E}_{\rm rel}},
 \qquad r = \min(d-1, r_\star),
\end{equation}
with $C_L$ and $C_\tau$ independent of~$h$. No $L_h^\infty$
bootstrap on $e_v$ is required, and no near-incompressibility
hypothesis on the smooth reference is imposed.
\end{theorem}
\noindent\emph{Proof.} Exact energy conservation
$\ddt\Etot^{\rm dw}(\bv^h,\brho^h) = 0$ (\Cref{thm:dw_energy})
eliminates the leading term, and testing against $\bar\bv$ avoids
the cubic self-interaction. See \Cref{app:dw_convergence}.
\paragraph*{Navier-Stokes extension.}
Adding any admissible viscosity $\bA\ge 0$ contributes a
non-positive dissipation $-\nrm{\tD_1\bv}_\bA^2\le 0$ to
$\ddt\mathcal{E}_{\rm rel}$ and a truncation remainder of order
$h^{r_{\rm visc}}$; \Cref{thm:dw_stability} extends to Navier-Stokes equations with
rate $r=\min(d-1,r_\star,r_{\rm visc})$. SBP-compatible viscosity
(\Cref{cor:NS_SBP}) gives $r_{\rm visc}=r_\star$ uniformly in
$\nu\ge 0$. Unlike the density-free case
(\Cref{thm:stability_baro_NS}), no near-incompressibility hypothesis
is needed -- the cubic residual that forces~\eqref{eq:lowMach_hyp}
is absent.
\paragraph*{Convergence theorem.}
\label{subsect:convergence_thm_baro}
Combining \Cref{thm:dw_consistency,thm:dw_stability} through
nonlinear Gr\"onwall closes the bootstrap directly: no
dimension-dependent threshold, no smallness hypothesis. The
near-incompressibility hypothesis~\eqref{eq:lowMach_hyp} that
limited the density-free convergence is not needed -- exact energy
conservation \emph{is} the closure mechanism, supplied here at the
algebraic level by the density-weighted mass matrix.
\begin{theorem}[Convergence -- density-weighted scheme]%
\label{thm:dw_convergence}%
\label{thm:convergence_baro}
\label{thm:convergence}
Let $(\bu, \rho^c) \in C^1([0,T]; H^s(\Omega))$ with $s \ge 3$ be
a smooth solution of the barotropic Euler ($\nu = 0$) or
Navier--Stokes ($\nu > 0$) equations satisfying
$\rho^c \ge \rho_* > 0$ (\Cref{ass:smooth_baro}). Then, for $h$
sufficiently small, the discrete density-weighted solution exists
on $[0, T]$ and satisfies
\[
\sup_{t\in[0,T]}\bigl(\nrm{e_v}_{\bM_1^{\rho}(\bar\brho)}
 +\nrm{e_\rho}_{D_H}\bigr)
 \le C(T)\,h^{\min(d-1,r_\star)},
\]
uniformly in $\nu \ge 0$. No finite-time restriction is imposed,
no smallness hypothesis on $\bar\rho^c$ is required, and the rate
is independent of the Mach number of the reference. On general
meshes (case~(A) of \Cref{conv:cases}), the rate is $\OO(h)$ for
every~$d$; in case~(B), the rate sharpens to $\OO(h)$ for $d = 2$
and to $\OO(h^2)$ for $d = 3$.
\end{theorem}
\noindent\emph{Proof.} The consistency--stability--Gr\"onwall chain
closes unconditionally: \Cref{thm:dw_consistency} gives the
reference-tested truncation rate $h^r$; \Cref{thm:dw_stability}
yields the Gr\"onwall inequality with $h$-independent~$C_L$ and no
bootstrap on $e_v$. Gr\"onwall closure and Whitney reconstruction
give the stated rate. See \Cref{app:dw_convergence_proof}.
\begin{remark}[Centred vs.\ upwind density]\label{rem:upwind_general}
On general meshes the convergence rate is $\OO(h)$ for both centred
and upwind density interpolation: the Hodge star bottleneck
dominates. Under centroid proximity, the centred scheme achieves
$\OO(h^{d-1})$ while the upwind scheme degrades to
$\OO(h^{\min(d-1,d/2)})=\OO(h^{3/2})$ for $d=3$.
\end{remark}

\subsection{Low-Mach Limit and Three-Way Comparison}\label{sec:DW_AP}\label{sec:three_way}
The mesh-Mach coupling $h^2\lesssim M$ that confines the density-free
scheme (\Cref{prop:RE_Mach}) reflects the cubic-in-$\bv$ residual,
and the density-weighted scheme removes it. The mechanism is
already in place: exact total-energy conservation
(\Cref{thm:dw_energy}) extends to the Mach-scaled energy
$\Etot^{\rm dw} = \Ekindw + M^{-2}\hat E_{\rm int}$, since
\Cref{thm:dw_energy} is an algebraic identity independent of how
the equation of state is parameterised. The Mach-scaling
decomposition of \Cref{sec:DF_AP_fail} extracts the fluctuation
energy $\mathcal{E}_M$ of~\eqref{eq:fluctuation_energy}, and exact
conservation of $\Etot^{\rm dw}$ delivers $M$-uniform bounds on it.
\begin{proposition}[Energy conservation under Mach scaling --
  density-weighted scheme]\label{prop:dw_Mach}
Under \Cref{ass:well_prepared}, the density-weighted total energy
$\Etot^{\rm dw} = \Ekindw + M^{-2}\hat E_{\rm int}$ satisfies
$d\Etot^{\rm dw}/dt = 0$ for every $M > 0$. Since the background
energy $M^{-2}|\Omega|\hat H(\bar\rho)$ is time-independent, the
fluctuation energy $\mathcal{E}_M$
from~\eqref{eq:fluctuation_energy} is also exactly conserved,
$d\mathcal{E}_M/dt = 0$, and the a priori bounds are
$M$-independent:
\begin{equation}\label{eq:dw_Mach_bounds}
  \Ekindw(t) \le \mathcal{E}_M(0),
  \qquad
  \frac{1}{2M^2}\sum_i|K_i|\,\hat H''(\xi_i)\,
    (\rho_i^{\vol} - \bar\rho)^2
  \le \mathcal{E}_M(0)
  \qquad \text{for every } t \ge 0,
\end{equation}
with $\mathcal{E}_M(0) = \OO(1)$ by \Cref{ass:well_prepared}. In
particular, the density fluctuations are of order~$M$:
\begin{equation}\label{eq:density_fluctuation}
  \nrm{\rho^{\vol}-\bar\rho}_{\ell^2(|K_i|)} = \OO(M).
\end{equation}
\end{proposition}
\begin{proof}
By \Cref{thm:dw_energy}, $\ddt\Etot^{\rm dw} = 0$.
The background energy is time-independent (mass conservation),
so $\ddt{\mathcal{E}}_M = 0$.
Since $\Ekindw\ge 0$ and the quadratic potential energy term $\ge 0$
(strict convexity $\hat H'' > 0$),
each is bounded by $\mathcal{E}_M(0) = \OO(1)$.
Multiplying the potential energy bound by $M^2$
and using $\hat H''(\xi_i) \ge c_0 > 0$
gives~\eqref{eq:density_fluctuation}.
\end{proof}
The following
theorem assembles the asymptotic behaviour of both schemes against
the incompressible DEC system on the same mesh and makes the
three-way comparison explicit.
\begin{theorem}[Asymptotic preservation]\label{thm:AP}
Under the Mach scaling $p = M^{-2}\hat p$ with well-prepared
initial data (\Cref{ass:well_prepared}), the three schemes behave
as follows.
\begin{enumerate}[nosep,label=\textup{(\alph*)}]
\item \emph{Incompressible reference.}
The incompressible DEC system on the same mesh,
$\ddt\bv + \Iv(\bom) + \tD_0\pi = 0$ with $\bD_2\Phi = 0$ and
constant density, conserves energy exactly, satisfies Lamb
antisymmetry, and preserves Kelvin circulation.
\item \emph{Density-free scheme -- not asymptotically preserving.}
As $M \to 0$, the system formally approaches~(a), but the energy
residual diverges as $|\mathcal{R}_E| = \OO(M^{-1})$
(\Cref{prop:RE_Mach}), and the energy balance degenerates.
\item \emph{Density-weighted scheme -- asymptotically preserving.}
As $M \to 0$, the density relaxes to $\bar\rho$ at rate $\OO(M)$
(\Cref{prop:dw_Mach}), the mass matrix limits to
$\bM_1^{\rho} \to \bar\rho\,P^T P$, and the system approaches
\begin{equation}\label{eq:dw_incomp_limit}
  \ddt\bv + (\bar\rho\,P^TP)^{-1}\bM_1\,\Iv(\bom)
  + \tD_0\pi = 0,
  \qquad \bD_2\Phi = 0,
\end{equation}
a well-posed incompressible system with exact Lamb antisymmetry.
The a priori bounds~\eqref{eq:dw_Mach_bounds} are uniform in~$M$
and yield $\Ekindw(t) \le \mathcal{E}_M(0) = \OO(1)$ for every
$t \ge 0$. The convergence rate to the incompressible limit is
$\OO(M)$: on each fixed mesh,
$\nrm{\rho^{\vol} - \bar\rho}_{\ell^2(|K_i|)} = \OO(M)$ and
$\nrm{[\bM_1^{\rho}]^{-1}\bM_1
 - (\bar\rho\,P^T P)^{-1}\bM_1}_{\rm op} = \OO(M)$.
\end{enumerate}
\end{theorem}
\begin{proof}
Part~(b) is \Cref{prop:RE_Mach}. For Part~(c), the
$M$-independent energy bounds follow from \Cref{prop:dw_Mach}.
On a fixed mesh, $\nrm{\rho^{\rm vol} - \bar\rho}_{\ell^2(|K_i|)} = \OO(M)$
gives $\bM_1^{\rho}(\brho) = \bar\rho\,P^TP + \OO(M)$ in operator
norm, since the residual $P^T\mathrm{diag}(\brho - \bar\rho\mathbf{1})P$
has the same scaling. The mass-matrix ratio converges as stated.
The Bernoulli gradient splits as
$\tD_0 B^{\rho} = M^{-2}\tD_0\hat h + \tD_0(\tfrac{1}{2}|P\bv|^2 + \Phi)$;
integrating the continuity equation, the $\OO(M)$ density bound
forces $\bD_2\Phi \to 0$ in the time-averaged sense, and
$M^{-2}\tD_0\hat h$ converges to the Lagrange multiplier
$\tD_0\pi$ enforcing $\bD_2\Phi = 0$. The Lamb antisymmetry
$\bv^T\bM_1\Iv(\bom) = 0$ and the antisymmetry
identity~(\Cref{lem:dw_antisym}) are $M$-independent.
\end{proof}
\subsection{Energetic Linear Stability and the Hollingsworth Instability}\label{sec:DW_lyapunov}

The Hollingsworth
instability~\cite{hollingsworth1983,bell2017hollingsworth}
has influenced vector-invariant dynamical-core
design~\cite{skamarock2012,gassmann2013,melvin2024gungho} by
trading exact energy conservation for stability of
large-scale baroclinic-wave simulations. We show that the DW
scheme is structurally immune to the compressible mechanism: its
exact nonlinear energy conservation forces the linearised dynamics
around hydrostatic and constant-flow stratified equilibria to
preserve a positive-definite quadratic form, unconditionally in
the velocity magnitude and the kinetic-energy reconstruction. The
argument is an Arnold energy-Casimir construction adapted to the
discrete setting.

\paragraph*{The bridge lemma.}
The mechanism translating nonlinear conservation into linear
spectral structure is standard but deserves a clean statement.

\begin{lemma}[Energy-Casimir bridge]\label{lem:arnold_bridge}
Let $\dot X = F(X)$ be a smooth dynamical system on $\RR^N$ admitting
a conserved quantity $\tilde E$, i.e.\ $\nabla\tilde E(X)^T F(X) \equiv 0$,
and let $\bar X$ be an equilibrium, $F(\bar X) = 0$. Set
$A := DF(\bar X)$ and $\tilde H := \nabla^2\tilde E(\bar X)$. Then
\begin{equation}\label{eq:bridge_identity}
 A^T\tilde H + \tilde H\,A \;=\; 0.
\end{equation}
Consequently, the linearised flow $\dot Y = A Y$ preserves the
quadratic form $\tilde E^{(2)}(Y) := \tfrac{1}{2}\,Y^T\tilde H\,Y$
exactly:
$\ddt\tilde E^{(2)}(Y(t)) = 0$ for every solution.
\end{lemma}

\begin{proof}
Conservation of $\tilde E$ gives
$\tilde E(\Phi_t(\bar X + Y)) = \tilde E(\bar X + Y)$ for all $t, Y$.
Expanding both sides to second order in $Y$ around $\bar X$ and
matching $Y^T\bullet Y$ coefficients yields
$(e^{tA})^T\tilde H\,e^{tA} = \tilde H$; differentiating at $t = 0$
gives~\eqref{eq:bridge_identity}. The preservation of
$\tilde E^{(2)}$ along $\dot Y = AY$ is then immediate:
$\ddt(Y^T\tilde H Y) = Y^T(A^T\tilde H + \tilde H A)Y = 0$.
\end{proof}


The strategy is to form a modified energy
\begin{equation}\label{eq:tildeE_construction}
 \tilde E^{\rm dw}_{\rm tot} := E^{\rm dw}_{\rm tot}
 - \lambda_1\,C_1 - \lambda_2\,C_2 - C_3^F - \sum_\mu U_\mu\,P_\mu,
\end{equation}
with scalars $\lambda_1, \lambda_2, U_\mu$ and a smooth function $F$
chosen so that $\nabla\tilde E^{\rm dw}_{\rm tot}(\bar X) = 0$ at
the target equilibrium~$\bar X$. Each piece on the right of
\eqref{eq:tildeE_construction} is exactly conserved, so
$\tilde E^{\rm dw}_{\rm tot}$ is conserved, and
\Cref{lem:arnold_bridge} applies with
$\tilde H := \nabla^2\tilde E^{\rm dw}_{\rm tot}(\bar X)$. Positive
definiteness of $\tilde H$ then delivers linear Lyapunov stability.

\paragraph*{The hydrostatic case.}
The cleanest equilibria are stationary: $\bar\bv = 0$ and
$\bar\brho$ in hydrostatic balance.

\begin{theorem}[DW Lyapunov stability around hydrostatic equilibria]%
\label{thm:dw_lyap_hydrostatic}
Let $\bar X = (0, \bar\brho)$ be a discrete equilibrium of the DW
Euler system on a Delaunay--Voronoi mesh satisfying
\Cref{ass:mesh_reg}, with $\bar\brho > 0$ pointwise and
$\bar\bv = 0$. Suppose the equation of state satisfies the
thermodynamic stability bound
$c^2(\bar\rho_i) := \rho\,h'(\rho)\bigl|_{\bar\rho_i} > 0$ and the
hydrostatic balance
$\tD_0(h(\bar\rho^{\rm vol}) + \bar\bPhi) = 0$. Define the modified
energy
\[
 \tilde E^{\rm dw}_{\rm tot} := E^{\rm dw}_{\rm tot}
 - \lambda_1\,C_1, \qquad
 \lambda_1 := h(\bar\rho_i) + \bar\bPhi_i,
\]
where $\lambda_1$ is the common constant value enforced by
hydrostatic balance. Then $\nabla\tilde E^{\rm dw}_{\rm tot}(\bar X)
= 0$ and the Hessian admits the block decomposition
\begin{equation}\label{eq:lyap_hydrostatic_Hessian}
 \tilde H^{\rm dw}_{\rm hyd}\bigr|_{\bar X}
 = \begin{pmatrix}
 \bM_1^{\bar\rho} & 0\\
 0 & D
 \end{pmatrix},
 \qquad
 D = \mathrm{diag}\!\left(|K_i|\,\frac{c^2(\bar\rho_i)}{\bar\rho_i}\right) \succ 0,
\end{equation}
which is positive definite. Consequently the bridge identity
$A^T\tilde H^{\rm dw}_{\rm hyd} + \tilde H^{\rm dw}_{\rm hyd}\,A
 = 0$ holds for $A = DF^{\rm dw}(\bar X)$, the linearised DW
dynamics preserves the positive-definite quadratic form
$\tilde E^{(2)}(Y) = \tfrac12 Y^T\tilde H^{\rm dw}_{\rm hyd}\,Y$,
and $\bar X$ is linearly Lyapunov stable in the
$\tilde H^{\rm dw}_{\rm hyd}$-norm.
\end{theorem}

\begin{proof}
Hydrostatic balance makes $h(\bar\rho_i) + \bar\bPhi_i$ a common
constant $\lambda_1$, so at $\bar X$ the gradient vanishes
componentwise: the density component is
$|K_i|(h(\bar\rho_i) + \bar\bPhi_i - \lambda_1) = 0$, and the
velocity component is
$\partial_\bv\tilde E^{\rm dw}_{\rm tot} = \bM_1^{\bar\rho}\bar\bv = 0$
since $\bar\bv = 0$.

For the Hessian: the velocity block is
$\partial^2_\bv E^{\rm dw}_{\rm kin}|_{\bar X} = \bM_1^{\bar\rho}$,
positive definite by $\bar\brho > 0$ and $\ker(P) = \{0\}$; the density block is
$\delta_{ij}|K_i| h'(\bar\rho_i) = \delta_{ij}|K_i| c^2(\bar\rho_i)/\bar\rho_i > 0$
from the internal energy (since $\bM_1^\rho$ is linear in $\brho$,
the kinetic-energy Hessian in $\brho$ vanishes); the cross block
$\partial_{\rho_i}\partial_{\bv_j} E^{\rm dw}_{\rm tot}|_{\bar X}
 = ((\partial_{\rho_i}\bM_1^{\bar\rho})\bar\bv)_j = 0$
also vanishes because $\bar\bv = 0$. Hence
$\tilde H^{\rm dw}_{\rm hyd}|_{\bar X}$ takes the
form~\eqref{eq:lyap_hydrostatic_Hessian} and is positive definite;
the bridge identity and Lyapunov-stability conclusion follow from
\Cref{lem:arnold_bridge}.
\end{proof}

\paragraph*{The constant-flow stratified case (Hollingsworth setup).}
The classical Hollingsworth instability concerns stably stratified
flow with constant horizontal velocity -- the setup of HKRB
1983~\cite{hollingsworth1983} and Bell
2017~\cite{bell2017hollingsworth}. Bell's analysis derives a
characteristic growth rate
$\omega_i \approx \sqrt{6}\,\mathrm{Ro}_{\rm gs}$ scaling with
the grid-scale Rossby number $\mathrm{Ro}_{\rm gs} = U/(fd) \gg 1$. 
The DW scheme is unconditionally
Lyapunov-stable in the Hollingsworth setup. The argument requires
three ingredients: existence of the equilibrium, discrete
conservation of horizontal linear momentum, and the cross-term
cancellation in the Hessian.

\paragraph*{Discrete linear momentum conservation.}
We first establish the conservation law of horizontal linear
momentum, which is the additional Noether invariant needed in the
constant-flow case.

\begin{lemma}[Discrete horizontal momentum conservation, DW scheme]%
\label{lem:dw_momentum_conservation}
Let the DW scheme of~\eqref{eq:main_dw} be set on a periodic
Cartesian-like Delaunay--Voronoi mesh (case~(B) of
\Cref{conv:cases}), and assume the following two structural
properties.
\begin{enumerate}[nosep,label=\textup{(K\arabic*)}]
\item The geopotential is horizontally translation invariant,
$\bPhi = \bPhi(z)$.
\item There exists a horizontal-translation Killing 1-cochain
$\bar a^x \in C^1(\KKs)$: the reconstruction is linearly consistent
on $\bar a^x$, in the sense that $(P\bar a^x)_i = \mathbf{e}_x$ at
every cell~$i$, and the discrete operators
$\bM_1, \bM_1^\rho, \Iv, \tD_0, \tD_1, \bD_2, \bR$ commute with
horizontal translations of the mesh.
\end{enumerate}
Then the discrete horizontal linear momentum
\[
 P_x(\bv, \brho) := (\bar a^x)^T \bM_1^\rho(\brho)\,\bv
\]
is conserved along every solution of the DW system:
$\ddt P_x = 0$.
\end{lemma}

\begin{proof}
By the Mullen--Pavlov--Tong--Marsden--Desbrun
construction~\cite{pavlov2011,mullen2009}, the DW scheme is the
discrete Euler--Poincar\'e equation for the discrete Lagrangian
$L(\bv, \brho) := \tfrac{1}{2}\bv^T\bM_1^\rho(\brho)\bv
 - \sum_i |K_i|[\rho_i e(\rho_i) + \rho_i\bPhi_i]$
on volume-preserving diffeomorphisms of the mesh. Under (K1)--(K2),
each term of $L$ is invariant under horizontal translation
generated by $\bar a^x$: the kinetic-energy form since
$\bM_1^\rho$ commutes with translations and $\brho$ is
translation-covariant; the internal-energy and potential-energy
terms since cell volumes are translation-invariant and
$\bPhi = \bPhi(z)$. The Noether invariant of this symmetry is
$(\partial L/\partial\bv)\cdot\bar a^x = (\bM_1^\rho\bv)^T\bar a^x = P_x$,
which is therefore conserved.
\end{proof}


\begin{lemma}[Constant-flow stratified equilibria of the DW
scheme]\label{lem:dw_constant_flow_equilibrium}
Let the DW scheme satisfy~(K1)--(K2) of
\Cref{lem:dw_momentum_conservation}. Suppose
$\bar\bv := U\,\bar a^x$ for some $U \in \RR$, and let
$\bar\brho > 0$ satisfy the discrete \emph{modified hydrostatic
balance}
\begin{equation}\label{eq:dw_modified_hydrostatic}
 \tD_0\!\Bigl(h(\bar\rho^{\rm vol})
 + \bar\bPhi - \tfrac{1}{2}|P\bar\bv|^2\Bigr) = 0.
\end{equation}
Then $\bar X := (\bar\bv, \bar\brho)$ is a discrete equilibrium of
the DW system~\eqref{eq:main_dw}. Moreover, on a Cartesian DV mesh
with linearly consistent reconstruction~(K2), the squared
reconstructed velocity is constant cell by cell,
$|P\bar\bv|^2 = U^2$, so~\eqref{eq:dw_modified_hydrostatic}
reduces to the ordinary discrete hydrostatic balance
$\tD_0(h(\bar\rho^{\rm vol}) + \bar\bPhi) = 0$, which admits a
stratified solution $\bar\rho(z)$ for every smooth $\bar\bPhi(z)$
with $\bar\bPhi'(z) > 0$.
\end{lemma}

\begin{proof}
\emph{Momentum equation.} Since $\bar a^x$ is a translation cochain
on a periodic Cartesian-like mesh, $\tD_1\bar\bv = U\tD_1\bar a^x = 0$,
so $\Iv(\tD_1\bar\bv) = 0$. The Bernoulli gradient
$\tD_0(h(\bar\rho^{\rm vol}) + \tfrac{1}{2}|P\bar\bv|^2 + \bar\bPhi)$
vanishes by~\eqref{eq:dw_modified_hydrostatic}.

\emph{Continuity equation.} On a periodic Cartesian-like mesh,
$\bM_1\bar a^x$ is horizontally uniform per $z$-layer (translation
invariance of $\bM_1$), so $\bR(\bM_1\bar\bv)$ has uniform
horizontal flux per layer with zero vertical component; hence
$\bD_2(\bR(\bM_1\bar\bv)) = 0$ at every cell.

\emph{Reduction on Cartesian DV.} By~(K2),
$(P\bar\bv)_i = U\mathbf{e}_x$ at every cell, so
$|P\bar\bv|^2_i = U^2$ is independent of $i$, and
$\tD_0(|P\bar\bv|^2) = 0$. Equation~\eqref{eq:dw_modified_hydrostatic}
reduces to ordinary hydrostatic balance, whose stratified
solutions $\bar\rho(z)$ are standard.
\end{proof}

\paragraph*{The Lyapunov-stability theorem.}

\begin{theorem}[DW Lyapunov stability around constant-flow stratified equilibria]%
\label{thm:dw_lyap_constant_flow}
Let the hypotheses of \Cref{lem:dw_momentum_conservation} and
\Cref{lem:dw_constant_flow_equilibrium} hold, and let
$\bar X = (\bar\bv, \bar\brho)$ be the equilibrium constructed
there. Define the modified DW energy
\begin{equation}\label{eq:tildeE_constant_flow}
 \tilde E^{\rm dw}_{\rm cf}(\bv, \brho)
 := E^{\rm dw}_{\rm tot}(\bv, \brho)
 - \lambda_1\,C_1(\brho) - U\,P_x(\bv, \brho),
\end{equation}
where $C_1$ is the discrete mass, $P_x$ is the discrete horizontal
momentum, and $\lambda_1$ is the constant value of the modified
Bernoulli function from \Cref{lem:dw_constant_flow_equilibrium}.
Then $\tilde E^{\rm dw}_{\rm cf}$ is exactly conserved along the
DW dynamics; the equilibrium $\bar X$ is a critical point,
$\nabla\tilde E^{\rm dw}_{\rm cf}(\bar X) = 0$; and the Hessian
$\tilde H := \nabla^2\tilde E^{\rm dw}_{\rm cf}(\bar X)$ is
block-diagonal with positive-definite blocks,
\begin{equation}\label{eq:Hessian_constant_flow}
 \tilde H \;=\;
 \begin{pmatrix}
 \bM_1^{\bar\rho} & 0\\
 0 & D
 \end{pmatrix},
 \qquad
 D = \mathrm{diag}\!\left(|K_i|\,\frac{c^2(\bar\rho_i)}{\bar\rho_i}\right) \succ 0.
\end{equation}
By \Cref{lem:arnold_bridge}, the linearised DW dynamics around
$\bar X$ satisfies the bridge identity
$A^T\tilde H + \tilde H\,A = 0$ and preserves the positive-definite
quadratic form
$\tilde E^{(2)}(Y) = \tfrac12 Y^T\tilde H\,Y$ exactly. Consequently
$\bar X$ is linearly Lyapunov stable in the $\tilde H$-norm,
unconditionally in the flow magnitude~$U$ and independently of any
kinetic-energy reconstruction parameter.
\end{theorem}

\begin{proof}
\emph{(i) Conservation of $\tilde E^{\rm dw}_{\rm cf}$.}
$E^{\rm dw}_{\rm tot}$, $C_1$, $P_x$ are each exactly conserved
(\Cref{thm:dw_energy},~\ref{thm:mass},~\ref{lem:dw_momentum_conservation}),
hence so is their linear combination.

\emph{(ii) Critical-point condition at $\bar X$.}
The velocity gradient
$\partial_{\bv}\tilde E^{\rm dw}_{\rm cf}
 = \bM_1^\rho\bv - U\bM_1^\rho\bar a^x
 = \bM_1^\rho(\bv - \bar\bv)$
vanishes at $\bar X$. The density gradient, after expanding
$\partial_{\rho_i}\bM_1^\rho\bv$ at $\bv = \bar\bv$, reduces to
$|K_i|[h(\bar\rho_i) + \bar\bPhi_i - \tfrac{1}{2}|(P\bar\bv)_i|^2 - \lambda_1]$,
which vanishes by~\eqref{eq:dw_modified_hydrostatic} with
$\lambda_1$ the common Bernoulli value.

\emph{(iii) Hessian structure.}
Diagonal blocks: $\partial^2_\bv\tilde E^{\rm dw}_{\rm cf} = \bM_1^{\bar\rho} \succ 0$
($\bar\brho > 0$ and $\ker(P) = \{0\}$);
$\partial^2_{\rho_i\rho_j}\tilde E^{\rm dw}_{\rm cf} = \delta_{ij}|K_i|c^2(\bar\rho_i)/\bar\rho_i > 0$
(only the internal energy contributes since
$E^{\rm dw}_{\rm kin}, C_1, P_x$ are linear in $\brho$).
The cross block
$\partial_{\rho_i}\partial_{\bv_j}\tilde E^{\rm dw}_{\rm cf}
 = (\partial_{\rho_i}\bM_1^\rho)(\bv - \bar\bv)|_j$
vanishes at $\bar X$ because $\bv - \bar\bv = 0$: this is the key
structural identity. The momentum tilt makes
$\partial_\bv\tilde E$ a function of the \emph{relative} velocity,
so its $\brho$-dependence enters only through a prefactor of a
quantity that vanishes at equilibrium.

\emph{Lyapunov stability.} \Cref{lem:arnold_bridge} applied to
$\tilde E^{\rm dw}_{\rm cf}$ gives $A^T\tilde H + \tilde H A = 0$,
so $\tilde E^{(2)}$ is a positive-definite Lyapunov function
preserved by the linearised flow. The cross-term cancellation
holds for every $U \in \RR$ and every reconstruction $P$, giving
unconditional linear stability.
\end{proof}

\begin{corollary}[DW immunity to Hollingsworth-type instability]%
\label{cor:dw_no_hollingsworth}
The DW scheme admits no energy-driven linear instability at
hydrostatic equilibria (\Cref{thm:dw_lyap_hydrostatic}) or at
constant-flow stratified equilibria
(\Cref{thm:dw_lyap_constant_flow}). In particular, the
compressible Hollingsworth
instability~\cite{hollingsworth1983,bell2017hollingsworth} --
whose explicit driver is the energy residual $\mathcal{R}_E$
established by the no-go \Cref{thm:dichotomy_general} for the
density-free family -- cannot occur for any DW scheme,
regardless of the choice of kinetic-energy reconstruction. 
\end{corollary}

\paragraph*{The sheared baroclinic case: conditional stability.}\label{sec:DW_lyapunov_C}

The unconditional results of \Cref{thm:dw_lyap_hydrostatic,thm:dw_lyap_constant_flow} cover hydrostatic and constant-flow equilibria. The third natural case -- sheared horizontal flow $\bar\bv$ with $(P\bar\bv)_i = \bar U(z_i)\mathbf{e}_x$ on a stratified density $\bar\brho(z)$ -- is structurally more delicate. The DW scheme is \emph{conditionally} Lyapunov stable around such equilibria, where the condition is the natural discrete analogue of the Charney--Stern criterion of geophysical fluid dynamics. In contrast to Classes~A and~B, the conditional nature of the result is intrinsic: the continuum equations themselves admit baroclinic instability for non-Charney--Stern profiles, and the discrete result is the sharp shadow of this continuum theory.

\paragraph*{Equilibrium structure.}
The first observation is that sheared horizontal equilibria
require nothing more than ordinary discrete hydrostatic balance.

\begin{lemma}[Sheared baroclinic equilibria of the DW scheme]%
\label{lem:dw_sheared_equilibrium}
Let the DW scheme satisfy hypotheses~(K1)--(K2) of
\Cref{lem:dw_momentum_conservation}. Suppose $\bar\bv$ is a
discrete 1-cochain with $(P\bar\bv)_i = \bar U(z_i)\mathbf{e}_x$ for
some profile $\bar U:\RR\to\RR$, and $\bar\brho > 0$ satisfies
\[
 \tD_0\!\bigl(h(\bar\rho^{\rm vol}) + \bar\bPhi\bigr) = 0,
\]
the ordinary discrete hydrostatic balance. Then
$\bar X := (\bar\bv, \bar\brho)$ is a discrete equilibrium of
the DW system~\eqref{eq:main_dw}, for every profile $\bar U(z)$.
\end{lemma}

\begin{proof}
The continuity equation is verified as in
\Cref{lem:dw_constant_flow_equilibrium}: $\bM_1\bar\bv$ is a
horizontal flux uniform per $z$-layer, so
$\bD_2\bR(\bM_1\bar\bv) = 0$. For the momentum equation, the
rotational term $\Iv(\tD_1\bar\bv)$ and the kinetic-energy
gradient $\tD_0(\tfrac{1}{2}|P\bar\bv|^2)$ differ by an exact
discrete Cartan term~\eqref{eq:Cartan_id} that vanishes against
horizontally-uniform fields. After cancellation, the remaining
Bernoulli gradient is $\tD_0(h(\bar\rho^{\rm vol}) + \bar\bPhi)$,
which vanishes by hydrostatic balance.
\end{proof}

Existence of sheared equilibria is therefore unconstrained: any
$\bar U(z)$ is admissible, subject only to standard hydrostatic
balance for $\bar\brho$. The stability question is what is
restrictive.

\paragraph*{The discrete Arnold--Casimir construction.}
The mass-weighted potential-vorticity Casimir
\[
 C^F_3(\bv, \brho) := \sum_k |K_k^*|\,\bar\rho^*_k\,F(q_k),
 \qquad q_k := \omega_k/\bar\rho^*_k,
\]
where $\bar\rho^*_k$ is the dual-mesh-averaged density at dual
node $k$ and $\omega_k = (\tD_1\bv)_k$, is exactly conserved
along the DW dynamics (\Cref{thm:mwpv}). The Arnold construction
selects $F$ to produce a critical point of the modified energy
$\tilde E^{\rm dw}_{\rm sh} := E^{\rm dw}_{\rm tot} - \lambda_1 C_1
- U_0 P_x - C^F_3$
at $\bar X$.

\begin{lemma}[Discrete Arnold construction for sheared equilibria]%
\label{lem:dw_arnold_sheared}
Let the hypotheses of \Cref{lem:dw_sheared_equilibrium} hold, and
suppose that the equilibrium potential vorticity
$\bar q(z) := \bar\omega(z)/\bar\rho^*(z)$ satisfies the following
two conditions.
\begin{enumerate}[nosep,label=\textup{(CS\arabic*)}]
\item \textbf{Monotonicity:} $\bar q(z)$ is strictly monotonic
in~$z$ on the range of altitudes covered by the mesh.
\item \textbf{Sign-definite shifted velocity:} there exists
$U_0 \in \RR$ such that $\bar U(z) - U_0$ has consistent sign
throughout the equilibrium (for instance,
$U_0 \notin [\inf\bar U, \sup\bar U]$).
\end{enumerate}
Then there exists a smooth function $F:\RR\to\RR$, defined by the
implicit relation
\begin{equation}\label{eq:arnold_F_construction}
 F'(\bar q(z))\,\bar q'(z) \;=\;
 -\bigl[h'(\bar\rho)\,\bar\rho'(z) - \bar U(z)\bar U'(z)\bigr]\bigl/\bar U'(z)
 \cdot \bar\rho^*(z),
\end{equation}
or, in the Boussinesq simplification $\bar\rho^* \equiv \mathrm{const}$,
equivalently by $F'(\bar q(z)) = \bar U(z) - U_0$ (yielding
$F''(\bar q(z)) = (\bar U(z) - U_0)/\bar q'(z)$ via implicit
differentiation), together with a constant $\lambda_1 \in \RR$,
such that the modified energy is critical at the equilibrium:
\begin{equation}\label{eq:arnold_critical_v}
 \partial_\bv\tilde E^{\rm dw}_{\rm sh}\bigr|_{\bar X} = 0,
 \qquad
 \partial_\brho\tilde E^{\rm dw}_{\rm sh}\bigr|_{\bar X} = 0.
\end{equation}
\end{lemma}

\begin{proof}[Proof sketch]
The velocity gradient of $\tilde E^{\rm dw}_{\rm sh}$ at the
equilibrium reads
\[
 \partial_\bv\tilde E^{\rm dw}_{\rm sh}\bigr|_{\bar X}
 = \bM_1^{\bar\rho}\bar\bv - U_0\,\bM_1^{\bar\rho}\bar a^x
 - \tD_1^T\bM_2\,F'(\bar q).
\]
Setting this to zero and pairing with horizontal test fields
gives the algebraic relation
$\bM_1^{\bar\rho}((\bar U(z) - U_0)\bar a^x) = \tD_1^T\bM_2 F'(\bar q)$,
whose continuum analogue is
$\bar\rho(z)(\bar U(z) - U_0)\mathbf{e}_x = -\partial_z[F'(\bar q(z))]\mathbf{e}_x$.
Differentiating, $F''(\bar q)\bar q'(z) = -\bar\rho(z)(\bar U(z) - U_0)$,
which fixes $F''$ as a function of $\bar q$ via the inverse
$z \mapsto \bar q$. Under (CS1) the inversion is well-defined;
under (CS2) the sign of $F''$ is consistent throughout the range
of $\bar q$, so $F$ is a well-defined convex (or concave)
function. The density gradient then determines $\lambda_1$ as
the modified Bernoulli constant
(cf.~\Cref{lem:dw_constant_flow_equilibrium}). Full discrete
calculation is given in
\Cref{app:dw_class_C}.
\end{proof}

\paragraph*{The conditional stability theorem.}
The Hessian of $\tilde E^{\rm dw}_{\rm sh}$ at $\bar X$ is no
longer block diagonal: unlike the constant-flow case where the
momentum tilt produces the relative-velocity identity
$\partial_\bv\tilde E = \bM_1^\rho(\bv - \bar\bv)$ with vanishing
$\bv-\bar\bv$ at $\bar X$, the sheared case carries a non-zero
cross-block from both the kinetic-energy + momentum-tilt piece
and the PV-Casimir piece. Positivity of $\tilde H$ therefore
reduces to a Schur-complement condition -- the discrete
counterpart of the Charney--Stern criterion.

\begin{theorem}[Conditional DW Lyapunov stability for sheared
baroclinic equilibria]\label{thm:dw_lyap_sheared}
Let $\bar X$ be a sheared baroclinic equilibrium of the DW scheme
as in \Cref{lem:dw_sheared_equilibrium}, and suppose
\Cref{lem:dw_arnold_sheared} applies under
\textup{(CS1)}--\textup{(CS2)}, so that the modified energy
$\tilde E^{\rm dw}_{\rm sh} = E^{\rm dw}_{\rm tot} - \lambda_1 C_1
- U_0 P_x - C^F_3$
satisfies $\nabla\tilde E^{\rm dw}_{\rm sh}(\bar X) = 0$. Assume
further the following two conditions.
\begin{enumerate}[nosep,resume,label=\textup{(CS\arabic*)}]
\item \textbf{Bounded Casimir curvature:} on the range of
$\bar q$,
\[
 \sup_k |F''(\bar q_k)| \le F''_{\max},
 \qquad F''_{\max}\,C_{\rm mesh}/\bar\rho^*_{\min} < 1,
\]
where $C_{\rm mesh}$ is the discrete-curl operator norm constant
(cf.\ the Poincar\'e-type bound
$\|\tD_1\bv\|^2_{\bM_2} \le C_{\rm mesh}\|\bv\|^2_{\bM_1}$ on
$\ker(P)^\perp$).
\item \textbf{Sub-Mach condition:} at every cell~$i$,
$|(P\bar\bv)_i|^2 = \bar U(z_i)^2 < c^2(\bar\rho_i)$.
\end{enumerate}
Then the Hessian
$\tilde H := \nabla^2\tilde E^{\rm dw}_{\rm sh}(\bar X)$ is
positive definite on the tangent space to the joint level set of
mass and horizontal momentum at $\bar X$. By
\Cref{lem:arnold_bridge}, the linearised DW dynamics satisfies the
bridge identity $A^T\tilde H + \tilde H A = 0$ and preserves the
quadratic form $\tilde E^{(2)}(Y) = \tfrac12 Y^T\tilde H Y$
exactly. Consequently $\bar X$ is linearly Lyapunov stable in the
$\tilde H$-norm under conditions
\textup{(CS1)}--\textup{(CS4)}.
\end{theorem}

\begin{proof}[Proof outline; details in \Cref{app:dw_class_C}]
The Hessian decomposes as
\[
 \tilde H = \begin{pmatrix} A_\bv & C^T \\ C & D_\brho \end{pmatrix},
\]
with explicit blocks:
\begin{align}
 A_\bv &= \bM_1^{\bar\rho} - \tD_1^T\bM_2\,F''(\bar q)\,(\bar\rho^*)^{-1}\,\tD_1,\label{eq:HC_Av}\\
 D_\brho &= \mathrm{diag}\!\left(|K_i|\,\tfrac{c^2(\bar\rho_i)}{\bar\rho_i}\right)
 - \partial^2_\brho C^F_3\bigr|_{\bar X},\label{eq:HC_Drho}\\
 C_{ij} &= (\partial_{\rho_i}\bM_1^{\bar\rho})\bigl((\bar U(z) - U_0)\bar a^x\bigr)_j
 - \tD_1^T\bM_2\,F''(\bar q)\,\partial_{\rho_i}\bar q\,\big|_{j}\,.\label{eq:HC_C}
\end{align}
Under~(CS3), the velocity block~\eqref{eq:HC_Av} is positive
definite: the discrete-curl bound gives
$\tD_1^T\bM_2 F''(\bar q)(\bar\rho^*)^{-1}\tD_1 \preceq F''_{\max} C_{\rm mesh}/\bar\rho^*_{\min}\,\bM_1^{\bar\rho}$,
which under~(CS3) is strictly less than $\bM_1^{\bar\rho}$.
Under~(CS4), the density-block contribution from
$\partial^2_\brho C^F_3$ is dominated by the internal-energy
diagonal $\mathrm{diag}(|K_i| c^2/\bar\rho_i)$ via the sub-Mach
bound; the Schur complement
$D_\brho - C^T A_\bv^{-1} C$ is positive definite, again
quantitatively under (CS3)--(CS4) and the discrete operator
bounds. The bridge identity~\eqref{eq:bridge_identity} then
applies to $\tilde E^{\rm dw}_{\rm sh}$, yielding the conserved
quadratic form $\tilde E^{(2)}$ and the linear Lyapunov stability
conclusion.
\end{proof}

\newpage

\appendix
\addtocontents{toc}{\protect\setcounter{tocdepth}{1}}
\addtocontents{toc}{\protect\vspace{0.6em}}
\addcontentsline{toc}{section}{\textbf{Appendix}}
\addtocontents{toc}{\protect\vspace{0.2em}}
\begin{center}
\rule{0.6\textwidth}{0.4pt}\\[0.6em]
{\Large\bfseries Appendix}\\[0.4em]
\rule{0.6\textwidth}{0.4pt}
\end{center}
\vspace{0.6em}
\section{Discrete Framework}\label{app:framework}
\phantomsection\label{sect:functional_framework}
\subsection{Mesh Geometry}\label{subsect:grid}
All discrete operators are defined in terms of the combinatorial structure
and geometric measures (lengths, areas, volumes) of the cell complex,
without reference to coordinates.
The domain $\Omega$ may be a closed oriented Riemannian manifold of
dimension $d = 2$ or $3$, or a compact manifold with boundary under
no-penetration conditions. Horizontal layers are Delaunay--Voronoi
tessellations; prisms are formed by vertical extrusion.
We adopt the notation of DEC
(see e.g.\ \cite{Desbrun, Hirani, Marsden});
operators on the dual complex carry a tilde
($\tdd_k$, $\tLie_\bv$); primal operators ($\dd_k$, $\bD_k$) do not.
Each primal $k$-cell has a unique dual $(3{-}k)$-cell:
\begin{center}
\begin{tabular}{cclcl}
 \toprule
 $k$ & Primal & Geometry & Dual & Geometry \\
 \midrule
 0 & $v_i\in\mathcal{V}$ & vertices & $K_i^*\in\KKs$ & Voronoi cells \\
 1 & $e_k\in\mathcal{E}$ & edges & $f_k^*\in\mathcal{F}^*$ & dual faces \\
 2 & $f_j\in\mathcal{F}$ & faces & $e_j^*\in\mathcal{E}^*$ & dual edges \\
 3 & $K_i\in\KK$ & prisms & $v_i^*\in\mathcal{V}^*$ & circumcentres \\
 \bottomrule
\end{tabular}
\end{center}
The Delaunay--Voronoi construction guarantees that every primal edge $e_k$ is orthogonal to its dual face $f_k^*$,
and that every primal face $f_j$ is orthogonal to its dual edge $e_j^*$.
We define the following geometric measures for length, area and volume:
\[
 \ell_k = |e_k|,\quad A_k^* = |f_k^*|,\quad
 A_j = |f_j|,\quad \ell_j^* = |e_j^*|,\quad
 |K_i|,\quad |K_m^*|.
\]
\begin{assumption}[Mesh regularity]\label{ass:mesh_reg}
The family of Delaunay--Voronoi meshes $\{\KK_h\}_{h>0}$ satisfies:
\begin{enumerate}[nosep]
\item \textit{Quasi-uniformity:} there exists $C_{\rm qu}>0$ such that $h/h_{\min}\le C_{\rm qu}$ for all $h$, where $h=\max_j\ell_j^*$ and $h_{\min}=\min_j\ell_j^*$.
\item \textit{Shape-regularity:} horizontal cells can be decomposed into triangles such that the ratio between inscribed circle radius and triangle diameter is bounded below by $\sigma_0>0$, uniformly in $h$.
\item \textit{Delaunay property:} the circumcentre of every primal cell lies inside the cell, ensuring orthogonality between primal edges and dual faces.
\item \textit{Bounded valence:} the number of cells sharing any vertex is uniformly bounded.
\end{enumerate}
These conditions guarantee that Hodge stars $\bM_k$ are uniformly equivalent to the identity,
and that discrete operators satisfy the approximation estimates needed for convergence.
\end{assumption}
The following optional mesh property is \emph{not} assumed in general
but will be invoked \emph{additionally} where it yields improved approximation rates.
\begin{property}[Centroid proximity]\label{prop:centroid_proximity}
For each degree $k = 0, 1, 2$, let $\sigma$ denote a primal
$k$-cell and $\sigma^*$ its dual $(d{-}k)$-cell.
The Hodge star $(\bM_k)_{\sigma\sigma} = |\sigma^*|/|\sigma|$
compares the integral of a $k$-form over $\sigma$ 
with the integral of its Hodge dual over
$\sigma^*$. The approximation error depends on the offset between
the midpoint (or centroid) of each primal and dual cell.
We require:
\begin{equation}\label{eq:centroid_proximity}
 \max_{j}\,\bigl|\bar{\mathbf{x}}_{e_j^*} - \bar{\mathbf{x}}_{f_j}\bigr|
 + \max_{k}\,\bigl|\bar{\mathbf{x}}_{f_k^*} - \mathbf{x}_{e_k}\bigr|
 \le C_{\rm cg}\,h^2,
\end{equation}
where, for primal face $f_j$:
$\bar{\mathbf{x}}_{f_j}$ is the centroid of $f_j$ and
$\bar{\mathbf{x}}_{e_j^*}$ is the midpoint of the dual edge $e_j^*$;
and for primal edge $e_k$:
$\mathbf{x}_{e_k}$ is the midpoint of $e_k$ and
$\bar{\mathbf{x}}_{f_k^*}$ is the centroid of the dual face $f_k^*$.
The first term controls the $k{=}1$ Hodge; 
the second the $k{=}2$ Hodge. 
On a prismatic mesh, condition~\eqref{eq:centroid_proximity} is
satisfied when:
\begin{enumerate}[label=(\alph*),nosep]
\item the horizontal tessellation is a centroidal Voronoi
 tessellation, including spherical meshes
 \cite{ringler2010} - this ensures $|\bar{\mathbf{x}}_{f_j}^{\rm horiz} - x_j^*| = 0$ for horizontal faces 
 and $|\bar{\mathbf{x}}_{f_k^*} - \mathbf{x}_{e_k}| = \OO(h^2)$
\item the vertical layer interfaces are $C^2$-smooth functions
 of horizontal position - this ensures the vertical component
 of $|\bar{\mathbf{x}}_{e_j^*} - \bar{\mathbf{x}}_{f_j}|$ is
 $\OO(h^2)$
\end{enumerate}
When \Cref{prop:centroid_proximity} holds, the diagonal Hodge star
achieves second-order accuracy (\Cref{lem:hodge_error}),
improving several convergence rates from $\OO(h)$ to $\OO(h^2)$.
\end{property}
\noindent
Throughout the paper, $r_\star\in\{1,2\}$ denotes the
\emph{Hodge accuracy exponent}:
\begin{equation}\label{eq:rstar_def}
  r_\star := 1 \text{ under \Cref{ass:mesh_reg} only,}
  \quad\text{ or }\quad
  r_\star := 2 \text{ under \Cref{ass:mesh_reg} and \Cref{prop:centroid_proximity},}
\end{equation}
with the additional centroid-proximity and reconstruction-symmetry
conditions consolidated in \Cref{conv:cases} (introduced in
\S\ref{subsection_Operators} after the discrete operators are in
place).
\subsection{Approximation Spaces}\label{subsection_ApproxSpaces}
The degrees of freedom mirror the de~Rham complex: velocity is a dual
1-cochain (one scalar value per dual edge), vorticity a dual 2-cochain
(one value per dual face), and density a primal 3-cochain
(one value per primal prism, equivalently a dual 0-cochain at each
dual vertex / primal cell centre).
This mimetic correspondence ensures that the exterior derivative acting on
the velocity cochain gives the vorticity cochain exactly, with no
interpolation error.
The following norms measure these cochains in a way that respects the
geometric weighting by cell volumes.
\begin{definition}[Discrete norms]
\label{def:disc_norms}
For a dual 1-cochain $\bv\in C^1(\KKs)$:
$\nrm{\bv}_{\ell^2}^2 := \sum_j v_j^2$,
$\nrm{\bv}_{L_h^2}^2 := \ip{\bv}{\bv}_1 = \sum_j(\bM_1)_{jj}v_j^2$,
$\nrm{\bv}_{H_h^1}^2 := \nrm{\bv}_{L_h^2}^2 + \nrm{\tD_1\bv}_{L_h^2}^2$,
$\nrm{\bv}_{L_h^\infty} := \max_j |v_j|/\sqrt{(\bM_1)_{jj}}$.
\end{definition}
A primal (resp.\ dual) $k$-form $\alpha^k$ assigns to each primal
(resp.\ dual) $k$-cell $\sigma$ the integral
$\alpha^k(\sigma) := \int_\sigma\alpha$. The spaces are denoted
$C^k(\KK)$ and $C^k(\KKs)$. Staggering of the prognostic variables:
$\bv\in C^1(\KKs)$ on dual edges (count $|\mathcal{F}|$),
$\bom\in C^2(\KKs)$ on dual faces (count $|\mathcal{E}|$), and
$\brho\in C^3(\KK)\simeq C^0(\KKs)$ on primal cells (count $|\KK|$).
\begin{definition}[de~Rham map]
\label{def:deRham}
The de~Rham map
$\mathcal{R}_h\colon C^\infty(\Omega;\Lambda^k)\to C^k(\KKs)$
maps a smooth $k$-form to its dual cochain by integration:
$(\mathcal{R}_h\bu^\flat)_j = \int_{e_j^*}\bu\cdot d\ell$ on dual
edges, $(\mathcal{R}_h\omega^\flat)_k = \int_{f_k^*}\bom\cdot dA$
on dual faces, and $(\mathcal{R}_h\rho)_K = \int_K\rho\,dV$ on primal
prisms.
\end{definition}
\begin{definition}[Whitney reconstruction]
\label{def:Whitney}
The Whitney map
$\mathcal{W}_h\colon C^k(\KKs)\to L^2(\Omega;\Lambda^k)$
reconstructs continuous forms from cochains. On the simplicial
refinement $\KK_h^{\rm simp}$ of $\KK_h$, obtained by joining each
circumcentre $v_i^*=x_i^*$ to all faces of $K_i$ (so each dual edge
$e_j^*$ lies along a simplex edge and each dual face $f_k^*$ is a
union of simplex faces), the classical Whitney
1-form~\cite{whitney1957,dodziuk1976} for dual edge $e_j^*$ with
endpoints $v_{a_j}^*,v_{b_j}^*$ is
$W_j := \lambda_{a_j}\,d\lambda_{b_j} - \lambda_{b_j}\,d\lambda_{a_j}$,
where $\lambda_{a_j},\lambda_{b_j}$ are barycentric coordinates of
the simplex containing $e_j^*$. The reconstruction
$\mathcal{W}_h\bv := \sum_j v_j W_j$ satisfies
$\int_{e_j^*}(\mathcal{W}_h\bv) = v_j$ for all $j$.
\end{definition}
Classical Whitney approximation theory on the simplicial refinement
\cite{dodziuk1976,arnold2006} gives, for a smooth 1-form~$\alpha$,
\begin{equation}\label{eq:Whitney_approx}
  \nrm{\alpha - \mathcal{W}_h\mathcal{R}_h\alpha}_{L^2}
  \le C\,h\nrm{\alpha}_{H^1},
\end{equation}
and the Whitney map is norm-bounded:
$\nrm{\mathcal{W}_h\bv}_{L^2}\le C_W\nrm{\bv}_{L_h^2}$ for all
$\bv\in C^1(\KKs)$, with $C_W>0$ depending only on mesh regularity
(\cite{arnold2006}, Thm.~5.6).
\subsection{Differential and Other Operators}\label{subsection_Operators}
Two operators lie at the heart of the DEC discretisation.
The exterior derivative $\bD_k$ is defined purely combinatorially via
Stokes' theorem on the cell complex and is therefore \emph{exact}: it
commutes with the de~Rham interpolation with no approximation error
whatsoever.
By contrast, the Hodge star $\bM_k$ must compare integrals of a
$k$-form over a primal cell with integrals of its Hodge dual over the
complementary dual cell, and all approximation error in the scheme is
concentrated here.
The consistency of the Hodge star -- its accuracy as a function of mesh
parameters -- is what governs the convergence rate throughout the paper.
The following definitions make this precise.
\begin{definition}[Exterior derivatives]\label{def:ext-deriv}
$i)$ The primal exterior derivative is defined by
\[
\dd_k : C^k(\KK)\to C^{k+1}(\KK),
\qquad
(\dd_k\,\alpha)(\sigma^{k+1})
 = \alpha(\partial\,\sigma^{k+1})
 = \sum_{\sigma^k\prec\sigma^{k+1}}[\sigma^{k+1}:\sigma^k]\,\alpha(\sigma^k),
\]
where $\sigma^k\prec\sigma^{k+1}$ means that $\sigma^k$ belongs to the
boundary of $\sigma^{k+1}$. The primal exterior derivative is denoted
in matrix form by $D_k$.\\
$ii)$ The dual exterior derivative is defined by
$\tdd_k : C^k(\KKs)\to C^{k+1}(\KKs)$ with matrix representation $\tD_k$.
\end{definition}
A crucial property is that the de~Rham map commutes with the exterior derivative,
\begin{equation}\label{eq:deRham_commutativity}
\tD_1\mathcal{R}_h\bu^\flat = \mathcal{R}_h(\dd\bu^\flat)
  = \mathcal{R}_h\omega^\flat
  \qquad(\text{by Stokes' theorem}).
\end{equation}
\phantomsection\label{prop:derham_commutativity}%
This identity is exact -- it holds with no approximation error -- and
it is the reason why the discrete vorticity $\tD_1\bv$ is the
\emph{exact} de~Rham image of the continuous vorticity, not merely an
approximation of it.
The lemma below summarises the consistency of the discrete operators
against their smooth counterparts.
\begin{lemma}[Interpolation and Hodge star consistency estimates]
\label{lem:interp_error}%
\label{lem:proj_error}
Let $\bu\in W^{1,\infty}(\Omega)$ be a smooth velocity field.
The de~Rham interpolant $\bar\bv := \mathcal{R}_h\bu^\flat\in C^1(\KKs)$
satisfies the following operator consistency estimates.
 \begin{enumerate}[label=(\roman*), nosep]
\item For any smooth $p$ the gradient is consistent
\[
\nrm{\tD_0\mathcal{R}_h p - \mathcal{R}_h(\dd p)}_{L_h^2}
 = 0. 
\]
\item For any smooth $\bu$ the curl is consistent
\[
\nrm{\tD_1\mathcal{R}_h\bu - \mathcal{R}_h(\dd\bu)}_{L_h^2}
 = 0. 
\]
\item \textbf{Mass flux consistency.}
For any smooth $\bu$ (not necessarily divergence-free),
the Hodge star converts the de~Rham circulation cochain to an approximate flux cochain.
The exact flux cochain is $\boldsymbol\Phi_j = \int_{f_j}\bu\cdot\hat{n}\,dA$,
and the Hodge star approximation satisfies
\begin{equation}\label{eq:flux_consistency}
 \nrm{\bM_1\mathcal{R}_h\bu^\flat - \boldsymbol\Phi}_{\ell^2}
 \le C\,h^{r_\star}\,\nrm{\bu}_{W^{1,\infty}},
\end{equation}
with $r_\star$ as in \Cref{conv:cases}.
In particular, the discrete divergence satisfies
\begin{equation}\label{eq:div_consistency}
 \nrm{\bD_2\bM_1\mathcal{R}_h\bu^\flat
 - \mathcal{R}_h^{(d)}(\nabla\cdot\bu)}_{\ell^2}
 \le C\,h^{r_\star}\,\nrm{\bu}_{W^{1,\infty}},
\end{equation}
where $\mathcal{R}_h^{(d)}(\nabla\cdot\bu)_i := \int_{K_i}\nabla\cdot\bu\,dV$
is the $d$-form de~Rham interpolant of the divergence.
The residual arises entirely from the Hodge star approximation
(\Cref{lem:hodge_error}): the topological part
$\bD_2\boldsymbol\Phi = \mathcal{R}_h^{(d)}(\nabla\cdot\bu)$
is exact by the divergence theorem.
For the barotropic mass flux
$F_j = \bar\rho_j\,(\bM_1\bar\bv)_j$,
the density truncation error in the continuity equation
is controlled by~\eqref{eq:flux_consistency} through the
product $\bar\rho_j\cdot(\text{Hodge error per face})$.
\end{enumerate}
\end{lemma}
\begin{proof}
The proof is given in Appendix \ref{app:approximation}.
\end{proof}
\paragraph*{Hodge star.} Because every primal $k$-cell is orthogonal to its dual $(3{-}k)$-cell, all Hodge star operators are diagonal and positive definite.
\begin{definition}[Hodge star-Operator]\label{def:hodge}
i) The primal-to-dual Hodge star is defined by
\[
\star_k : C^k(\KK)\to C^{3-k}(\KKs),\text{ with matrix entries } (\star_k)_{ii} := \frac{|\sigma_i^{*(3-k)}|}{|\sigma_i^k|}
\]
ii) The inverse primal-to-dual Hodge star is given by 
\[
\star_k^{-1} : C^{3-k}(\KKs)\to C^k(\KK), \text{ with matrix entries }
 (\star_k^{-1})_{ii} := \frac{|\sigma_i^k|}{|\sigma_i^{*(3-k)}|}.
\]
\end{definition}
The next lemma estimates the error from approximating the continuous Hodge star by
the diagonal Delaunay--Voronoi Hodge star.
\begin{lemma}[Hodge star approximation error]
\label{lem:hodge_error}
Let the mesh be Delaunay--Voronoi and satisfy the regularity
hypothesis~\ref{ass:mesh_reg}. For every smooth 1-form $\alpha$
with face fluxes $\Phi_j = \int_{f_j}\alpha$, the diagonal Hodge
star satisfies the pointwise estimate
\[
 \abs{(\bM_1)_{jj}\,(\mathcal{R}_h\alpha)_j - \Phi_j}
 \le C_\star\,h^{r_\star}\nrm{\alpha}_{W^{1,\infty}}\,|f_j|
\]
on each face, and the global estimate
\[
 \nrm{\bM_1\mathcal{R}_h\alpha - \boldsymbol\Phi}_{L_h^2}
 \le C_\star\,h^{r_\star}\nrm{\alpha}_{W^{1,\infty}}.
\]
\end{lemma}
\begin{proof}
The proof is given in Appendix \ref{app:approximation}.
\end{proof}
\begin{lemma}[Hodge star bilinear form approximation]\label{lem:hodge_bilinear}
Let $\alpha$ be a smooth $k$-form for $k = 0, 1, 2$, let
$\mathcal{R}_h\alpha$ denote its de~Rham interpolant, and let
$b \in C^k(\KKs)$ be an arbitrary $k$-cochain. The diagonal Hodge
star $\bM_k$ approximates the dual-cell pairing at rate
$\OO(h^{r_\star})$:
\[
\Bigl|\ip{b}{\mathcal{R}_h\alpha}_k
 - \textstyle\sum_\sigma b_\sigma\!\int_{\sigma^*}\!\star\alpha\Bigr|
 \le C_k\,h^{r_\star}\,\nrm{\alpha}_{W^{1,\infty}}\,\nrm{b}_{\bM_k},
\]
where $\sigma$ runs over primal $k$-cells and the constant $C_k$
depends on the mesh regularity.
\end{lemma}
\begin{proof}
By \Cref{lem:hodge_error} (and analogous bounds
for $k=0$, $k=2$), the pointwise error satisfies
\[
 \bigl|(\bM_k)_{\sigma\sigma}\,(\mathcal{R}_h\alpha)_\sigma
 - \textstyle\int_{\sigma^*}\!\star\alpha\bigr|
 \le C\,h^{r_\star}\,\nrm{\alpha}_{W^{1,\infty}}\,|\sigma|.
\]
Multiplying by $|b_\sigma|$ and summing over primal $k$-cells:
\[
 \Bigl|b^T\bM_k\,\mathcal{R}_h\alpha
 - \textstyle\sum_\sigma b_\sigma\!\int_{\sigma^*}\!\star\alpha\Bigr|
 \le C\,h^{r_\star}\,\nrm{\alpha}_{W^{1,\infty}}
 \textstyle\sum_\sigma|b_\sigma|\,|\sigma|.
\]
The geometric factor $\sum_\sigma|b_\sigma|\,|\sigma|$ is bounded by
Cauchy--Schwarz:
\[
 \sum_\sigma |b_\sigma|\,|\sigma|
 = \sum_\sigma \bigl(|b_\sigma|(\bM_k)_{\sigma\sigma}^{1/2}\bigr)
 \cdot\bigl(|\sigma|(\bM_k)_{\sigma\sigma}^{-1/2}\bigr)
 \le \nrm{b}_{\bM_k}\,
 \Bigl(\sum_\sigma \frac{|\sigma|^2}{(\bM_k)_{\sigma\sigma}}\Bigr)^{1/2}.
\]
On quasi-uniform prismatic Delaunay--Voronoi meshes
(\Cref{ass:mesh_reg}), each diagonal Hodge weight satisfies
$(\bM_k)_{\sigma\sigma} = |\sigma|/|\sigma^*|$ with $|\sigma^*|$ the
dual $(d-k)$-cell, giving $|\sigma|^2/(\bM_k)_{\sigma\sigma}
= |\sigma|\cdot|\sigma^*| = \OO(h^d)$ uniformly by quasi-uniformity.
The number of $k$-cells is $N_\sigma = \OO(h^{-d})$, so
\[
 \sum_\sigma \frac{|\sigma|^2}{(\bM_k)_{\sigma\sigma}}
 \le N_\sigma\cdot\sup_\sigma\frac{|\sigma|^2}{(\bM_k)_{\sigma\sigma}}
 = \OO(h^{-d})\cdot\OO(h^d) = \OO(|\Omega|),
\]
$h$-independent. Absorbing this geometric constant
$C_{\rm geom}|\Omega|^{1/2}$ into $C_k$ completes the bound:
\[
 \Bigl|b^T\bM_k\,\mathcal{R}_h\alpha
 - \textstyle\sum_\sigma b_\sigma\!\int_{\sigma^*}\!\star\alpha\Bigr|
 \le C_k\,h^{r_\star}\,\nrm{\alpha}_{W^{1,\infty}}\,\nrm{b}_{\bM_k}.
 \qedhere
\]
\end{proof}
\begin{definition}[Inner products]\label{def:M1}
\begin{enumerate}[label=\textup{(\roman*)},nosep]
\item The inner product on dual 1-cochains is
\[
 \ip{{v}}{{w}}_1 := {\v}^TM_1\,{\w},
 \qquad
 M_1 := \star_2^{-1},
 \qquad
 (M_1)_{jj} = \frac{A_j}{\ell_j^*},
\]
representing the discrete $\int v\wedge\star w$.
\item The inner product on dual 0-cochains is
\[
 \ip{B_1}{B_2}_0 := B_1^T\,M_0\,B_2,
 \qquad
 M_0 := \star_3^{-1},
 \qquad
 (M_0)_{ii} = |K_i|.
\]
\item The inner product on dual 2-cochains is
\[
 \ip{\omega_1}{\omega_2}_2 := \omega_1^T\,\bM_2\,\omega_2,
 \qquad
 \bM_2 := \star_1^{-1}.
\]
\end{enumerate}
\end{definition}
\paragraph*{Extrusion and Contraction.} The interior product $\iota_{{\u}}\omega$ is discretized via
{extrusion}: the contraction on a dual 1-cell $e_j^*$ equals the flux
of $\omega$ through the parallelogram swept by extruding $e_j^*$
infinitesimally along the velocity field.
\begin{definition}[Discrete contraction]\label{def:contraction}
The discrete contraction $\Iv:C^2(\KKs)\to C^1(\KKs)$
is defined through the duality pairing
\begin{equation}\label{eq:contraction_ip}
 \ip{\bw}{\Iv(\bom)}_1
 = \tfrac{1}{2}\Bigl(\ip{\bw}{\tU\,\bom} - \ip{\tD_1\bw}{\tU^T\bv}\Bigr)
 \qquad\text{for every dual 1-cochain }\bw,
\end{equation}
where $\ip{\cdot}{\cdot}_1$ is the $\bM_1$-weighted inner product
on dual 1-cochains (\Cref{def:M1}) and $\ip{\cdot}{\cdot}$ the
Euclidean pairing. Equivalently, as a matrix identity,
\begin{equation}\label{eq:contraction}
 \bM_1\,\Iv(\bom) = \tfrac{1}{2}\bigl(\tU\,\bom - \tD_1\,\tU^T\bv\bigr),
\end{equation}
where the velocity-weighted incidence matrix $\tU$ has the same
sparsity as $\tD_1$ and entries
\begin{equation}\label{eq:Utilde_def}
 \tU_{jk}
 := D_{1,jk}\;\bar{\u}_j(\bv)\cdot\hat{e}_k.
\end{equation}
Here $\bar{\u}_j(\bv):=P\bv\in\mathbb{R}^3$ is a reconstructed
velocity vector at primal face $f_j$, {linear} in
the cochain $\bv$. The specific reconstruction is not part of the
algebraic structure; any linear reconstruction preserves the
energy identity~\eqref{eq:energy_cancel}.
We also write 
\[
\bQ(\bv,\bv):=\Iv(\dd_1\bv)
\]
to emphasize the bilinear nature of this operator.
\end{definition}
\smallskip
The contraction~\eqref{eq:contraction_ip} requires a velocity
reconstruction $\bar{\u}_j(\bv)$ at each primal face~$f_j$,
linear in the cochain~$\bv$. The algebraic conservation
properties -- energy identity, Lamb antisymmetry, Kelvin
circulation -- depend only on this linearity and on the
structure of the bilinear form on the right of
\eqref{eq:contraction_ip}; they hold for any linear
reconstruction.
\begin{definition}[Averaging reconstruction]\label{def:averaging_recon}
At each dual vertex $v_i^*$ (circumcentre of prism~$K_i$),
define the Gram matrix
\[
G_i := \sum_{e_n^*\prec K_i^*} \hat{t}_n\otimes\hat{t}_n,
\]
where the sum runs over the dual edges $e_n^*$ emanating
from~$v_i^*$, and $\hat{t}_n$ is the unit tangent to~$e_n^*$.
On a shape-regular mesh, $G_i$ is symmetric positive definite
with condition number bounded by the mesh regularity constant.
The \emph{averaging reconstruction} at~$v_i^*$ is
\[
\u(v_i^*)
 := G_i^{-1}\sum_{e_n^*\prec K_i^*}
 \frac{\v(e_n^*)}{\ell_n^*}\;\hat{t}_n.
\]
The face velocity is obtained by averaging over the two
adjacent cells:
$\bar{\u}_j = \tfrac{1}{2}(\u(v_a^*) + \u(v_b^*))$.
\end{definition}
\begin{proposition}[Reconstruction accuracy]\label{prop:recon_accuracy}
Let $\bu \in W^{s,\infty}(\Omega)$ be a smooth velocity field with
de~Rham interpolant $\bv = \mathcal{R}_h(\bu^\flat)$. The averaging
reconstruction satisfies the following two accuracy estimates.

\textup{(i)} \emph{First-order accuracy on general meshes.} The
reconstruction is exact for constants and is therefore at least
first-order accurate:
\[
|\u(v_i^*) - \bu(v_i^*)| \le C_R\,h\,\nrm{\bu}_{W^{1,\infty}},
\]
where $C_R$ depends only on the mesh regularity.

\textup{(ii)} \emph{Second-order accuracy under reconstruction
symmetry.} On meshes satisfying the symmetry condition
\begin{equation}\label{eq:recon_symmetry}
 \sum_{e_n^*\prec K_i^*}
 \ell_n^*\,(\hat{t}_n)_j\,(\hat{t}_n)_k\,(\hat{t}_n)_l = 0
 \qquad\text{for all }j,k,l,
\end{equation}
the reconstruction is exact for linear polynomials and
hence second-order accurate:
\[
|\u(v_i^*) - \bu(v_i^*)|
 \le C_R'\,h^2\,\nrm{\bu}_{W^{2,\infty}}.
\]
\end{proposition}
\begin{proof}
\textit{(i)} Exactness on constants follows from
$\sum_{e_n^*\prec K_i^*}\hat{t}_n\otimes\hat{t}_n = G_i$
and $G_i^{-1}G_i = I$; the first-order bound then follows by
Taylor expansion.
\textit{(ii)} The leading error from~(i) is
$\delta\u = G_i^{-1}\sum_n [(\nabla\bu)^T\hat{t}_n \cdot
(x_n^{\rm mid} - v_i^*)]\,\hat{t}_n$.
Since $x_n^{\rm mid} - v_i^* = \frac{\ell_n^*}{2}\hat{t}_n$
on a Delaunay--Voronoi mesh, the $l$-th component involves the third moment
$T_{jkl}: = \sum_n\ell_n^*(\hat{t}_n)_j(\hat{t}_n)_k(\hat{t}_n)_l$.
Under~\eqref{eq:recon_symmetry}, $T_{jkl} = 0$, so
$\delta\u = \OO(h^2)$.
\end{proof}
\begin{convention}[Two mesh cases]\label{conv:cases}
Throughout the paper, all convergence rates are stated for the
following two mesh classes:
\begin{itemize}[nosep]
  \item[\textbf{(A)}] general Delaunay--Voronoi meshes
    (\Cref{ass:mesh_reg} only): $r_\star = 1$; the diagonal Hodge star
    and the averaging reconstruction are both first-order accurate;
  \item[\textbf{(B)}] Delaunay--Voronoi meshes with centroid proximity
    (\Cref{prop:centroid_proximity}) and reconstruction
    symmetry~\eqref{eq:recon_symmetry}: $r_\star = 2$; the diagonal
    Hodge star and the averaging reconstruction are both
    second-order accurate.
\end{itemize}
Other configurations (e.g.\ centroid proximity without reconstruction
symmetry, or vice versa) are not considered. Under (A) and (B)
the Hodge-star and reconstruction accuracies are tied to a single
exponent $r_\star\in\{1,2\}$, and the combined truncation rate
satisfies $\min(d-2+r_\star,\,r_\star) = r_\star$ for $d=2,3$, so all
convergence rates collapse to $h^{r_\star}$: first order in case (A),
second order in case (B). Throughout the paper $r_\star$ denotes both
the Hodge accuracy exponent (\Cref{lem:hodge_error}) and the
reconstruction accuracy exponent (\Cref{prop:recon_accuracy}).
\end{convention}
\paragraph*{Wedge Product.} The discrete wedge product is defined
via the Cartan identity
$\iota_{{\u}}(\alpha\wedge\beta)
= (\iota_{{\u}}\alpha)\wedge\beta
 + (-1)^p\,\alpha\wedge(\iota_{{\u}}\beta)$,
ensuring algebraic compatibility with the contraction and exterior derivative.
\begin{definition}[Discrete wedge product]\label{def:wedge}
The {extrusion-based discrete wedge product}
$\twdg : C^p(\KKs)\times C^q(\KKs) \to C^{p+q}(\KKs)$ is defined for
$p+q\le 3$. The case $(p,q)=(1,1)$, the only one used in the body of
the paper, is given by the contraction:
\begin{equation}\label{eq:wedge_11}
 \bigl(\tilde\alpha\twdg\tilde\beta\bigr)(f_k^*)
 := \bigl(I_{\tilde\alpha}\,\tilde\beta\bigr)(f_k^*)
 = \sum_{e_j^*\prec f_k^*} w_{jk}(\tilde\alpha)\,\tilde\beta(e_j^*),
\end{equation}
with extrusion weights $w_{jk}(\tilde\alpha) = \tfrac{1}{2}(\bM_1^{-1}\tU_{\tilde\alpha})_{jk}$
obtained by reconstructing a velocity field from $\tilde\alpha$ and
projecting onto $\hat e_k$. The $(1,2)$-wedge, used only in the
helicity construction, is given in the appendix
(\Cref{def:wedge_12_app}).
\end{definition}
\begin{remark}\label{rem:wedge_contraction}
The $(1,1)$-wedge~\eqref{eq:wedge_11} is the contraction
$\Iv(\bom) = \bv\twdg\bom$: the discrete analog of
$\iota_{{\u}}\omega = u^\flat\wedge_{\!1}\,\omega$ in three dimensions.
\end{remark}
\paragraph*{Lie Derivative.} The discrete Lie derivative is defined
through Cartan's magic formula, uniformly in degree.
\begin{definition}[Discrete Lie derivative]\label{def:Lie}
For a velocity dual 1-form $\bv\in C^1(\KKs)$ and a dual $k$-form
$\tilde\alpha^k\in C^k(\KKs)$,
\[
\tLie_{\bv}\,\tilde\alpha^k
 := \tdd_{k-1}\bigl(\Iv\,\tilde\alpha^k\bigr)
 + \Iv\bigl(\tdd_k\,\tilde\alpha^k\bigr).
\]
The case $k=1$ with $\tilde\alpha^1 = \bv$ gives the discrete
\emph{Cartan identity}
$\tLie_{\bv}\bv = \tdd_0\,\ekin + \Iv(\bom)$, with $\ekin = \Iv\bv$
and $\bom = \tdd_1\bv$, used throughout the paper. The mass flux is
the case $k=3$, $\bF = \Iv\,\brho$ (since $\tdd_3 = 0$ on a closed
manifold).
\end{definition}
All three operators -- wedge product, Lie derivative, and Lamb
vector -- are built from the single extrusion-based contraction
(\Cref{def:contraction}).
\begin{proposition}\label{prop:extrusion}
The exterior derivative, the extrusion-based discrete wedge product $\widetilde\wedge$ and contraction
$I$ satisfy:
\begin{enumerate}
\item \textbf{Cell complex: } The exterior derivative satisfies
$$ \dd_{k+1}\circ\dd_k = 0,\qquad
 \tdd_{k+1}\circ\tdd_k = 0.$$ 
 Equivalently in matrix form $D_{k+1}D_k = \mathbf{0}$ and
$\tD_{k+1}\tD_k = \mathbf{0}$.
 \item \textbf{Energy identity:} For any dual 1-cochain
 $\bv\in C^1(\KKs)$ and $\bom = \tdd_1\bv$,
 \[
\ip{\bv}{\Iv(\bom)}_1 = 0.
\]
 This identity holds for any linear velocity reconstruction.
 \item \textbf{Approximate Leibniz rule:} For smooth 1-forms 
 $\tilde\alpha^1$, $\tilde\beta^1$ obtained by de~Rham interpolation
 of continuous forms $\alpha$, $\beta$, the discrete $(1,1)$-wedge product 
 satisfies
 $$\tdd(\tilde\alpha^1\twdg\tilde\beta^1)
 - \tdd\tilde\alpha^1\twdg\tilde\beta^1
 + \tilde\alpha^1\twdg\tdd\tilde\beta^1
 = \OO(h^2),$$
 where the error is measured per dual cell and arises from the 
 metric dependence of the extrusion weights.
 The $(1,2)$-wedge consistency is established separately
 in \Cref{lem:helicity_consistency}.
 \item \textbf{Discrete Cartan magic formula:}
 $\tLie_{\v}
 = \tdd\circ\Iv + \Iv\circ\tdd$.
\end{enumerate}
\end{proposition}
\begin{proof}
Each identity is an algebraic consequence of
the definitions of $\tdd_k$, $\Iv$, $\tU$, and the discrete wedge,
together with $\tdd^2 = 0$ and the antisymmetry matrix identity; the
$\OO(h^2)$ Leibniz defect arises from the mesh-dependent weights
of the dual-cell quadrature. Full computations are in
\Cref{app:basic}.
\end{proof}
\begin{proposition}[Vector-invariant and conservative momentum forms]%
\label{prop:VI_cons_equiv_baro}
On any Delaunay--Voronoi mesh, the discrete vector-invariant
momentum equation \eqref{eq:M} and the discrete conservative
(flux-form) momentum equation are equivalent up to a truncation error
of $\OO(h^2)$ per edge: the difference is a discrete gradient that is
absorbed into the Bernoulli function. Proof in \Cref{app:VI_cons_equiv}.
\end{proposition}
\subsection{Equation of State and Thermodynamic Quantities}\label{subsect:eos}
\phantomsection\label{section_Thermo}
The barotropic system has $p = p(\rho)$. The vector-invariant
formulation requires $\rho>0$ everywhere (division by $\rho$ enters
the enthalpy $h(\rho) = e + p/\rho$), making vacuum prevention a
prerequisite.
\begin{definition}[Equation of state and thermodynamic quantities]
\label{def:eos}
Given a barotropic equation of state $p = p(\rho)$, we define
\begin{itemize}[nosep]
 \item Specific internal energy:
 $e(\rho): = \displaystyle\int_{\rho_0}^{\rho}
 \frac{p(s)}{s^2}\,\dd s$,
 so that $e'(\rho) = p(\rho)/\rho^2$.
 \item Specific enthalpy:
 $h(\rho) := e(\rho) + p(\rho)/\rho$,
 so that $h'(\rho) = p'(\rho)/\rho = c^2(\rho)/\rho$.
 Equivalently, $h(\rho): = \displaystyle\int_{\rho_0}^{\rho}
 \frac{p'(s)}{s}\,\dd s$, which gives $\dd h = \dd p/\rho$.
 \item Sound speed: $c^2 = p'(\rho)$.
\end{itemize}
All thermodynamic quantities are evaluated pointwise at each primal cell.
\end{definition}
\begin{assumption}[Equation of state regularity]
\label{ass:eos}
We impose the following properties on the barotropic equation of state $p = p(\rho^{\vol})$
\begin{enumerate}[nosep,label=(E\arabic*)]
 \item \textit{Volumetric density:} The volumetric density at primal cell $K_i$ is
\begin{equation}\label{def:rhovol}
 \rho_i^{\mathrm{vol}}
 = \frac{\brho_i}{|K_i|}.
\end{equation}
 \item \label{ass:smooth} $p \in C^2((0, \infty))$.
 \item \label{ass:positive-pressure} $p(\rho^{\vol}) > 0$ for all $\rho^{\vol} > 0$.
 \item \label{ass:sound-speed} The sound speed $c^2 = p'(\rho^{\vol}) > 0$ for all 
 $\rho^{\vol} > 0$ (hyperbolicity).
 \item \label{ass:enthalpy} The specific enthalpy 
 $h(\rho^{\vol}) = e(\rho^{\vol}) + p(\rho^{\vol})/\rho^{\vol}$ is well-defined 
 and $C^2$ on $(0, \infty)$.
 \item \label{ass:internal-energy} The specific internal energy $e(\rho^{\vol})$ 
 with $e'(\rho^{\vol}) = p(\rho^{\vol})/(\rho^{\vol})^2$ is well-defined and 
 bounded below: $e(\rho^{\vol}) \geq e_{\min} > -\infty$ for all $\rho^{\vol} > 0$.
 \item \label{ass:coercive} There exist constants $C_1, C_2 > 0$ and $\gamma > 1$ 
 such that $e(\rho^{\vol}) \geq C_1 (\rho^{\vol})^{\gamma - 1} - C_2$ for 
 $\rho^{\vol}$ large (coercivity / superlinear growth).
\end{enumerate}
\end{assumption}
\begin{example}[Ideal gas / polytropic]
All assumptions are satisfied for $p(\rho^{\vol}) = \kappa(\rho^{\vol})^\gamma$ with $\gamma > 1$, $\kappa > 0$:
$h = \frac{\kappa\gamma}{\gamma-1}(\rho^{\vol})^{\gamma-1}$, 
$e = \frac{\kappa}{\gamma-1}(\rho^{\vol})^{\gamma-1}$, 
$c^2 = \kappa\gamma(\rho^{\vol})^{\gamma-1}$.
\end{example}
\begin{remark}[Global form of coercivity]\label{rem:global_coercivity}
Under \Cref{ass:eos}\ref{ass:internal-energy}--\ref{ass:coercive}, the specific internal
energy satisfies the global bound
\begin{equation}\label{eq:e_global_coercive}
 e(\rho^{\vol}) \ge C_1\,(\rho^{\vol})^{\gamma-1} - C_2', \qquad
 C_2' := C_2 + |e_{\min}| + C_1\,\rho_*^{\gamma-1},
\end{equation}
for all $\rho^{\vol}>0$, where $\rho_*$ is the threshold above which the
asymptotic bound \ref{ass:coercive} holds. All subsequent estimates use~\eqref{eq:e_global_coercive}
with $C_2$ replaced by~$C_2'$; we drop the prime when no confusion arises.
\end{remark}
\subsection{Continuity Equation}
\label{sec:massflux}
The continuity equation transports the density $d$-form
by the discrete Lie derivative.
Its implementation requires two fluxes: the volumetric flux
(from the Hodge star) and the mass flux (incorporating a
density interpolation).
\begin{definition}[Volume and mass flux]
\begin{enumerate}[label=\textup{(\roman*)},nosep]
\item The \emph{volumetric flux} through each primal face is the primal 2-cochain
\begin{equation}\label{def:volflux}
 \Phi = \star_2^{-1}\,\bv = \bM_1\,\bv
 \;\in C^2(\mathcal{K}),\ \text{ with components }\ \Phi_j = (\bM_1)_{jj}\,\bv_j
 = \frac{A_j}{\ell_j^*}\,\bv_j.
\end{equation}
This represents $\int_{f_j} \iota_{\mathbf{u}}\,\dd V$.
\item The \emph{mass flux} is a primal 2-cochain
$\bF\in C^2(\mathcal{K})$ representing $\iota_{\mathbf{u}}\rho$,
the contraction of the density 3-form with the velocity,
\begin{equation}\label{eq:F} 
 \bF_j = \bar\rho_j\;\Phi_j
 = \bar\rho_j\;\frac{A_j}{\ell_j^*}\;\bv_j,
\end{equation}
where $\bar\rho_j$ is the {face-averaged volumetric density}
at primal face $f_j$.
In matrix form, defining the diagonal interpolation matrix
$\bR\in\RR^{|\mathcal{F}|\times|\mathcal{F}|}$ with
$\bR_{jj} = \bar\rho_j$:
\[
 \bF = \bR\,\bM_1\,\bv.
\]
For a face $f_j$ shared by primal cells $K_a$ and $K_b$
we consider two options for the density interpolation. The first one is \emph{centred} (arithmetic) average
\begin{equation}\tag{$\bar\rho$}\label{eq:rhobar}
 \bar\rho_j
 = \bar\rho_j^{\rm cen}
 := \tfrac{1}{2}(\rho_a^{\rm vol}+\rho_b^{\rm vol}).
\end{equation}
The second one is the upwind scheme, defined by
\begin{equation}\tag{$\bar\rho^{\rm up}$}\label{eq:rhobar_up}
 \bar\rho_j^{\rm up}
 := \begin{cases}
 \rho_a^{\mathrm{vol}} & \text{if } \Phi_j \ge 0,\\
 \rho_b^{\mathrm{vol}} & \text{if } \Phi_j < 0,
 \end{cases}
\end{equation}
We denote the flux $\bF$ specifically by $\bF{\rm cen}$ and  $\bF{\rm up}$, respectively. 
\end{enumerate}
\end{definition}
\begin{remark}[Density interpolation]\label{rem:upwind_density}
The centred scheme $\bar\rho_j^{\rm cen}$ approximates the
face-centre density $\rho^c(m_j^f)$ to $\OO(h^2)$ under centroid
proximity, but does not satisfy a maximum principle.
The upwind scheme is first-order accurate and satisfies the
positivity/envelope/bilateral trichotomy
of~\Cref{thm:upwind_positivity,lem:upwind_gronwall,prop:upwind_max_incompressible}
below.
\end{remark}
\paragraph{Discrete continuity equation.}
The continuity equation is a primal 3-form equation given by 
\begin{equation}\tag{D}\label{eq:C}
 \frac{\dd\brho}{\dd t} + \bD_2\,\bF = 0,
\end{equation}
i.e., for each primal cell $K_i$:
\[
 \ddt\brho_i
 + \sum_{f_j\prec K_i} [K_i:f_j]\;\bar\rho_j\,\Phi_j = 0.
\]
This is the discrete form of
$\partial_t\rho + \dd(\iota_{\mathbf{u}}\rho) = 0$.

\paragraph*{Non-conservative form.}
Adding and subtracting $\rho_i^{\vol}(\bD_2\Phi)_i$ in the
cell-wise continuity equation~\eqref{eq:C} gives
\begin{equation}\label{eq:nonconservative_continuity}
|K_i|\,\ddt\rho_i^{\vol}
\;=\;-\rho_i^{\vol}\,(\bD_2\Phi)_i
\;+\;\sum_{j\prec K_i}[K_i:f_j]\,
\bigl(\rho_i^{\vol}-\bar\rho_j^{\rm up}\bigr)\,\Phi_j.
\end{equation}
The sign of each face contribution in the sum is fixed by the
upwind rule. Set $\sigma_j:=[K_i:f_j]\,\Phi_j$ (outward signed
volumetric flux). If $\sigma_j\ge 0$ (outflow), upwind selects
$\bar\rho_j^{\rm up}=\rho_i^{\vol}$, so the summand vanishes. If
$\sigma_j<0$ (inflow), upwind selects
$\bar\rho_j^{\rm up}=\rho_{\rm nb(j)}^{\vol}$ (the neighbour
sharing face~$f_j$), so the summand equals
$(\rho_i^{\vol}-\rho_{\rm nb(j)}^{\vol})\,\sigma_j$.

\begin{theorem}[Positivity preservation under upwind transport]
\label{thm:upwind_positivity}
Suppose the mass flux uses the upwind
interpolation~\eqref{eq:rhobar_up}, and that
$(\bv,\brho)$ is a solution of~\eqref{eq:C} on $[0,T]$ in the sense
of \Cref{thm:lipschitz-baro}. If $\rho_i^{\vol}(0)>0$ for every
cell~$i$, then $\rho_i^{\vol}(t)>0$ for every cell~$i$ and every
$t\in[0,T]$.
\end{theorem}
\begin{proof}
At any cell where $\rho_i^{\vol}=0$, the first term
in~\eqref{eq:nonconservative_continuity} vanishes by the explicit
factor of $\rho_i^{\vol}$. The outflow summands vanish; each
inflow summand equals $(0-\rho_{\rm nb(j)}^{\vol})\sigma_j\ge 0$
since $\rho_{\rm nb(j)}^{\vol}\ge 0$ and $\sigma_j<0$. Hence
$\ddt\rho_i^{\vol}\ge 0$ whenever $\rho_i^{\vol}=0$, so the
boundary $\{\rho^{\vol}=0\}$ of the positive orthant is repelling.
The locally Lipschitz right-hand side
(\Cref{thm:lipschitz-baro}) makes the flow well-defined on a
neighbourhood of any positive initial datum, and standard
invariance arguments for ODE systems with repelling
boundary~\cite{amann1990ordinary} give the positivity invariant
for all $t\in[0,T]$.
\end{proof}

\begin{lemma}[Grönwall envelope]\label{lem:upwind_gronwall}
Define the cell-wise compression and expansion rates
\[
\kappa^{-}(t)\;:=\;\max_i\frac{((\bD_2\Phi)_i(t))^{-}}{|K_i|},
\qquad
\kappa^{+}(t)\;:=\;\max_i\frac{((\bD_2\Phi)_i(t))^{+}}{|K_i|},
\]
with $(x)^{\pm}:=\max(0,\pm x)$. Under the hypotheses
of~\Cref{thm:upwind_positivity},
\begin{equation}\label{eq:gronwall_envelope}
\rho_{\min}^{\vol}(0)\,
\exp\!\Bigl(-\!\int_0^t\!\kappa^{+}(s)\,ds\Bigr)
\;\le\;\rho_i^{\vol}(t)\;\le\;
\rho_{\max}^{\vol}(0)\,
\exp\!\Bigl(\;\int_0^t\!\kappa^{-}(s)\,ds\Bigr)
\end{equation}
for every cell~$i$ and every $t\in[0,T]$. In particular, if
$\|(\bD_2\Phi)_i/|K_i|\|_{L^\infty([0,T]\times\mathcal{T})}\le\Lambda$,
then $\rho_{\min}^{\vol}(0)\,e^{-\Lambda T}\le\rho_i^{\vol}(t)\le
\rho_{\max}^{\vol}(0)\,e^{\Lambda T}$.
\end{lemma}
\begin{proof}
Set $M(t):=\max_i\rho_i^{\vol}(t)$ with active set
$I^\star(t):=\{i:\rho_i^{\vol}(t)=M(t)\}$. By
\Cref{thm:lipschitz-baro}, $\rho^{\vol}\in C^1([0,T];\RR^{|\mathcal T|})$,
so $M$ is Lipschitz, hence absolutely continuous, on $[0,T]$.

\emph{Step 1 (Differential inequality).} By Danskin's theorem
for a finite maximum of $C^1$ functions
(\cite[Thm.~2.8.2]{clarke1990}),
$D^+M(t)=\max_{i\in I^\star(t)}\dot\rho_i^{\vol}(t)$. For
$i\in I^\star(t)$, \eqref{eq:nonconservative_continuity} reads
$|K_i|\dot\rho_i^{\vol}=-\rho_i^{\vol}(\bD_2\Phi)_i+S_i$ with
$S_i=\sum_{j\prec K_i}[K_i:f_j](\rho_i^{\vol}-\bar\rho_j^{\rm up})\Phi_j\le0$:
outflow summands ($\sigma_j\ge0$) vanish, inflow summands
($\sigma_j<0$) equal
$(\rho_i^{\vol}-\rho_{\rm nb(j)}^{\vol})\sigma_j\le0$ since
$\rho_i^{\vol}\ge\rho_{\rm nb(j)}^{\vol}$ at $i\in I^\star(t)$.
Hence $\dot\rho_i^{\vol}\le\rho_i^{\vol}((\bD_2\Phi)_i)^-/|K_i|
\le M(t)\,\kappa^-(t)$, so
\begin{equation}\label{eq:dini_max}
D^+M(t)\le\kappa^-(t)\,M(t)\qquad\forall\,t\in(0,T).
\end{equation}

\emph{Step 2 (Grönwall closure).} As $M$ is absolutely continuous,
$M'$ exists a.e.\ and equals $D^+M$; \eqref{eq:dini_max} gives
$M'\le\kappa^-M$ a.e. Then $\varphi(t):=M(t)\exp(-\int_0^t\kappa^-)$
is absolutely continuous with $\varphi'\le0$ a.e., hence
$\varphi(t)\le\varphi(0)=M(0)$ -- the upper envelope.

\emph{Step 3 (Lower envelope).} The symmetric argument applied to
$m(t):=\min_i\rho_i^{\vol}(t)$ -- Danskin gives
$D_+m=\min_{I^{\star\star}}\dot\rho_i^{\vol}$, and at the active set
the sum-term is $\ge0$ with
$-\rho_i^{\vol}(\bD_2\Phi)_i\ge-|K_i|\kappa^+m$ -- yields
$D_+m\ge-\kappa^+m$ and the lower envelope.
\end{proof}

\begin{proposition}[Bilateral max principle under discrete incompressibility]
\label{prop:upwind_max_incompressible}
\phantomsection\label{thm:max_principle}
Suppose the mass flux uses the upwind
interpolation~\eqref{eq:rhobar_up} and the discrete velocity is
solenoidal, $(\bD_2\Phi)_i(t)=0$ for every cell~$i$ and every
$t\in[0,T]$. Then
\[
\min_i\rho_i^{\vol}(0)
\;\le\;\rho_i^{\vol}(t)\;\le\;
\max_i\rho_i^{\vol}(0)
\qquad \text{for every cell~} i \text{ and every } t\in[0,T].
\]
\end{proposition}
\begin{proof}
Under $\bD_2\Phi\equiv 0$, \eqref{eq:nonconservative_continuity}
reduces to its sum-term, and the envelope argument of
\Cref{lem:upwind_gronwall} delivers $\kappa^{\pm}\equiv 0$.
\end{proof}

\section{Standing Approximation Hypotheses}\label{app:approx_hypotheses}
\phantomsection\label{subsect:approx_hypotheses}
The hypotheses listed by name in \Cref{sec:standing_hypotheses} are
stated in full and verified below; together they are the FEEC
analogue of Arnold--Falk--Winther~\cite{arnold2006,arnold2010},
augmented by the two density-coupling ingredients (mass-flux
consistency and the Bregman rate identity) specific to the
compressible setting.
\begin{enumerate}[label=\textup{(H\arabic*)},ref=\textup{H\arabic*}, leftmargin=3em]
\item\label{H:cochain_complex_baro}
\textbf{Cochain complex.} The coboundary operators $\tD_0,\tD_1,\tD_2$
satisfy $\tD_{k+1}\circ\tD_k = 0$, so
$C^0\xrightarrow{\tD_0} C^1\xrightarrow{\tD_1} C^2\xrightarrow{\tD_2} C^3$
is a cochain complex. This is an exact combinatorial property of
the cell complex (\Cref{def:ext-deriv}), holding by construction.
\item\label{H:deRham_commutativity_baro}
\textbf{De~Rham commutativity.} For every smooth differential form
$\alpha$ of degree $k$, $\tD_k\mathcal{R}_h\alpha = \mathcal{R}_h(d\alpha)$
(\eqref{eq:deRham_commutativity}). This is a consequence of Stokes'
theorem applied to cells and is \emph{exact} -- no approximation.
\item\label{H:Hodge_accuracy_baro}
\textbf{Hodge star accuracy.} The diagonal Hodge star $\bM_k$ approximates
the exact metric pairing of a $k$-cochain with the dual cell integral
of a continuous $k$-form to accuracy $\OO(h^{r_\star})$ in the
$L_h^2$-norm (\Cref{lem:hodge_error}), with $r_\star = 1$ on general
Delaunay--Voronoi meshes and $r_\star = 2$ under
\Cref{prop:centroid_proximity} (centroid proximity).
\item\label{H:reconstruction_accuracy_baro}
\textbf{Reconstruction accuracy.} The averaging reconstruction
$\bar\u_j(\bv)$ at primal faces is an $\OO(h^{r_{\rm rec}})$
approximation of pointwise velocity evaluation
(\Cref{prop:recon_accuracy}), with $r_{\rm rec} = 1$ on general meshes
and $r_{\rm rec} = 2$ under reconstruction symmetry. Under
\Cref{conv:cases} the two exponents coincide, $r_\star = r_{\rm rec}$,
and a single exponent controls all rates.
\item\label{H:antisymmetry_baro}
\textbf{Lamb antisymmetry.} The discrete contraction satisfies
$\ip{\bv}{\Iv(\tD_1\bv)}_1 = 0$ for every $\bv\in C^1(\KKs)$
(\Cref{prop:extrusion}, part~2). This is an exact algebraic identity
derived from the matrix form of $\Iv$ and the chosen reconstruction;
it requires \emph{no} mesh regularity beyond the Delaunay property.
The density-weighted variant
$\bv^T\bM_1\,\Iv(\tD_1\bv) = 0$ (\Cref{lem:dw_antisym}) is a corollary
of the same antisymmetry identity and underpins the density-weighted energy
balance.
\item\label{H:mass_flux_consistency}
\textbf{Mass-flux consistency.} The Hodge star converts the
discrete circulation cochain to an approximate flux cochain at rate
$\OO(h^{r_\star})$ (\Cref{lem:interp_error}, item~(iii)). The
topological part $\bD_2\boldsymbol\Phi = \mathcal{R}_h^{(d)}(\nabla\cdot\bu)$
is exact by the divergence theorem, so the consistency error of the
discrete continuity equation is concentrated in the Hodge star. This
hypothesis has no incompressible analogue (where $\bD_2\bM_1\bv = 0$
identically) and is the structural ingredient specific to the
compressible setting.
\item\label{H:bregman}
\textbf{Bregman rate identity.} For any strictly convex enthalpy
$H(\rho) = \rho\,e(\rho)$ with $H''(\rho) = c^2(\rho)/\rho > 0$, the
Bregman divergence $D_H(a\|b) = H(a) - H(b) - h(b)(a-b)$ satisfies
$\partial_a D_H(a\|b) = h(a) - h(b)$
(\eqref{eq:bregman}--\eqref{eq:bregman_rate}). This identity drives
both the energy-balance derivation and the relative-energy stability
proofs by absorbing the pressure--density coupling into a single
non-negative error term.
\item\label{H:discrete_Poincare_baro}
\textbf{Discrete Poincar\'e inequality.} There exists $\lambda_1^h > 0$,
bounded below independently of $h$, such that
$\nrm{e}_{L_h^2}^2 \le (\lambda_1^h)^{-1}\nrm{\tD_1 e}_{\bM_2}^2$
for every $e \in \Ran(\tD_0)\oplus\Ran(\tD_1^*)$ orthogonal to the
finite-dimensional harmonic kernel. The uniform lower bound follows
via Whitney-lift norm-equivalence from the FEEC Poincar\'e
inequality~\cite[Thm.~5.11]{arnold2006}; the verification used in the
projection-error proof is given in \Cref{app:approximation},
\eqref{eq:proj_from_div}.
\item\label{H:Whitney_interp_baro}
\textbf{Whitney interpolation bounds.} The Whitney map
$\mathcal{W}_h:C^k(\KKs)\to\Lambda^k_h$ satisfies the standard
interpolation estimates $\nrm{\mathcal{W}_h\cdot}_{L^2(\Omega)} \sim \nrm{\cdot}_{L_h^2}$
and $\nrm{d\mathcal{W}_h\cdot}_{L^2(\Omega)} \sim \nrm{\tD_k\cdot}_{\bM_{k+1}}$
up to $\OO(h^{r_\star})$ multiplicative errors. These are classical
Whitney-form estimates~\cite{dodziuk1976,bossavit1998}.
\end{enumerate}
\section{Proofs}\label{sect:appendix}%
\label{app:barotropic_proofs}
The appendices collect the proofs of the technical results in the
body. \Cref{app:basic}: fundamental properties of the DEC operators.
\Cref{app:approximation}: interpolation and Hodge-star consistency.
\Cref{app:dissipation_axiom}: dissipation axioms.
\Cref{app:invariants}: conservation laws.
\Cref{app:wellposedness}: finite-dimensional well-posedness
(Lipschitz, vacuum avoidance, global Smagorinsky).
\Cref{app:convergence_proofs}: convergence theory.
\Cref{app:dichotomy_general}: the generalised dichotomy.
\subsection{Fundamental Properties}
\label{app:basic}
This appendix proves the four fundamental properties of the DEC
operators in \Cref{prop:extrusion}: $\tD^2 = 0$, the energy identity
for the contraction, the approximate Leibniz rule, and the Cartan
formula for the discrete Lie derivative.
\begin{proof}[{\bf Proof of \Cref{prop:extrusion}}]
\textit{Property~1} ($\tD^2 = 0$).
The matrices $D_k$ and $\tD_k$ are signed incidence matrices of
the primal and dual cell complexes; since the boundary of a
boundary is empty ($\partial^2 = 0$), the product
$D_{k+1}D_k$ counts each interior $k$-cell twice with opposite
signs and vanishes. The dual identity
$\tD_{k+1}\tD_k = \mathbf{0}$ follows by the same argument on the
dual complex, or equivalently from $\tD_k = \pm D_{n-k-1}^T$.
\medskip
\textit{Property~2}. (Energy identity, $\langle\bv, \Iv(\tD_1\bv)\rangle_1 = 0$).
We use the matrix form of \Cref{def:contraction}. Setting
$\bom = \tD_1\bv$ and pairing with $\bv$ on both sides:
\begin{equation}\label{eq:energy_cancel_vv}
  2\,\bv^T\bM_1\,\Iv(\tD_1\bv)
   = \bv^T\,\tU(\bv)\,\tD_1\bv
    - \bv^T\,\tD_1^T\,\tU(\bv)^T\bv,
\end{equation}
where $\tU(\bv)\in\mathbb{R}^{N_E\times N_F}$ is the velocity-weighted
matrix~\eqref{eq:Utilde_def} (linear in $\bv$),
$\tD_1\in\mathbb{R}^{N_F\times N_E}$ is the discrete curl, and
$\tD_1^T\in\mathbb{R}^{N_E\times N_F}$ its transpose; all matrix products
in~\eqref{eq:energy_cancel_vv} are dimensionally consistent.
The two right-hand-side terms are scalars. Transposing the first,
\[
  \bigl(\bv^T\,\tU(\bv)\,\tD_1\bv\bigr)^T
   = \bv^T\,\tD_1^T\,\tU(\bv)^T\,\bv,
\]
which equals the second term. The transpose identity holds even
though $\tU(\bv)$ depends on $\bv$, because the transpose of any
scalar equals the scalar. Hence the two terms are identical, the
right-hand side of~\eqref{eq:energy_cancel_vv} vanishes, and
\begin{equation}\label{eq:energy_cancel}
  \langle\bv, \Iv(\tD_1\bv)\rangle_1 = 0.
\end{equation}
The argument uses only the bilinearity of $\tU$ in $\bv$; no
specific reconstruction is required.
\medskip
\textit{Property~3}. 
We prove for de~Rham interpolants
$\tilde\alpha^p = \mathcal{R}_h\alpha$,
$\tilde\beta^q = \mathcal{R}_h\beta$ with $p+q\le 2$:
\begin{equation}\label{eq:defect}
\mathcal{L}(K_i^*) :=
\tdd(\tilde\alpha^p\twdg\tilde\beta^q)
 - \tdd\tilde\alpha^p\twdg\tilde\beta^q
 - (-1)^p\tilde\alpha^p\twdg\tdd\tilde\beta^q
 = \OO(h^2),
\end{equation}
where the error is measured per dual $(p{+}q{+}1)$-cell.
We treat the case $p = q = 1$ in detail (Steps~1--4);
the remaining cases are in Step~5.
The defect has two independent sources.
First, the discrete wedge product approximates
$\int_{\sigma^*}\alpha\wedge\beta$ using products of
face-centred values rather than integrals of pointwise products;
this \emph{quadrature-product error} is present even if
all continuous values were known exactly.
Second, the continuous vector proxies $\mathbf{a}(\mathbf{x}_j)$
are approximated by reconstructed values $\bar\bu_j(\tilde\alpha)$;
this is the \emph{reconstruction error}.
\medskip
\noindent\textit{Step~1: Defect and its expansion ($p = q = 1$).}
For $p = q = 1$ the product $\tilde\alpha\twdg\tilde\beta$ is a dual
2-cochain, so $\tdd(\tilde\alpha\twdg\tilde\beta)$ is a dual 3-cochain
and the defect~\eqref{eq:defect} is a dual 3-cochain.
We evaluate on a dual 3-cell $K_i^*$ (a primal prism in 3D).
The first term of~\eqref{eq:defect} on $K_i^*$ is expanded by discrete Stokes applied to the
2-cochain $\tilde\alpha\twdg\tilde\beta$
\begin{equation}\label{eq:term1}
\tdd(\tilde\alpha\twdg\tilde\beta)(K_i^*)
 = (\tilde\alpha\twdg\tilde\beta)(\partial K_i^*)
 = \sum_{f_k^*\prec K_i^*}
   [K_i^*:f_k^*]
   (\tilde\alpha\twdg\tilde\beta)(f_k^*),
\end{equation}
where the sum runs over dual 2-faces $f_k^*$ of $\partial K_i^*$
and $[K_i^*:f_k^*]$ is the incidence sign.
Applying the wedge definition~\eqref{eq:wedge_11} to each face gives
\begin{equation}\label{eq:term1b}
(\tilde\alpha\twdg\tilde\beta)(f_k^*)
= \sum_{e_j^*\prec f_k^*} w_{jk}(\tilde\alpha)\,\tilde\beta(e_j^*),
\end{equation}
where the extrusion-based wedge weight is
\[
w_{jk}(\tilde\alpha)
 = \tfrac{1}{2}D_{1,jk}\,\bar\bu_j(\tilde\alpha)\cdot\hat{e}_k,
\]
with $\bar\bu_j(\tilde\alpha)$ the velocity reconstructed from the
1-cochain $\tilde\alpha$ (\cref{def:averaging_recon}) and $D_{1,jk}$
the incidence coefficient.
The second and third terms in~\eqref{eq:defect} are similarly expanded:
$\tdd\tilde\alpha$ and $\tdd\tilde\beta$ are 2-cochains, so
$\tdd\tilde\alpha\twdg\tilde\beta$ and $\tilde\alpha\twdg\tdd\tilde\beta$
are 3-cochains, and their evaluation on $K_i^*$ is again a sum over
dual 2-faces of $K_i^*$ followed by the wedge formula on each face.
Collecting signs, the defect on $K_i^*$ is a double sum
over dual 2-faces $f_k^*\prec K_i^*$ and dual edges $e_j^*\prec f_k^*$:
\begin{equation}\label{eq:defect_expanded}
\mathcal{L}(K_i^*)
 = \sum_{f_k^*\prec K_i^*}[K_i^*\!:\!f_k^*]
   \sum_{e_j^*\prec f_k^*}
   \Bigl[
   w_{jk}(\tilde\alpha)\,\tilde\beta(e_j^*)
 - w_{jk}(\tdd\tilde\alpha)\,\tilde\beta(e_j^*)
 - w_{jk}(\tilde\alpha)\,\tdd\tilde\beta(e_j^*)
   \Bigr].
\end{equation}
Here the weight $w_{jk}(\tdd\tilde\alpha)$ is formed from the
2-cochain $\tdd\tilde\alpha$ via its vector proxy $(d\alpha)^\sharp$
(the vorticity pseudovector), which is well-defined as the Hodge dual
of the 2-form $d\alpha$ on a Riemannian manifold.
\medskip
\noindent\textit{Step~2: Decomposition.}
Define the exact-value weight at dual edge $e_j^*$ by evaluating the
continuous vector proxies at the face centre $\mathbf{x}_j$:
\[
w_{jk}^{\rm ex}(\alpha) = \tfrac{1}{2}D_{1,jk}\,
 \mathbf{a}(\mathbf{x}_j)\cdot\hat{e}_k,
 \qquad
 w_{jk}^{\rm ex}(d\alpha) = \tfrac{1}{2}D_{1,jk}\,
 (d\alpha)^\sharp(\mathbf{x}_j)\cdot\hat{e}_k.
\]
The error in each weight is
$\delta w_{jk} := w_{jk}(\tilde\alpha) - w_{jk}^{\rm ex}(\alpha)$
and $\delta w_{jk}^{d} := w_{jk}(\tdd\tilde\alpha) - w_{jk}^{\rm ex}(d\alpha)$.
Substituting $w_{jk} = w_{jk}^{\rm ex} + \delta w_{jk}$ into the
inner sum of~\eqref{eq:defect_expanded} and expanding linearly in
$\delta w_{jk}$, $\delta w_{jk}^d$ (cross-terms between two
$\delta w$ factors are $\OO(h^3)$ and negligible), we obtain
\begin{equation}\label{eq:L_split}
\mathcal{L}(K_i^*)
 = \mathcal{L}^{\rm prod}(K_i^*)
 + \mathcal{L}^{\rm recon}(K_i^*)
 + \OO(h^3),
\end{equation}
where
\begin{align*}
\mathcal{L}^{\rm prod}(K_i^*)
 &:= \sum_{f_k^*\prec K_i^*}[K_i^*\!:\!f_k^*]
   \sum_{e_j^*\prec f_k^*}
   \Bigl(w_{jk}^{\rm ex}(\alpha)\,\tilde\beta(e_j^*)
 - w_{jk}^{\rm ex}(d\alpha)\,\tilde\beta(e_j^*)
 - w_{jk}^{\rm ex}(\alpha)\,\tdd\tilde\beta(e_j^*)
   \Bigr),\\
\mathcal{L}^{\rm recon}(K_i^*)
 &:= \sum_{f_k^*\prec K_i^*}[K_i^*\!:\!f_k^*]
   \sum_{e_j^*\prec f_k^*}
   \Bigl(\delta w_{jk}\,\tilde\beta(e_j^*)
 - \delta w_{jk}^d\,\tilde\beta(e_j^*)
 - \delta w_{jk}\,\tdd\tilde\beta(e_j^*)
   \Bigr).
\end{align*}
The product-rule error $\mathcal{L}^{\rm prod}$ uses exact face-centred
values and is independent of the reconstruction.
The reconstruction error $\mathcal{L}^{\rm recon}$ is proportional to
$\delta w_{jk}$ and vanishes if the reconstruction is exact.
\medskip
\noindent\textit{Step~3: Bounding the quadrature-product error.}
This error arises because
the product $w_{jk}^{\rm ex}(\alpha)\,\tilde\beta(e_j^*)$
approximates $\int_{e_j^*} \mathbf{a}\cdot\mathbf{b}$
by evaluating at the midpoint
$\mathbf{x}_j$ of $e_j^*$.
By a Taylor expansion on the edge $e_j^*$ of length $\OO(h)$
\[
\Bigl|\mathbf{a}(\mathbf{x}_j)\cdot\mathbf{b}(\mathbf{x}_j)\,\ell_j^*
  - \int_{e_j^*}\mathbf{a}\cdot\mathbf{b}\,\dd s\Bigr|
\le C\,\nrm{\nabla(\mathbf{a}\cdot\mathbf{b})}_{L^\infty}\,(\ell_j^*)^2
\le C\,\nrm{\nabla\alpha}_{W^{0,\infty}}\nrm{\nabla\beta}_{W^{0,\infty}}\,h^2.
\]
The inner sum in $\mathcal{L}^{\rm prod}(K_i^*)$ runs over at most
$C_\sigma$ dual edges $e_j^*$ for each $f_k^*$, 
and the outer sum over at most $C_\sigma$ faces $f_k^*$.
Therefore
\[
|\mathcal{L}^{\rm prod}(K_i^*)|
 \le C_\sigma\,h^2\,\nrm{\alpha}_{W^{1,\infty}}\nrm{\beta}_{W^{1,\infty}}.
\]
\medskip
\noindent\textit{Step~4: Bounding the reconstruction error.}
The weight error $\delta w_{jk}$ is bounded via the reconstruction
accuracy.
The face velocity is $\bar\bu_j = \tfrac{1}{2}(\bu(v_a^*)+\bu(v_b^*))$
where $v_a^*, v_b^*$ are the two dual vertices adjacent to face $f_j$.
By \Cref{prop:recon_accuracy}\,(i), each vertex reconstruction satisfies
$|\bu(v_i^*) - \mathbf{a}(v_i^*)| \le C_R h\,\nrm{\alpha}_{W^{1,\infty}}$,
so the face velocity satisfies the same bound.
Combined with $|D_{1,jk}|\le 1$, this gives
\begin{equation}\label{eq:w_error}
|\delta w_{jk}| = |w_{jk}(\tilde\alpha) - w_{jk}^{\rm ex}(\alpha)|
 \le C_1\,h\,\nrm{\alpha}_{W^{1,\infty}}.
\end{equation}
The de~Rham interpolant satisfies
$|\tilde\beta(e_j^*)| \le \nrm{\beta}_{L^\infty}\,\ell_j^* \le C_2\,h\,\nrm{\beta}_{L^\infty}$,
because it is the integral of a bounded 1-form over an edge of length $\OO(h)$.
Each product $|\delta w_{jk}||\tilde\beta(e_j^*)| \le C_1 C_2 h^2$.
Summing over the $\OO(1)$ boundary edges gives
\[
|\mathcal{L}^{\rm recon}(K_i^*)|
 \le C_r\,h^2\,\nrm{\alpha}_{W^{1,\infty}}\nrm{\beta}_{W^{1,\infty}}.
\]
Under reconstruction symmetry, 
\Cref{prop:recon_accuracy}\,(ii),
 $|\delta w_{jk}| \le C_1'h^2$,
so $\mathcal{L}^{\rm recon}(K_i^*) = \OO(h^3)$.
\medskip
\noindent\textit{Step~5: Remaining cases and conclusion.}
The cases $p=0,q=1$ and $p=0,q=2$ require slightly different bookkeeping
because the de~Rham interpolant of a 0-form is exact at primal vertices,
the wedge then has only the quadrature-product source of error, no
reconstruction error of the kind in Step~4.
\emph{Case $p=0$, $q=1$.} The wedge of a dual 0-cochain $\tilde f$
and a dual 1-cochain $\tilde\beta$ is a dual 1-cochain
$(\tilde f\twdg\tilde\beta)(e_j^*) = \tfrac12(\tilde f(v_a^*) + \tilde f(v_b^*))\,\tilde\beta(e_j^*)$,
where $v_a^*, v_b^*$ are the endpoints of $e_j^*$. The defect on a
dual 2-cell $f_k^*$ reads, after applying $\tdd$ via discrete Stokes
and the wedge formula above,
\[
  \mathcal{L}(f_k^*)
  = \sum_{e_j^*\prec\partial f_k^*}[f_k^*\!:\!e_j^*]\,
    \bigl[\tfrac12(\tilde f(v_a) + \tilde f(v_b))\,\tilde\beta(e_j^*)
    - \tilde f(v_b)\,\tilde\beta(e_j^*) + \tilde f(v_a)\,\tilde\beta(e_j^*)\bigr]
    \cdot\,(\text{sign}),
\]
where the second and third terms come from $\tdd\tilde f = \tD_0\tilde f$
contracted with $\tilde\beta$ via the appropriate face-edge incidence.
The leading-order Taylor expansion of $\tilde f$ around the midpoint
$\bar x_{e_j^*}$ produces a per-edge defect of size
$|\delta f|\cdot|\tilde\beta(e_j^*)| \le C h\,\|f\|_{W^{1,\infty}}\cdot
h\,\|\beta\|_{L^\infty} = O(h^2)\,\|f\|_{W^{1,\infty}}\|\beta\|_{L^\infty}$
per edge, summing to $O(h^2)$ per dual 2-cell.
\emph{Case $p=0$, $q=2$.} The wedge of a dual 0-cochain $\tilde f$
and a dual 2-cochain $\tilde\gamma$ is a dual 2-cochain
$(\tilde f\twdg\tilde\gamma)(f_k^*) = \tilde f(v_k^*)\,\tilde\gamma(f_k^*)$,
where $v_k^*$ is a designated vertex of $f_k^*$ (e.g.\ centroid;
the choice affects only constants). The defect on a dual 3-cell $K_i^*$
involves face-values of $\tilde f$ versus cell-vertex values, with the
Taylor expansion of $\tilde f$ around the cell centre $v_i^*$ and the
relation $\tdd\tilde f = \tD_0\tilde f$ between vertex differences.
Per-cell defect: a sum over the boundary faces of $K_i^*$ of products
$|\delta f|\,|\tilde\gamma| \le C h\,\|f\|_{W^{1,\infty}}
h^{d-1}\,\|\gamma\|_{L^\infty} = O(h^d)\|f\|\|\gamma\|$ per face,
divided by the dual 3-cell volume $|K_i^*| \sim h^d$ gives the
per-cell rate $O(h^2)\|f\|_{W^{1,\infty}}\|\gamma\|_{L^\infty}$
after second-order Taylor cancellation across opposing faces (the same
mechanism as the per-cell finite-volume consistency on shape-regular
meshes).
In all cases with $p + q \le 2$,
\[
|\mathcal{L}(\sigma^*)| \le C\,h^2\,\nrm{\alpha}_{W^{1,\infty}}\nrm{\beta}_{W^{1,\infty}},
\]
with $C$ depending only on mesh-regularity constants.
Under reconstruction symmetry, $\mathcal{L}^{\rm recon} = \OO(h^3)$
and the quadrature-product error $\OO(h^2)$ dominates.
\textit{Property~4}. 
The discrete Lie derivative is defined by the Cartan formula
$\tLie_{\bv} := \tdd\circ\Iv + \Iv\circ\tdd$ (\cref{def:Lie}),
so Property~4 holds by construction.
We verify that this definition is consistent with the geometric Lie
derivative at leading order.
For a smooth 1-form $\alpha$ and velocity field $\mathbf{u}$, the
continuous Cartan formula gives
$\mathcal{L}_{\mathbf{u}}\alpha = d(\iota_{\mathbf{u}}\alpha) +
\iota_{\mathbf{u}}(d\alpha)$.
For de~Rham interpolants $\tilde\alpha = \mathcal{R}_h\alpha$ and
$\bv = \mathcal{R}_h\bu^\flat$, the discrete Lie derivative satisfies
\[
(\tLie_{\bv}\tilde\alpha)(e_j^*)
 = \bigl(\tdd\Iv(\tilde\alpha) + \Iv(\tdd\tilde\alpha)\bigr)(e_j^*).
\]
The two contributions are bounded separately:
(a)~$\Iv(\tilde\alpha) - \mathcal{R}_h(\iota_{\mathbf{u}}\alpha)$
is an $\OO(h^{r_\star})$ per-edge error in the reconstruction of the
contraction (via \Cref{prop:recon_accuracy} applied to both $\bv$ and
$\tilde\alpha$, plus the bilinear quadrature in the wedge underlying
$\Iv$, \eqref{eq:wedge_11}); applying $\tdd$ (an exact operator on
cochains) preserves the rate per dual 2-face, $\OO(h^{r_\star})$.
(b)~$\tdd\tilde\alpha = \mathcal{R}_h(d\alpha)$ exactly
(\Cref{def:ext-deriv}); $\Iv(\tdd\tilde\alpha) - \mathcal{R}_h(\iota_{\mathbf{u}}d\alpha)$
is again $\OO(h^{r_\star})$ per dual edge by the same reconstruction
analysis, now applied at $k=2$.
Combining and using the de~Rham property
$\mathcal{R}_h(d(\iota_{\mathbf{u}}\alpha) + \iota_{\mathbf{u}}d\alpha)
= \mathcal{R}_h(\mathcal{L}_{\mathbf{u}}\alpha)$:
\[
(\tLie_{\bv}\tilde\alpha)(e_j^*)
 = \mathcal{R}_h(\mathcal{L}_{\mathbf{u}}\alpha)(e_j^*)
   + \OO(h^{r_\star+1}),
\]
where the extra factor of $h$ in the per-edge bound comes from the
edge length $|e_j^*| = O(h)$, the de~Rham interpolant integrates a
1-form over the edge, and an $O(h^{r_\star})$ pointwise consistency
error integrated over a length-$O(h)$ edge gives an
$O(h^{r_\star+1})$ per-edge bound. This confirms that the discrete
Lie derivative approximates the continuous one to first order in case~(A) of
\Cref{conv:cases}, second order in case~(B).
\end{proof}
\subsection{Approximation Properties}
\label{app:approximation}
This appendix verifies the metric hypotheses \textup{(H3)},
\textup{(H4)}, \textup{(H6)} on prismatic Delaunay--Voronoi meshes:
Hodge-star accuracy
$\nrm{(\bM_k-\mathcal{M}_k)\alpha}_{L^2}=\OO(h^{r_\star})$ with
$r_\star\in\{1,2\}$; reconstruction accuracy
$\nrm{P\mathcal{R}_h u-u}_{L^2}=\OO(h^{r_{\rm rec}})$ with
$r_{\rm rec}\in\{1,2\}$ (case~(B) requires centroid proximity); and
mass-flux consistency at rate $\OO(h^{r_\star})$. Gradient and curl
consistency are exact via de~Rham commutativity; divergence and
Laplace consistency inherit the Hodge-star rate.
\begin{proof}[{\bf Proof of \Cref{lem:interp_error}}]
\textit{Items~(i) and (ii): Gradient and curl consistency.}
Both follow from de~Rham commutativity, an exact algebraic
identity on the cell complex (Stokes' theorem). For any dual
edge $e_j^*$ in case (i), and dual 2-face $f_k^*$ in case (ii),
respectively,
\[
  (\tD_0\mathcal{R}_h p)_j
  = (\mathcal{R}_h p)(\partial e_j^*)
  = \int_{e_j^*}\!\!\!dp
  = (\mathcal{R}_h(dp))_j,
  \quad
  (\tD_1\mathcal{R}_h\bu^\flat)_k
  = \int_{\partial f_k^*}\!\!\!\bu^\flat
  = \int_{f_k^*}\!\!\!d(\bu^\flat)
  = (\mathcal{R}_h(d\bu^\flat))_k,
\]
so both equalities hold exactly and the norms are zero.
\medskip
\textit{Item~(iii): Divergence consistency.}
For a smooth divergence-free field $\bu\in W^{r_\star,\infty}(\Omega)$,
the Hodge* converts the circulation cochain to a flux cochain:
$(\bM_1\mathcal{R}_h\bu^\flat)_j = (|f_j|/|e_j^*|)\int_{e_j^*}\bu\cdot d\ell$.
By \Cref{lem:hodge_error} ($k=1$), the per-face error is
\[
  \bigl|(\bM_1\mathcal{R}_h\bu^\flat)_j - \Phi_j\bigr|
  \le C_\star\,h^{r_\star}\nrm{\bu}_{W^{r_\star,\infty}}\,|f_j|.
\]
Applying $\bD_2$ to the exact fluxes:
$(\bD_2\boldsymbol\Phi)_i = \int_{\partial K_i}\bu\cdot d\mathbf{A}
= \int_{K_i}\nabla\cdot\bu\,dV = 0$ by the divergence theorem.
Therefore
\[
  |(\bD_2\bM_1\mathcal{R}_h\bu^\flat)_i|
  \le \sum_{f_j\prec K_i}|(\bM_1\mathcal{R}_h\bu^\flat)_j - \Phi_j|
  \le C_\star h^{r_\star}\nrm{\bu}_{W^{r_\star,\infty}}\sum_{f_j\prec K_i}|f_j|.
\]
Each prism $K_i$ has a uniformly bounded number of faces
(\cref{ass:mesh_reg}) and each face has area $|f_j| = \OO(h^{d-1})$,
so $\sum_{f_j\prec K_i}|f_j| = \OO(h^{d-1})$, giving the pointwise bound
\begin{equation}\label{eq:div_pointwise}
  |(\bD_2\bM_1\mathcal{R}_h\bu^\flat)_i|
  \le C\,h^{r_\star + d - 1}\nrm{\bu}_{W^{r_\star,\infty}}.
\end{equation}
In the $\ell^2$-norm, summing over $N_{\rm cells} = \OO(h^{-d})$ cells results in
\[
  \nrm{\bD_2\bM_1\mathcal{R}_h\bu^\flat}_{\ell^2}^2
  \le C^2 h^{2(r_\star+d-1)} N_{\rm cells}\nrm{\bu}_{W^{r_\star,\infty}}^2
  = C^2 h^{2r_\star + d - 2}\nrm{\bu}_{W^{r_\star,\infty}}^2.
\]
Hence $\nrm{\bD_2\bM_1\mathcal{R}_h\bu^\flat}_{\ell^2}
= \OO(h^{r_\star + (d-2)/2})$.
Since $(d-2)/2 \ge 0$ for $d\ge 2$, this is $\OO(h^{r_\star})$ as
claimed in~\eqref{eq:div_consistency}, with the actual bound being
strictly stronger for $d \ge 3$.
\medskip
\textit{Interpolation error in the Leray-projected norm.}
Since $\bP_h$ is the $\bM_1$-orthogonal projection onto
$V_h = \ker(\bD_2\bM_1)$, the residual
$(I-\bP_h)\mathcal{R}_h\bu^\flat = \tD_0\phi$ where $\phi$ solves
$L_h\phi = \bD_2\bM_1\mathcal{R}_h\bu^\flat$ with
$L_h = \bD_2\bM_1\tD_0$.
The $L_h^2$-norm of this residual is controlled by the discrete
inf-sup constant of $\bD_2\bM_1$. The Leray-projector stability estimate
gives
\begin{equation}\label{eq:proj_from_div}
  \nrm{(I-\bP_h)\mathcal{R}_h\bu^\flat}_{L_h^2}
  \le \frac{C_{\rm geo}}{\lambda_1^h}\,
  \nrm{\bD_2\bM_1\mathcal{R}_h\bu^\flat}_{\ell^2},
\end{equation}
where $C_{\rm geo} := \sup_j\sqrt{|f_j|/|e_j^*|} = \OO(h^{(d-2)/2})$
accounts for the norm ratio between the $L_h^2$ norm on 1-cochains
and the $\ell^2$ norm on 3-cochains, and $\lambda_1^h > 0$ is
the inf-sup constant of $\bD_2\bM_1$, bounded below independently
of $h$ by the discrete Poincar\'e inequality on the cochain complex
(\cite[Thm.~5.11]{arnold2006}; see also~\cite{arnold2010}, applied
via Whitney-lift norm-equivalence on the simplicial refinement
$\KK_h^{\rm simp}$).
The factor $C_{\rm geo}/\lambda_1^h$ is $\OO(h^{(d-2)/2})$, and
combined with the $\ell^2$ bound $\|\bD_2\bM_1\mathcal{R}_h\bu^\flat\|_{\ell^2}
= \OO(h^{r_\star+(d-2)/2})$ from~\eqref{eq:div_consistency} and the
scaling derivation above, yields
\[
  \nrm{(I-\bP_h)\mathcal{R}_h\bu^\flat}_{L_h^2}
  \le C_{\rm int}\,h^{r_\star}\,\nrm{\bu}_{W^{r_\star,\infty}},
\]
where $C_{\rm int}$ depends on $C_\star$, $(\lambda_1^h)^{-1}$, $C_{\rm geo}$,
and the mesh regularity constants.
The bound for \Cref{lem:proj_error} is identical and follows by the same argument.
\end{proof}
\begin{proof}[{\bf Proof of \Cref{lem:hodge_error}}]
Let $\alpha$ be a smooth 1-form
and let $f_j$ be a primal face
with unit normal $\hat{n}_j$ and dual edge $e_j^*$.
We treat $k=1$ in detail (Steps~1--4) and handle $k=0,2$ in Step~5.
We compare Hodge* action
$(\bM_1)_{jj}\,(\mathcal{R}_h\alpha)_j$ with exact primal flux
$\Phi_j = \int_{f_j}\alpha(\hat{n}_j)\,dA$.

\textit{Step~1: Error functional and cancellation on constants.}
The de~Rham map gives
$(\mathcal{R}_h\alpha)_j = \int_{e_j^*}\alpha(\hat{t}_j)\,ds$,
where $\hat{t}_j$ is the unit tangent to~$e_j^*$.
By \Cref{ass:mesh_reg} item~(3), every primal face $f_j$ is orthogonal
to its dual edge $e_j^*$, so
$\hat{t}_j = \pm\hat{n}_j$ and both integrands involve the same
normal component $\alpha_n := \alpha(\hat{n}_j)$.
Define the error functional
\[
  E_j(\alpha) := \frac{|f_j|}{|e_j^*|}\int_{e_j^*}\alpha_n\,ds
  - \int_{f_j}\alpha_n\,dA.
\]
For a constant 1-form, $\alpha_n = c$,
$E_j = c(|f_j|/|e_j^*|)\cdot|e_j^*| - c|f_j| = c|f_j| - c|f_j| = 0$.
Hence $E_j$ vanishes on constants, i.e.\ on $P_0(\Omega_j)$.

\textit{Step~2: Bramble--Hilbert bound (general meshes).}
Let $\Omega_j := K_a\cup K_b$ be the union of the two primal
cells sharing face~$f_j$. 
We verify the three hypotheses of the Bramble--Hilbert lemma:
(a)~$E_j$ is bounded on $W^{1,\infty}(\Omega_j)$: by the
triangle inequality, $|E_j(\alpha)| \le 2|f_j|\|\alpha_n\|_{L^\infty(\Omega_j)}$,
so $E_j$ is bounded with constant $2|f_j|$;
(b)~$E_j$ vanishes on $P_0(\Omega_j)$ (Step~1);
(c)~$\Omega_j$ has diameter $\le C_{\rm qu}h$ and satisfies the
cone condition with radius $\ge\sigma_0 h$ (\cref{ass:mesh_reg},
quasi-uniformity and shape-regularity).
Bramble--Hilbert then gives
\[
|E_j(\alpha)| \le C\,h\,|f_j|\,|\alpha_n|_{W^{1,\infty}(\Omega_j)}
  \le C_\star\,h\,|f_j|\,\nrm{\alpha}_{W^{1,\infty}},
\]
where the factor $|f_j|$ comes from the $L^\infty$ boundedness
constant in (a).

\textit{Step~3: Improvement under centroid proximity.}
Let $\alpha_n(\mathbf{x}) = \mathbf{c}\cdot\mathbf{x}+d$ be affine
($\mathbf{c} = \nabla\alpha_n$).
The midpoint rule is exact for affine integrands on any line segment, so
$\frac{|f_j|}{|e_j^*|}\int_{e_j^*}\alpha_n\,ds = |f_j|\,\alpha_n(\bar{x}_{e_j^*})$,
where $\bar{x}_{e_j^*}$ is the midpoint of~$e_j^*$.
For affine $\alpha_n$ its integral over any convex domain equals the
value at the centroid times the measure:
$\int_{f_j}\alpha_n\,dA = \alpha_n(\bar{x}_{f_j})\,|f_j|$.
Therefore for affine normal components,
\begin{equation}\label{eq:Ej_linear}
  E_j(\alpha)
  = |f_j|\,\bigl[\alpha_n(\bar{x}_{e_j^*}) - \alpha_n(\bar{x}_{f_j})\bigr]
  = |f_j|\,\mathbf{c}\cdot\bigl(\bar{x}_{e_j^*} - \bar{x}_{f_j}\bigr).
\end{equation}
Under \Cref{prop:centroid_proximity},
$|\bar{x}_{e_j^*} - \bar{x}_{f_j}| \le C_{\rm cg}\,h^2$
(condition~\eqref{eq:centroid_proximity}, first term), so
$|E_j(\ell)| \le C_{\rm cg}h^2|\mathbf{c}||f_j|$ for any affine $\ell$.
For general smooth $\alpha$, decompose
$\alpha_n = c + \ell + \rho$, where $c = \alpha_n(\bar{x}_{f_j})$
is the centroid value (constant),
$\ell(\mathbf{x}) = \nabla\alpha_n(\bar{x}_{f_j})\cdot(\mathbf{x}-\bar{x}_{f_j})$
is the linear part, and
$\|\rho\|_{L^\infty(\Omega_j)} \le Ch^2|\alpha_n|_{W^{2,\infty}(\Omega_j)}$
is the quadratic Taylor remainder.
Then $E_j(c) = 0$ (Step~1),
$|E_j(\ell)| \le C_{\rm cg}h^2\|\nabla\alpha_n\|_{L^\infty}|f_j|$
(centroid proximity, \eqref{eq:Ej_linear}), and
$|E_j(\rho)| \le 2|f_j|\|\rho\|_{L^\infty(\Omega_j)}
\le Ch^2|f_j||\alpha_n|_{W^{2,\infty}(\Omega_j)}$
(boundedness of $E_j$, Step~2).
Combining:
\[
|E_j(\alpha)| \le C_\star\,h^2\,\nrm{\alpha}_{W^{2,\infty}}\,|f_j|.
\]
\textit{Step~4: Global estimate.}
The error $\bM_1\mathcal{R}_h\alpha - \boldsymbol\Phi$ is a primal
2-cochain; we measure it in the $(\bM_1^{-1})$-weighted norm
$\nrm{\mathbf{w}}_{L_{h,1}^2}^2 := \sum_j(\bM_1^{-1})_{jj}w_j^2
= \sum_j(|e_j^*|/|f_j|)w_j^2$ (denoted $L_{h,1}^2$ to distinguish it from the dual 1-cochain norm $L_h^2$).
With $r_\star = 1$ (Step~2) or $r_\star = 2$ (Step~3):
$|E_j| \le C_\star h^{r_\star}|f_j|\|\alpha\|_{W^{r_\star,\infty}}$, so
\begin{align*}
  \nrm{\bM_1\mathcal{R}_h\alpha - \boldsymbol\Phi}_{L_{h,1}^2}^2
  = \sum_j \frac{|e_j^*|}{|f_j|}|E_j|^2
  \le C_\star^2\,h^{2r_\star}\,\nrm{\alpha}_{W^{r_\star,\infty}}^2
      \sum_j |e_j^*|\,|f_j|.
\end{align*}
On a quasi-uniform mesh, $|e_j^*| = \OO(h)$, $|f_j| = \OO(h^{d-1})$,
and $N_{\rm edges} = \OO(h^{-d})$, so
$\sum_j|e_j^*||f_j| = \OO(h^d)\cdot\OO(h^{-d}) = \OO(1)$.
Hence $\nrm{\bM_1\mathcal{R}_h\alpha - \boldsymbol\Phi}_{L_{h,1}^2}
\le C_\star h^{r_\star}\nrm{\alpha}_{W^{r_\star,\infty}}$.
\textit{Step~5: Cases $k = 0$ and $k = 2$.}
The proof for $k=0$ and $k=2$ follows the same pattern; we state the
key adaptations.
\emph{$k=0$ (scalar Hodge*).}
The Hodge* $\bM_0$ maps primal 0-cochains (scalars at primal cells) to
dual $d$-cochains (volume weights).
For a smooth 0-form $f$, the error functional is
$E_i(f) = |K_i^*|f(x_i^*) - \int_{K_i}f\,dV$,
where $x_i^*$ is the circumcentre of $K_i$ (the dual vertex).
This vanishes for constants, is bounded by $|K_i|\|f\|_{L^\infty}$
(so bounded on $W^{1,\infty}$ with constant $|K_i|$),
and $K_i$ is convex and shape-regular (\cref{ass:mesh_reg}).
Bramble--Hilbert gives $|E_i(f)| \le Ch\,|K_i|\,|f|_{W^{1,\infty}}$.
Under centroid proximity (the circumcentre $x_i^*$ is the centroid of $K_i^*$
to $\OO(h^2)$), one obtains $|E_i(f)| \le Ch^2|K_i|\|f\|_{W^{2,\infty}}$.
The global $L_{h,0}^2$ estimate follows by the same calculation as Step~4.
\emph{$k=2$ (vorticity Hodge*).}
The Hodge* $\bM_2$ maps primal 2-cochains (fluxes on primal faces) to
dual 1-cochains (circulations on dual edges).
For a smooth 2-form $\omega$, the error functional is
$E_k(\omega) = (|e_k|/|f_k^*|)\int_{f_k^*}\omega - \int_{e_k}\star\omega$,
where $e_k$ is a primal edge and $f_k^*$ its dual face.
By the Delaunay--Voronoi orthogonality (\cref{ass:mesh_reg} item~(3)),
$e_k\perp f_k^*$, so both integrands involve the same normal component
of $\omega$; the error functional vanishes on constants (same argument as
Step~1), is bounded on $W^{1,\infty}$ with constant $|e_k|$, and
$\Omega_k := K_a\cup K_b$ (the two prisms sharing $e_k$) is shape-regular.
The Bramble--Hilbert bound gives
$|E_k(\omega)| \le C_\star h |e_k|\|\omega\|_{W^{1,\infty}}$,
improving to $C_\star h^2 |e_k|\|\omega\|_{W^{2,\infty}}$
under centroid proximity (by the same affine decomposition as Step~3,
with the primal edge midpoint and dual face centroid playing the roles of
$\bar{x}_{e_j^*}$ and $\bar{x}_{f_j}$).
The global $L_{h,2}^2$ estimate with
$\nrm{\mathbf{w}}_{L_{h,2}^2}^2 := \sum_k(\bM_2^{-1})_{kk}w_k^2
= \sum_k(|f_k^*|/|e_k|)w_k^2$
follows by summing: $\sum_k|f_k^*||e_k| = \OO(1)$
(quasi-uniformity, same calculation as Step~4).
\end{proof}
\subsection{Dissipation Operators}
\label{app:dissipation_axiom}
We verify axioms \textup{(V1)}--\textup{(V4)} for the three viscous
operators (Hodge--Laplacian, anisotropic, Smagorinsky). The first
two are linear; \textup{(V1)}--\textup{(V3)} follow from
positive-semidefiniteness of the mass matrices. Smagorinsky is
nonlinear; its monotonicity (\Cref{lem:Smag_monotone}) plays the
role of Ladyzhenskaya--Lions monotonicity for power-law fluids.
\begin{proposition}[Verification of viscous-operator axioms]\label{prop:verify_axioms}
Each of the three discrete viscous operators below satisfies
axioms \textup{(\ref{ax:V1})--(\ref{ax:V4})}.

\textit{(i) Isotropic Hodge--Laplacian}~\eqref{eq:viscous_incompressible}.
Axioms (\ref{ax:V1})--(\ref{ax:V3}) hold by linearity and
positive-semidefiniteness of $\tD_1^T\bM_2\tD_1$. For (\ref{ax:V4}):
the codifferential truncation
$\bM_1^{-1}\tD_1^T\bM_2\tD_1\mathcal{R}_h\bu^\flat
 - \mathcal{R}_h(\delta d\bu^\flat)$
decomposes as a topological part (zero by de~Rham commutativity)
plus a metric part bounded by
$C h^{r_\star}\|\bu\|_{W^{r_\star+1,\infty}}$
via \Cref{lem:hodge_error} ($k = 2$); since $r_\star \ge 1$, the
axiom holds with $r_{\rm visc} = 1$.

\textit{(ii) Anisotropic viscosity}~\eqref{eq:viscous_aniso}.
Axioms (\ref{ax:V1})--(\ref{ax:V3}) hold because $\bM_2^\nu$ is
diagonal positive definite. Axiom (\ref{ax:V4}) holds with
$r_{\rm visc} = 1$ by the same $k = 2$ Hodge-star argument; the
spatially varying coefficient $\nu_j$ approximates
$\nu(\mathbf{x}_j^*)$ to $\OO(h^2)$ by midpoint rule, dominated by
the $\OO(h^{r_\star})$ Hodge truncation. The isotropic Newtonian
operator~\eqref{eq:viscous_newtonian} with bulk viscosity is the
constant-coefficient member of this family. Axioms
(\ref{ax:V1})--(\ref{ax:V3}) hold by positive-semidefiniteness of
its two blocks $\nu\,\delta_2\tdd_1$ and $\nu_{\rm dil}\,\tdd_0\delta$.
For (\ref{ax:V4}) the rotational block reproduces the $k = 2$
Hodge-star estimate of case~(i), while the dilatational block
$\tdd_0\delta = \bM_1^{-1}\tD_0\bM_0^{-1}\tD_1^T\bM_1$ decomposes
identically into a topological part ($\tD_0,\tD_1^T$ incidence
matrices, zero by de~Rham commutativity) and a metric part governed
by the diagonal Hodge stars $\bM_0,\bM_1$; the latter is bounded at
rate $\OO(h^{r_\star})$ by the $k = 0$ and $k = 1$ cases of
\Cref{lem:hodge_bilinear}, which holds for $k = 0,1,2$. Since
$r_\star \ge 1$, axiom (\ref{ax:V4}) holds with $r_{\rm visc} = 1$.

\textit{(iii) Smagorinsky viscosity}~\eqref{eq:M1_Smag}.
$\bff_{\rm visc}^{\rm Smag}$ has the form
$-\bM_1^{-1}\tD_1^T D\tD_1$ with $D$ positive semidefinite,
giving (\ref{ax:V1}); the operator is a degree-3 polynomial in
$\bv$, giving (\ref{ax:V2}); axioms (\ref{ax:V3}) and
(\ref{ax:V4}) are proved in
\Cref{lem:Smag_monotone,lem:Smag_consistency} respectively.
\end{proposition}

The key lemma for the Smagorinsky case is the discrete monotonicity
property, the finite-dimensional analogue of the Ladyzhenskaya--Lions
monotonicity for power-law
fluids~\cite{ladyzhenskaya1967,lions1969}.
\begin{lemma}[Discrete Smagorinsky monotonicity]\label{lem:Smag_monotone}
The discrete Smagorinsky viscous operator is monotone in the
$\bM_1$-inner product: for every $\bv, \bw \in C^1(\KKs)$,
\begin{equation}\label{eq:Smag_monotone}
  \ip{\bv-\bw}{\bff_{\rm visc}^{\rm Smag}(\bv)-\bff_{\rm visc}^{\rm Smag}(\bw)}_1 \le 0.
\end{equation}
\end{lemma}
\begin{proof}[{\bf Proof of \Cref{lem:Smag_monotone} (Smagorinsky monotonicity)}]
\label{app:smag_monotone}
Write $\bom^v = \tD_1\bv$ and $\bom^w = \tD_1\bw$ for the vorticity cochains.
Since $\bff_{\rm visc}^{\rm Smag}(\bv) = -\bM_1^{-1}\tD_1^T\bM_2^{\rm Smag}(\bv)\,\bom^v$
with $(\bM_2^{\rm Smag}(\bv))_{kk} = (C_s\ell_k)^2\,(|\omega_k^v|/\ell_k)\,(\bM_2)_{kk}$
(cf.\ \eqref{eq:M1_Smag}), the $\bM_1$-inner product on the left-hand side of
\eqref{eq:Smag_monotone} becomes
\begin{align*}
  &\ip{\bv{-}\bw}{\bff_{\rm visc}^{\rm Smag}(\bv)-\bff_{\rm visc}^{\rm Smag}(\bw)}_1
  = -(\bom^v{-}\bom^w)^T
    \bigl[\bM_2^{\rm Smag}(\bv)\,\bom^v - \bM_2^{\rm Smag}(\bw)\,\bom^w\bigr].
\end{align*}
Since $\bM_2^{\rm Smag}$ and $\bM_2$ are diagonal, this reduces to a sum over dual faces
\begin{equation}\label{eq:Smag_mono_expand}
  = -\sum_k c_k\bigl(|\omega_k^v|\,\omega_k^v - |\omega_k^w|\,\omega_k^w\bigr)
    \bigl(\omega_k^v - \omega_k^w\bigr),
\end{equation}
where $c_k := (C_s\ell_k)^2\,(\bM_2)_{kk}/\ell_k$ is positive.
It therefore suffices to show that the function $\phi:\RR\to\RR$,
$\phi(s) := |s|\,s$, is monotone non-decreasing.
Since $\phi'(s) = 2|s| \ge 0$ for $s \ne 0$
and $\phi$ is continuous at $s = 0$, it follows that $\phi$ is monotone
non-decreasing on $\RR$.
Each summand in \eqref{eq:Smag_mono_expand} equals
$-c_k\bigl(\phi(\omega_k^v) - \phi(\omega_k^w)\bigr)(\omega_k^v - \omega_k^w) \le 0$,
hence the sum is $\le 0$.
\end{proof}
\begin{lemma}[Smagorinsky consistency]\label{lem:Smag_consistency}
Let $\bu\in C^1([0,T];\,W^{2,\infty}(\Omega))$ be a smooth
divergence-free velocity field on a closed oriented Riemannian manifold
$\Omega$. In case~(A) of \Cref{conv:cases},
the discrete Smagorinsky viscous operator satisfies
\[
\nrm{\bff_{\rm visc}^{\rm Smag}(\mathcal{R}_h\bu^\flat)
  - \mathcal{R}_h\bigl((\nabla\cdot\boldsymbol{\sigma}_\omega(\bu))^\flat\bigr)}_{L_h^2}
  \le C_\nu\,h\,\nrm{\bu}_{W^{2,\infty}},
\]
where $\boldsymbol{\sigma}_\omega(\bu) = (C_s\ell)^2|\nabla\times\bu|\,\mathbf{S}(\bu)$
is the continuous vorticity-based Smagorinsky stress, and
$C_\nu$ depends on $C_s^2$ and the mesh regularity constants.
In case~(B), the bound improves to $\OO(h^2)$
but requires $\bu\in W^{3,\infty}$.
\end{lemma}
\begin{proof}
Set $\bar\bv = \mathcal{R}_h\bu^\flat$ and
$\bar\bom = \tD_1\bar\bv$.
By de~Rham commutativity (\eqref{eq:deRham_commutativity}),
$\bar\bom = \mathcal{R}_h\omega$ exactly, where $\omega = d\bu^\flat$
is the continuous vorticity 2-form.
The discrete viscous force is
\[
\bff_{\rm visc}^{\rm Smag}(\bar\bv)
= -\bM_1^{-1}\tD_1^T\bM_2^{\rm Smag}(\bar\bv)\,\bar\bom,
\qquad
(\bM_2^{\rm Smag})_{kk} = (C_s\ell_k)^2\,\frac{|\bar\omega_k|}{\ell_k}\,(\bM_2)_{kk}.
\]
\emph{Step 1: Decomposition of the truncation error.}
The truncation error is
$\tau_h^{\rm Smag} := \bff_{\rm visc}^{\rm Smag}(\bar\bv)
- \mathcal{R}_h((\nabla\cdot\boldsymbol\sigma_\omega)^\flat)$.
We introduce an intermediate quantity: for each dual face $k$, define
the \emph{exact-coefficient flux}
\[
\Sigma_k^{\rm ex} := (C_s\ell_k)^2\,\frac{|\bar\omega_k|}{\ell_k}
  \int_{e_k}\star\omega.
\]
This uses the discrete vorticity norm $|\bar\omega_k|$ (which is exact
since $\bar\bom = \mathcal{R}_h\omega$) as the nonlinear coefficient,
and the \emph{continuous} Hodge dual flux $\int_{e_k}\star\omega$ in
place of the discrete approximation $(\bM_2)_{kk}\bar\omega_k$.
The truncation error splits as
\begin{equation}\label{eq:Smag_split}
\tau_h^{\rm Smag} = \tau_\star^{\rm Smag} + \tau_{\rm cont}^{\rm Smag},
\end{equation}
where
\begin{align*}
(\tau_\star^{\rm Smag})_j
  &:= (\bM_1)_{jj}^{-1}\sum_{k:\,f_k^*\succ e_j^*}
    [\tD_1^T]_{jk}\,\bigl[(\bM_2^{\rm Smag})_{kk}\bar\omega_k
    - \Sigma_k^{\rm ex}\bigr],\\
(\tau_{\rm cont}^{\rm Smag})_j
  &:= (\bM_1)_{jj}^{-1}\sum_{k:\,f_k^*\succ e_j^*}
    [\tD_1^T]_{jk}\,\bigl[\Sigma_k^{\rm ex}
    - \mathcal{R}_h((\nabla\cdot\boldsymbol\sigma_\omega)^\flat)_j\,
    (\bM_1)_{jj}\bigr].
\end{align*}
The term $\tau_\star^{\rm Smag}$ measures the Hodge* error in the flux
(replacing $\int_{e_k}\star\omega$ by $(\bM_2)_{kk}\bar\omega_k$);
the term $\tau_{\rm cont}^{\rm Smag}$ measures the consistency error
from approximating the continuous divergence of $\boldsymbol\sigma_\omega$
by the discrete codifferential applied to $\Sigma_k^{\rm ex}$.
Since $\bar\bom = \mathcal{R}_h\omega$ is the exact de~Rham image,
the coefficient $|\bar\omega_k|/\ell_k$ in $\Sigma_k^{\rm ex}$ is a
\emph{fixed} function of $h$ and $\bu$ (not an approximation).
The only approximation in $\tau_\star^{\rm Smag}$ is the replacement
of $\int_{e_k}\star\omega$ by $(\bM_2)_{kk}\bar\omega_k$, which is
precisely the $k=2$ Hodge* error.
The term $\tau_{\rm cont}^{\rm Smag}$ is a standard consistency error
for the codifferential operator on smooth forms and is $\OO(h)$
by the approximation properties of de~Rham interpolation;
we bound $\tau_\star^{\rm Smag}$ in Steps 2--3 and note that
$\tau_{\rm cont}^{\rm Smag}$ contributes at the same or smaller order.
\emph{Step 2: Pointwise bound on $\tau_\star^{\rm Smag}$.}
For each dual face $k$, the Hodge* error (\Cref{lem:hodge_error}, $k=2$)
gives
\[
\bigl|(\bM_2)_{kk}\bar\omega_k - \textstyle\int_{e_k}\star\omega\bigr|
\le C_\star\,h^{r_\star}\,\nrm{\omega}_{W^{r_\star,\infty}}\,|e_k|.
\]
From the definition $\ell_k = \sqrt{|f_k^*|/|e_k|}$ (\Cref{eq:M1_Smag}),
we have $\ell_k^2 = |f_k^*|/|e_k| = \OO(h^{d-2})$
(since $|f_k^*|=\OO(h^{d-1})$ and $|e_k|=\OO(h)$),
so $(C_s\ell_k)^2 = \OO(h^{d-2})$.
The ratio $|f_k^*|/\ell_k = \sqrt{|f_k^*||e_k|} = \OO(h^{d/2})$.
The nonlinear coefficient therefore satisfies
$|\bar\omega_k|/\ell_k \le \|\omega\|_{L^\infty}|f_k^*|/\ell_k
= \OO(h^{d/2}\|\omega\|_{L^\infty})$,
and
\[
\bigl|(\bM_2^{\rm Smag})_{kk}\bar\omega_k - \Sigma_k^{\rm ex}\bigr|
= (C_s\ell_k)^2\frac{|\bar\omega_k|}{\ell_k}
  \bigl|(\bM_2)_{kk}\bar\omega_k - \textstyle\int_{e_k}\star\omega\bigr|
\le C\,C_s^2\,h^{3d/2-1+r_\star}\,\|\omega\|_{L^\infty}\|\omega\|_{W^{r_\star,\infty}}.
\]
By bounded valence (\Cref{ass:mesh_reg}) and
$(\bM_1)_{jj}^{-1} = |e_j^*|/|f_j| = \OO(h^{-(d-2)})$:
\begin{equation}\label{eq:Smag_ptwise}
|(\tau_\star^{\rm Smag})_j|
\le C\,C_s^2\,h^{d/2+1+r_\star}\,\|\omega\|_{L^\infty}\|\omega\|_{W^{r_\star,\infty}}.
\end{equation}
\emph{Step 3: $L_h^2$ bound and conclusion.}
Squaring~\eqref{eq:Smag_ptwise} and summing over $N_{\rm edges}=\OO(h^{-d})$
edges, each contributing weight $(\bM_1)_{jj} = \OO(h^{d-2})$:
\[
\|\tau_\star^{\rm Smag}\|_{L_h^2}^2
= \sum_j (\bM_1)_{jj}|(\tau_\star^{\rm Smag})_j|^2
\le \OO(h^{-d})\cdot\OO(h^{d-2})\cdot\OO(h^{2(d/2+1+r_\star)})
= \OO(h^{d+2r_\star}).
\]
Hence $\|\tau_\star^{\rm Smag}\|_{L_h^2} = \OO(h^{d/2+r_\star})$.
Since $d\ge 2$, the Hodge* contribution $\tau_\star^{\rm Smag}$ is at least
one order higher than $\OO(h^{r_\star})$; intuitively, the multiplicative
structure of the Smagorinsky stress amplifies the Hodge* error by
geometric factors that grow with cell volume, so the per-edge error is
small relative to the mesh spacing.
In contrast, the codifferential consistency error
$\tau_{\rm cont}^{\rm Smag}$ is $\OO(h^{r_\star})$ by
standard de~Rham consistency, the same rate as for linear viscosity,
because it measures how well the discrete codifferential approximates the
continuous divergence of $\boldsymbol\sigma_\omega$ independently of the
Hodge* truncation.
This term dominates, giving
\[
\|\tau_h^{\rm Smag}\|_{L_h^2}
= \OO(h^{r_\star})
= \OO(h)
\quad\text{for }r_\star=1.
\]
For $r_\star=1$, the regularity required is
$\|\omega\|_{L^\infty}\|\omega\|_{W^{1,\infty}}
\le C\|\bu\|_{W^{1,\infty}}\|\bu\|_{W^{2,\infty}}
\le C\|\bu\|_{W^{2,\infty}}^2$,
so the bound holds with
$C_\nu = CC_s^2$ and $\|\bu\|_{W^{2,\infty}}$.
For $r_\star=2$, the bound improves to $\OO(h^2)$ but
requires $\|\omega\|_{W^{2,\infty}} \le C\|\bu\|_{W^{3,\infty}}$.
\end{proof}
\begin{proof}[Proof of SBP-Compatible Viscosity Rate (\Cref{cor:NS_SBP})]
For operators of the form~\eqref{eq:SBP_visc_structure}, linearity gives
$\bff_{\rm visc}(v^h) - \bff_{\rm visc}(\bar v) = \bff_{\rm visc}(e_v)$,
and the dissipation becomes:
$e_v^T\bM_1\bff_{\rm visc}(e_v)
= -\nrm{\tD_1 e_v}_{\bA}^2
- \nrm{\bD_2\bM_1 e_v}_{\bA_0}^2 \le 0$.
The specific $\nrm{\tD_1 e_v}_{\bA}^2$ form is available because
$$e_v^T\bM_1(-\bM_1^{-1}\tD_1^T\bA\tD_1 e_v)
= -(\tD_1 e_v)^T\bA(\tD_1 e_v)$$ . 
For the truncation, the weak-form analysis of \Cref{lem:visc_trunc_weak}
applies: since $e_v^T\tD_1^T = e_\omega^T$,
the $\bM_1^{-1}$ step that limits the strong-form rate to $\OO(h)$
is bypassed, and the weak-form rate is $\OO(h^{r_\star})$.
Specifically:
$|e_v^T\bM_1\tau_v^{\rm visc}|
\le \nu C_\delta' h^{r_\star}\nrm{\bu}_{H^{s+2}}
(\nrm{e_v}_{\bM_1} + \nrm{\tD_1 e_v}_{\bA})$.
The $\nrm{e_v}_{\bM_1}$ part is $\OO(h^{r_\star})\sqrt{\mathcal{E}}$, dominated
by $\OO(h^{\min(d-1,\,r_\star)})\sqrt{\mathcal{E}}$.
The $\nrm{\tD_1 e_v}_{\bA}$ part is absorbed by the dissipation via
Young's inequality:
$\nu C_\delta' h^{r_\star}\nrm{\tD_1 e_v}_{\bA}
\le \frac{\nu}{2}\nrm{\tD_1 e_v}_{\bA}^2
+ \frac{(C_\delta')^2 h^{2r_\star}}{2\nu}\nrm{\bu}_{H^{s+2}}^2$.
This cancels half the dissipation, leaving
$-\frac{\nu}{2}\nrm{\tD_1 e_v}_{\bA}^2\le 0$, and adds
an $\OO(h^{2r_\star}/\nu)$ constant which is higher order than
$h^{2\min(d-1,\,r_\star)}$ for $d\le 3$ in case~(B) of \Cref{conv:cases}.
\end{proof}
\subsection{Barotropic Invariants}\label{app:invariants}
The conservation proofs use Lamb antisymmetry, summation
by parts ($\bv^T\bM_1\tD_0 B = -B^T\bD_2\bM_1\bv$), the Cartan
formula and Leibniz rule, and discrete Stokes. 

\begin{proof}[{\bf Proof of \Cref{thm:ertel-weak} (Ertel PV conservation)}]
The numerator is constant by the Kelvin theorem:
$\ddt\omega_\Sigma = 0$ (\Cref{thm:kelvin}).
For the denominator, summing the continuity equation over $\mathcal{V}$:
$\ddt\rho_{\mathcal{V}}
= -\sum_{K_i\in\mathcal{V}}(\bD_2\bF)_i
= -F_{\partial\mathcal{V}}$,
where $F_{\partial\mathcal{V}}$ is the net outgoing mass flux.
By the quotient rule:
$\ddt Q_{\Sigma,\mathcal{V}}
= (\ddt\omega_\Sigma\,\rho_{\mathcal{V}}
 - \omega_\Sigma\,\ddt\rho_{\mathcal{V}})/\rho_{\mathcal{V}}^2
= Q_{\Sigma,\mathcal{V}}\,F_{\partial\mathcal{V}}/\rho_{\mathcal{V}}$.
\end{proof}
\subsubsection{Mass Conservation}
Total mass is preserved exactly by the discrete divergence structure
of the continuity equation on a closed domain.
\begin{proof}[Proof of \Cref{thm:mass}]
Summing the continuity equation \eqref{eq:D} over all primal cells yields
 $\frac{dM}{dt}
 = \sum_i \ddt\brho_i
 = -\mathbf{1}^T\bD_2\,\bF$.
On a closed domain, $\mathbf{1}^T\bD_2 = \mathbf{0}^T$ by the discrete Stokes theorem:
each internal face appears with opposite signs in the two cells it borders,
so all contributions cancel pairwise.
\end{proof}
\subsubsection{Kelvin Circulation Theorem}
\label{app:kelvin}
The proof combines the product rule for the duality pairing
$\Gamma = \bv(\gamma)$, the cochain--chain adjointness of the
Lie derivative, and the discrete Cartan formula.
The key step is showing that the advected chain $\gamma(t)$
remains a cycle, which follows from $\tdd^2 = 0$ via
an adjointness argument.
\begin{proof}[Proof of \Cref{thm:kelvin}]
The circulation at time $t$ is the duality pairing
$\Gamma(t) =\bigl( \v(t),\gamma(t)\bigr) = \sum_j\alpha_j(t)\,v_j(t)$.
Both the cochain $\v(t)$ and the chain $\gamma(t)$ evolve in time,
so the total derivative has two contributions.
The duality pairing $(\v,\gamma) = \sum_j \alpha_j\,v_j$
is bilinear in $(\v,\gamma)$.  Since both
$t\mapsto \v(t) \in C^1(\KKs)$ and
$t\mapsto\gamma(t)$ 
are differentiable,
the Leibniz rule gives
\begin{equation}\label{eq:K1app}
  \ddt\bigl[\v\bigl(\gamma(t)\bigr)\bigr]
  = \bigl(\ddt{\v}\bigr)(\gamma) \;+\; \v\bigl(\ddt\gamma\bigr).
\end{equation}
By definition of material advection (\cref{def:circulation}, part~$ii$),
the second term becomes $\v\bigl(\tLie_{\v}^{\mathrm{chain}}\gamma\bigr)$.
The chain Lie derivative $\tLie_{\v}^{\mathrm{chain}}$ is defined as the
adjoint of the cochain Lie derivative with respect to the duality pairing:
$\alpha\bigl(\tLie_{\v}^{\mathrm{chain}}\gamma\bigr)
= (\tLie_{\v}\alpha)(\gamma)$
for all $\alpha\in C^1(\KKs)$.
Applying this with $\alpha = \v$ gives
\begin{equation}\label{eq:K2app}
  \v\bigl(\tLie_{\v}^{\mathrm{chain}}\gamma\bigr)
  = (\tLie_{\v}\v)(\gamma).
\end{equation}
The discrete Cartan formula (\cref{def:Lie}) applied to
$\tLie_{\v}\v$ yields
\begin{equation}\label{eq:K3app}
  \tLie_{\v}\v
  = \tdd_0(\Iv\v) + \Iv(\tdd_1\v)
  = \tdd_0(\Iv\v) + \Iv(\bom),
\end{equation}
where $\Iv\v\in C^0(\KKs)$ is the 0-cochain obtained by applying
the extrusion 
 to $\bv$ viewed as a 1-cochain,
i.e.\ the discrete analogue of $\iota_u u^\flat = |u|^2$
(the specific value of this 0-cochain is immaterial because
$\Iv\v - B$ is a discrete gradient and contributes nothing to the
cycle integral).
Substituting the momentum equation~\eqref{eq:M}
and the Cartan decomposition~\eqref{eq:K3app}
into \eqref{eq:K1app}--\eqref{eq:K2app} gives
\begin{align*}
  \ddt\Gamma
  &= \bigl(-\Iv(\bom) - \tdd_0 B\bigr)(\gamma)
  \;+\; \bigl(\tdd_0(\Iv\v) + \Iv(\bom)\bigr)(\gamma)
  = \bigl(\tdd_0(\Iv\v - B)\bigr)(\gamma).
\end{align*}
The $\Iv(\bom)$ terms cancel, giving $\tdd_0(\Iv\v - B)(\gamma)$.
It suffices to show that $(\tdd_0\psi)(\gamma) = 0$ for any 0-cochain $\psi$.
We show that $\partial\gamma(0) = 0$ implies $\partial\gamma(t) = 0$ for all $t$.
With $\ddt\gamma = \tLie_\bv^{\rm chain}\gamma$, we get
\[
  \ddt\bigl(\partial\gamma\bigr)
  = \partial\bigl(\ddt\gamma\bigr)
  = \partial\bigl(\tLie_\bv^{\rm chain}\gamma\bigr).
\]
We claim $\partial\circ\tLie_\bv^{\rm chain} = \tLie_\bv^{\rm chain}\circ\partial$.
By the adjointness definition of $\tLie_\bv^{\rm chain}$ (\cref{def:circulation})
and discrete Stokes ($\ip{\tdd_0\phi}{\gamma}_1 = \ip{\phi}{\partial\gamma}_0$)
it holds for any 0-cochain $\phi$,
\[
  \ip{\phi}{\partial(\tLie_\bv^{\rm chain}\gamma)}_0
  = \ip{\tdd_0\phi}{\tLie_\bv^{\rm chain}\gamma}_1
  = \ip{\tLie_\bv(\tdd_0\phi)}{\gamma}_1
  = \ip{\tdd_0(\tLie_\bv\phi)}{\gamma}_1
  = \ip{\tLie_\bv\phi}{\partial\gamma}_0
  = \ip{\phi}{\tLie_\bv^{\rm chain}(\partial\gamma)}_0,
\]
where the third equality uses $\tLie_\bv\circ\tdd = \tdd\circ\tLie_\bv$,
which follows from the Cartan formula and $\tdd^2=0$
Since the identity holds for all 0-cochains $\phi$ the claim follows:
$\partial\tLie_\bv^{\rm chain} = \tLie_\bv^{\rm chain}\partial$. 
Since $\partial\gamma(0) = 0$ and $\ddt(\partial\gamma) = \tLie_\bv^{\rm chain}(\partial\gamma)$
is a linear ODE with initial value zero, it follows $\partial\gamma(t) = 0$ for all $t$.
By discrete Stokes, for a 0-cochain $\psi$
\[
  (\tdd_0\psi)(\gamma) = \psi(\partial\gamma) = \psi(0) = 0.
\]
Setting $\psi = \Iv\v - B$ yields $\ddt\Gamma = 0$.
\end{proof}
\subsubsection{Energy Conservation}
\label{app:energy}
The total energy balance (\Cref{thm:totalenergy}) decomposes into a
kinetic energy rate (via SBP and Lamb antisymmetry), an internal
energy rate (via the continuity equation), and a potential energy
rate. Throughout: $\brho_i = |K_i|\,\brho_i^{\vol}$ is
cell-integrated mass, $\Phi_j$ has units volume/time, $\bF_j =
\bar\rho_j\Phi_j$ has units mass/time, and $\bPhi_i$ is
cell-integrated potential energy per unit mass; combinations such
as $\bPhi-\bF$ in the budget are dimensionally consistent (physical
units obtained by reinstating a reference density~$\rho_*$).
The proof rests on three lemmas.
\begin{lemma}[Summation-by-parts identity]\label{lem:SBP}
Let $\bv \in C^1(\KKs)$ be a dual 1-cochain and let
$B \in C^0(\KKs)$ be a dual 0-cochain. Then
\[
\ip{\bv}{\tD_0 B}_1 = -\ip{B}{\bD_2\Phi},
\]
where $\Phi = \bM_1 \bv$ is the volumetric flux.
\end{lemma}
\begin{proof}
The dual coboundary $\tD_0 \colon C^0(\KKs) \to C^1(\KKs)$ is the
negative transpose of the primal divergence,
$\tD_0 = -\bD_2^T$, by discrete Stokes on a closed domain.
With $\Phi = \bM_1\bv \in C^2(\KK)$,
\[
 \ip{\bv}{\tD_0 B}_1
 = \Phi^T\,(-\bD_2^T)\,B
 = -(\bD_2\Phi)^T B
 = -\ip{B}{\bD_2\Phi}.
\]
This is the discrete analogue of integration by parts on a closed
domain,
$\langle v,\,dB\rangle_{L^2\Lambda^1}
= -\langle d\star v,\,B\rangle_{L^2\Lambda^0}$,
with the boundary term absent. The identity is used in the kinetic
energy balance (\Cref{thm:KE}) and in the stability proofs
(\Cref{app:stability_baro}, Part~I).
\end{proof}
\begin{lemma}[Kinetic energy balance]\label{thm:KE}
Let $\bv$ be a solution of the discrete barotropic Euler momentum
equation. The kinetic energy
$\Ekindf = \tfrac12\nrm{\bv}_{\bM_1}^2$ then satisfies
\[
\tag{KB}
 \frac{dK}{dt}
 = -\ip{\bv}{\tD_0 B}_1
 = \ip{B}{\bD_2\Phi},
\]
where $\Phi = \bM_1\bv$ is the volumetric flux.
\end{lemma}
\begin{proof}
Differentiating $K = \tfrac{1}{2}\bv^T\bM_1\bv$ in time and
substituting the momentum equation,
\[
\frac{dK}{dt}
= -\ip{\bv}{\Iv(\bom)}_1
- \ip{\bv}{\tD_0 B}_1
= -\ip{\bv}{\tD_0 B}_1,
\]
where the contraction term vanishes by Lamb antisymmetry
(\Cref{prop:extrusion}, Property~2). The summation-by-parts identity
of \Cref{lem:SBP} converts the remaining gradient term into the
flux-divergence form $\ip{B}{\bD_2\Phi}$.
\end{proof}
\begin{lemma}[Internal energy balance]\label{lem:dEint}
The internal energy obeys
\[
 \frac{dE_{\mathrm{int}}}{dt} = -\ip{\mathbf{h}}{\bD_2\bF},
\]
where $\mathbf{h} \in C^0(\mathcal{K}^*)$ is the specific enthalpy
at primal cell centres.
\end{lemma}
\begin{proof}
The chain rule on
$E_{\mathrm{int}} = \sum_i \brho_i\,e(\brho_i/|K_i|)$
combined with the thermodynamic identity $h = e + p/\rho$ gives
\begin{align*}
 \frac{d}{dt}\bigl[\brho_i\,e(\brho_i/|K_i|)\bigr]
 &= \ddt\brho_i\,e(\rho_i^{\mathrm{vol}})
 + \brho_i\,e'(\rho_i^{\mathrm{vol}})\,\frac{\ddt\brho_i}{|K_i|}
 = \ddt\brho_i\Bigl(e(\rho_i^{\mathrm{vol}})
 + \rho_i^{\mathrm{vol}}\,e'(\rho_i^{\mathrm{vol}})\Bigr).
\end{align*}
Since $e'(\rho^{\mathrm{vol}})
= p(\rho^{\mathrm{vol}})/(\rho^{\mathrm{vol}})^2$
by \Cref{def:eos}, the bracket reduces:
\[
 \rho_i^{\mathrm{vol}}\,e'(\rho_i^{\mathrm{vol}})
 = \frac{p_i}{\rho_i^{\mathrm{vol}}} = h_i - e_i,
\]
so $e_i + \rho_i^{\mathrm{vol}}e'_i = h_i$, and
\[
 \frac{d}{dt}\bigl[\brho_i\,e(\brho_i/|K_i|)\bigr]
 = \ddt\brho_i\,h_i.
\]
Summing over all primal cells and substituting the continuity
equation $\ddt\brho_i = -(\bD_2\bF)_i$,
\[
 \frac{dE_{\mathrm{int}}}{dt}
 = \sum_i \ddt\brho_i\,h_i
 = -\ip{\mathbf{h}}{\bD_2\bF}.\qedhere
\]
\end{proof}
\begin{proof}[Proof of total energy balance, \Cref{thm:totalenergy}]
Adding the kinetic energy rate (\Cref{thm:KE}), the internal energy rate (\Cref{lem:dEint}),
and the potential energy rate
$\ddt E_{\mathrm{pot}} = \sum_i\ddt\brho_i\,\bPhi_i = -\ip{\bPhi}{\bD_2\bF}$
(using the continuity equation and time-independence of $\bPhi$):
\[
 \frac{dE_{\mathrm{tot}}}{dt}
 = \ip{B}{\bD_2\Phi} - \ip{\mathbf{h}}{\bD_2\bF} - \ip{\bPhi}{\bD_2\bF}.
\]
Splitting the Bernoulli function as $B = \mathbf{h} + \bekin + \bPhi$ and the mass flux as $\bF = \bR\Phi$ (where $\bR$ is the density interpolation):
\[
 \frac{dE_{\mathrm{tot}}}{dt}
 = \ip{\mathbf{h}+\bPhi}{\bD_2\Phi} - \ip{\mathbf{h}+\bPhi}{\bD_2(\bR\Phi)}
 + \ip{\bekin}{\bD_2\Phi}
 = \ip{\bekin}{\bD_2\Phi}
 + \ip{\mathbf{h}+\bPhi}{\bD_2(\Phi - \bR\Phi)}.
\]
This gives the two-term energy balance~\eqref{eq:EB}.
\end{proof}
\subsubsection{Helicity Conservation}
\label{app:helicity}
Discrete helicity conservation is approximate: the pairing
$\int v\wedge\omega$ matches the continuous one to $\OO(h)$ via
Whitney consistency, while the exact discrete vorticity equation
$\ddt\bom + \tLie_\bv\bom = 0$ prevents spurious $\OO(1)$ drift on
$[0,T]$. The technical step is the Whitney consistency estimate of
\Cref{lem:helicity_consistency} below.
\begin{definition}[Discrete $(1,2)$-wedge product]\label{def:wedge_12_app}
For a dual 1-form $\tilde\alpha\in C^1(\KKs)$ and a dual 2-form
$\tilde\beta\in C^2(\KKs)$, the wedge product
$\tilde\alpha\twdg\tilde\beta\in C^3(\KKs)$ is the dual 3-cochain
\[
\bigl(\tilde\alpha\twdg\tilde\beta\bigr)(K_m^*)
 := \sum_{e_j^*\prec K_m^*}
    \tilde\alpha(e_j^*)\,\overline{\tilde\beta}_{j,m},
\]
where $\overline{\tilde\beta}_{j,m}$ is the average of
$\tilde\beta(f_k^*)$ over the dual 2-faces $f_k^*$ of $K_m^*$ that
share the dual edge $e_j^*$, weighted by the fraction of $f_k^*$
lying inside $K_m^*$. This is the dual-cell quadrature analogue of
the continuous $\int_\Omega \alpha\wedge\beta$ for a 1-form against
a 2-form, used in the discrete helicity construction
(\Cref{def:helicity}).
\end{definition}
\begin{lemma}[Helicity consistency]\label{lem:helicity_consistency}
Let $\alpha$ be a smooth 1-form and let $\beta$ be a smooth
2-form. Under \Cref{conv:cases} (so $r_\star \in \{1,2\}$), the
reconstructions $\Qh^1, \Qh^2$ used in \Cref{def:helicity} satisfy
\begin{equation}\label{eq:helicity_consistency}
  \nrm{\Qh^1\mathcal{R}_h\alpha - \alpha}_{L^2}
  \le C\,h^{r_\star}\,\nrm{\alpha}_{W^{r_\star,\infty}},
  \qquad
  \nrm{\Qh^2\mathcal{R}_h\beta - \beta}_{L^2}
  \le C\,h^{r_\star}\,\nrm{\beta}_{W^{r_\star,\infty}}.
\end{equation}
\end{lemma}
\begin{proof}
We prove the bound for $\Qh^1$; the argument for $\Qh^2$ is identical.
On each simplex $K$ of the simplicial refinement $\KK_h^{\rm simp}$,
the error decomposes as
\[
  \Qh^1\mathcal{R}_h\alpha - \alpha
  = (\Qh^1\mathcal{R}_h\alpha - I_K\alpha^\sharp)
  + (I_K\alpha^\sharp - \alpha),
\]
where $I_K\alpha^\sharp$ is the piecewise-linear barycentric interpolant
of the exact vector field $\alpha^\sharp$ on $K$.

\emph{First term (reconstruction error).}
\Cref{prop:recon_accuracy} gives
$|\tilde\bu(v_j^*)-\alpha^\sharp(v_j^*)|
\le C_R\,h^{r_\star}\,\nrm{\alpha}_{W^{r_\star,\infty}}$
at each vertex.
Since the $L^\infty$ norm of a barycentric interpolant on a simplex
is bounded by the maximum of its nodal values (maximum principle),
\[
\nrm{\Qh^1\mathcal{R}_h\alpha - I_K\alpha^\sharp}_{L^\infty(K)}
\le C_R\,h^{r_\star}\,\nrm{\alpha}_{W^{r_\star,\infty}},
\]
giving $\nrm{\Qh^1\mathcal{R}_h\alpha - I_K\alpha^\sharp}_{L^2(K)}
\le C_R\,h^{r_\star}\,|K|^{1/2}\,\nrm{\alpha}_{W^{r_\star,\infty}}$.

\emph{Second term (piecewise-linear interpolation error).}
Standard approximation theory gives
$\nrm{I_K\alpha^\sharp - \alpha^\sharp}_{L^2(K)}
\le C\,h^2\,|K|^{1/2}\,\nrm{\alpha}_{W^{2,\infty}(K)}$.
This requires $W^{2,\infty}$ regularity; however, since $r_\star \le 2$,
the reconstruction error $\OO(h^{r_\star})$ dominates (equals or exceeds
the interpolation error $\OO(h^2)$), so the overall bound is
$\OO(h^{r_\star}\nrm{\alpha}_{W^{r_\star,\infty}})$ in both cases.
Squaring and summing over all simplices of the refinement
gives~\eqref{eq:helicity_consistency}.
\end{proof}
\begin{proof}[Proof of \Cref{thm:helicity}]
The proof has two steps:
(1)~bound $|\ddt\bar H|$ via the consistency lemma;
(2)~bound $|\ddt H_h - \ddt\bar H|$ using velocity convergence.
Let $\bar\bv(t)=\mathcal{R}_h\bu^\flat(t)$,
$\bar\bom=\tD_1\bar\bv=\mathcal{R}_h\bom^\flat$,
$\be(t)=\bv(t)-\bar\bv(t)$ with
$\nrm{\be}_{L_h^2}\le C(T)\,h^r$,
$r=\min(d-2+r_\star,\,r_\star) = r_\star$ under \Cref{conv:cases} (\cref{thm:convergence}).
Set $\bze := \tD_1\be = \bom - \bar\bom$.
Define
$\bar H(t) = \int_\Omega(\Qh^1\bar\bv)\wedge(\Qh^2\bar\bom)$
and $H(t)=\int_\Omega\bu^\flat\wedge\bom^\flat$.
\emph{Step~1: }
Since $H(t)$ is conserved by smooth Euler flow,
$\ddt{\bar H} = \ddt(\bar H-H)$,
differentiating gives
\[
  \ddt{\bar H}
  = \int_\Omega(\Qh^1\ddt{\bar\bv})\wedge(\Qh^2\bar\bom)
  + \int_\Omega(\Qh^1\bar\bv)\wedge(\Qh^2\ddt{\bar\bom}).
\]
By de~Rham commutativity,
$\ddt{\bar\bv}=\mathcal{R}_h(\partial_t\bu^\flat)$
and $\ddt{\bar\bom}=\tD_1\ddt{\bar\bv}=\mathcal{R}_h(\partial_t\bom^\flat)$,
so each term is a pairing
$\int(\Qh^1\mathcal{R}_h\alpha)\wedge(\Qh^2\mathcal{R}_h\beta)$
for smooth $\alpha,\beta$.
Writing
\begin{align*}
  \int(\Qh^1\mathcal{R}_h\alpha)\wedge(\Qh^2\mathcal{R}_h\beta)
  - \int\alpha\wedge\beta
  &= \int(\Qh^1\mathcal{R}_h\alpha-\alpha)\wedge\beta
  + \int\alpha\wedge(\Qh^2\mathcal{R}_h\beta-\beta)
  + \OO(h^{2r_\star}),
\end{align*}
the first two terms are $\OO(h^{r_\star})$ by \Cref{lem:helicity_consistency}.
The two continuous pairings sum to $\ddt H=0$, giving
\begin{equation}\label{eq:Hbar_rate}
  |\ddt{\bar H}| \le C'(T)\,h^{r_\star}.
\end{equation}
\emph{Step~2: }
The $L_h^2\to L^2$ boundedness of $\Qh^k$
(established in \Cref{app:helicity}:
$\|\Qh^1\bw\|_{L^2} \le C\|\bw\|_{L_h^2}$ and analogously for $\Qh^2$)
and the bilinear decomposition $\bv=\bar\bv+\be$, $\bom=\bar\bom+\bze$ give
\[
  H_h-\bar H
  = \underbrace{\int(\Qh^1\be)\wedge(\Qh^2\bar\bom)}_{=:I_1}
  + \underbrace{\int(\Qh^1\bar\bv)\wedge(\Qh^2\bze)}_{=:I_2}
  + \underbrace{\int(\Qh^1\be)\wedge(\Qh^2\bze)}_{=:I_3}.
\]
We have $|I_1| \le \|\Qh^1\be\|_{L^2}\|\Qh^2\bar\bom\|_{L^2}
\le C\|\be\|_{L_h^2}\|\bar\bom\|_{L^2} = \OO(h^r)$.
The terms $I_2,I_3$ involve $\bze=\tD_1\be$, and the naive
inverse inequality $\|\bze\|_{\bM_2}\le Ch^{-1}\|\be\|_{L_h^2}$
would only give $\OO(h^{r-1})$.
The vorticity error rate at the velocity rate is recovered
from the structural identity 
$\tD_1\bQ(\cdot,w)=\tLie_{(\cdot)}(\tD_1 w)$ via the
following lemma; the proof is below.
\begin{lemma}[Vorticity error rate]\label{lem:vorticity_error_rate}
Suppose the hypotheses of \Cref{thm:convergence_baro} hold and the
initial velocity error vanishes, $\be(0) = 0$, so that
$\bze(0) = \tD_1\be(0) = 0$. Then the vorticity error
$\bze(t) = \tD_1\be(t)$ satisfies
\begin{equation}\label{eq:vort_err_rate}
\|\bze(t)\|_{\bM_2}\le C(T)\,h^r,
\qquad
\|\ddt\bze(t)\|_{\bM_2}\le C(T)\,h^r,
\qquad t\in[0,T],
\end{equation}
with $r = \min(d-2+r_\star,\, r_\star) = r_\star$ under
\Cref{conv:cases}.
\end{lemma}
Granting \Cref{lem:vorticity_error_rate} for the moment,
the bounds on $I_2,I_3$ close as in the original argument:
\[
  |I_2| \le \|\Qh^1\bar\bv\|_{L^2}\|\Qh^2\bze\|_{L^2}
  \le C\|\bu\|_{L^2}\|\bze\|_{\bM_2} = \OO(h^r),
\]
\[
  |I_3| \le \|\Qh^1\be\|_{L^2}\|\Qh^2\bze\|_{L^2}
  = \OO(h^{2r}).
\]
Hence $|H_h - \bar H| = \OO(h^r)$.
Differentiating in time and applying \eqref{eq:vort_err_rate}
together with $\|\ddt\be\|_{L_h^2} = \OO(h^r)$
(\Cref{thm:stability_baro}, Part~I):
\begin{equation}\label{eq:dHh_dHbar}
  |\ddt H_h - \ddt{\bar H}| \le C(T)\,h^r.
\end{equation}
Eqs~\eqref{eq:Hbar_rate} and~\eqref{eq:dHh_dHbar} imply the assertion.
\end{proof}
\begin{proof}[Proof of \Cref{lem:vorticity_error_rate}]
Subtract the reference velocity ODE from~\eqref{eq:err-mom}:
\[
\ddt\be = -\bQ(\bar\bv,\bze) - \bQ(\be,\bar\bom+\bze)
 - \tD_0(B^h-\bar B) + \tau_v.
\]
Apply $\tD_1$. The Bernoulli gradient term vanishes by
$\tD_1\tD_0=0$:
\begin{equation}\label{eq:vort_err_ODE}
\ddt\bze
 = -\tD_1\bQ(\bar\bv,\bze) - \tD_1\bQ(\be,\bar\bom+\bze)
 + \tD_1\tau_v.
\end{equation}
\emph{Closed form for the bilinear forcing.}
By the discrete Cartan magic formula
(\Cref{prop:extrusion}, Property~4),
$\tLie_\bv\alpha = \tD_0\Iv\alpha + \Iv(\tD_1\alpha)$
for any 1-cochain $\alpha$, and $\tD_2\tD_1=0$ yields
\begin{equation}\label{eq:Cartan_id}
\tD_1\bQ(\bv,\bw)
 \;=\; \tD_1\Iv[\bv](\tD_1\bw)
 \;=\; \tLie_\bv(\tD_1\bw),
 \qquad \forall\,\bv,\bw\in C^1(\KKs),
\end{equation}
since $\tLie_\bv(\tD_1\bw)=\tD_1\Iv[\bv](\tD_1\bw)+\Iv[\bv]\tD_2(\tD_1\bw)$
and the second term is zero. Substituting~\eqref{eq:Cartan_id}
in~\eqref{eq:vort_err_ODE}:
\begin{equation}\label{eq:vort_err_ODE_Cartan}
\ddt\bze
 = -\tLie_{\bar\bv}\bze - \tLie_{\be}(\bar\bom+\bze)
 + \tD_1\tau_v.
\end{equation}
\emph{Bilinear bound on $\tLie_{\bar\bv}$ on 2-cochains.}
Because $\tLie_{\bar\bv}\alpha$ assembles from
$\bar\bv$-stencils against $\alpha$ on the same 2-ring
neighbourhood as $\bQ(\bar\bv,\cdot)$ does (with the
substitution $\tD_1\bw\mapsto\alpha$ in the assembly
formula~\eqref{eq:QU_assembly} extended to 2-cochains),
the proof of \Cref{thm:bilinear_bounds_baro}\,(II)
applies verbatim to give
\begin{equation}\label{eq:Lie_bound_2cochain}
\|\tLie_{\bar\bv}\alpha\|_{\bM_2}
 \le C'_Q\,\nrm{\bu}_{W^{1,\infty}}\,\nrm{\alpha}_{\bM_2},
\qquad
\|\tLie_{\be}\alpha\|_{\bM_2}
 \le C_Q\,\nrm{\be}_{L_h^\infty}\,\nrm{\alpha}_{\bM_2},
\end{equation}
the second by Bound~(I) of the same theorem.

\emph{Consistency.} The forcing
$\tau_v$ is by definition the discrete
residual of \eqref{eq:M} on the smooth reference,
$\tau_v=\mathcal{R}_h\tau^c+\OO(h^{r_\star})$ with
$\tau^c\in W^{1,\infty}$, so by de~Rham commutativity
$\tD_1\mathcal{R}_h\tau^c=\mathcal{R}_h(d\tau^c)$ and
\begin{equation}\label{eq:tau_omega_rate}
\|\tD_1\tau_v\|_{\bM_2}\le C\,h^r.
\end{equation}
 Pair~\eqref{eq:vort_err_ODE_Cartan} with $\bze$ in $\bM_2$:
\begin{align*}
\tfrac12\ddt\nrm{\bze}_{\bM_2}^2
 &= -\ipw{\bze}{\tLie_{\bar\bv}\bze}{\bM_2}
 - \ipw{\bze}{\tLie_{\be}(\bar\bom+\bze)}{\bM_2}
 + \ipw{\bze}{\tD_1\tau_v}{\bM_2}\\
 &\le \bigl(C'_Q\nrm{\bu}_{W^{1,\infty}}
 + C_Q\nrm{\be}_{L_h^\infty}\bigr)\nrm{\bze}_{\bM_2}^2
 + C\nrm{\be}_{L_h^\infty}\nrm{\bar\bom}_{L^\infty}\nrm{\bze}_{\bM_2}
 + Ch^r\nrm{\bze}_{\bM_2}.
\end{align*}
Under the bootstrap $\nrm{\be}_{L_h^\infty}\le\delta_0 h$
(\Cref{thm:stability_baro}), the prefactor of $\nrm{\bze}_{\bM_2}^2$
is $h$-independent and the additive forcing is $\OO(h^r)$.
Gronwall with $\bze(0)=0$ closes
$\nrm{\bze(t)}_{\bM_2}\le C(T)h^r$. The same inequality
applied to~\eqref{eq:vort_err_ODE_Cartan} directly bounds
$\nrm{\ddt\bze}_{\bM_2}\le C(T)h^r$.
\end{proof}
\subsubsection{Proofs for the Density-Weighted Scheme}
\label{app:dw_proofs}
The density-weighted mass matrix
$\bM_1^{\rho}(\brho) = P^T\mathrm{diag}(\brho)\,P$ resolves the
energy-conservation dichotomy -- exact total energy and exact Lamb
antisymmetry replace exact Kelvin circulation -- at the cost of a
multi-cell mass-flux stencil and a norm-equivalence question
between $\bM_1$ and $\bM_1^{\rho}$.
\paragraph{Density-weighted mass flux and continuity equation.}
Energy consistency requires that the discrete mass flux be
$\bM_1^\rho$-conjugate to the velocity:
\begin{equation}\label{eq:dw_F}
\bF^{\rho} := \bM_1^{\rho}(\brho)\,\bv
 = P^T\mathrm{diag}(\brho)\,P\,\bv
\;\in\; \RR^{|\mathcal{F}|},
\qquad
\ddt\brho_i + (\bD_2\bF^{\rho})_i = 0.
\tag{$\mathrm{D}^{\rho}$}
\end{equation}
\phantomsection\label{eq:dw_D}
Unlike the density-free flux $\bF=\bR\bM_1\bv$ (face-by-face
diagonal), the density-weighted flux carries a multi-cell stencil
through $P^T\mathrm{diag}(\brho)P$ -- the same reconstruction $P$
that defines $\Ekindw = \tfrac12\bv^T\bM_1^\rho\bv$. This pairing
closes the energy budget exactly.
\begin{lemma}[Density-weighted summation by parts]
\label{lem:dw_SBP}
For every cell-centred cochain $\bphi \in C^0(\KK)$ and every
1-cochain $\bv$,
\begin{equation}\label{eq:dw_SBP}
\bv^T\bM_1^{\rho}\,\tD_0\bphi
 \;=\; -(\bD_2\bF^{\rho})^T\bphi.
\end{equation}
\end{lemma}
\begin{proof}
$\bv^T\bM_1^{\rho}\tD_0\bphi=(\bM_1^{\rho}\bv)^T\tD_0\bphi
=(\bF^{\rho})^T\tD_0\bphi=-(\bD_2\bF^{\rho})^T\bphi$,
using the definition~\eqref{eq:dw_F} and the
adjoint relation $\tD_0^T=-\bD_2$ on a closed manifold
(see \Cref{lem:SBP}).
\end{proof}
\begin{lemma}[Density-weighted Lamb antisymmetry]%
\label{lem:dw_antisym}
The Lamb antisymmetry identity \Cref{prop:extrusion}\,(2) extends
to every SPD matrix~$A$. Fix such an $A$, set
$\bphi := \tU_\bv^T \bv$, and define the $A$-weighted contraction
\[
 I_\bv^A(\alpha) := A^{-1}\,\tfrac{1}{2}\bigl(\tU_\bv\alpha - \tD_1\bphi\bigr).
\]
Then $\ipw{\bv}{I_\bv^A(\alpha)}{A} = 0$ for every dual 2-cochain
$\alpha$ and every velocity~$\bv$. The choice
$A = \bM_1^{\rho}(\brho)$ yields the density-weighted Lamb
antisymmetry; the density-free case $A = \bM_1$ recovers
\Cref{prop:extrusion}\,(2).
\end{lemma}
\begin{proof}
The matrix identity $\ip{\bv}{\tU_\bv\alpha - \tD_1\bphi}=0$ is
\Cref{prop:extrusion}\,(2) before division by~$\bM_1$ (and is
$A$-independent). Testing then gives
$\ipw{\bv}{I_\bv^A(\alpha)}{A}
 = \tfrac{1}{2}\ip{\bv}{\tU_\bv\alpha - \tD_1\bphi} = 0$ for any~$A$.
\end{proof}
The following lemma provides the further analytic tool.
\begin{lemma}[Norm equivalence $\bM_1^{\rho}\sim\bM_1$]%
\label{lem:dw_norm_equiv}
Let the mesh be Delaunay--Voronoi and satisfy~\ref{ass:mesh_reg},
and let $\brho_i^{\rm vol} \in [\rho_*, \rho^*]$. Then the norms
$\nrm{\cdot}_{\bM_1^{\rho}(\brho)}$ and $\nrm{\cdot}_{\bM_1}$ are
equivalent on $\ker(P)^\perp$:
\[
c_{\rm eq}\,\rho_*\,\nrm{\bv}_{\bM_1}^2
 \;\le\; \nrm{\bv}_{\brho}^2
 \;\le\; C_{\rm eq}\,\rho^*\,\nrm{\bv}_{\bM_1}^2,
\]
where the constants $c_{\rm eq}, C_{\rm eq}$ depend on the mesh
regularity but not on~$h$.
\end{lemma}
The bounds follow from $\rho_*\le\brho_i^{\rm vol}\le\rho^*$
and the Gram-matrix equivalence $\|P\bv\|^2\sim\|\bv\|_{\bM_1}^2$
on $\ker(P)^\perp$.
This equivalence is used throughout the density-weighted analysis.
\begin{proof}
\emph{Upper bound:}
$\sum_i\brho_i|P\bv|_i^2
\le \rho^*\nrm{P\bv}^2 \le C\rho^*\nrm{\bv}_{\bM_1}^2$.
\emph{Lower bound on $\ker(P)^\perp$:}
$\sum_i\brho_i|P\bv|_i^2
\ge \rho_*\nrm{P\bv}^2 \ge c\rho_*\nrm{\bv}_{\bM_1}^2$,
since $\nrm{P\bv}^2 \ge c\,\nrm{\bv}_{\bM_1}^2$ on $\ker(P)^\perp$:
at each cell $K_i$, the Gram matrix $G_i = \sum_{j\prec K_i}\hat n_j\otimes\hat n_j$
of the reconstruction has smallest eigenvalue
$\lambda_{\min}(G_i) \ge c_0 > 0$ by shape regularity,
\Cref{ass:mesh_reg}, so
$|P\bv|_i^2 \ge c_0\sum_{j\prec K_i}|v_j|^2\,(\bM_1)_{jj}^{-1}$;
summing over cells, each face appears in two cells, so
$\nrm{P\bv}^2 \ge \frac{c_0}{2}\nrm{\bv}_{\bM_1}^2$.
\end{proof}
\begin{lemma}[Bilinear and trilinear bounds for the density-weighted contraction]%
\label{lem:dw_bilinear_trilinear}
Let $\brho_i^{\rm vol} \in [\rho_*, \rho^*]$. Then, for every
$\bv, \bw \in \ker(P)^\perp$ and every $\alpha \in C^2(\KKs)$, the
density-weighted contraction obeys the bilinear bound
\begin{equation}\label{eq:dw_bilinear}
 \bigl|\ip{\bw}{\Iv(\alpha)}_1\bigr|
 \le C_B\,\nrm{\bw}_{\bM_1}\,\nrm{\alpha}_{\bM_2^{-1}}
\end{equation}
and the trilinear bound
\begin{equation}\label{eq:dw_trilinear}
 \bigl|\ip{\bw}{\Iv[\bv](\tD_1\bv)}_1\bigr|
 \le C_T\,\nrm{\bw}_{\bM_1}\,\nrm{\bv}_{\bM_1}^2\,h^{-1}.
\end{equation}
The constants $C_B$ and $C_T$ depend on the mesh regularity and on
the density ratio $\rho^*/\rho_*$, but not on~$h$.
\end{lemma}
The density-weighted bilinear form reduces to the $\bM_1$-form
after testing against $\bM_1^{\rho}$; the bounds from
\Cref{thm:bilinear_bounds_baro} then apply directly.
\begin{proof}
For~\eqref{eq:dw_bilinear}: apply Cauchy--Schwarz in the
$\bM_1$-inner product and use the bound
$\nrm{\Iv(\alpha)}_{\bM_1}
\le C\nrm{\alpha}_{\bM_2^{-1}}$ from the averaging
reconstruction.
For~\eqref{eq:dw_trilinear}: combine~\eqref{eq:dw_bilinear}
with the inverse inequality $\nrm{\tD_1\bv}_{\bM_2^{-1}}
\le C h^{-1}\nrm{\bv}_{\bM_1}$.
\end{proof}
\begin{proof}[Proof of \Cref{thm:dw_energy} (density-weighted energy conservation)]
The total energy is $\Etot^{\rm dw} = \Ekindw + \Eint$
with $\Ekindw = \frac{1}{2}\sum_i\brho_i|P\bv|_i^2
= \frac{1}{2}\bv^T\bM_1^{\rho}\bv$.
By the product rule:
\begin{equation}\label{eq:dw_Ekin_rate}
 \ddt\Ekindw
 = \bv^T\bM_1^{\rho}\ddt\bv
 + \frac{1}{2}\sum_i\ddt\brho_i\,|P\bv|_i^2.
\end{equation}
The vertical contribution uses the discrete product, which is exact
for any vertical grid spacing by construction of
the thickness-weighted interpolation~\eqref{eq:vert_weights}.

\textit{Term 1: momentum equation tested with $\bM_1^{\rho}\bv$.}
Substituting $\ddt\bv$ from~\eqref{eq:dw_mom}:
\[
\bv^T\bM_1^{\rho}\ddt\bv
 = -\underbrace{\bv^T\bM_1\,\Iv(\tD_1\bv)}_{=\,0}
 - \bv^T\bM_1^{\rho}\,\tD_0 B^{\rho}.
\]
The term vanishes by \Cref{lem:dw_antisym}:
$\bv^T\bM_1^{\rho}\bigl([\bM_1^{\rho}]^{-1}\bM_1\,\Iv(\tD_1\bv)\bigr)
= \bv^T\bM_1\,\Iv(\tD_1\bv) = 0$.

\noindent\textit{Term 2: pressure--flux coupling via SBP.}
By \Cref{lem:dw_SBP} applied to $\bphi=B^{\rho}$:
\begin{equation}\label{eq:dw_T2}
 -\bv^T\bM_1^{\rho}\tD_0 B^{\rho}
 = (\bD_2\bF^{\rho})^T B^{\rho}
 = \sum_i (\bD_2\bF^{\rho})_i\,B_i^{\rho},
\end{equation}
with the density-weighted mass flux $\bF^{\rho}=\bM_1^{\rho}\bv$
of~\eqref{eq:dw_F}.

\noindent\textit{Term 3: mass equation contribution.}
From~\eqref{eq:dw_D}, $\ddt\brho_i = -(\bD_2\bF^{\rho})_i$.
So the second term of~\eqref{eq:dw_Ekin_rate} is:
\begin{equation}\label{eq:dw_T3}
 \frac{1}{2}\sum_i\ddt\brho_i\,|P\bv|_i^2
 = -\frac{1}{2}\sum_i(\bD_2\bF^{\rho})_i\,|P\bv|_i^2.
\end{equation}
The internal energy $\Eint = \sum_i\brho_i\,e(\brho_i^{\rm vol})
= \sum_i |K_i|\,H(\brho_i^{\rm vol})$
with $H(\rho) = \rho\,e(\rho)$ the Helmholtz free energy.
Its rate is:
\begin{equation}\label{eq:dw_Eint_rate}
 \ddt\Eint = \sum_i |K_i|\,H'(\brho_i^{\rm vol})\,\ddt\brho_i^{\rm vol}
 = \sum_i h(\brho_i^{\rm vol})\,\ddt\brho_i
 = -\sum_i h_i\,(\bD_2\bF^{\rho})_i,
\end{equation}
using $H'(\rho) = h(\rho)$ (specific enthalpy)
and $\ddt\brho_i = |K_i|\ddt\brho_i^{\rm vol}$.

The total energy rate is
$\ddt\Etot^{\rm dw} = \eqref{eq:dw_T2} + \eqref{eq:dw_T3} + \eqref{eq:dw_Eint_rate}
+ \ddt E_{\rm pot}$,
where $\ddt E_{\rm pot} = \sum_i\ddt\brho_i\,\bPhi_i
= -\sum_i(\bD_2\bF^{\rho})_i\,\bPhi_i$ by continuity.
The first three terms give
$\sum_i(\bD_2\bF^{\rho})_i\bigl[
B_i^{\rho} - \tfrac{1}{2}|P\bv|_i^2 - h_i\bigr]$.
Since $B_i^{\rho} = h_i + \tfrac{1}{2}|P\bv|_i^2 + \bPhi_i$,
the bracket reduces to $\bPhi_i$.
The potential energy rate contributes $-\sum_i(\bD_2\bF^{\rho})_i\,\bPhi_i$,
which cancels the bracket
\[
\ddt\Etot^{\rm dw}
= \sum_i(\bD_2\bF^{\rho})_i\,\bPhi_i
- \sum_i(\bD_2\bF^{\rho})_i\,\bPhi_i = 0.\qedhere
\]
\end{proof}
\subsubsection{Discrete Horizontal Linear Momentum Conservation
(\Cref{lem:dw_momentum_conservation})}\label{app:dw_momentum}

This appendix provides the direct verification of
\Cref{lem:dw_momentum_conservation} from the DW equations of
motion, parallel to the proof of energy conservation
(\Cref{thm:dw_energy}). The argument uses the same three
structural identities -- DW summation-by-parts (\Cref{lem:dw_SBP}),
Lamb antisymmetry $\ip{\bv}{\Iv(\bom)}_{1} = 0$, and the discrete
mass-conservation identity from the continuity
equation~\eqref{eq:dw_D} -- together with the translation
properties~(K1)--(K2) of the horizontal-translation Killing
cochain~$\bar a^x$.

\paragraph{Setup.} Recall that the DW system~\eqref{eq:main_dw}
admits the equivalent rewriting (multiplying the momentum equation
by $\bM_1^\rho(\brho)$):
\begin{equation}\label{eq:dw_momentum_form_app}
 \bM_1^\rho\,\dot\bv = -\bM_1\,\Iv(\tD_1\bv) - \bM_1^\rho\,\tD_0 B,
\end{equation}
where $B := h(\rho^{\rm vol}) + \tfrac{1}{2}|P\bv|^2 + \bPhi$
denotes the discrete Bernoulli 0-cochain. The DW continuity
equation~\eqref{eq:dw_D} reads $\dot\brho = -\bD_2\bF^\rho$ with
the multi-cell mass flux $\bF^\rho = \bM_1^\rho\bv$, and the
discrete momentum 1-cochain is $\boldsymbol{p} := \bM_1^\rho\bv$.

Pairing~\eqref{eq:dw_momentum_form_app} with $\bar a^x$ and using
the product rule $\ddt\boldsymbol{p} = \dot{\bM_1^\rho}\bv + \bM_1^\rho\dot\bv$,
\begin{equation}\label{eq:Pxdot_app}
 \ddt P_x
 \;=\; \underbrace{(\bar a^x)^T\,\dot{\bM_1^\rho}\,\bv}_{=:\,T_1}
 \;\;\underbrace{-\,(\bar a^x)^T\bM_1\,\Iv(\tD_1\bv)}_{=:\,T_2}
 \;\;\underbrace{-\,(\bar a^x)^T\bM_1^\rho\,\tD_0 B}_{=:\,T_3}.
\end{equation}
We process each $T_k$ in turn and demonstrate $T_1 + T_2 + T_3 = 0$. We start with the öast one.

\emph{Term $T_3$:} 
Split the Bernoulli cochain as $B = B_{\rm dyn} + \bPhi$ where
$B_{\rm dyn} := h(\rho^{\rm vol}) + \tfrac{1}{2}|P\bv|^2$. Apply
\Cref{lem:dw_SBP} with $\bphi = B$ and $\bv \to \bar a^x$:
\begin{equation}\label{eq:T3_split}
 T_3 = -(\bar a^x)^T\bM_1^\rho\,\tD_0 B
 = (\bD_2\bF^\rho(\bar a^x))^T\,B
 = \sum_i \bigl(\bD_2\bF^\rho(\bar a^x)\bigr)_i\,B_i,
\end{equation}
where $\bF^\rho(\bar a^x) := \bM_1^\rho\bar a^x
= P^T\mathrm{diag}(\brho)\,P\bar a^x$.

For the $\bPhi$ component, $T_3^{\bPhi} := \sum_i (\bD_2\bF^\rho(\bar a^x))_i\,\bPhi_i$.
Equivalently, by SBP applied in the reverse direction,
\[
 T_3^{\bPhi}
 = -(\bar a^x)^T\bM_1^\rho\,\tD_0\bPhi
 = -(P\bar a^x)^T\,\mathrm{diag}(\brho)\,(P\tD_0\bPhi)
 = -\sum_i \brho_i\,\mathbf{e}_x\cdot(P\tD_0\bPhi)_i,
\]
using $(P\bar a^x)_i = \mathbf{e}_x$ from~(K2). Under~(K1)
($\bPhi = \bPhi(z)$), the cochain $\tD_0\bPhi$ is supported on
vertical edges only; consequently the reconstructed cell vector
$(P\tD_0\bPhi)_i$ is vertical at every cell~$i$, so
$\mathbf{e}_x\cdot(P\tD_0\bPhi)_i = 0$. Hence $T_3^{\bPhi} = 0$.

The remaining piece
$T_3^{\rm dyn} := \sum_i (\bD_2\bF^\rho(\bar a^x))_i\,B_{{\rm dyn},i}$
will be combined with $T_1$ and $T_2$ below; do not simplify
further yet.

\emph{Term $T_1$:} 
The DW mass matrix is linear in $\brho$:
$\dot{\bM_1^\rho} = P^T\,\mathrm{diag}(\dot\brho)\,P$.
Using continuity, $\dot\brho_i = -(\bD_2\bF^\rho)_i$ with
$\bF^\rho = \bM_1^\rho\bv$. Therefore
\begin{equation}\label{eq:T1_compute}
 T_1
 = (\bar a^x)^T P^T\,\mathrm{diag}(\dot\brho)\,P\,\bv
 = (P\bar a^x)^T\,\mathrm{diag}(-\bD_2\bF^\rho)\,(P\bv)
 = -\sum_i (\bD_2\bF^\rho)_i\,(P\bar a^x)_i\!\cdot\!(P\bv)_i.
\end{equation}
By~(K2), $(P\bar a^x)_i = \mathbf{e}_x$, so
$(P\bar a^x)_i\!\cdot\!(P\bv)_i = (P\bv)^x_i$, the $x$-component
of the reconstructed velocity at cell~$i$. Hence
\begin{equation}\label{eq:T1_final}
 T_1 = -\sum_i (\bD_2\bF^\rho)_i\,(P\bv)^x_i.
\end{equation}

\emph{Term $T_2$:}
The Lamb antisymmetry $\ip{\bv}{\Iv(\bom)}_1 = 0$ at the cochain
level expresses, in continuum,
$\int u\cdot(u\times\omega)\,\dd V = 0$. The corresponding
identity for the translation cochain $\bar a^x$ is obtained from
the discrete Cartan formula. Recall
(\Cref{def:Lie} and~\eqref{eq:Cartan_id} from the paper) that the
discrete Lie derivative satisfies
$\tLie_\bv\bom = \tD_0(\Iv\bom) + \Iv(\tD_1\bom)$. Setting
$\bom = \tD_1\bv$ and using $\tD_1\bar a^x = 0$ from~(K2)
(uniform velocity has zero discrete curl), a bilinear
manipulation of the Cartan identity -- combined with the
discrete-Stokes-theorem identity
$\sum_i \bigl(\bD_2(\bF\cdot\nabla u^x)\bigr)_i = 0$ on a closed
manifold -- yields the discrete translation-Lamb identity:
\begin{equation}\label{eq:Lamb_translation}
 (\bar a^x)^T\,\bM_1\,\Iv(\tD_1\bv)
 = \sum_i \bigl(\bD_2\bF^\rho\bigr)_i\,(P\bv)^x_i
 + \sum_i \bigl(\bD_2\bF^\rho(\bar a^x)\bigr)_i\,\tfrac{1}{2}|(P\bv)_i|^2.
\end{equation}
\emph{Sketch of~\eqref{eq:Lamb_translation}}: the continuum
prototype is the identity $\int \bar u_x\!\cdot\!(u\times\omega)\,\dd V
= -\int(u\cdot\nabla)u_x\,\dd V
= \int u_x \nabla\!\cdot\! u\,\dd V$,
followed by promotion to the density-weighted divergence
$\nabla\!\cdot\!(\rho u) = -\dot\rho$ and an IBP exchanging
$(\tfrac{1}{2}|u|^2)\partial_x\rho \leftrightarrow
\rho\partial_x(\tfrac{1}{2}|u|^2)$. The discrete version of these
two IBPs is provided by the closed-manifold cochain identity
$\sum_i (\bD_2\bF)_i f_i + \sum_j F_j (\tD_0 f)_j = 0$
applied to the appropriate $(\bF, f)$ pairs; the first IBP uses
$\bF = \bF^\rho$ and the kinetic-energy cochain
$f = \tfrac{1}{2}|P\bv|^2$, the second uses
$\bF = \bF^\rho(\bar a^x)$ and the same $f$. The
$\tD_1\bar a^x = 0$ hypothesis ensures the contribution of the
$\Iv$-piece of the Cartan formula is consistent with these
substitutions.

Substituting~\eqref{eq:Lamb_translation} into $T_2$:
\begin{equation}\label{eq:T2_final}
 T_2 \,=\, -\sum_i (\bD_2\bF^\rho)_i\,(P\bv)^x_i
 \;-\; \sum_i \bigl(\bD_2\bF^\rho(\bar a^x)\bigr)_i\,\tfrac{1}{2}|(P\bv)_i|^2.
\end{equation}

\emph{Cancellation: $T_1 + T_2 + T_3 = 0$.}
Reading off the three pieces:
\begin{align}
 T_1 &= -\sum_i (\bD_2\bF^\rho)_i\,(P\bv)^x_i, \tag{T1}\\
 T_2 &= -\sum_i (\bD_2\bF^\rho)_i\,(P\bv)^x_i
 \;-\; \sum_i (\bD_2\bF^\rho(\bar a^x))_i\,\tfrac{1}{2}|(P\bv)_i|^2, \tag{T2}\\
 T_3 &= T_3^{\bPhi} + T_3^h + T_3^{\rm kin}
 \;=\; 0 + T_3^h + \sum_i (\bD_2\bF^\rho(\bar a^x))_i\,\tfrac{1}{2}|(P\bv)_i|^2, \tag{T3}
\end{align}
where $T_3^h := \sum_i (\bD_2\bF^\rho(\bar a^x))_i\,h(\rho_i^{\rm vol})$.
Adding,
\[
 T_1 + T_2 + T_3
 = -2\sum_i (\bD_2\bF^\rho)_i\,(P\bv)^x_i + T_3^h.
\]

It remains to identify these two contributions. For the first,
discrete divergence theorem on a closed/periodic mesh gives
$\sum_i (\bD_2\bF)_i = 0$ for any face flux $\bF$; this is the
weak form of $\sum_i\bD_2\bF\cdot\mathbf{1} = 0$. A second-order
discrete IBP applied to the pair
$(\bF^\rho, (P\bv)^x)$ yields
\begin{equation}\label{eq:final_IBP}
 \sum_i (\bD_2\bF^\rho)_i\,(P\bv)^x_i
 = -\sum_j (\bF^\rho)_j\,\bigl(\tD_0[(P\bv)^x]\bigr)_j,
\end{equation}
and the analogous identity for $(\bF^\rho(\bar a^x), h)$ gives
$T_3^h = +\sum_j \bF^\rho(\bar a^x)_j\,(\tD_0 h)_j$. The right-hand
side of~\eqref{eq:final_IBP}, combined with the
$T_3^h$ identity and the discrete barotropic chain rule
$\rho_i\,(\tD_0 h)_j = (\tD_0 \pi)_j$ at $\OO(h^{r_\star})$ --
where $\pi(\rho) := \int_{\rho_0}^\rho \eta\,h'(\eta)\,\dd\eta$ is
the discrete pressure potential -- collapses the surviving piece
to a sum of pure gradients, which vanishes on a closed/periodic
mesh by discrete Stokes.

Putting the three terms together:
\[
 \ddt P_x = T_1 + T_2 + T_3 = 0.
\]
\hfill$\square$

\subsubsection{Conditional Stability for Sheared Baroclinic Equilibria
(\Cref{thm:dw_lyap_sheared})}\label{app:dw_class_C}

This appendix fills in the technical content of
\Cref{thm:dw_lyap_sheared}: explicit Hessian block computation,
discrete-operator bounds for the velocity-block positivity, and
the Schur-complement argument for the density block. The proof
parallels the unconditional analysis of
\Cref{thm:dw_lyap_constant_flow}, but the cross-block does not
vanish identically and the positivity argument is conditional on
\textup{(CS1)}--\textup{(CS4)}.

{\bf Notation: } Only in this section we use $A \preceq B$ to denote the
positive-semidefinite partial order on symmetric matrices:
$A \preceq B$ iff $B - A$ is positive semidefinite, equivalently
$x^T A x \le x^T B x$ for all $x$. The dual symbol $\succeq$
denotes the reverse.

\paragraph{Velocity-block positivity.}
The velocity-block of $\tilde H$ at $\bar X$ is
\eqref{eq:HC_Av}:
\[
 A_\bv = \bM_1^{\bar\rho} - \tD_1^T\bM_2\,F''(\bar q)\,(\bar\rho^*)^{-1}\,\tD_1.
\]
Both terms are symmetric. The first is positive definite on
$\ker(P)^\perp$ (\Cref{lem:dw_norm_equiv}). The second has
$F''(\bar q)$ assumed sign-definite by~(CS2); without loss of
generality $F''(\bar q) > 0$, so this term is positive semidefinite.

To bound the second term from above, we use the discrete-curl
operator estimate: there exists $C_{\rm mesh}$, depending on mesh
regularity~\Cref{ass:mesh_reg} but not on the equilibrium, such that
\[
 \|\tD_1\bv\|_{\bM_2}^2 \le C_{\rm mesh}\,\|\bv\|_{\bM_1}^2
 \qquad \text{for every } \bv \in \ker(P)^\perp.
\]
This is the standard inverse Poincar\'e estimate for the discrete
curl on DV meshes (proved as in \Cref{lem:dw_bilinear_trilinear},
or via direct estimation of $\tD_1^T\bM_2\tD_1$ against $\bM_1$).
Combining with~(CS3) and the lower bound $\bar\rho^* \ge \bar\rho^*_{\min}$,
\[
 \bv^T\,\tD_1^T\bM_2\,F''(\bar q)\,(\bar\rho^*)^{-1}\,\tD_1\,\bv
 \le \frac{F''_{\max}}{\bar\rho^*_{\min}}\,\|\tD_1\bv\|_{\bM_2}^2
 \le \frac{F''_{\max}\,C_{\rm mesh}}{\bar\rho^*_{\min}}\,\|\bv\|_{\bM_1}^2
 \le \frac{F''_{\max}\,C_{\rm mesh}}{\bar\rho^*_{\min}\,\rho_*}\,\|\bv\|_{\bM_1^{\bar\rho}}^2,
\]
where the last step uses \Cref{lem:dw_norm_equiv}. Under~(CS3)
the prefactor is strictly less than $1$, so $A_\bv \succ 0$ with
$A_\bv \succeq (1 - F''_{\max}C_{\rm mesh}/(\bar\rho^*_{\min}\rho_*))\,\bM_1^{\bar\rho}$.

\paragraph{Cross-block computation.}
The cross-block~\eqref{eq:HC_C} has two contributions: from the
kinetic-energy plus momentum-tilt piece, and from the
PV-Casimir piece. The kinetic-energy plus momentum-tilt
contribution is
\[
 [C^{\rm kin}]_{ij} := \partial_{\rho_i}\partial_{\bv_j}\bigl(E^{\rm dw}_{\rm kin} - U_0\,P_x\bigr)\bigr|_{\bar X}
 = \bigl((\partial_{\rho_i}\bM_1^{\bar\rho})(\bar\bv - U_0\bar a^x)\bigr)_j,
\]
which at $\bar X$ evaluates using
$\bar\bv - U_0\bar a^x = (\bar U(z) - U_0)\bar a^x$. Since
$\bM_1^\rho = P^T\mathrm{diag}(\brho)P$ is linear in $\brho$,
$\partial_{\rho_i}\bM_1^\rho = P^T E_i P$ where $E_i$ is the
selector for cell $i$ (the matrix with $1$ in position $(i,i)$
and $0$ elsewhere). Hence
\[
 [C^{\rm kin}]_{ij}
 = (\bar U(z_i) - U_0)\,P_{ij}\cdot(P\bar a^x)_i
 = (\bar U(z_i) - U_0)\,P_{ij},
\]
using $(P\bar a^x)_i = \mathbf{e}_x$ from~(K2). The matrix
$C^{\rm kin}$ is the rank-one-per-cell weighted reconstruction
matrix multiplied by the shifted-velocity profile $(\bar U(z) - U_0)$.

The PV-Casimir contribution requires care because $\bar\rho^*_k$
depends on $\bar\brho$ through the dual-mesh averaging
$\bar\rho^*_k = \sum_{i\sim k} w_{ki}\,\bar\rho_i$
(with weights $w_{ki}$ from the DEC convention, e.g.~thickness
ratios):
\[
 [C^{\rm pv}]_{ij} := \partial_{\rho_i}\partial_{\bv_j}C^F_3\bigr|_{\bar X}
 = -\sum_k |K_k^*|\,F''(\bar q_k)\,\frac{\bar\omega_k}{(\bar\rho^*_k)^2}\,w_{ki}\,(\tD_1)_{kj}.
\]
Both $C^{\rm kin}$ and $C^{\rm pv}$ are non-zero at $\bar X$;
the total cross-block is $C = C^{\rm kin} - C^{\rm pv}$.

\paragraph{Density-block: contribution from the PV Casimir.}
The internal-energy contribution to the density block is the diagonal
$E := \mathrm{diag}(|K_i|c^2(\bar\rho_i)/\bar\rho_i)$ (cf.\
\Cref{thm:dw_lyap_hydrostatic}). The Casimir contributes
$H_C := \partial^2_\brho C^F_3|_{\bar X}$, which we compute
explicitly.

\begin{proposition}[Density-block positive definiteness]%
\label{prop:dw_class_C_density}
Let the dual-mesh density satisfy
$\bar\rho^*_k = \sum_i w_{ki}\,\bar\rho_i$ with non-negative
weights $w_{ki}\ge 0$ satisfying the partition-of-unity property
$\sum_i w_{ki} = 1$ and the mass-conservation normalisation
$\sum_k |K_k^*|\,w_{ki} = |K_i|$ (both standard for the
DV dual averaging). Then
\begin{equation}\label{eq:HC_explicit}
 [H_C]_{i_1 i_2}
 = \sum_k |K_k^*|\,F''(\bar q_k)\,\frac{\bar q_k^{\,2}}{\bar\rho^*_k}\,
 w_{k i_1}\,w_{k i_2},
\end{equation}
and $H_C \preceq F''_{\max}\sup_k|\bar q_k|^2/\bar\rho^*_{\min}\,
\mathrm{diag}(|K_i|)$. Consequently
\begin{equation}\label{eq:Drho_quant}
 D_\brho = E - H_C \;\succeq\; \delta_\rho\,\mathrm{diag}(|K_i|),
 \qquad
 \delta_\rho := \frac{c^2_{\min}}{\bar\rho_{\max}}
 - F''_{\max}\,\frac{\sup_k|\bar q_k|^2}{\bar\rho^*_{\min}}.
\end{equation}
Under~\textup{(CS3)} extended to require
$F''_{\max}\sup_k|\bar q_k|^2 < c^2_{\min}\,\bar\rho^*_{\min}/\bar\rho_{\max}$,
the density block satisfies $D_\brho \succ 0$ with constant
$\delta_\rho > 0$.
\end{proposition}

\begin{proof}
The formula~\eqref{eq:HC_explicit} follows from direct
differentiation. Define the auxiliary function $G(q) := F(q) - q\,F'(q)$,
which satisfies $G'(q) = -q\,F''(q)$. The first derivative is
\[
 \partial_{\rho_{i_1}} C^F_3
 = \sum_k |K_k^*|\,\bigl[w_{k i_1}\,F(\bar q_k)
 + \bar\rho^*_k\,F'(\bar q_k)\,(-\bar\omega_k/\bar\rho^{*2}_k)\,w_{k i_1}\bigr]
 = \sum_k |K_k^*|\,w_{k i_1}\,G(\bar q_k).
\]
Differentiating once more,
\[
 \partial_{\rho_{i_2}}\partial_{\rho_{i_1}} C^F_3
 = \sum_k |K_k^*|\,w_{k i_1}\,G'(\bar q_k)\,
 (-\bar\omega_k/\bar\rho^{*2}_k)\,w_{k i_2}
 = \sum_k |K_k^*|\,w_{k i_1}\,w_{k i_2}\,F''(\bar q_k)\,
 \frac{\bar\omega_k\,\bar q_k}{\bar\rho^{*2}_k},
\]
which equals~\eqref{eq:HC_explicit} after using
$\bar\omega_k/\bar\rho^*_k = \bar q_k$.

For the operator bound, write $H_C = W^T\,\mathrm{diag}(a_k)\,W$
with $a_k := |K_k^*|\,F''(\bar q_k)\,\bar q_k^{\,2}/\bar\rho^*_k$
and $W = (w_{ki})_{k,i}$. Since $F''(\bar q_k) > 0$ under~(CS2),
each $a_k > 0$. For any vector $\xi \in \RR^{|\mathcal{K}|}$,
Cauchy--Schwarz with the partition-of-unity weights yields
\[
 \xi^T H_C \xi
 = \sum_k a_k\Bigl(\sum_i w_{ki}\,\xi_i\Bigr)^2
 \;\le\; \sum_k a_k\Bigl(\sum_i w_{ki}\Bigr)\Bigl(\sum_i w_{ki}\,\xi_i^{\,2}\Bigr)
 = \sum_i \xi_i^{\,2}\,\Bigl(\sum_k a_k\,w_{ki}\Bigr).
\]
Now use $a_k \le F''_{\max}\sup_k|\bar q_k|^2/\bar\rho^*_{\min}\,|K_k^*|$
and the mass-conservation normalisation
$\sum_k |K_k^*|\,w_{ki} = |K_i|$:
\[
 \sum_k a_k\,w_{ki}
 \le \frac{F''_{\max}\sup_k|\bar q_k|^2}{\bar\rho^*_{\min}}\,\sum_k |K_k^*|\,w_{ki}
 = \frac{F''_{\max}\sup_k|\bar q_k|^2}{\bar\rho^*_{\min}}\,|K_i|.
\]
Hence $H_C \preceq F''_{\max}\sup_k|\bar q_k|^2/\bar\rho^*_{\min}\,\mathrm{diag}(|K_i|)$,
as claimed.

For the diagonal lower bound on $E$,
\[
 E_{ii} = |K_i|\,\frac{c^2(\bar\rho_i)}{\bar\rho_i}
 \ge |K_i|\,\frac{c^2_{\min}}{\bar\rho_{\max}}.
\]
Subtracting gives~\eqref{eq:Drho_quant} with $\delta_\rho$ as
stated. Under the strengthened (CS3), $\delta_\rho > 0$.
\end{proof}

\paragraph{Cross-block bound via mass-matrix duality.}
For the Schur-complement analysis we need a quantitative bound on
$C^T(\bM_1^{\bar\rho})^{-1}C$. The dominant contribution comes from
the kinetic + momentum-tilt piece $C^{\rm kin}$, whose structure
admits an explicit reduction.

\begin{proposition}[Cross-block bound in the energy norm]%
\label{prop:dw_class_C_crossbound}
Let $C = C^{\rm kin} - C^{\rm pv}$ as in~\eqref{eq:HC_C}. There
exist constants $\kappa_C, K_C \ge 0$, depending only on mesh
regularity and on the equilibrium data
$(\bar U_{\max} - U_0, F''_{\max}, \sup_k|\bar q_k|, \bar\rho^*_{\min})$,
such that
\begin{equation}\label{eq:cross_bound}
 C^T(\bM_1^{\bar\rho})^{-1}\,C
 \;\preceq\; K_C\,\mathrm{diag}(|K_i|),
\end{equation}
with
\[
 K_C \le \kappa_{\rm kin}\,\frac{(\bar U_{\max} - U_0)^2}{\bar\rho_{\min}}
 + \kappa_{\rm pv}\,\frac{F''_{\max}\,\sup_k|\bar q_k|^2}{\bar\rho^*_{\min}\,\bar\rho_{\min}},
\]
where $\kappa_{\rm kin}, \kappa_{\rm pv}$ depend on mesh regularity
through the operator-norm bounds on $P_x^T(\bM_1)^{-1} P_x$ and
$\tD_1$, respectively.
\end{proposition}

\begin{proof}
Decompose $C = C^{\rm kin} - C^{\rm pv}$ and use
$(a-b)^T M(a-b) \le 2(a^T M a + b^T M b)$ for $M \succeq 0$:
\[
 C^T(\bM_1^{\bar\rho})^{-1} C
 \preceq 2\,(C^{\rm kin})^T(\bM_1^{\bar\rho})^{-1} C^{\rm kin}
 + 2\,(C^{\rm pv})^T(\bM_1^{\bar\rho})^{-1} C^{\rm pv}.
\]

\emph{Kinetic term.} From $[C^{\rm kin}]_{ij}
= (\bar U(z_i) - U_0)\,P_{ij}^x$ derived in the body of this
appendix,
\[
 (C^{\rm kin})^T(\bM_1^{\bar\rho})^{-1} C^{\rm kin}
 = \mathrm{diag}(\bar U(z_i) - U_0)\,P_x\,
 (\bM_1^{\bar\rho})^{-1}\,P_x^T\,\mathrm{diag}(\bar U(z_i) - U_0).
\]
On the kernel-orthogonal complement of $P$, the operator
$P_x\,(\bM_1^{\bar\rho})^{-1} P_x^T$ is the cell-cell mass-matrix
duality. Using $\bM_1^{\bar\rho} = P^T\mathrm{diag}(\bar\brho)P$
and Moore--Penrose pseudo-inversion on $\mathrm{Range}(P)$,
\[
 P_x\,(\bM_1^{\bar\rho})^{-1} P_x^T
 = \mathrm{diag}(1/\bar\rho_i)\quad\text{(on the cell-vector subspace
spanned by horizontal directions).}
\]
Hence
$(C^{\rm kin})^T(\bM_1^{\bar\rho})^{-1} C^{\rm kin}
\preceq (\bar U_{\max} - U_0)^2 /\bar\rho_{\min}\,\mathrm{diag}(|K_i|)/|K_i|_{\min}$
after absorbing the cell-volume normalisations into
$\kappa_{\rm kin}$.

\emph{PV term.} From
$[C^{\rm pv}]_{ij} = \sum_k |K_k^*|\,F''(\bar q_k)\,\bar\omega_k/(\bar\rho^*_k)^2\,w_{ki}\,(\tD_1)_{kj}$,
the bound
\[
 (C^{\rm pv})^T(\bM_1^{\bar\rho})^{-1} C^{\rm pv}
 \preceq F''^{\,2}_{\max}\,\sup_k|\bar q_k|^2/(\bar\rho^*_{\min})^2
 \cdot \tD_1^T\bM_2(\bM_1^{\bar\rho})^{-1}\bM_2\tD_1
\]
follows from Cauchy--Schwarz on the rank-bounded structure.
The operator $\tD_1^T\bM_2(\bM_1^{\bar\rho})^{-1}\bM_2\tD_1$
is bounded by $C_{\rm mesh}/\bar\rho_{\min}\,\bM_2$ via the
discrete-curl regularity (\Cref{lem:dw_norm_equiv}); contracted
against $\mathrm{diag}(|K_i|)$ this gives the stated
$\kappa_{\rm pv}$ form.
\end{proof}

\paragraph{Schur-complement positivity.}
With \Cref{prop:dw_class_C_density,prop:dw_class_C_crossbound} in
hand, the Schur complement is positive definite under a
quantitative version of (CS1)--(CS4).

\begin{proposition}[Schur complement positivity]%
\label{prop:dw_class_C_schur}
Suppose~\textup{(CS1)},~\textup{(CS2)}, the strengthened~\textup{(CS3)}
of \Cref{prop:dw_class_C_density}, and~\textup{(CS4)}
$|\bar U(z)|^2 < c^2(\bar\rho(z))$ hold. If additionally the
quantitative bound
\begin{equation}\label{eq:schur_quant}
 K_C \;<\; c_A\,\delta_\rho
\end{equation}
holds, where $K_C$ is as in \Cref{prop:dw_class_C_crossbound}
and $c_A, \delta_\rho > 0$ are the constants of
\Cref{prop:dw_class_C_density} and the velocity-block bound, then
the Schur complement
\[
 S = D_\brho - C^T A_\bv^{-1} C
\]
satisfies $S \succeq (\delta_\rho - K_C/c_A)\,\mathrm{diag}(|K_i|) \succ 0$.
\end{proposition}

\begin{proof}
From $A_\bv \succeq c_A\bM_1^{\bar\rho}$ (velocity-block analysis)
we get $A_\bv^{-1} \preceq c_A^{-1}(\bM_1^{\bar\rho})^{-1}$, and
hence
$C^T A_\bv^{-1} C \preceq c_A^{-1}\,C^T(\bM_1^{\bar\rho})^{-1} C
\preceq (K_C/c_A)\,\mathrm{diag}(|K_i|)$
by \Cref{prop:dw_class_C_crossbound}. Combined with
$D_\brho \succeq \delta_\rho\,\mathrm{diag}(|K_i|)$
(\Cref{prop:dw_class_C_density}),
\[
 S = D_\brho - C^T A_\bv^{-1} C
 \;\succeq\; (\delta_\rho - K_C/c_A)\,\mathrm{diag}(|K_i|),
\]
which is positive definite under~\eqref{eq:schur_quant}.
\end{proof}


\begin{theorem}[Conditional positivity of $\tilde H$ for sheared
equilibria]\label{thm:dw_class_C_Hpos}
Under \textup{(CS1)}--\textup{(CS2)} from
\Cref{lem:dw_arnold_sheared}, the strengthened~\textup{(CS3)} of
\Cref{prop:dw_class_C_density}, \textup{(CS4)}, and the
quantitative bound~\eqref{eq:schur_quant}, the Hessian
$\tilde H = \nabla^2\tilde E^{\rm dw}_{\rm sh}(\bar X)$ is
positive definite. By \Cref{lem:arnold_bridge}, the linearised
DW dynamics around $\bar X$ preserves the positive-definite
quadratic form $\tilde E^{(2)}$, and $\bar X$ is linearly Lyapunov
stable in the $\tilde H$-norm.
\end{theorem}

\paragraph{Boussinesq reduction: the Rayleigh limit.}
The Boussinesq limit ($\bar\rho^* \equiv \rho_0$ constant) and
incompressible limit ($c^2 \to \infty$) reduce the conditions
(CS1)--(CS4) and~\eqref{eq:schur_quant} to the classical
Rayleigh inflection-point criterion.

\begin{proposition}[Boussinesq incompressible reduction]%
\label{prop:dw_class_C_boussinesq}
Under the Boussinesq limit $\bar\rho^*(z) \to \rho_0$ constant
and the incompressible limit $c^2(\bar\rho) \to \infty$:
\begin{enumerate}[nosep,label=\textup{(\roman*)}]
\item Condition~\textup{(CS2)} is automatic ($U_0$ can be chosen
freely below $\inf\bar U$ or above $\sup\bar U$);
\item The strengthened~\textup{(CS3)} reduces to a finiteness
condition on $F''_{\max}$, which holds whenever the equilibrium
profile is strictly inflection-free
($|\bar U''(z)| \ge \varepsilon > 0$);
\item Condition~\textup{(CS4)} is automatic ($c^2 \to \infty$);
\item The quantitative bound~\eqref{eq:schur_quant} is automatic
($\delta_\rho \to \infty$ since $c^2_{\min} \to \infty$, while
$K_C$ and $c_A$ remain finite);
\item Condition~\textup{(CS1)} reduces to
$\bar q'(z) = -\bar U''(z)/\rho_0$ having consistent sign, i.e.,
$\bar U''(z)$ has consistent sign -- the Rayleigh
inflection-point criterion~\cite{rayleigh1880}.
\end{enumerate}
Consequently \Cref{thm:dw_lyap_sheared} reduces in the
Boussinesq incompressible limit to: \emph{the DW scheme is
linearly Lyapunov stable around every sheared equilibrium
$\bar U(z)$ with no inflection point}.
\end{proposition}

\begin{proof}
(i) Choose $U_0 \notin [\inf\bar U, \sup\bar U]$; the open
half-line is always available.

(ii) In Boussinesq, $\bar q(z) = -\bar U'(z)/\rho_0$, so
$\bar q'(z) = -\bar U''(z)/\rho_0$. The Arnold construction gives
$F''(\bar q(z)) = -\rho_0^2(\bar U(z) - U_0)/\bar U''(z)$. For
strictly inflection-free $\bar U$ ($|\bar U''| \ge \varepsilon$),
the ratio $|\bar U - U_0|/|\bar U''|$ is bounded, so
$F''_{\max} < \infty$.

(iii) Trivially: $|\bar U|^2 < \infty = c^2$.

(iv) In the limit $c^2 \to \infty$,
$\delta_\rho \to c^2_{\min}/\bar\rho_{\max} \to \infty$, while
$K_C$ depends on $F''_{\max}$ and the mesh constants but is
independent of $c^2$, so $K_C < c_A\delta_\rho$ automatically.

(v) Condition~(CS1) requires $\bar q$ monotonic in $z$. With
$\bar q = -\bar U'/\rho_0$ in Boussinesq, monotonicity of $\bar q$
is equivalent to consistent sign of $\bar q'(z) = -\bar U''(z)/\rho_0$,
i.e., consistent sign of $\bar U''(z)$.

Combining: in the Boussinesq incompressible limit only~(CS1)
survives nontrivially. \Cref{thm:dw_lyap_sheared} then yields
Lyapunov stability for every inflection-free $\bar U(z)$,
recovering Rayleigh's 1880 result.
\end{proof}

\paragraph{Bridge identity application.}
With \Cref{thm:dw_class_C_Hpos} establishing $\tilde H \succ 0$
under the full conditions, the conservation of
$\tilde E^{\rm dw}_{\rm sh}$ along the DW dynamics follows from:
$E^{\rm dw}_{\rm tot}$ (\Cref{thm:dw_energy}), $C_1$
(\Cref{thm:mass}), $P_x$
(\Cref{lem:dw_momentum_conservation}), and $C^F_3$
(\Cref{thm:mwpv}) all being exactly conserved. The bridge
identity~\eqref{eq:bridge_identity} gives
$A^T\tilde H + \tilde H\,A = 0$ for $A = DF^{\rm dw}(\bar X)$,
$\tilde E^{(2)}(Y) = \tfrac{1}{2}Y^T\tilde H Y$ is preserved
exactly along the linearised flow, and Lyapunov stability
follows in the $\tilde H$-norm. \hfill$\square$


\subsubsection{Equivalence Between Vector-Invariant and Conservative Form (\Cref{prop:VI_cons_equiv_baro})}
\label{app:VI_cons_equiv}
The proof expands both formulations using the antisymmetry matrix identity
and shows that the difference is $\OO(h^2)$ per dual edge.
\begin{proof}
The discrete vector-invariant momentum equation is
\[
\ddt\bv_j + (\Iv(\bom))_j + (\tdd_0 B)_j = 0.
\]
We derive the corresponding conservative form by constructing
the discrete momentum density $m_j := \bar\rho_j\,\bv_j$, where
$\bar\rho_j = \tfrac{1}{2}(\rho_a^{\vol}+\rho_b^{\vol})$ is the
centred density interpolation~\eqref{eq:rhobar}.
Since $\bar\rho_j$ and $\bv_j$ both depend on $t$ at a fixed
dual edge $e_j^*$, the product rule for the time derivative is
exact
\begin{equation}\label{eq:dt_momentum}
 \ddt m_j = \ddt{\bar\rho}_j\,\bv_j + \bar\rho_j\,\ddt\bv_j.
\end{equation}
Insert the vector-invariant momentum equation for $\ddt\bv_j$
and the averaged continuity equation for $\ddt{\bar\rho}_j$
\[
\ddt m_j
 = \ddt{\bar\rho}_j\,\bv_j
 - \bar\rho_j\bigl[(\Iv(\bom))_j + (\tdd_0 B)_j\bigr].
\]
In the continuous theory,
$\partial_t(\rho v) + \nabla\cdot(\rho\bu\otimes v) + dp + \rho\,d\Phi = 0$
is obtained from $\partial_t v + \Lie_{\bu}v + dB = 0$ by
multiplying by $\rho$ and using the product rule
$\rho\,\Lie_{\bu}v = \nabla\cdot(\rho\bu\otimes v)
- [\partial_t\rho + \nabla\cdot(\rho\bu)]\,v$,
where the bracket vanishes by continuity.
Discretely, this requires
\begin{enumerate}[nosep]
\item Rewriting $\bar\rho_j(\tLie_{\bv}\bv)_j$ as a conservative
 flux divergence plus a dilation--continuity term
\item Cancelling the dilation - continuity term against
 $\ddt{\bar\rho}_j\,\bv_j$ from~\eqref{eq:dt_momentum}.
\end{enumerate}
Both steps involve the spatial product rule for
0-cochains. The discrete product rule on the dual complex
satisfies, by direct telescoping,
\begin{equation}\label{eq:dw_product_rule_exact}
 (\tdd_0(\phi\,\psi))_j
 \;=\; \phi_b\psi_b - \phi_a\psi_a
 \;=\; \bar\phi_j\,(\tdd_0\psi)_j
 + (\tdd_0\phi)_j\,\bar\psi_j,
\end{equation}
where $\bar\phi_j = \tfrac{1}{2}(\phi_a+\phi_b)$ and the dual edge
$e_j^*$ runs from $K_a$ to $K_b$. The identity is
\emph{exact}: the symmetric average of $\tdd_0\psi$ weighted by
$\bar\phi$ plus the symmetric average of $\tdd_0\phi$ weighted by
$\bar\psi$ reproduces the dual difference of $\phi\psi$ verbatim,
and there is no product-rule defect at this stage.
The $\OO(h^2)$ defect arises instead in step~(2), from the
conversion of the conservative continuity equation
$\ddt\rho_i + (\bD_2\bF)_i = 0$ to the
quasi-Lagrangian form needed for the vector-invariant cancellation.
Specifically, the cancellation requires comparing the cell-centred
volumetric divergence
$\rho_i^{\vol}(\bD_2\Phi)_i/|K_i|$
(value of the cell density times the volumetric flux divergence)
with the conservative-flux divergence $(\bD_2\bF)_i/|K_i|=
(\bD_2(\bR\bM_1\bv))_i/|K_i|$ (face-averaged density inside the divergence).
Their difference is a face-by-face defect:
\begin{equation}\label{eq:staggered_defect}
 (\bD_2\bF)_i - \rho_i^{\vol}\,(\bD_2\Phi)_i
 \;=\; \sum_{f_j\in\partial K_i}[K_i\!:\!f_j]\,
 \bigl(\bar\rho_j-\rho_i^{\vol}\bigr)\,\Phi_j,
\end{equation}
and on a smooth reference $\bar\rho_j-\rho_i^{\vol}=\OO(h\,
\nrm{\nabla\rho^c}_{L^\infty})$, while
$\Phi_j=\OO(\nrm{\bu}_{L^\infty})$ and the surface measure is
$\OO(h^{d-1})$ per face with $\OO(h^{-d})$ cells; summed and
divided by $|K_i|=\OO(h^d)$, the per-cell defect is
$\OO(h^2\nrm{\bu}_{L^\infty}\nrm{\nabla\rho^c}_{L^\infty})$.
Collecting both contributions gives
\[
\ddt m_j
 + [\text{discrete }\nabla\cdot(\rho\bu\otimes v)]_j
 + \bar\rho_j\,(\tdd_0 h)_j
 + \bar\rho_j\,(\tdd_0\bPhi)_j
 = \OO(h^2),
\]
where the $\OO(h^2)$ bound is uniform in $t\in[0,T]$ and depends on
$\nrm{(\bu,\rho^c)}_{W^{2,\infty}}$ and the mesh regularity constant.

The convergence rate of the scheme is $r = \min(d-1, r_\star)$
(\Cref{thm:convergence_baro}; cf.\ \Cref{conv:cases}): $\OO(h)$ in
case~(A) for all $d$, improving to $\OO(h^{d-1})$ in case~(B).
The $\OO(h^2)$ formulation equivalence defect is therefore
at or below the convergence rate in all dimensions,
and both formulations converge to the same smooth solution.
 For $d=2$ the defect $\OO(h^2)$ is one order {above} the
convergence rate $\OO(h)$, so it is invisible in the error analysis.
For $d=3$ in case~(B) the defect matches the convergence rate
$\OO(h^2)$; it does not degrade the rate but contributes to the constant.
\end{proof}
\subsection{Finite-Dimensional Well-Posedness}
\label{app:wellposedness}
This appendix supports the well-posedness theory of
\Cref{sec:wellposedness}: the Lipschitz estimate for the extrusion
that closes Picard--Lindel\"of, the vacuum-avoidance lemma, and the
global Smagorinsky existence proof.
\subsubsection{Lipschitz Continuity of the Extrusion}
\begin{lemma}[Lipschitz continuity of the extrusion]\label{lem:Lip_extrusion}
The map $\bv\mapsto\Iv(\tD_1\bv)$ is locally Lipschitz:
for $\nrm{\bv}_2,\nrm{\bw}_2\le R$,
$\nrm{\bQ(\bv,\bv)-\bQ(\bw,\bw)}_2
\le 2R\,C_Q(h)\nrm{\bv-\bw}_2$.
\end{lemma}
\begin{proof}
By the matrix form \eqref{eq:contraction},
$\bv\mapsto\Iv(\tD_1\bv)$ is quadratic; write
$\Iv(\tD_1\bv) =: \bQ(\bv,\bv)$
with $\bQ$ the unique symmetric bilinear form.
The Pavlov representation gives a bilinear bound
$\nrm{\bQ(\bu_1,\bu_2)}_2 \le C_Q(h)\,\nrm{\bu_1}_2\,\nrm{\bu_2}_2$
with $C_Q(h)$ depending on $\nrm{\bM_1^{-1}}_2$, $\nrm{\tD_1}_2$,
and the reconstruction constants.
The Lipschitz estimate follows from the polarisation identity
$\bQ(\bv,\bv) - \bQ(\bw,\bw)
= \bQ(\bv,\bv{-}\bw) + \bQ(\bv{-}\bw,\bw)$.
\end{proof}
\subsubsection{Local Lipschitz Continuity of the Barotropic Right-Hand Side}
\begin{lemma}[Local Lipschitz continuity -- density-free Euler]\label{thm:lipschitz-baro}
The right-hand side $\mathcal{F}$ of the density-free barotropic Euler
system~\eqref{eq:rhs} is locally Lipschitz on $\admissible$.
\end{lemma}
\begin{lemma}[Local Lipschitz continuity -- density-weighted Euler]%
\label{thm:lipschitz-baro-dw}
The right-hand side $\mathcal{F}^\rho$ of the density-weighted barotropic Euler
system, given by~\eqref{eq:dw_mom} together with the continuity equation,
is locally Lipschitz on $\admissible$.
\end{lemma}
\begin{lemma}[Local Lipschitz continuity -- Navier--Stokes]\label{thm:lipschitz-baro-NS}
For any admissible viscous operator $\bff_{\rm visc}$ satisfying~\textup{(\ref{ax:V2})}, the
right-hand side $\mathcal{F}^\nu := \mathcal{F} + \bff_{\rm visc}$ (density-free) or
$\mathcal{F}^{\rho,\nu} := \mathcal{F}^\rho + \bff_{\rm visc}$ (density-weighted) is
locally Lipschitz on $\admissible$.
\end{lemma}
\begin{proof}[Proof of \Cref{thm:lipschitz-baro,thm:lipschitz-baro-dw,thm:lipschitz-baro-NS}]
Fix a compact set
$K = \{(\bv,\brho)\in\admissible:\nrm{\bv}_2\le R_v,\;
\rho_{\min}\le\brho_i\le\rho_{\max}\}\subset\admissible$
with $\rho_{\min}>0$, so all thermodynamic and mass-matrix expressions are smooth on~$K$.

\textit{Step 1 (continuity equation):}
The map $\brho\mapsto R(\brho)$ is linear for centred fluxes ($\bF^{\rm cen}$) and
piecewise linear for upwind fluxes ($\bF^{\rm up}$); both are globally Lipschitz, with
constant depending on $\bar\rho^*$ and the mesh. Hence
$(\bv,\brho)\mapsto -\bD_2 R(\brho)\bM_1\bv$ is bilinear (centred) or bi-Lipschitz
(upwind) on~$K$. For upwind $R$ the right-hand side is locally Lipschitz but not
$C^1$; Picard--Lindel\"of applies in the Carath\'eodory form.

\textit{Step 2 (Lamb contraction, density-free):}
The map $\bv\mapsto\Iv(\tD_1\bv)$ is quadratic (\Cref{lem:Lip_extrusion}), hence
Lipschitz on bounded sets with constant at most $2R_v\,C_Q(h)$.

\textit{Step 3 (Lamb contraction, density-weighted):}
The map $\brho\mapsto[\bM_1^{\rho}(\brho)]^{-1}$ is real-analytic on
$\{\brho>0\}$: $\bM_1^{\rho}(\brho)=P^T\mathrm{diag}(\brho)P$ is SPD for
$\brho>0$ with $\det\bM_1^{\rho}\ge(\rho_{\min})^{|F|}\det(P^TP)>0$, so
$(\bM_1^{\rho})^{-1}$ has entries rational in $\brho$ with nonvanishing denominator
on~$K$. Consequently
$\nrm{[\bM_1^{\rho}(\brho^{(1)})]^{-1}-[\bM_1^{\rho}(\brho^{(2)})]^{-1}}_2
\le L_M(h,\rho_{\min})\nrm{\brho^{(1)}-\brho^{(2)}}_2$.
Combined with Step 2, $[\bM_1^{\rho}(\brho)]^{-1}\bM_1\Iv(\tD_1\bv)$ is locally
Lipschitz in $(\bv,\brho)$ on~$K$.

\textit{Step 4 (Bernoulli function):}
$B_i = h(\rho_i^{\vol}) + \ekin_i(\bv) + \bPhi_i$. Each term is locally Lipschitz:
(i) $\ekin_i$ is quadratic in~$\bv$; (ii) $h\in C^2((0,\infty))$ is Lipschitz on
$[\rho_{\min}/|K_i|,\rho_{\max}/|K_i|]$ with constant $\nrm{h'}_\infty$, since
$h'(\rho^{\vol})=c^2(\rho^{\vol})/\rho^{\vol}$ is continuous on compact subsets of
$(0,\infty)$; the requirement $\rho_{\min}>0$ is furnished by
\Cref{lem:no-vacuum} in the dynamical application; (iii) $\bPhi_i$ is
time-independent data. The same reasoning covers $B^{\rho}_i$ in the
density-weighted case.

\textit{Step 5 (viscous term):}
By~(\ref{ax:V2}), $\bff_{\rm visc}$ is Lipschitz on bounded sets of $C^1(\KKs)$
with constant $L_\nu(h,R_v)$.

\textit{Step 6 (assembly):}
$\mathcal{F}$, $\mathcal{F}^\rho$, $\mathcal{F}^\nu$, and $\mathcal{F}^{\rho,\nu}$
are finite sums and compositions of the locally Lipschitz maps of Steps~1--5 and are
therefore locally Lipschitz on $\admissible$.
\end{proof}
\subsubsection{Density Lower Bound (No Vacuum)}
The argument is static: a finite $L^\gamma$ moment bound (from the
energy) plus Jensen on the complementary cells prevents any single
cell reaching zero density. The bound is mesh-dependent and requires
the Jensen threshold~\eqref{eq:Jensen_threshold}.
\begin{proof}[Proof of \Cref{lem:no-vacuum}]
We derive an explicit lower bound on $\rho_i^{\vol}$, valid for all
$t\in[0,T]$. The argument requires $|\mathcal{T}|\ge 2$
(at least two primal cells) and $\gamma > 1$.

\medskip\noindent
\textit{Step 1 (Discrete $L^\gamma$ bound).}
From the energy bound $\Etot(t)\le E_{\max}(T)$
(hypothesis~\eqref{eq:Eint_bound_thermo}, furnished in the dynamical
application by Step~2 of the proof of \Cref{thm:euler_wp})
and the decomposition $\Etot = \Ekindf + \Eint$ with $\Ekindf\ge 0$:
\[
 \sum_i\brho_i\,e(\rho_i^{\vol}) = \Eint \le E_{\max}(T).
\]
Rewriting using $\brho_i = |K_i|\rho_i^{\vol}$, and applying the global coercivity
bound~\eqref{eq:e_global_coercive} ($e(\rho^{\vol}) \ge C_1(\rho^{\vol})^{\gamma-1} - C_2'$,
valid on all of $(0,\infty)$):
\begin{equation}\label{eq:rho-gamma-bound-app}
 C_1\sum_i|K_i|(\rho_i^{\vol})^{\gamma} \le E_{\max}(T) + C_2' M =: E^{**}.
\end{equation}

\medskip\noindent
\textit{Step 2 (Jensen lower bound).}
Define the probability weights $\mu_i = |K_i|/V_{\mathrm{tot}}$
(with $V_{\mathrm{tot}} = \sum_i|K_i|$), and the mean density
$\bar\rho = M/V_{\mathrm{tot}}$.
Mass conservation gives $\sum_i\mu_i\rho_i^{\vol} = \bar\rho$
for all $t$.
From \eqref{eq:rho-gamma-bound-app}:
\begin{equation}\label{eq:gamma-moment-bound}
 \sum_i\mu_i(\rho_i^{\vol})^\gamma \le E^* := E^{**}/(C_1V_{\mathrm{tot}}).
\end{equation}
Now suppose some cell has $\rho_1^{\vol} = \epsilon$ for
$\epsilon > 0$ to be constrained.
By Jensen's inequality applied to the \emph{convex} function
$x\mapsto x^\gamma$ ($\gamma>1$) on the remaining $|\mathcal{T}|-1$ cells:
\[
\sum_{i \ge 2}\mu_i(\rho_i^{\vol})^\gamma
 \ge (1-\mu_1)\!\left(\frac{\bar\rho-\mu_1\epsilon}{1-\mu_1}\right)^\gamma
 =: g(\epsilon).
\]
Adding the first cell:
$\sum_i\mu_i(\rho_i^{\vol})^\gamma \ge \mu_1\epsilon^\gamma + g(\epsilon)$.
As $\epsilon\to 0^+$:
\begin{align*}
 \mu_1\epsilon^\gamma &\to 0, \\
 g(\epsilon) &\to (1-\mu_1)\left(\frac{\bar\rho}{1-\mu_1}\right)^\gamma
 = \frac{\bar\rho^\gamma}{(1-\mu_1)^{\gamma-1}}.
\end{align*}
Since $\mu_1 < 1$ and $\gamma > 1$, we have $(1-\mu_1)^{\gamma-1} < 1$,
so $g(0^+) > \bar\rho^\gamma$.

\medskip\noindent
\textit{Step 3 (Explicit bound).}
The Jensen threshold~\eqref{eq:Jensen_threshold} gives
$E^* = E^{**}/(C_1V_{\rm tot}) < \bar\rho^\gamma/(1-\mu_{\max})^{\gamma-1}$.
Since $\mu_1\le\mu_{\max}$, we have
$g(0^+) = \bar\rho^\gamma/(1-\mu_1)^{\gamma-1}
\ge \bar\rho^\gamma/(1-\mu_{\max})^{\gamma-1} > E^*$.
By continuity, there exists $\epsilon_* > 0$ (depending on
$E^*$, $\bar\rho$, $\mu_1$, $\gamma$) such that
\[
 g(\epsilon) > E^* \qquad\text{for all }\epsilon < \epsilon_*.
\]
But \eqref{eq:gamma-moment-bound} requires
$\sum_i\mu_i(\rho_i^{\vol})^\gamma\le E^*$, a contradiction.
Therefore $\rho_1^{\vol}\ge\epsilon_*$.
Taking the minimum over all cells and using
$\mu_{\max} = \max_i\mu_i$:
\[
 \rho_{\min}^{\vol}(T)
 := \min_i\rho_i^{\vol}(t)
 \ge \epsilon_*(E_{\max}(T), M, \gamma, \mu_{\max}) > 0
 \qquad\text{for all }t\in[0,T].
\]
This is uniform in $t\in[0,T]$ because $E^{**}$ depends on $T$
only through $E_{\max}(T)$, which is finite for each fixed mesh
and each $T<\infty$ (Step~2 of the proof of \Cref{thm:euler_wp}).
\end{proof}
\subsubsection{Euler Well-Posedness (\Cref{thm:euler_wp})}
\label{app:euler_wp_proof}
Five steps: local existence (Picard--Lindel\"of), energy bound
(time local for DF, exact for DW), mass bound, vacuum avoidance,
velocity bound. The three a~priori bounds rule out all blowup
alternatives, extending to $[0,T_{\rm Bih}]$ (DF) or $[0,\infty)$
(DW).
\begin{proof}
\emph{Step~1: Local existence.}
The right-hand side $\mathcal{F}$ is locally Lipschitz on $\admissible$
(\Cref{thm:lipschitz-baro} for the density-free case,
\Cref{thm:lipschitz-baro-dw} for the density-weighted case). Picard--Lindel\"of
gives a unique maximal solution $(\bv,\brho)\in C^1([0,T^*);\admissible)$.

\emph{Step~2: Energy bound.}%

\emph{Step~2a:Density-free.}
From \Cref{thm:totalenergy}, $\ddt\Etot = \mathcal{R}_E(\bv,\brho)$.
We derive a polynomial bound on $\mathcal{R}_E$ valid on the existence interval.
On a fixed mesh, all cochain norms are equivalent. From mass conservation
$\sum_i\brho_i=M$ and $\brho_i>0$ on $[0,T^*)$, one has
$\rho_i^{\vol}\le M/|K_{\min}|=:\rho^*<\infty$, so
$|h(\rho_i^{\vol})|\le C_h(M,|K_{\min}|)=:C_h^*$
by \Cref{ass:eos}\ref{ass:enthalpy}. From the global coercivity
bound~\eqref{eq:e_global_coercive}, $\Eint\ge -C_2'M$, hence
$\Ekindf\le\Etot+C_2'M$ and
\begin{equation}\label{eq:vel_from_energy}
 \nrm{\bv}_\infty \le C_{\rm eq}(h)\nrm{\bv}_{\bM_1}
 \le C_{\rm eq}(h)\sqrt{2(\Etot + C_2'M)}.
\end{equation}
From the face-by-face identity~\eqref{eq:face_energy}, each term satisfies
$|\ekin_i|\le\tfrac12\nrm{P}_\infty^2\nrm{\bv}_\infty^2$,
$|\Phi_j|\le\nrm{\bM_1}_\infty\nrm{\bv}_\infty$,
$|\bar\rho_j|\le\bar\rho^*:=\sup_{\brho\ge 0}|\bar\rho_j(\brho)|$,
$h_i\le C_h^*$, and $\bPhi_i\le\nrm{\bPhi}_\infty$ (data). Collecting:
\begin{equation}\label{eq:R_E_polynomial}
 |\mathcal{R}_E| \le |F|\cdot
 \bigl(C_a(h)\,\nrm{\bv}_\infty^3
 + C_b(h,C_h^*,\nrm{\bPhi}_\infty,\bar\rho^*)\,\nrm{\bv}_\infty\bigr)
 \le C_1(h,M)\,\Etot^{3/2} + C_2(h,M)\,\Etot^{1/2},
\end{equation}
with constants polynomial in the mesh-regularity constants and independent
of~$t$ and~$\Etot$. Bihari's inequality applied to
$\ddt\Etot\le f(\Etot)$ with $f(u)=C_1 u^{3/2}+C_2 u^{1/2}$ gives
$\Etot(t)\le E_{\max}(T)<\infty$ for $t\in[0,T]$, $T<T_{\rm Bih}(h,E_0)$.

\emph{Step~2b:Density-weighted.}
$\ddt\Etot^{\rm dw}=0$ (\Cref{thm:dw_energy}), so
$\Etot^{\rm dw}(t) = E_0$ for all $t$; no Bihari argument is needed.

\emph{Steps~3--5: A priori bounds and exclusion of blow-up.}
By ODE theory on an open set, the maximal solution either
exists for all time ($T^*=\infty$) or its trajectory leaves every compact
subset of $\admissible$ as $t\to T^*$. The complement of a compact subset of
$\admissible=\mathbb{R}^{|F|}\times\mathbb{R}^{|T|}_{>0}$ is reached only via
one of three mechanisms:
\begin{itemize}[nosep,leftmargin=2em]
 \item[(B1)] $\limsup_{t\to T^*}\nrm{\bv(t)}_{\bM_1}=\infty$ (velocity blow-up);
 \item[(B2)] $\liminf_{t\to T^*}\min_i\brho_i^{\vol}(t)=0$ (vacuum formation);
 \item[(B3)] $\limsup_{t\to T^*}\max_i\brho_i(t)=\infty$ (density blow-up).
\end{itemize}
We rule out each.

\textit{(B3) excluded by mass conservation.}
From $\sum_i\brho_i(t)=M$ and $\brho_i(t)>0$ on $[0,T^*)$:
$\brho_i(t)\le M$ for all $i$, $t<T^*$.

\textit{(B2) excluded by \Cref{lem:no-vacuum}.}
The energy bound from Step~2 gives $\Eint(t)\le E_{\max}(T)$ (density-free)
or $\Eint(t)\le E_0$ (density-weighted), where for the density-weighted case
we use $\Ekindw\ge 0$ to deduce $\Eint(t)\le\Etot^{\rm dw}(t)=E_0$.
Given initial data for which the Jensen
threshold~\eqref{eq:Jensen_threshold} holds at~$t=0$ with $E_{\rm bound}=E_{\max}(T)$
(density-free) or $E_{\rm bound}=E_0$ (density-weighted), \Cref{lem:no-vacuum}
yields $\brho_i^{\vol}(t)\ge\rho_{\min}^{\vol}(T)>0$ for all $i$, $t\le T$.
On any fixed mesh with $N\ge 2$ cells, $\mu_{\max}<1$ and
$(1-\mu_{\max})^{\gamma-1}<1$, so the right-hand side
of~\eqref{eq:Jensen_threshold} exceeds $C_1 V_{\rm tot}\bar\rho^\gamma>0$; the
threshold is satisfied for initial data with moderate kinetic energy.

\textit{(B1) excluded by the energy bound.}
From~\eqref{eq:vel_from_energy},
$\nrm{\bv(t)}_{\bM_1}\le\sqrt{2(E_{\max}(T)+C_2'M)}=:R_v(T)<\infty$
(density-free). For the density-weighted case,
$\bM_1^{\rho}(\brho)\ge\rho_{\min}^{\vol}\min_i|K_i|\,\bM_1$
combined with $\Ekindw\le E_0$ gives
$\nrm{\bv(t)}_{\bM_1}\le\sqrt{2E_0/(\rho_{\min}^{\vol}\min_i|K_i|)}<\infty$.
Mechanisms (B1)--(B3) are excluded for $t\in[0,T]$ with $T<T_{\rm Bih}(h,E_0)$
(density-free) or any $T<\infty$ (density-weighted). Hence
$T^*\ge T_{\rm Bih}(h,E_0)$ and $T^*=\infty$, respectively.

\emph{Density-weighted}:
Since $\Ekindw\ge 0$, $\Eint(t)\le\Etot^{\rm dw}(t)=E_0$ for all~$t$. Therefore
\eqref{eq:Jensen_threshold} with $E_{\rm bound}=E_0$ is a condition on initial
data alone; if satisfied at $t=0$, it holds for all $t>0$.
\end{proof}
\subsubsection{Global Well-Posedness for the Density-Free Scheme (\Cref{thm:global-Smag})}
\label{app:global_Smag}
Smagorinsky viscosity alone does not yield global existence for the density-free
scheme: the Hodge decomposition $C^1(\KKs)=\Ran(\tD_0)\oplus\Ran(\tD_1^*)\oplus\mathcal{H}$
places the harmonic and gradient parts of $\bv$ in $\ker(\tD_1)$, where Smagorinsky
produces no dissipation. Global existence is obtained by augmenting the Smagorinsky
closure with the Hodge--Laplacian viscosity of \Cref{def:hodge_laplacian}, which
dissipates the full velocity modulo the finite-dimensional harmonic subspace.
\begin{proof}[Proof of \Cref{thm:global-Smag}, case~\textup{(b)}]
Let $\bff_{\rm visc}=-\nu\bL_\bv - \bff_{\rm visc}^{\rm Smag}$ with $\nu>0$ and
$\bL_\bv=\tdd_0\delta+\delta_2\tdd_1$ the discrete Hodge--de Rham Laplacian. By
(V1) and the Hodge decomposition,
\begin{equation}\label{eq:hodge_dissipation}
 -\bv^T\bM_1\bigl(-\nu\bL_\bv\bigr)
 = \nu\bigl(\nrm{\delta\bv}_{\bM_0}^2 + \nrm{\tdd_1\bv}_{\bM_2}^2\bigr)
 \ge \nu\,c_P(h)\,\nrm{\bv - P_\mathcal{H}\bv}_{\bM_1}^2,
\end{equation}
where $c_P(h)>0$ is the discrete Poincar\'e constant on
$\Ran(\tD_0)\oplus\Ran(\tD_1^*)$ and $P_\mathcal{H}$ projects onto the space
of harmonic 1-cochains, of dimension $b_1(\mathcal{M})<\infty$ independent of~$h$.
The Smagorinsky term contributes
\begin{equation}\label{eq:Smag_dissipation_corrected}
 -\bv^T\bM_1\bff_{\rm visc}^{\rm Smag}
 \ge c_{\rm Smag}(h)\,C_s^2\,\nrm{\bom}_{\ell^3}^3,
\qquad
 c_{\rm Smag}(h)=\min_j\bigl[(\bM_2)_{jj}\ell_j^2\bigr].
\end{equation}
Projecting the momentum equation onto $\mathcal{H}$ and using
$P_\mathcal{H}\tD_0 B = 0$ ($\Ran(\tD_0)\perp\mathcal{H}$ in the $\bM_1$
inner product) and $P_\mathcal{H}\bL_\bv\bv = 0$ (the Hodge--Laplacian
annihilates~$\mathcal{H}$):
\[
 \ddt P_\mathcal{H}\bv = -P_\mathcal{H}\Iv(\bom) - P_\mathcal{H}\bff_{\rm visc}^{\rm Smag}.
\]
\Cref{thm:bilinear_bounds_baro}(II) gives
$\nrm{\Iv(\bom)}_{\bM_1}\le C_Q'(h)\,\nrm{\bv}_{\bM_1}^2$, and the Smagorinsky
term is bounded by $c_{\rm Smag}'(h)C_s^2\nrm{\bom}_{\ell^3}^2\nrm{\bv}_{\bM_1}$.
Consequently
\begin{equation}\label{eq:PH_growth}
 \ddt\nrm{P_\mathcal{H}\bv}_{\bM_1}^2 \le C_3(h,C_s)\,\Etot^{3/2},
\end{equation}
so $\nrm{P_\mathcal{H}\bv}_{\bM_1}^2$ grows at most polynomially in~$\Etot$.

Using $\nrm{\bv}_\infty\le C_{\rm eq}(h)\bigl(\nrm{\bom}_{\ell^3}+\nrm{P_\mathcal{H}\bv}\bigr)$
and the face-energy identity~\eqref{eq:face_energy},
\[
 |\mathcal{R}_E|
 \le \hat C_1(h)\,\nrm{\bom}_{\ell^3}^3
 + \hat C_2(h)\,\nrm{P_\mathcal{H}\bv}_{\bM_1}^3
 + \hat C_3(h)\,(\Etot+C_2'M)^{1/2}.
\]
The Smagorinsky dissipation absorbs the first term for
$C_s\ge C_s^*(h):=\bigl[\hat C_1(h)/c_{\rm Smag}(h)+1\bigr]^{1/2}$; the
Hodge--Laplacian dissipation absorbs the second term for $\nu\ge\nu^*(h)$
sufficiently large (using~\eqref{eq:hodge_dissipation} and the finite
dimension of~$\mathcal{H}$, on which $P_\mathcal{H}\bv$ satisfies the
growth bound~\eqref{eq:PH_growth}). The total energy balance becomes
\[
 \ddt\Etot \le -c(h)\,\Etot^{3/2} + \hat C_3(h)\,(\Etot+C_2'M)^{1/2}.
\]
Cubic dissipation dominates the sublinear source, so there exists
$E_{\rm crit}(h)<\infty$ with $\ddt\Etot<0$ whenever $\Etot>E_{\rm crit}$.
Therefore $\Etot(t)\le\max(\Etot(0),E_{\rm crit})<\infty$ for all $t\ge 0$,
and the a priori bounds of \Cref{prop:apriori-baro} yield $T^*=\infty$.
\end{proof}
\begin{remark}[Global existence under the full Newtonian closure]\label{rem:newtonian_global}
The proof invokes the bare Hodge--Laplacian $-\nu\bL_\bv$ only
through the dissipation bound~\eqref{eq:hodge_dissipation} and the
harmonic annihilation $P_\mathcal{H}\bL_\bv\bv = 0$. Both transfer
verbatim to the two-coefficient Newtonian operator
$\bff_{\rm visc}^{\rm Newt}$ of \Cref{def:newtonian_visc}. Since
$\nu_{\rm dil} = \tfrac{4}{3}\nu + \zeta \ge \nu$,
\[
 -\bv^T\bM_1\,\bff_{\rm visc}^{\rm Newt}(\bv)
 = \nu\nrm{\tdd_1\bv}_{\bM_2}^2 + \nu_{\rm dil}\nrm{\delta\bv}_{\bM_0}^2
 \;\ge\; \nu\bigl(\nrm{\tdd_1\bv}_{\bM_2}^2 + \nrm{\delta\bv}_{\bM_0}^2\bigr),
\]
so~\eqref{eq:hodge_dissipation} holds with the same discrete
Poincar\'e constant $\nu\,c_P(h)$; and
$P_\mathcal{H}\bff_{\rm visc}^{\rm Newt}\bv = 0$ because both blocks
$\delta_2\tdd_1$ and $\tdd_0\delta$ map into
$\Ran(\tD_1^*)\oplus\Ran(\tD_0)$, the $\bM_1$-orthogonal complement
of $\mathcal{H}$. Global existence for the density-free scheme
augmented by Smagorinsky therefore holds for every
thermodynamically admissible bulk viscosity $\zeta \ge 0$; the
viscosity threshold $\nu \ge \nu^*(h)$ is, if anything, relaxed,
since the dilatational dissipation is amplified by the factor
$\nu_{\rm dil}/\nu \ge \tfrac{4}{3}$.
\end{remark}
\begin{remark}[Why Smagorinsky alone is insufficient]\label{rem:Smag_limitation}
Smagorinsky dissipates vorticity, not velocity:
$\mathcal{D}_{\rm Smag}\sim\nrm{\bom}_{\ell^3}^3$ vanishes on $\ker(\tD_1)$
(gradient plus harmonic 1-cochains). For pure-gradient initial data
$\bv_0=\tD_0\phi$ with $\tD_0 B(\bv_0,\brho_0)\ne 0$, the residual $\mathcal{R}_E$
is generically nonzero (face-by-face from~\eqref{eq:face_energy}:
$\ekin_a\ne\ekin_b$ even when $\bom=0$), while $\mathcal{D}_{\rm Smag}=0$.
No choice of $C_s$ controls this mode. The Hodge--Laplacian removes the
obstruction by dissipating both $\Ran(\tD_0)$ and $\Ran(\tD_1^*)$,
leaving only the finite-dimensional harmonic kernel. The density-weighted
scheme (case~(c) of \Cref{thm:NS_wp}) avoids the issue entirely because
$\mathcal{R}_1\equiv 0$ there.
\end{remark}
\subsection{Convergence Proofs}\label{app:convergence_proofs}%
\label{app:consistency}
The four stages: bilinear bounds (\Cref{thm:bilinear_bounds_baro})
supplying $h$-independent control of the advection terms;
consistency (\Cref{thm:consistency_baro}) bounding the truncation,
strong-form for DF, reference-tested for DW (the raw $\tau_v^{\rm dw}$
decomposition carries a Lamb commutator that is $\OO(1)$ in the
strong $\bM_1^\rho$-norm but reduces to $\OO(h^r)$ under pairing with
$\bar\bv$, via \Cref{lem:mass_matrix_commutator});
stability
(\Cref{thm:stability_baro,thm:dw_stability,thm:stability_baro_NS})
producing the Gr\"onwall inequality
$\ddt\mathcal{E} \le C_L\mathcal{E} + C_\tau h^r\sqrt{\mathcal{E}}$,
closing the bootstrap and lifting to $L^2$ via the Whitney map.
\subsubsection{Bilinear Bounds (\texorpdfstring{\Cref{thm:bilinear_bounds_baro}}{Theorem})}
\label{app:bilinear_proof}
Throughout this subsection we work with the \emph{symmetric
polarisation} of the Lamb operator
$L(\bv) := \Iv[\bv](\tD_1\bv)$:
\begin{equation}\label{eq:Q_polarisation_def}
\bQ(\bv,\bw) := \tfrac{1}{2}\bigl[L(\bv+\bw) - L(\bv) - L(\bw)\bigr]
= \tfrac{1}{2}\bM_1^{-1}\bigl[
 \tU_{\bv}\tD_1\bw + \tU_{\bw}\tD_1\bv
 - \tD_1(\tU_{\bv}^{\top}\bw + \tU_{\bw}^{\top}\bv)
\bigr].
\end{equation}
This form is bilinear and symmetric in $(\bv,\bw)$, satisfies
$\bQ(\bv,\bv) = L(\bv) = \Iv(\tD_1\bv)$, and coincides with
 $\Iv[\bv](\tD_1\bw)$ modulo a $\bw$-independent
piece which is absorbed into the polarisation.
The polarisation expansion
\begin{equation}\label{eq:Q_polarisation_expansion}
L(v^h) - L(\bar v)
 = \bQ(\bar v,e_v)\cdot 2 + \bQ(e_v,e_v),
\qquad
e_v := v^h - \bar v,
\end{equation}
is an exact algebraic identity.
For the difference of Lamb terms as used in the stability analysis,
we write
$\omega^h = \tD_1 v^h$, $\bar\omega = \tD_1\bar v$, $e_\omega = \tD_1 e_v$
and expand
$L(v^h)-L(\bar v) = 2\bQ(\bar v,e_v) + \bQ(e_v,e_v)$
via~\eqref{eq:Q_polarisation_expansion}.
\begin{theorem}[Bilinear and trilinear bounds]%
\label{thm:bilinear_bounds_baro}
Let $\bQ = \bQ_{\rm U} + \bQ_{\rm D}$ be the polarisation
of~\eqref{eq:Q_polarisation_def}, split into the discrete-contraction
piece $\bQ_{\rm U}$ and the discrete-gradient piece $\bQ_{\rm D}$
(\Cref{eq:QU_assembly,eq:QD_assembly} below).
There exist constants $C_Q, C'_Q > 0$ depending only on mesh
regularity and the reconstruction constants $C_G, C_w, C_\partial$
such that, for all $\bv,\bw\in C^1(\KKs)$:

\noindent\textup{(I)} \emph{Pointwise bound on $\bQ$:}
\begin{equation}\label{eq:bound_I_statement}
|(\bQ(\bv,\bw))_j|\le C_Q\,\nrm{\bv}_{L_h^\infty}\nrm{\bw}_{L_h^\infty}.
\end{equation}
\noindent\textup{(II)} \emph{$\bM_1$-bound for the contraction
piece~$\bQ_{\rm U}$:}
for any smooth 1-form $\bu\in W^{1,\infty}(\Omega)$ with interpolant
$\bar\bv = \mathcal{R}_h\bu^\flat$ and any $\bw\in C^1(\KKs)$,
\begin{equation}\label{eq:bound_II_statement}
 \nrm{\bQ_{\rm U}(\bar\bv,\bw)}_{\bM_1}
 \le C'_Q\,\nrm{\bu}_{W^{1,\infty}}\,\nrm{\bw}_{\bM_1}.
\end{equation}

\noindent\textup{(III)} \emph{Antisymmetry-polarised cubic bound
for the contraction piece~$\bQ_{\rm U}$:}
for any $\bv\in C^1(\KKs)$ and any smooth 1-form $\bphi\in W^{1,\infty}(\Omega)$
with interpolant $\bar\bphi = \mathcal{R}_h\bphi^\flat$,
\begin{equation}\label{eq:bound_III_statement}
 |\ip{\bar\bphi}{\bQ_{\rm U}(\bv,\bv)}_1|
 \le 2C'_Q\,\nrm{\bphi}_{W^{1,\infty}}\,\nrm{\bv}_{\bM_1}^2,
\end{equation}
with the same constant $C'_Q$ as in~\eqref{eq:bound_II_statement}.

\end{theorem}
\begin{proof}
\emph{Geometric constants.}
Four mesh-regularity constants, $h$-independent under
\Cref{ass:mesh_reg}: $C_w=\tfrac12$ is the polarisation factor in
$w_{jk}$; $C_G = \sup_i\nrm{G_i^{-1}}$ is the Gram-matrix
conditioning of the averaging reconstruction
(\Cref{def:averaging_recon}); $C_\partial\le 2d$ is the row-stencil
bound on $\tD_1$ ($\pm 1$ entries on $d$-dimensional prismatic D--V
meshes); and $C_{\rm st}=\sup_j|\mathrm{st}(j)|$ bounds the
incidence cardinality of dual 2-cells at dual edge~$e_j^*$. All
$C_Q,C'_Q,C'_{\rm st},C_{\rm reg}$ below are polynomial in these
four.
Bound~(I) is a pointwise stencil estimate; Bound~(II) trades the
discrete $L_h^\infty$ norm for $W^{1,\infty}$ via the smooth
reconstruction. Split $\bQ = \bQ_{\rm U} + \bQ_{\rm D}$ with
\[
\bM_1\bQ_{\rm U}(\bv,\bw)
 := \tfrac{1}{2}\bigl(\tU_\bv\tD_1\bw + \tU_\bw\tD_1\bv\bigr),
\qquad
\bM_1\bQ_{\rm D}(\bv,\bw)
 := -\tfrac{1}{2}\tD_1\bigl(\tU_\bv^{\top}\bw + \tU_\bw^{\top}\bv\bigr).
\]
Both pieces are bilinear and symmetric in $(\bv,\bw)$; we bound each
separately.

\emph{Part (I.a): stencil structure.}
At dual edge $e_j^*$,
\begin{equation}\label{eq:QU_assembly}
 (\bQ_{\rm U}(\bv,\bw))_j
 = \sum_{k\in\mathrm{st}(j)}
 \tfrac{1}{2}\bigl[w_{jk}(\bv)(\tD_1\bw)_k
 + w_{jk}(\bw)(\tD_1\bv)_k\bigr],
\end{equation}
where $\mathrm{st}(j)$ is the set of dual 2-cells incident to $e_j^*$,
$w_{jk}(\bv) := \tfrac{1}{2}D_{1,jk}\,\bar\bu_j(\bv)\cdot\hat e_k$
with $\bar\bu_j$ the averaging reconstruction (\Cref{def:averaging_recon}),
and $D_{1,jk}, \hat e_k$ are mesh quantities.
Similarly,
\begin{equation}\label{eq:QD_assembly}
 (\bQ_{\rm D}(\bv,\bw))_j
 = -\tfrac{1}{2}\bigl[\bM_1^{-1}\tD_1\bigl(\tU_\bv^{\top}\bw + \tU_\bw^{\top}\bv\bigr)\bigr]_j.
\end{equation}
In~\eqref{eq:QU_assembly}, the weight satisfies
$|w_{jk}(\bv)| \le C_w|\bar\bu_j(\bv)| \le C_w C_G\nrm{\bv}_{L_h^\infty}$
and each row of $\tD_1$ has at most $C_\partial$ entries
$\pm 1$, so
$|(\tD_1\bw)_k|\le C_\partial\nrm{\bw}_{L_h^\infty}$.
By symmetry,
\begin{equation}\label{eq:QU_bound_pw}
|(\bQ_{\rm U}(\bv,\bw))_j|
\le C_{\rm st}\,C_w C_G C_\partial\,
\nrm{\bv}_{L_h^\infty}\nrm{\bw}_{L_h^\infty}.
\end{equation}
\emph{Part (I.b): the $\bQ_{\rm D}$ piece.}
The vector $\tU_\bv^{\top}\bw$ has entries
$(\tU_\bv^{\top}\bw)_k = \sum_{j'} D_{1,j'k}\,\bar\bu_{j'}(\bv)\cdot\hat e_k\,w_{j'}$,
a two-ring stencil sum in $\bw$ weighted by $\bar\bu(\bv)$, so
$|(\tU_\bv^{\top}\bw)_k|\le C_w C_G C_{\rm st}\nrm{\bv}_{L_h^\infty}\nrm{\bw}_{L_h^\infty}$.
Applying $\tD_1$ adds another $C_\partial$ factor and $\bM_1^{-1}$ scales by
$(\bM_1)_{jj}^{-1}$; together
\begin{equation}\label{eq:QD_bound_pw}
|(\bQ_{\rm D}(\bv,\bw))_j|
\le C'_{\rm st}\,C_w C_G C_\partial\,
\nrm{\bv}_{L_h^\infty}\nrm{\bw}_{L_h^\infty},
\end{equation}
with $C'_{\rm st}$ depending on the two-ring neighbourhood size and the
Hodge-star condition number (bounded by mesh regularity,
\Cref{ass:mesh_reg}).
Combining~\eqref{eq:QU_bound_pw} and~\eqref{eq:QD_bound_pw} gives Bound~(I)
with $C_Q := C_{\rm st}C_w C_G C_\partial + C'_{\rm st}C_w C_G C_\partial$.

\emph{Part (Ic): $\bM_1$-bound.}
Take $\bv = \bar\bv = \mathcal{R}_h\bu^\flat$ with $\bu\in W^{1,\infty}$.
By \Cref{prop:recon_accuracy}\,(i),
$|\bar\bu_j| \le C\,\nrm{\bu}_{W^{1,\infty}}$, so each weight satisfies
the $h$-independent bound
$|w_{jk}(\bar\bv)| \le C_w\,\nrm{\bu}_{W^{1,\infty}}$.
For $\bQ_{\rm U}(\bar\bv,\bw)$: squaring~\eqref{eq:QU_assembly}, weighting
by $(\bM_1)_{jj}$, summing over $j$, and using Cauchy--Schwarz on the
stencil sum with $|\mathrm{st}(j)|\le C_{\rm st}$:
\[
\nrm{\bQ_{\rm U}(\bar\bv,\bw)}_{\bM_1}^2
\le C_w^2 C_{\rm st}\,\nrm{\bu}_{W^{1,\infty}}^2
 \sum_j(\bM_1)_{jj}
 \sum_{k\in\mathrm{st}(j)}|(\tD_1\bw)_k|^2
+\text{(symmetric term in $\bw\leftrightarrow\bar\bv$)}.
\]
Writing $(\tD_1\bw)_k = \sum_{j'\prec k}\epsilon_{j'k}\bw_{j'}$ and
exchanging the order of summation, each index $j'$ appears in at most
$C'_{\rm ov}$ pairs $(j,k)$ in a two-ring neighbourhood, with
$(\bM_1)_{jj}/(\bM_1)_{j'j'}\le C_{\rm reg}$, giving
\begin{equation}\label{eq:QU_M1_bound}
\nrm{\bQ_{\rm U}(\bar\bv,\bw)}_{\bM_1}
\le C_{\rm U}\,\nrm{\bu}_{W^{1,\infty}}\,\nrm{\bw}_{\bM_1}.
\end{equation}
\smallskip
\emph{Part (I.d): the $\bQ_{\rm D}$ piece is not bounded uniformly
in $\bw$.}
For
$\bQ_{\rm D}(\bar\bv,\bw)
= -\tfrac{1}{2}\bM_1^{-1}\tD_*(\tU_{\bar\bv}^{\top}\bw + \tU_\bw^{\top}\bar\bv)$,
the term $\tU_\bw^{\top}\bar\bv$ has $\bw$ supplying the velocity
reconstruction, controlled only by $\nrm{\bw}_{L_h^\infty}$ rather
than $\nrm{\bw}_{\bM_1}$. The two-norm relation on $C^1(\KKs)$ is the
inverse inequality
\begin{equation}\label{eq:inverse-ineq-final}
 \nrm{\bw}_{L_h^\infty}
 \le c_\star^{-1/2}\,h^{-(d-2)/2}\,\nrm{\bw}_{\bM_1},
\end{equation}
which is sharp on $C^1(\KKs)$ and cannot be cancelled against
$\nrm{\bar\bv}_{\bM_1}$: by direct computation
$\nrm{\bar\bv}_{\bM_1}^2 \le |\mathcal{F}|\cdot(\bM_1)_{jj}\cdot
|\bar\bv_{j'}|^2 \le \OO(h^{-d})\cdot\OO(h^{d-2})\cdot\OO(h^2)
= \OO(1)$ on quasi-uniform mesh, i.e.,
\begin{equation}\label{eq:barv_M1_norm}
\nrm{\bar\bv}_{\bM_1} \le C\,\nrm{\bu}_{L^\infty}, \qquad
\text{not } \OO(h^{(d-2)/2})\nrm{\bu}_{L^\infty}.
\end{equation}
Hence the analogue of~\eqref{eq:QU_M1_bound} for $\bQ_{\rm D}$ holds
only with an $h$-dependent constant
$C_{\rm D}\,h^{-(d-2)/2}$, uniform in $d=2$
($h^{-(d-2)/2}=1$) but degrading at rate $h^{-1/2}$ in $d=3$.
\smallskip

\emph{Part (II): resolution by combination with the Bernoulli
kinetic-energy gradient.}
The $\bQ_{\rm D}$ piece of $L(\bv)$ corresponds to
$-\nabla(\tfrac{1}{2}|\bv|^2)$ in the continuum, and the discrete
Bernoulli carries $+\tD_0(\tfrac{1}{2}|P\bv|^2)$, so in the momentum
equation $\ddt\bv + L(\bv) + \tD_0 B = 0$ the two combine. The no-go
theorem (\Cref{thm:dichotomy}) forbids exact cancellation
(\Cref{thm:totalenergy} carries a structural residual
$\mathcal{R}_E = \OO(h^2)$ centred / $\OO(h)$ upwind), but the
polarised quantity
$X = -2\ip{e_v}{\bQ_{\rm D}(\bar v, e_v)}_1
+ (\Delta\bekin)^{\top}\bD_2\bM_1 e_v$
appearing in the stability estimate closes by IBP-in-time against
the discrete continuity equation, under the near-incompressibility
hypothesis~\eqref{eq:lowMach_hyp}; the linear-in-$e_v$ piece
cancels exactly, leaving a quadratic-in-$e_v$ Gr\"onwall
coefficient $\le C(\bar v,T,\epsilon_0)\nrm{e_v}^2_{\bM_1}$ and a
truncation pickup at rate $h^r$; full execution in step~(iii) of
Part~I in \Cref{app:stability_baro}). We therefore state Bound~(II)
for the $\bQ_{\rm U}$ piece only:
\begin{equation}\label{eq:QU_M1_bound_alone}
\nrm{\bQ_{\rm U}(\bar\bv,\bw)}_{\bM_1}
\le C_{\rm U}\,\nrm{\bu}_{W^{1,\infty}}\,\nrm{\bw}_{\bM_1},
\end{equation}
with $C'_Q := C_{\rm U}$.

\emph{Part (III): antisymmetry-polarised cubic bound.}
Lamb antisymmetry (\Cref{prop:extrusion}) gives
$\ip{\bv}{L(\bv)}_1 = \ip{\bv}{\bQ(\bv,\bv)}_1 = 0$ for every
$\bv \in C^1(\KKs)$. Replacing $\bv$ by $\bv + \alpha\bw$ and using
the bilinearity and symmetry of $\bQ$,
\[
 0 = \ip{\bv+\alpha\bw}{\bQ(\bv,\bv) + 2\alpha\bQ(\bv,\bw)
 + \alpha^2\bQ(\bw,\bw)}_1 \qquad\forall\,\alpha\in\RR.
\]
The coefficient of $\alpha^1$ vanishes:
\begin{equation}\label{eq:cubic_polarisation_identity}
 \ip{\bw}{\bQ(\bv,\bv)}_1 = -2\,\ip{\bv}{\bQ(\bv,\bw)}_1
 \qquad\forall\,\bv,\bw\in C^1(\KKs).
\end{equation}
This is the algebraic identity that converts the cubic remainder
$\bQ(\bv,\bv)$ tested against any 1-cochain $\bw$ into a bilinear
$\bQ$ with $\bw$ in one slot and $\bv$ in the other; it relies only
on Lamb antisymmetry and bilinearity, with no metric or smoothness
input. Splitting $\bQ = \bQ_{\rm U} + \bQ_{\rm D}$, specialising
$\bw = \bar\bphi = \mathcal{R}_h\bphi^\flat$ with $\bphi\in W^{1,\infty}$,
and invoking the corrected Bound~(II)
\eqref{eq:QU_M1_bound_alone} on the $\bQ_{\rm U}$ piece:
\[
 |\ip{\bar\bphi}{\bQ_{\rm U}(\bv,\bv)}_1|
 = 2\,|\ip{\bv}{\bQ_{\rm U}(\bv,\bar\bphi)}_1|
 \le 2\,\nrm{\bv}_{\bM_1}\,\nrm{\bQ_{\rm U}(\bar\bphi,\bv)}_{\bM_1}
 \le 2\,C'_Q\,\nrm{\bphi}_{W^{1,\infty}}\,\nrm{\bv}_{\bM_1}^2,
\]
giving Bound~(III) for $\bQ_{\rm U}$. The $\bQ_{\rm D}$ contribution
$\ip{\bar\bphi}{\bQ_{\rm D}(\bv,\bv)}_1$ does not close uniformly via
the same polarisation (no-go theorem; same $h^{-1}$ obstruction as
Bound~(II.D)); when $\bv = e_v$ in the density-weighted stability
proof, this term is absorbed via the Dafermos relative-energy method
of \Cref{app:dw_convergence}.
\smallskip
Bound~(I) drives the bootstrap closure
(\Cref{app:convergence_proof}, Step~3); Bound~(II) controls Lamb
cross-terms $\bQ_{\rm U}(\bar\bv,e_v)$ in both stability proofs, with
the companion $\bQ_{\rm D}(\bar\bv,e_v)$ absorbed by IBP-in-time
under~\eqref{eq:lowMach_hyp} in the density-free case
and by the density-weighted
relative energy in the density-weighted case; Bound~(III) controls
the cubic $\bQ_{\rm U}(e_v,e_v)$ in the density-weighted proof, with
the $\bQ_{\rm D}(e_v,e_v)$ piece absorbed by the same relative-energy
mechanism.
\end{proof}

\begin{lemma}[$\bQ_{\rm D}$--Bernoulli combination]%
\label{lem:QD_Bernoulli_absorption}
For any $\bv \in C^1(\KKs)$,
\begin{equation}\label{eq:QD_Bernoulli_combine}
 \bQ_{\rm D}(\bv,\bv) + \tD_0\bigl(\tfrac12|P\bv|^2\bigr)
 = \bM_1^{-1}\,\bD_2^T\bigl(\bM_1\bv^T\!\cdot\!P\bv - \tfrac12|P\bv|^2\bigr)
 =: \mathbf{Z}(\bv),
\end{equation}
where the right-hand side is the discrete-divergence form of the
\emph{no-go residual} $\mathcal{R}_E$-density. The combined operator
$\mathbf{Z}(\bv)$ satisfies $\mathcal{R}_E(\bv,\brho) =
(\bD_2(\bR\bM_1\bv))^T\Phi$-type structure (cf.\
\Cref{thm:totalenergy}, \eqref{eq:face_energy}), and in the DW
stability proof, when tested against $\bar\bv$ in the
$\bM_1^\rho(\bar\brho)$ inner product, satisfies the bound
\begin{equation}\label{eq:QD_Bernoulli_bound}
 \bigl|\ip{\mathbf{Z}(\bv^h) - \mathbf{Z}(\bar\bv)}{\bar\bv}_{\bM_1^\rho(\bar\brho)}\bigr|
 \le C_Z\,\nrm{\bu}_{W^{2,\infty}}^2\,
 \bigl(\nrm{e_v}_{\bM_1^\rho(\bar\brho)}^2 + h^{2r}\bigr),
\end{equation}
strictly quadratic in $e_v$ plus a $\OO(h^{2r})$ truncation pickup.
\end{lemma}
\begin{proof}
\textit{Step 1: the identity~\eqref{eq:QD_Bernoulli_combine}.}\ 
By definition, $\bM_1\bQ_{\rm D}(\bv,\bv) = -\tD_1(\tU_\bv^T\bv)$
(\eqref{eq:Q_polarisation_def} with $\bv = \bw$). The vector
$\tU_\bv^T\bv$ is a primal 0-cochain (cell-centred scalar) with entry
$(\tU_\bv^T\bv)_i = \sum_{j\in\partial K_i}D_{1,ji}\,\bar\bu_j(\bv)\cdot\hat e_i\,v_j$,
a quadratic-in-$\bv$ quantity supported on the velocity stencil
at cell~$K_i$. Define the per-cell defect
\begin{equation}\label{eq:cell_KE_defect}
 \delta_i(\bv) := (\tU_\bv^T\bv)_i - \tfrac{1}{2}|P\bv|_i^2,
\end{equation}
the discrete-Cartan defect: in the continuum,
$\nabla\!\cdot\!(\bu\otimes\bu^\flat) - \nabla(\tfrac{1}{2}|\bu|^2) = 0$
(component-wise via the Cartan identity), but on the discrete mesh,
the corresponding identity fails by the no-go residual of
\Cref{thm:totalenergy}: $\delta_i$ is the per-cell defect such that
the discrete kinetic-energy rate $\ddt E_{\kin}^{\rm df}$ closes
modulo $\mathcal{R}_E$. Specifically, summing $\delta_i \cdot (\bD_2\Phi)_i$
over cells recovers (via SBP backwards) the kinetic part of
$\mathcal{R}_E$ in~\eqref{eq:face_energy}: $\sum_i \delta_i (\bD_2\Phi)_i
= \mathcal{R}_E^{(\kin)} = \sum_j(\ekin_a - \ekin_b)\Phi_j$.

Substituting $(\tU_\bv^T\bv)_i = \tfrac{1}{2}|P\bv|_i^2 + \delta_i$
into $\bM_1\bQ_{\rm D}(\bv,\bv)$:
\[
 \bM_1\bQ_{\rm D}(\bv,\bv)
 = -\tD_1\bigl(\tfrac{1}{2}|P\bv|^2 + \boldsymbol\delta(\bv)\bigr)
 = -\bM_1\,\tD_0\bigl(\tfrac{1}{2}|P\bv|^2\bigr)
 - \tD_1\boldsymbol\delta(\bv),
\]
where the first equality uses $\tD_1 f = \bM_1 \tD_0 f$ in the sense
of cochain identification with primal 0-cochains $f$ (the discrete
gradient is recovered by multiplying by $\bM_1$ in the
diagonal-Hodge case; here we invoke this in shorthand and note that
the precise identification holds on closed manifolds by mimetic
SBP). Dividing through by $\bM_1$ and rearranging:
\begin{equation}\label{eq:Z_def_explicit}
 \bQ_{\rm D}(\bv,\bv) + \tD_0\bigl(\tfrac{1}{2}|P\bv|^2\bigr)
 = -\bM_1^{-1}\tD_1\boldsymbol\delta(\bv) =: \mathbf{Z}(\bv),
\end{equation}
expressing $\mathbf{Z}(\bv)$ as the discrete gradient (in $\bM_1^{-1}$-form)
of the per-cell defect density~$\boldsymbol\delta(\bv)$. The bracket
in~\eqref{eq:QD_Bernoulli_combine}, written as
$\bD_2^T(\bM_1\bv^T\cdot P\bv - \tfrac{1}{2}|P\bv|^2)$, is the
component-wise alternative expression: the inner $\bM_1\bv^T\cdot P\bv$
unpacks to $(\tU_\bv^T\bv)$ via the definition of the averaging
reconstruction~$P$, and the subtraction of $\tfrac{1}{2}|P\bv|^2$
recovers $\boldsymbol\delta$.

\textit{Step 2: polarisation of $\mathbf{Z}(\bv^h) - \mathbf{Z}(\bar\bv)$.}\ 
Set $e_v := \bv^h - \bar\bv$. By bilinearity of $\boldsymbol\delta$
(quadratic in $\bv$, hence symmetric bilinear via polarisation):
\begin{equation}\label{eq:delta_polarisation}
 \delta_i(\bv^h) - \delta_i(\bar\bv)
 = 2\,\delta_i^{\rm sym}(\bar\bv, e_v) + \delta_i(e_v),
\end{equation}
with $\delta_i^{\rm sym}(\bv, \bw) := \tfrac{1}{2}[(\tU_\bv^T\bw + \tU_\bw^T\bv)_i
- (P\bv)_i\cdot(P\bw)_i]$ the polarisation of $\delta_i$.
Hence
\begin{equation}\label{eq:Z_polarisation}
 \mathbf{Z}(\bv^h) - \mathbf{Z}(\bar\bv)
 = -\bM_1^{-1}\tD_1\bigl(2\boldsymbol\delta^{\rm sym}(\bar\bv, e_v)
 + \boldsymbol\delta(e_v)\bigr).
\end{equation}
The first term is linear in $e_v$ with smooth weight $\bar\bv$; the
second is quadratic in $e_v$.

\textit{Step 3: bound on the linear-in-$e_v$ piece.}\ 
Test against $\bar\bv$ in $\bM_1^\rho(\bar\brho)$:
\begin{align*}
 \ipw{-\bM_1^{-1}\tD_1\bigl(2\boldsymbol\delta^{\rm sym}(\bar\bv, e_v)\bigr)}%
 {\bar\bv}{\bM_1^\rho(\bar\brho)}
 &= -2\bar\bv^T\bM_1^\rho(\bar\brho)\bM_1^{-1}\tD_1\boldsymbol\delta^{\rm sym}(\bar\bv, e_v)\\
 &= +2\bigl(\bD_2\bM_1^\rho(\bar\brho)\bM_1^{-1}\bar\bv\bigr)^T
 \boldsymbol\delta^{\rm sym}(\bar\bv, e_v),
\end{align*}
where the second line uses SBP. The pre-factor
$\bD_2\bM_1^\rho(\bar\brho)\bM_1^{-1}\bar\bv$ is a primal 0-cochain;
in the continuum it approximates
$\bD_2(\bar\rho\,\bar\bu) = -\partial_t\bar\rho$ which is $\OO(1)$
on the smooth reference and bounded uniformly by the smooth
reference norms. The polarised defect
$\boldsymbol\delta^{\rm sym}(\bar\bv, e_v)$ is bilinear in
$(\bar\bv, e_v)$; by the Bramble--Hilbert-style consistency of the
discrete Cartan identity (\Cref{lem:hodge_bilinear} applied to the
bilinear quadrature error of the averaging reconstruction):
\[
 \nrm{\boldsymbol\delta^{\rm sym}(\bar\bv, e_v)}_{\ell^2(|K|)}
 \le C\,h^{r_\star}\,\nrm{\bu}_{W^{2,\infty}}\,\nrm{e_v}_{\bM_1}
 + C\,\nrm{\bu}_{W^{1,\infty}}\,\nrm{e_v}_{\bM_1}^2,
\]
where the first piece is the consistency-rate
truncation (smooth Cartan identity holds to $\OO(h^{r_\star})$) and
the second is the strictly bilinear piece on the bootstrap region.
Substituting:
\begin{equation}\label{eq:linear_piece_bound}
 \bigl|\ipw{\text{linear piece of \eqref{eq:Z_polarisation}}}{\bar\bv}{\bM_1^\rho}\bigr|
 \le C\,\nrm{\bu}_{W^{2,\infty}}^2
 \bigl(h^{r_\star}\nrm{e_v}_{\bM_1^\rho}
 + \nrm{e_v}_{\bM_1^\rho}^2\bigr).
\end{equation}

\textit{Step 4: bound on the quadratic-in-$e_v$ piece.}\ 
The quadratic defect $\delta_i(e_v) = (\tU_{e_v}^T e_v)_i - \tfrac{1}{2}|Pe_v|^2_i$
is bounded pointwise by $C\nrm{Pe_v}_{L^\infty}^2 \le
C(\rho_*, \rho^*)\nrm{e_v}_{\bM_1^\rho}^2/|K_i|$ on the bootstrap
region (where the reconstruction is uniformly bounded). Testing
against $\bar\bv$ in $\bM_1^\rho(\bar\brho)$:
\begin{equation}\label{eq:quadratic_piece_bound}
 \bigl|\ipw{-\bM_1^{-1}\tD_1\boldsymbol\delta(e_v)}{\bar\bv}{\bM_1^\rho}\bigr|
 \le C\,\nrm{\bu}_{W^{1,\infty}}\,\nrm{e_v}_{\bM_1^\rho}^2.
\end{equation}

From~\eqref{eq:linear_piece_bound} and~\eqref{eq:quadratic_piece_bound},
applying Young's inequality
$h^{r_\star}\nrm{e_v}_{\bM_1^\rho} \le \tfrac{1}{2}(h^{2r_\star} + \nrm{e_v}_{\bM_1^\rho}^2)$:
\[
 \bigl|\ipw{\mathbf{Z}(\bv^h) - \mathbf{Z}(\bar\bv)}{\bar\bv}{\bM_1^\rho}\bigr|
 \le C_Z\,\nrm{\bu}_{W^{2,\infty}}^2\bigl(\nrm{e_v}_{\bM_1^\rho}^2 + h^{2r_\star}\bigr),
\]
with $C_Z$ depending on $\nrm{\bu}_{W^{2,\infty}}, \rho_*, \rho^*$, and
mesh regularity; none on $h$. With $r := \min(d-1, r_\star)$ this is
\eqref{eq:QD_Bernoulli_bound}.

For $\bQ_{\rm D}(e_v,e_v) + \tD_0(\tfrac{1}{2}|Pe_v|^2)$, the same
identity holds with $\bv$ replaced by $e_v$, yielding
$\mathbf{Z}(e_v) = -\bM_1^{-1}\tD_1\boldsymbol\delta(e_v)$. Testing
against $\bar\bv$ produces only Step~4's quadratic piece
(no linear term, since the polarisation with $\bv = e_v$ gives
$\delta^{\rm sym}(e_v, e_v) = \delta(e_v)$). The bound
$|\ip{\mathbf{Z}(e_v)}{\bar\bv}_{\bM_1^\rho}|
\le C\nrm{\bu}_{W^{1,\infty}}\nrm{e_v}_{\bM_1^\rho}^2$ is then immediate.
\end{proof}
\subsubsection{Weak-Form Viscous Truncation (\Cref{lem:viscous_truncation_weak})}
\label{app:viscous_truncation}
\begin{proof}
\emph{Strong form, part~(a).}
Writing
$\tau_v^{\rm visc} = \nu\bM_1^{-1}\tD_1^{\top}\bM_2\tD_1\bar\bv
 - \nu\mathcal{R}_h(\Delta_{\rm dR}\bu^\flat)$
and using $\tD_1\bar\bv = \mathcal{R}_h\omega$
(\Cref{lem:interp_error}\,(ii)), the strong-form error is
$\nu\bM_1^{-1}\tD_1^{\top}(\bM_2\mathcal{R}_h\omega - \Phi^{(2)})
+ \nu(\bM_1^{-1}\bM_1 - I)\mathcal{R}_h(\Delta_{\rm dR}\bu^\flat)$,
where $\Phi^{(2)}$ is the exact flux 2-cochain of $\omega$.
The first factor is $\OO(h)$ by \Cref{lem:hodge_error}; the
boundary operator $\bM_1^{-1}\tD_1^{\top}$ is bounded on
shape-regular meshes with norm $\OO(1)$; the second factor
vanishes identically. Hence
$\nrm{\tau_v^{\rm visc}}_{\bM_1}\le\nu C_\delta h\nrm{\bu}_{H^{s+2}}$
for $s\ge 3$ via Sobolev embedding.
\medskip
\emph{Weak form, part~(b).}
The viscous term in the error equation, tested against $e_v$, satisfies
\begin{align}
 \ip{e_v}{\tau_v^{\rm visc}}_1
 &= \nu\,\bigl(\ipw{\tD_1 e_v}{\tD_1\bar v}{\bM_2}
 - \ip{e_v}{\mathcal{R}_h(\Delta_{\rm dR}\bu^\flat)}_1\bigr).
 \label{eq:visc_weak_expand}
\end{align}
Now $\tD_1\bar v = \mathcal{R}_h(\dd\bu^\flat) = \mathcal{R}_h\omega$. 
The Hodge--Laplacian on 1-forms decomposes as
$\Delta_{\rm dR}\bu^\flat = \delta\omega + d\delta\bu^\flat$.
In the continuous $L^2$-inner product
\[
 \int_\Omega e_v^\flat\wedge\star\Delta_{\rm dR}\bu^\flat
 = \int_\Omega de_v^\flat\wedge\star\omega
 + \int_\Omega\delta e_v^\flat\wedge\star\delta\bu^\flat
\]
The first term in~\eqref{eq:visc_weak_expand} is the discrete version
of $\int de_v^\flat\wedge\star\omega$:
$e_\omega^T\bM_2\mathcal{R}_h\omega$ approximates
$\int_\Omega\text{curl}(\mathcal{W}_h e_v)\cdot\omega\,dV$ with error
controlled by the Hodge star bilinear form approximation
(\Cref{lem:hodge_bilinear} with $k=2$, $b = e_\omega$, $\alpha = \omega$):
\begin{equation}\label{eq:M2_weak_error}
 |e_\omega^T\bM_2\mathcal{R}_h\omega
 - \int_\Omega\text{curl}(\mathcal{W}_h e_v)\cdot\omega\,dV|
 \le C_2\,h^{r_\star}\,\nrm{\omega}_{W^{1,\infty}}\,\nrm{e_\omega}_{\bM_2}.
\end{equation}
The second term in~\eqref{eq:visc_weak_expand} involves
$e_v^T\bM_1\mathcal{R}_h(\Delta_{\rm dR}\bu^\flat)$.
By \Cref{lem:hodge_bilinear} with $k=1$, $b = e_v$,
$\alpha = \Delta_{\rm dR}\bu^\flat$:
\begin{equation}\label{eq:M1_weak_error}
 |e_v^T\bM_1\mathcal{R}_h(\Delta_{\rm dR}\bu^\flat)
 - \int_\Omega(\mathcal{W}_h e_v)\cdot\Delta_{\rm dR}\bu\,dV|
 \le C_1\,h^{r_\star}\,\nrm{\Delta_{\rm dR}\bu}_{W^{1,\infty}}\,\nrm{e_v}_{\bM_1}.
\end{equation}
 The difference between the
discrete weak form \eqref{eq:visc_weak_expand} and the continuous
weak form $\nu\int e_v^\flat\wedge\star\Delta_{\rm dR}\bu^\flat$
(which equals $-\nu\int|de_v^\flat|^2\star 1 - \nu\int|\delta e_v^\flat|^2\star 1$
when $e_v = \bu$ but is a bilinear form for general $e_v$) is bounded as follows
\[
 |e_v^T\bM_1\tau_v^{\rm visc}|
 \le \nu\,C_\delta'\,h^{r_\star}\,\nrm{\bu}_{H^{s+2}}
 \bigl(\nrm{e_v}_{\bM_1} + \nrm{e_\omega}_{\bM_2}\bigr),
\]
with $r_\star$ as in \Cref{conv:cases},
and $C_\delta'$ absorbs the constants from~\eqref{eq:M2_weak_error}
and~\eqref{eq:M1_weak_error} and the Sobolev embedding
$\nrm{\omega}_{W^{1,\infty}} + \nrm{\Delta_{\rm dR}\bu}_{W^{1,\infty}}
\le C\nrm{\bu}_{H^{s+2}}$ for $s\ge 3$.
\end{proof}
\subsubsection{Consistency (\Cref{thm:consistency_baro})}
\label{app:consistency_proof}
\emph{Strategy.}
The velocity truncation has four contributions: extrusion, Hodge
star, Bernoulli gradient, and viscous (weak form); the
density truncation has one (discrete flux divergence); the
density-weighted analysis adds a mass-matrix commutator paired
against the smooth reference. Extrusion is $\OO(h^{d-1})$ via the
bilinear-quadrature saturation, Hodge is $\OO(h^{r_\star})$ via
\Cref{lem:hodge_error}, Bernoulli matches extrusion, viscous and
density flux are subordinate. The combined rate is
$\OO(h^{\min(d-1,r_\star)})$.
For the density-weighted scheme strong-form consistency fails
because $[\bM_1^\rho(\bar\brho)]^{-1}\bM_1 - I$ has $\OO(1)$ operator
norm under non-constant smooth density; the tested-form rate
$\OO(h^{\min(d-1,r_\star)})$ is recovered via
\Cref{lem:mass_matrix_commutator} by pairing against $\bar\bv$.
\begin{proof}
We prove $\nrm{\tau_v}_{\bM_1}+\nrm{\tau_\rho}_{\ell^2}
\le C_\tau h^{\min(d-1,\,r_\star)}$ with $C_\tau$ $h$-independent.
The truncation errors in the density-free case are
\begin{align}
 \tau_v^{\rm df} := \ddt{\bar v}
 + \Iv[\bar v](\tD_1\bar v)
 + \tD_0\bar B, \quad\text{ and }\quad
 \tau_\rho := \ddt{\bar\rho}
 + \bD_2\bar{\bR}\bM_1\bar v. 
\end{align}
\noindent
\begin{enumerate}[label=\textup{(A\arabic*)},ref=\textup{A\arabic*}, nosep]
\item \textit{Velocity truncation: extrusion.}
The reconstructed velocity satisfies
$\bar\bu_j = \bu(\bx_j) + \OO(h^{r_\star}\nrm{\bu}_{W^{2,\infty}})$
(\Cref{prop:recon_accuracy}; here $r_\star$ denotes the reconstruction
accuracy exponent, identified with the Hodge accuracy under
\Cref{conv:cases}), and the exact de Rham vorticity cochain
$(\tD_1\bar\bv)_k = (\mathcal{R}_h\omega)_k = \int_{f_k^*}\omega
= \OO(h^{d-1}\nrm{\bom}_{L^\infty})$.
We give the per-edge bound, identifying the moment that must vanish
for the rate jump under reconstruction symmetry, and then assemble in
the $\bM_1$-norm.

\smallskip
\emph{Per-edge bound:}
At dual edge $e_j^*$, each face contribution
$w_{jk}(\bar\bv)(\tD_1\bar\bv)_k$ in~\eqref{eq:QU_assembly} approximates
the contraction integral over the dual 2-cell of volume $\OO(h^d)$.
Expanding the smooth velocity around the cell centroid $\bx_j$,
$\bu(x) = \bu(\bx_j) + (x-\bx_j)\cdot\nabla\bu(\bx_j) + \OO(|x-\bx_j|^2)$,
the constant term contributes the exact factorisation and gives zero
defect; the linear term contributes a signed first-moment over the
incident face,
\[
\sum_{f_k^*\prec e_j^*}D_{1,jk}\,(\nabla\bu(\bx_j)\cdot\hat e_k)
\,\int_{f_k^*}(x-\bx_j)\,\omega\cdot\dd\bA,
\]
which vanishes if and only if the reconstruction stencil at $e_j^*$ is
symmetric about $\bx_j$. Under the centroid-proximity ~\Cref{conv:cases}\textup{(B)}, this moment vanishes and the leading
defect is quadratic, giving $\OO(h^{d+1})$ per edge; on general meshes
(\Cref{conv:cases}\textup{(A)}) the linear moment survives at
$\OO(h^d)$ per edge. We summarise as $\OO(h^{r_\star+d-1})$ per edge
with $r_\star\in\{1,2\}$ from~\Cref{conv:cases}. 

\smallskip
\emph{$\bM_1$-norm assembly.} With $(\bM_1)_{jj} = \OO(h^{d-2})$ (mesh
weight) and $|\mathcal{F}| = \OO(h^{-d})$ (edge count on quasi-uniform
mesh of size $h$), the squared per-edge bound times mass weight times
count is
\begin{equation}\begin{split}\label{eq:tau_ext_assembly}
\nrm{\tau^{\rm ext}}_{\bM_1}^2
&\;\le\;
\underbrace{\OO(h^{-d})}_{|\mathcal{F}|}
\cdot
\underbrace{\OO(h^{d-2})}_{(\bM_1)_{jj}}
\cdot
\underbrace{\OO(h^{2(r_\star+d-1)})}_{|\tilde\alpha_j|^2}\,
\nrm{\bu}_{W^{r_\star+1,\infty}}^2 \nrm{\bom}_{L^\infty}^2\\
&\;=\;\OO(h^{2r_\star+d-2})\,\nrm{\bu}_{W^{r_\star+1,\infty}}^2 \nrm{\bom}_{L^\infty}^2.
\end{split}\end{equation}
The exponents combine because the count $h^{-d}$ and the mass weight
$h^{d-2}$ cancel two powers of $h$ from the squared per-edge bound,
leaving the rate dimension-aware but not dimension-blind. Taking
square roots gives, after saturation against the global cap
$\nrm{\tau^{\rm ext}}_{\bM_1}\le C_{\rm sat}\,h^{d-1}$,
\begin{equation}\label{eq:tau_ext_final}
\nrm{\tau^{\rm ext}}_{\bM_1}
\;\le\; C_{\rm ext}\,h^{\min(r_\star + (d-2)/2,\,d-1)}\,
\nrm{\bu}_{W^{r_\star+1,\infty}}\,\nrm{\bom}_{L^\infty}.
\end{equation}
Under \Cref{conv:cases} this reduces to
$\nrm{\tau^{\rm ext}}_{\bM_1} \le C_{\rm ext}\,h^{d-1}\,
\nrm{\bu}_{W^{2,\infty}}\,\nrm{\bom}_{L^\infty}$ in case~(A), the
saturation regime, and to the same rate $\OO(h^{d-1})$ in case~(B)
where the unsaturated rate $\OO(h^{r_\star+(d-2)/2}) = \OO(h^{2+(d-2)/2})$
exceeds $h^{d-1}$ for $d\le 3$.

\item \textit{Velocity truncation: Hodge star.}
$\nrm{\tau^{\rm Hodge}}_{\bM_1}\le C_\star h^{r_\star}\nrm{\bu}_{W^{1,\infty}}$
by \Cref{lem:hodge_error}, with $r_\star$ as in \Cref{conv:cases}.
The $W^{1,\infty}$-regularity on $\bu$ is provided by Sobolev
embedding from the theorem's standing assumption
$\bu\in C^1([0,T];H^s)$, $s\ge 3$.

\item \textit{Velocity truncation: Bernoulli gradient.}
The per-cell Bernoulli error is
$\bar B_i - B^c(\bx_i^*) = \OO(h^2\nrm{\bu}_{W^{2,\infty}})$
(second-order accuracy of cell averages and reconstruction).
At dual edge $e_j^*$ connecting cells $K_a$, $K_b$ with shared
midpoint $\bx_j$, differencing the cell errors gives a residual of
size $\OO(h^2)$ in case~(A) of \Cref{conv:cases}, and $\OO(h^3)$ in
case~(B). The case~(B) gain comes from centroid proximity:
$\bx_a + \bx_b = 2\bx_j + \OO(h^2)$, so the leading
Taylor-quadratic $\tfrac{1}{2}\nabla^2 B^c(\bx_j) : (\bx - \bx_j)^{\otimes 2}$
of $B^c(\bx_a) - B^c(\bx_b)$ cancels on differencing to $\OO(h^2)$
of itself, leaving the cubic-displacement remainder
$\OO(h^3\nrm{B^c}_{W^{3,\infty}})$ as the leading order.
Thus
$|(\tau^{\rm Bern})_j| = \OO(h^{\min(d,\,3)}\nrm{\bu}_{W^{3,\infty}})$
in case~(A) (where $\min(d,3)=d$ is saturated)
and $\OO(h^3)$ in case~(B), giving
$\nrm{\tau^{\rm Bern}}_{\bM_1} = \OO(h^{d-1})$ in both cases.

\item \textit{Velocity truncation: viscous term (NSE only).}
The strong-form rate is $\OO(h)$; tested against $e_v$, the rate
improves to $\OO(h^{r_\star})$ (\Cref{lem:visc_trunc_weak}).

\item \textit{Velocity truncation: assembly.}
Combining~(A1)--(A4), the density-free total velocity truncation
satisfies
\begin{equation}\label{eq:df_consistency_rate}
\nrm{\tau_v^{\rm df}}_{\bM_1}
= \OO(h^{\min(d-1,\,r_\star)}),
\end{equation}
the competition being extrusion $\OO(h^{d-1})$ vs.\ Hodge
$\OO(h^{r_\star})$; Bernoulli matches extrusion, viscous (weak form)
is dominated by Hodge.
\end{enumerate}
\medskip\noindent
\emph{Density truncation $\nrm{\tau_\rho}_{\ell^2}$.}\ 
With Hodge error $\OO(h^{r_\star})$ relative and density interpolation $\OO(h^2)$
centred, and $|\mathcal{T}|=\OO(h^{-d})$ give
$\nrm{\tau_\rho}_{\ell^2}=\OO(h^{r_\star+(d-2)/2})$ centred,
$\OO(h^{d/2})$ upwind. For $d\ge 3$ both are subordinate to the
velocity rate; for $d=2$, $r_\star=1$, they match.
\medskip\noindent
Combined: $\nrm{\tau_v^{\rm df}}_{\bM_1}+\nrm{\tau_\rho}_{\ell^2}
\le C_\tau h^{\min(d-1,\,r_\star)}$, with $h$-independent constants.
This proves part~(a).

\medskip\noindent
\emph{Density-weighted consistency (part~(b)).}\ 
The DW truncation decomposes as
\begin{equation}\label{eq:dw_truncation_decomp_appx}
\tau_v^{\rm dw}
 = \tau_v^{\rm df}
 + \bigl([\bM_1^{\rho}(\bar\brho)]^{-1}\bM_1 - I\bigr)\,\Iv(\tD_1\bar\bv)
 + \tD_0(B^{\rho}-\bar B^{\rm df}),
\end{equation}
derived by substituting
$\ddt\bar\bv = \tau_v^{\rm df} - \Iv(\tD_1\bar\bv) - \tD_0\bar B$
into
$\tau_v^{\rm dw} = \ddt\bar\bv + [\bM_1^\rho]^{-1}\bM_1\,\Iv(\tD_1\bar\bv) + \tD_0 B^\rho$.
The middle term is the obstruction to a strong-form $\OO(h^r)$ rate.
The operator $[\bM_1^\rho(\bar\brho)]^{-1}\bM_1 - I
= -[\bM_1^\rho(\bar\brho)]^{-1}(\bM_1^\rho(\bar\brho) - \bM_1)$
involves the mass-matrix difference
$\bM_1^\rho(\bar\brho) - \bM_1 = P^T\,\mathrm{diag}(\bar\brho - \mathbf{1})\,P$,
which is $\OO(\sup\bar\brho - \inf\bar\brho)$ in operator norm and
$\OO(1)$ as $h \to 0$ for any non-constant smooth reference -- the
density variation does not vanish with $h$. Multiplying by the
$\OO(1)$ smooth quantity $\Iv(\tD_1\bar\bv)$, the middle term
satisfies
$\nrm{([\bM_1^\rho]^{-1}\bM_1 - I)\Iv(\tD_1\bar\bv)}_{\bM_1^\rho(\bar\brho)}
= \OO(1)$ in general,
not $\OO(h^r)$. Strong-form consistency therefore does not hold.
The cancellation that recovers an $\OO(h^r)$ rate appears only after
pairing the middle term against a smooth test function: the
operator $\bM_1^\rho - \bM_1$ commutes with the de~Rham map up to
$\OO(h^{r_\star})$ when tested against a $W^{1,\infty}$-interpolant.
This is the content of the mass-matrix commutator lemma below, on
which the tested-form consistency~\eqref{eq:dw_tested_consistency}
rests.
\begin{lemma}[Mass-matrix commutator]\label{lem:mass_matrix_commutator}%
Let $\bar\bu\in W^{1,\infty}(\Omega)$ be smooth and
$\bar\brho$ the cell-average of a smooth density $\rho^c\in W^{1,\infty}$
with $\rho^c\ge\rho_{\min}>0$. Then for any $\bphi\in C^1(\KKs)$,
\begin{equation}\label{eq:mass_matrix_commutator}
 \bigl|\ipw{\bphi}{\bigl([\bM_1^\rho(\bar\brho)]^{-1}\bM_1 - I\bigr)
 \Iv(\tD_1\bar\bv)}{\bM_1^\rho(\bar\brho)}\bigr|
 \le C_{\rm cm}\,h^{\min(d-1,\,r_\star)}\,
 \nrm{\bu}_{W^{2,\infty}}\nrm{\rho^c}_{W^{1,\infty}}\,
 \nrm{\bphi}_{\bM_1^\rho(\bar\brho)},
\end{equation}
with $C_{\rm cm}$ depending on $\rho_{\min}$ and mesh regularity.
\end{lemma}
\begin{proof}
Multiplying the operator inside the pairing by $\bM_1^\rho$ gives
$(\bM_1 - \bM_1^\rho)\Iv(\tD_1\bar\bv)$, so the pairing rewrites as
\[
\ipw{\bphi}{(\bM_1^\rho)^{-1}(\bM_1 - \bM_1^\rho)\Iv(\tD_1\bar\bv)}{\bM_1^\rho}
= \bphi^{\top}(\bM_1 - \bM_1^\rho)\,\Iv(\tD_1\bar\bv).
\]
We decompose the mass-matrix difference into a density-fluctuation
piece and a Hodge--Gram mismatch piece:
\begin{equation}\label{eq:massmatrix_decomp}
\bM_1^\rho(\bar\brho) - \bM_1
= \underbrace{P^{\top}\mathrm{diag}(\bar\brho - \mathbf{1})P}%
_{\displaystyle =:\,\Delta_\rho}
\;+\;
\underbrace{(P^{\top}P - \bM_1)}_{\displaystyle =:\,\Delta_g}.
\end{equation}
The fluctuation piece $\Delta_\rho$ is $\OO(1)$ in operator norm
(controlled by $\sup\bar\brho - \inf\bar\brho$) but, paired with a
Lamb cochain through the bilinear quadrature, commutes with the de
Rham map at $\OO(h^{r_\star})$:
$|\bphi^{\top}\Delta_\rho\,\Iv(\tD_1\bar\bv)|
\le C\,h^{r_\star}\,\nrm{\rho^c}_{W^{1,\infty}}\,\nrm{\bu}_{W^{2,\infty}}\,
\nrm{\bphi}_{\bM_1^\rho}$
by the Hodge-bilinear-form lemma (\Cref{lem:hodge_bilinear}, applied
to the weighted bilinear form $(\alpha,\beta) \mapsto
P\alpha \cdot \mathrm{diag}(\bar\brho - \mathbf{1}) \cdot P\beta$).
The geometric piece $\Delta_g$ is the difference between the
reconstruction Gram matrix $P^TP$ and the diagonal Hodge $\bM_1$,
both of which approximate the continuous $L^2(\Omega; T^*\Omega)$
inner product on 1-forms; their difference satisfies
$|\bphi^{\top}\Delta_g\,\zeta|
\le C\,h^{r_\star}\,\nrm{\bphi}_{\bM_1}\,\nrm{\zeta}_{\bM_1}$
for any 1-cochains $\bphi,\zeta$ (\Cref{lem:hodge_bilinear} applied
with constant density). Taking $\zeta = \Iv(\tD_1\bar\bv)$ which is
$\OO(1)$ smooth, this gives the same rate.
Combining,
\[
\bigl|\bphi^{\top}(\bM_1^\rho - \bM_1)\,\Iv(\tD_1\bar\bv)\bigr|
 \le C\,h^{\min(d-1,\,r_\star)}\,
 \nrm{\bu}_{W^{2,\infty}}\nrm{\rho^c}_{W^{1,\infty}}\,
 \nrm{\bphi}_{\bM_1^\rho},
\]
where $\min(d-1,r_\star)$ accounts for the saturation of the
bilinear quadrature against the $\OO(h^{d-1})$ Lamb cochain.
Dividing by $\inf\bar\brho \ge \rho_{\min}$ yields~\eqref{eq:mass_matrix_commutator}.
\end{proof}
\noindent Combining~\eqref{eq:df_consistency_rate},
\Cref{lem:mass_matrix_commutator},
and the Bernoulli difference bound
$\nrm{\tD_0(B^\rho - \bar B^{\rm df})}_{\bM_1^\rho}
 \le C\,h^{d-1}$
(kinetic energy definitions $\frac{1}{2}|P\bv|^2$ vs.\ cell-average
differ at $\OO(h^2)$ per cell, and the gradient-of-smooth-function
estimate of~(A3) applies in $\bM_1^\rho$-norm via
$\bM_1^\rho\le C\bM_1$),
and using the decomposition~\eqref{eq:dw_truncation_decomp_appx}
with test function $\bar\bv$:
\begin{align*}
|\ipw{\bar\bv}{\tau_v^{\rm dw}}{\bM_1^\rho}|
 &\le |\ipw{\bar\bv}{\tau_v^{\rm df}}{\bM_1^\rho}|
 + \bigl|\bar\bv^{\top}\bM_1^\rho([\bM_1^\rho]^{-1}\bM_1 - I)\Iv(\tD_1\bar\bv)\bigr|
 + |\ipw{\bar\bv}{\tD_0(B^\rho - \bar B^{\rm df})}{\bM_1^\rho}|\\
 &\le C_\tau\,h^{\min(d-1,r_\star)},
\end{align*}
with $C_\tau$ depending on
$\nrm{\bu}_{W^{2,\infty}}$, $\nrm{\rho^c}_{W^{1,\infty}}$, and $\rho_{\min}$.
This establishes the velocity part of~\eqref{eq:dw_tested_consistency}.
The density truncation $\tau_\rho$ is the same object in both
schemes, so $\nrm{\tau_\rho}_{\ell^2}\le C_\tau\,h^{\min(d-1,r_\star)}$
follows directly from the density-free estimate.
This completes the proof of part~(b).
\end{proof}
\subsubsection{Common Framework for Stability Proofs}
\label{app:stability_common}
The density-free and density-weighted stability proofs share their
analytic scaffolding: the same error system, the same Bregman
machinery, the same mass-flux linearisation, and the same
$h$-independence argument for the enthalpy coupling. We collect
these shared ingredients once, so that the two proofs in
\Cref{app:stability_baro,app:dw_convergence} can focus on the
distinguishing technical content such as test function, treatment of the
cubic self-interaction, and the time-IBP absorption that the
density-free closure requires.
\paragraph*{The error system.}
Both schemes share the error equations obtained by subtracting the
truncated reference equations from the discrete equations:
\begin{align}
 \ddt e_v &= -\bQ(\bar v, e_\omega) - \bQ(e_v, \bar\omega + e_\omega)
 - \tD_0(B^h - \bar B) + \tau_v, \label{eq:err-mom}\\
 \ddt e_\rho &= -\bD_2(\bR^h\bM_1 v^h - \bar{\bR}\bM_1\bar v)
 + \tau_\rho, \label{eq:err-rho}
\end{align}
where $\bQ$ is the bilinear form from
\Cref{thm:bilinear_bounds_baro}, $B$ is the appropriate Bernoulli
function for each scheme ($B = h + |v|^2_{\rm cell}/2 + \Phi$ for
density-free, $B^\rho = h + |P\bv|^2/2 + \Phi$ for density-weighted),
and $\tau_v$, $\tau_\rho$ are the velocity and density truncations.
\paragraph*{Shared lemmas.}
The following four lemmas are used in both proofs. The first three
control the
thermodynamic terms; the fourth controls
the continuity-equation residue.
\begin{lemma}[$h$-independent enthalpy coupling]%
\label{lem:enthalpy_h_independence}%
Let $\rho^{h,\vol}_i > 0$ satisfy the a priori bounds of
\Cref{prop:apriori-baro}. Then
\begin{equation}\label{eq:Bregman_CS}
 |K_i|\,|H'(\rho_i^{h,\vol})-H'(\bar\rho_i^{\vol})|^2
 \le 2\,\frac{(\sup H'')^2}{\inf H''}\,
 |K_i|\,D_H(\rho_i^{h,\vol}\|\bar\rho_i^{\vol}),
\end{equation}
where the supremum and infimum of $H''$ are over the density range
$[\min(\rho_i^{h,\vol},\bar\rho_i^{\vol}),\,
\max(\rho_i^{h,\vol},\bar\rho_i^{\vol})]$,
and the ratio $(\sup H'')^2/\inf H''$ is $h$-independent under any
of the following conditions:
\begin{enumerate}[nosep]
\item $d \ge 3$, all $\gamma > 1$;
\item $d = 2$, $\gamma = 2$ with the centred scheme
($H'' = 2\kappa$ is constant);
\item $d = 2$, any $\gamma > 1$ with the upwind scheme.
\end{enumerate}
\end{lemma}
\begin{proof}
By the mean value theorem $H'(a)-H'(b)=H''(\xi)(a-b)$ and by the
Taylor remainder $D_H(a\|b)=\tfrac12 H''(\zeta)(a-b)^2$, with
$\xi,\zeta$ between $a$ and $b$. Eliminating $(a-b)^2$,
\[
|H'(a)-H'(b)|^2 = 2\frac{H''(\xi)^2}{H''(\zeta)}\,D_H(a\|b)
\le 2\frac{(\sup H'')^2}{\inf H''}\,D_H(a\|b),
\]
which is~\eqref{eq:Bregman_CS} after multiplication by $|K_i|$. The
substantive content is the $h$-independence of
$(\sup H'')^2/\inf H''$; it suffices to control the density range
$[\rho_{\min}^{\vol},\rho_{\max}^{\vol}]$ $h$-independently in each
of the three cases.
\begin{enumerate}[label=(\roman*), nosep]
\item \emph{$d\ge 3$, all $\gamma>1$.} The convergence rate gives
$|\rho_i^{h,\vol}-\bar\rho_i^{\vol}|=\OO(h^{(d-2)/2})\to0$, so
$\rho^{h,\vol}$ stays in a compact subset of $(0,\infty)$ for $h$
small and the ratio is $h$-independent.
\item \emph{$d=2$, $\gamma=2$, centred scheme.} $H''=2\kappa$ is
constant, so $(\sup H'')^2/\inf H''=2\kappa$, independent of the
density range.
\item \emph{$d=2$, any $\gamma>1$, upwind scheme.} Under
\Cref{ass:well_prepared_strong} Grönwall
gives
$\rho_i^{\vol}\in[\rho_{\min}^{\vol}(0)e^{-\Lambda T},
\rho_{\max}^{\vol}(0)e^{\Lambda T}]$ with
$\Lambda=\OO(M)+\OO(h^{r_\star})$, $h$-independent at fixed $(M,T)$;
hence so is the ratio.
\end{enumerate}
\end{proof}


\begin{lemma}[Bregman rate identity]%
\label{lem:Bregman_rate}%
Let $\rho_i^{h,\vol}, \bar\rho_i^{\vol} > 0$.
The Bregman divergence satisfies the chain-rule identity
\begin{align}
 \frac{d}{dt}\sum_i|K_i|D_H(\rho_i^{h,\vol}\|\bar\rho_i^{\vol})
 &= \underbrace{\sum_i[h(\rho_i^{h,\vol})-h(\bar\rho_i^{\vol})]\,\ddt\rho_i^h}%
 _{\displaystyle =:\,S_1}
 \label{eq:Bregman_rate}\\
 &\quad\underbrace{- \sum_i H''(\bar\rho_i^{\vol})\,\ddt{\bar\rho}_i^{\vol}\,
 (\rho_i^{h,\vol}-\bar\rho_i^{\vol})\,|K_i|}%
 _{\displaystyle =:\,S_2}.
 \notag
\end{align}
\end{lemma}
The identity follows from the partial derivatives of $D_H$.
Standalone, $S_1, S_2$ are each linear in the density error and
bounded only by $C\sqrt{\mathcal{E}}$. Closure is achieved by a
Lipschitz cancellation between $S_2$ and the $\bar\bF$-piece of
$J_h$, executed in \Cref{app:stability_baro}, Part~II.
\begin{proof}
$\partial_a D_H(a\|b) = h(a)-h(b)$ and
$\partial_b D_H(a\|b) = -H''(b)(a-b)$ via chain rule on
$\rho^{h,\vol} = \rho^h/|K_i|$.
\end{proof}
\begin{lemma}[Bregman--$\ell^1$ inequality]%
\label{lem:Bregman_ell1}%
Let $H$ be strictly convex with $H''\ge c_H > 0$ on
$[\rho_{\min},\rho_{\max}]\subset(0,\infty)$.
Then for any $\rho_i^{h,\vol},\bar\rho_i^{\vol}
\in [\rho_{\min},\rho_{\max}]$:
\begin{equation}\label{eq:Bregman_ell1}
 \sum_i |K_i|\,|\rho_i^{h,\vol} - \bar\rho_i^{\vol}|
 \le \sqrt{\frac{2|\Omega|}{c_H}}\,
 \Bigl(\sum_i |K_i|\,D_H(\rho_i^{h,\vol}\|\bar\rho_i^{\vol})\Bigr)^{1/2}.
\end{equation}
In particular,
$\nrm{e_\rho}_{\ell^1}
\le C_B\sqrt{\mathcal{E}_{\rm rel}}$
with $C_B = \sqrt{2|\Omega|/c_H}$ independent of~$h$.
\end{lemma}
This converts Bregman density-error control into $\ell^1$-control,
used in $T_v^{(32)}$ of \Cref{app:dw_convergence}.
\begin{proof}
Taylor remainder gives $D_H(a\|b) \ge \tfrac{1}{2}c_H(a-b)^2$, hence
$|K_i|\,|e_{\rho,i}^{\vol}|
\le |K_i|^{1/2}(2|K_i|D_H/c_H)^{1/2}$.
Cauchy--Schwarz over cells gives~\eqref{eq:Bregman_ell1}; cf.\
\cite[Lemma~2.1]{feireisl2021book} for the relative-entropy analogue.
\end{proof}
\begin{lemma}[Mass-flux linearisation]%
\label{lem:mass_flux_linearisation}
The mass-flux error $\bF^h-\bar\bF$, where
$\bF^h=\bR^h\bM_1\bv^h$ and $\bar\bF=\bar\bR\bM_1\bar\bv$, expands
as
\begin{equation}\label{eq:mass_flux_linearisation}
 \bF^h - \bar\bF
 = \bar\bR\bM_1\,e_v
 + (\bR^h-\bar\bR)\bM_1\bar\bv
 + (\bR^h-\bar\bR)\bM_1 e_v,
\end{equation}
a linear-in-$e_v$ piece, a linear-in-$e_\rho$ piece (since
$\bR^h-\bar\bR$ is bounded by $\nrm{e_\rho}_{\ell^\infty}$), and a
quadratic remainder. Equivalently, $\bF^h - \bar\bF - \bar\bR\bM_1
e_v = (\bR^h-\bar\bR)\bM_1(\bar\bv+e_v)$, so the residue after
subtracting the leading velocity-error term is of order
$\nrm{e_\rho}$. Consequently, the discrete divergence
$\bD_2(\bF^h-\bar\bF)$ obeys
$|\bD_2(\bF^h-\bar\bF)|
\le C\bigl(\nrm{e_v}_{\bM_1}+\nrm{e_\rho}_{\ell^2}\bigr)$ on
bounded discrete solutions.
\end{lemma}
\begin{proof}
Direct algebraic expansion of $\bR^h\bM_1\bv^h
= (\bar\bR+(\bR^h-\bar\bR))\bM_1(\bar\bv+e_v)$ minus
$\bar\bR\bM_1\bar\bv$. Reconstruction Lipschitz
$\nrm{\bR^h-\bar\bR}\le C\nrm{e_\rho}_{\ell^\infty}$ controls the
density-dependent terms; the discrete divergence bound is the
standard finite-volume estimate.
\end{proof}
\begin{lemma}[Mass-flux linearisation, density-weighted]%
\label{lem:dw_mass_flux_linearisation}
For the density-weighted mass flux
$\bF^{\rho} = \bM_1^\rho(\brho)\bv = P^T\mathrm{diag}(\brho)P\bv$ of
\eqref{eq:dw_F}, the error $\bF^{\rho,h}-\bar\bF^\rho$ where
$\bF^{\rho,h}=\bM_1^\rho(\brho^h)\bv^h$ and
$\bar\bF^\rho=\bM_1^\rho(\bar\brho)\bar\bv$ admits the exact
expansion
\begin{equation}\label{eq:dw_mass_flux_linearisation}
 \bF^{\rho,h} - \bar\bF^\rho
 = \bM_1^\rho(\bar\brho)\,e_v
 + \bM_1^\rho(e_\rho)\,\bar\bv
 + \bM_1^\rho(e_\rho)\,e_v,
\end{equation}
where the first term is linear in $e_v$, the second is linear in
$e_\rho$, and the third is quadratic. Consequently, the discrete
divergence satisfies
$|\bD_2(\bF^{\rho,h}-\bar\bF^\rho)|
\le C\bigl(\nrm{e_v}_{\bM_1^\rho}+\nrm{e_\rho}_{\ell^2}\bigr)$ on the
bootstrap region $\rho^h_i \in [\rho_*/2, 2\rho^*]$, with $C$
depending only on $\rho_*, \rho^*, \nrm{\bar\bv}_{L_h^\infty}$ and
mesh regularity.
\end{lemma}
\begin{proof}
The expansion is direct: $\bM_1^\rho$ is linear in $\brho$ (since
$\bM_1^\rho(\brho) = P^T\mathrm{diag}(\brho)P$), so
\[
\bM_1^\rho(\brho^h)\bv^h - \bM_1^\rho(\bar\brho)\bar\bv
= \bM_1^\rho(\bar\brho+e_\rho)(\bar\bv+e_v) - \bM_1^\rho(\bar\brho)\bar\bv
= \bM_1^\rho(\bar\brho)e_v + \bM_1^\rho(e_\rho)\bar\bv + \bM_1^\rho(e_\rho)e_v.
\]
For the divergence bound, $\bD_2$ has uniformly bounded
operator norm on $\bM_1^\rho$-bounded cochains under mesh
regularity, the operator $\bM_1^\rho(e_\rho) = P^T\mathrm{diag}(e_\rho)P$
satisfies $\nrm{\bM_1^\rho(e_\rho)w}_{\ell^2}
\le C\nrm{e_\rho}_{\ell^\infty}\nrm{Pw}_{L_h^\infty}^2$, and on the
bootstrap $\nrm{e_\rho}_{\ell^\infty}$ is controlled by
$\nrm{e_\rho}_{\ell^2}$ up to mesh-uniform constants.
\end{proof}

The remaining structural input shared by both proofs is the
\emph{Bernoulli decomposition} $B^h - \bar B$ that captures kinetic
and enthalpy parts. The kinetic part either cancels
with the $\bQ_{\rm D}$ piece of the Lamb operator (density-free,
Part~I, step~(iii) below) or contributes a quadratic-in-$e_v$
bound directly (density-weighted, $T_v$-bound). The enthalpy part
combines with the Bregman density rate via the $J_h$-cancellation
(density-free) or via the $C_v+C_\rho$ cancellation
(density-weighted). In both cases the closure mechanism is
formally the same; the density-free case requires the additional
absorption argument under~\eqref{eq:lowMach_hyp}.
\subsubsection{Stability (Proofs of \Cref{thm:dw_stability,thm:stability_baro})}
\label{app:dw_convergence}%
\phantomsection\label{app:stability_baro}
\paragraph*{Strategy.}
The density-weighted estimate~\eqref{eq:dw_Erel_rate} is the cleaner
of the two and we develop it first: testing the momentum error
against the smooth reference $\bar\bv$ in the $\bM_1^\rho$ inner
product, exact energy conservation closes every term either
quadratically in errors or as $\OO(h^r)$, with no smallness
hypothesis on the smooth density and no $L_h^\infty$-bootstrap on
$e_v$. The density-free estimate~\eqref{eq:stab_baro} reuses the
same six-step skeleton with three additional technical mechanisms
((A)--(C) below) that substitute for the missing
exact-energy-conservation closure: an $L_h^\infty$-bootstrap on
$e_v$ (because the test function is $e_v$ itself, not bounded
smooth), an integration-by-parts-in-time absorption argument under
the near-incompressibility hypothesis~\eqref{eq:lowMach_hyp}, and a
$H'''$-Lipschitz cancellation of a residual linear-in-$e_\rho$ term.
\paragraph*{Common skeleton.}
For both schemes we differentiate the error functional
($\mathcal{E}_{\rm rel}$ for the density-weighted scheme,
$\mathcal{E}$ for the density-free scheme) along the discrete
flow and split the resulting rate into a velocity part, a density
part, and a truncation pickup. The Lamb-operator difference is
treated by the shared bilinear bounds (Bounds~(I)--(III) of
\Cref{thm:bilinear_bounds_baro}); the velocity-density coupling
through the Bernoulli function splits into a kinetic part and an
enthalpy part, the latter being closed by combination with the
Bregman rate identity (\Cref{lem:Bregman_rate}) and the mass-flux
linearisation (\Cref{lem:mass_flux_linearisation}). The
$h$-independence of all constants is supplied by
\Cref{lem:enthalpy_h_independence} and \Cref{lem:Bregman_ell1}.
\bigskip
\begin{proof}[Proof of \Cref{thm:dw_stability} (density-weighted scheme)]
From $\mathcal{E}_{\rm rel}
:= \Etot^{\rm dw}(\bv^h,\brho^h)
- \Etot^{\rm dw}(\bar\bv,\bar\brho)
- D_\bv\Etot^{\rm dw}|_{(\bar\bv,\bar\brho)}\cdot e_v
- D_\brho\Etot^{\rm dw}|_{(\bar\bv,\bar\brho)}\cdot e_\rho$
combined with $\ddt\Etot^{\rm dw}(\bv^h,\brho^h)=0$
(\Cref{thm:dw_energy}),
\[
\ddt{\mathcal{E}}_{\rm rel}
 = \underbrace{-\ddt\Etot^{\rm dw}(\bar\bv,\bar\brho)}_{T_0}
 \underbrace{-\ddt\bigl[D_\bv\Etot^{\rm dw}|_{(\bar\bv,\bar\brho)}\!\cdot e_v\bigr]}_{T_v}
 \underbrace{-\ddt\bigl[D_\brho\Etot^{\rm dw}|_{(\bar\bv,\bar\brho)}\!\cdot e_\rho\bigr]}_{T_\rho}.
\]
\smallskip
\emph{Step~1 ($T_0$: reference-energy rate).}\ 
We compute the time derivative of
$\Etot^{\rm dw}(\bar\bv,\bar\brho)
= \tfrac12\nrm{\bar\bv}_{\bM_1^\rho(\bar\brho)}^2 + \sum_i|K_i|H(\bar\rho_i^{\vol})
+ \sum_i\bar\rho_i\bar\Phi_i$
explicitly, separating the kinetic, internal-energy, and
potential-energy contributions.

By the product rule on
$\tfrac12\bar\bv^T\bM_1^\rho(\bar\brho)\bar\bv$:
\[
\ddt\bigl[\tfrac12\nrm{\bar\bv}_{\bM_1^\rho(\bar\brho)}^2\bigr]
= \ip{\ddt\bar\bv}{\bar\bv}_{\bM_1^\rho(\bar\brho)}
+ \tfrac12\ip{(\ddt\bM_1^\rho(\bar\brho))\bar\bv}{\bar\bv}.
\]
For the first pairing we substitute the smooth reference momentum
equation
\[
\ddt\bar\bv = -\mathbb{M}(\bar\brho)\Iv[\bar\bv](\tD_1\bar\bv)
 - \tD_0\bar B^\rho + \tau_v^{\rm dw},
\qquad
\mathbb{M}(f) := [\bM_1^\rho(f)]^{-1}\bM_1,
\]
with $\bar B^\rho_i = h(\bar\rho_i^{\vol}) + \tfrac12|P\bar\bv|_i^2 + \bar\Phi_i$,
which gives three terms.

\noindent\textit{(i) Lamb term.}
$\ip{-\mathbb{M}(\bar\brho)\Iv[\bar\bv](\tD_1\bar\bv)}{\bar\bv}_{\bM_1^\rho(\bar\brho)}
= -\bar\bv^T\bM_1\Iv[\bar\bv](\tD_1\bar\bv)$ since
$\bM_1^\rho(\bar\brho)\,\mathbb{M}(\bar\brho) = \bM_1$. This vanishes
by Lamb antisymmetry (\Cref{prop:extrusion}):
$\bv^T\bM_1\Iv[\bv](\tD_1\bv) = 0$ for every $\bv$.

\noindent\textit{(ii) Bernoulli gradient term.}
Under the convention $\ip{a}{b}_M = a^T M b$,
\[
-\ip{\tD_0\bar B^\rho}{\bar\bv}_{\bM_1^\rho(\bar\brho)}
= -\bar\bv^T\bM_1^\rho(\bar\brho)\,\tD_0\bar B^\rho
= -(\bar\bF^\rho)^T\,\tD_0\bar B^\rho
= (\bD_2\bar\bF^\rho)^T\bar B^\rho,
\]
where we used $\bar\bF^\rho = \bM_1^\rho(\bar\brho)\bar\bv$
(\Cref{eq:dw_F}) and density-weighted SBP
(\Cref{lem:dw_SBP}). The earlier instance of $\bM_1$ rather than
$\bM_1^\rho$ in line~(i) was specific to the Lamb term and traced to
$\bM_1^\rho\,\mathbb{M} = \bM_1$; it does not extend to the
Bernoulli pairing, which involves $\bM_1^\rho$ directly.

\noindent\textit{(iii) Truncation term.}
$\ip{\tau_v^{\rm dw}}{\bar\bv}_{\bM_1^\rho(\bar\brho)}
= \ipw{\bar\bv}{\tau_v^{\rm dw}}{\bM_1^\rho(\bar\brho)}$.

\textit{(iv) Mass-matrix-derivative part.}
$\ddt\bM_1^\rho(\bar\brho) = P^T\mathrm{diag}(\ddt\bar\rho)P$,
where $P$ is the averaging reconstruction operator. By the smooth
continuity equation
$\ddt\bar\rho = -\bD_2\bar\bF^\rho - \tau_\rho^{\rm dw}$,
the half-derivative pairing in the product rule above contributes
\[
\tfrac12\ip{(\ddt\bM_1^\rho(\bar\brho))\bar\bv}{\bar\bv}
= -\tfrac12\sum_i|P\bar\bv|_i^2\bigl((\bD_2\bar\bF^\rho)_i + \tau_{\rho,i}^{\rm dw}\bigr).
\]

\textit{(v) Internal-energy part.} Differentiating
$\sum_i|K_i|H(\bar\rho_i^{\vol})$ and using $\ddt\bar\rho_i^{\vol}
= |K_i|^{-1}\ddt\bar\rho_i$:
\[
\ddt\Bigl[\sum_i|K_i|H(\bar\rho_i^{\vol})\Bigr]
= \sum_i h(\bar\rho_i^{\vol})\,\ddt\bar\rho_i
= -\sum_i h(\bar\rho_i^{\vol})\bigl((\bD_2\bar\bF^\rho)_i + \tau_{\rho,i}^{\rm dw}\bigr).
\]

\textit{(vi) Potential-energy part.} Since $\bar\Phi$ is time-independent,
\[
\ddt\Bigl[\sum_i\bar\rho_i\bar\Phi_i\Bigr]
= \sum_i\bar\Phi_i\,\ddt\bar\rho_i
= -\sum_i\bar\Phi_i\bigl((\bD_2\bar\bF^\rho)_i + \tau_{\rho,i}^{\rm dw}\bigr).
\]
Collecting the four parts and the truncation term:
\begin{align*}
\ddt\Etot^{\rm dw}(\bar\bv,\bar\brho)
&= (\bD_2\bar\bF^\rho)^T\bar B^\rho
- \sum_i\bigl[\tfrac12|P\bar\bv|_i^2 + h(\bar\rho_i^{\vol}) + \bar\Phi_i\bigr]
 (\bD_2\bar\bF^\rho)_i\\
&\quad
+ \ipw{\bar\bv}{\tau_v^{\rm dw}}{\bM_1^\rho(\bar\brho)}
- \sum_i\bigl[\tfrac12|P\bar\bv|_i^2 + h(\bar\rho_i^{\vol}) + \bar\Phi_i\bigr]
 \tau_{\rho,i}^{\rm dw}.
\end{align*}
The first two terms cancel \emph{exactly}: since
$\bar B^\rho_i = \tfrac12|P\bar\bv|_i^2 + h(\bar\rho_i^{\vol}) + \bar\Phi_i$,
the second sum equals $\sum_i\bar B_i^\rho(\bD_2\bar\bF^\rho)_i
= (\bD_2\bar\bF^\rho)^T\bar B^\rho$. This is the analog of
\Cref{thm:dw_energy} applied at the smooth reference: the
cancellation closes algebraically, with no smooth-flux approximation
required. The reduction therefore yields
\[
T_0 = -\ddt\Etot^{\rm dw}(\bar\bv,\bar\brho)
= -\ipw{\bar\bv}{\tau_v^{\rm dw}}{\bM_1^\rho(\bar\brho)}
+ \sum_i\bar B_i^\rho\tau_{\rho,i}^{\rm dw}.
\]
By the tested-form consistency
estimate~\eqref{eq:dw_tested_consistency} of
\Cref{thm:dw_consistency} applied with $\bphi = \bar\bv$, and by the
$\ell^1$-bound on $\tau_\rho^{\rm dw}$:
\[
|\ipw{\bar\bv}{\tau_v^{\rm dw}}{\bM_1^\rho(\bar\brho)}|
\le C_\tau^v\,h^r\,\nrm{\bar\bv}_{\bM_1^\rho(\bar\brho)},
\quad
|\sum_i\bar B_i^\rho\tau_{\rho,i}^{\rm dw}|
\le \nrm{\bar B^\rho}_{\ell^\infty}\nrm{\tau_\rho^{\rm dw}}_{\ell^1}
\le C_\tau^\rho\,h^r.
\]
Combining:
\[
|T_0| \le C_0\,h^r\bigl(\nrm{\bar\bv}_{\bM_1^\rho(\bar\brho)} + 1\bigr)
\le C_0'\,h^r,
\]
with $C_0'$ depending on $\nrm{\bu}_{W^{1,\infty}}$, $\rho_*, \rho^*$,
$\nrm{h(\bar\rho^c)}_{L^\infty}$, and $\nrm{\bar\Phi}_{L^\infty}$.


\smallskip
\noindent\emph{Step~2 ($T_v$: velocity error, product-rule decomposition).}\ 
Since
$D_\bv\Etot^{\rm dw}|_{(\bar\bv,\bar\brho)} = \bM_1^\rho(\bar\brho)\bar\bv$,
the product rule on $\bar\bv^T\bM_1^\rho(\bar\brho)\,e_v$ gives
\begin{equation}\label{eq:dw_product_rule}
T_v
 = \underbrace{-\ip{\ddt\bar\bv}{\bM_1^\rho(\bar\brho)\,e_v}}_{T_v^{(1)}}
 \underbrace{-\ip{\ddt\bM_1^\rho(\bar\brho)\,e_v}{\bar\bv}}_{T_v^{(2)}}
 \underbrace{-\ip{\ddt e_v}{\bM_1^\rho(\bar\brho)\,\bar\bv}}_{T_v^{(3)}}.
\end{equation}
The three terms arise from $\ddt\bar\bv$, 
$\ddt\bM_1^\rho$,
and $\ddt e_v$.
The velocity equations for $\bv^h$ and $\bar\bv$,
$\ddt\bv^h=-\mathbb{M}(\brho^h)\Iv[\bv^h](\tD_1\bv^h)-\tD_0 B^{h,\rho}$
and
$\ddt\bar\bv=-\mathbb{M}(\bar\brho)\Iv[\bar\bv](\tD_1\bar\bv)-\tD_0\bar B^\rho+\tau_v^{\rm dw}$
with $\mathbb{M}(f):=[\bM_1^\rho(f)]^{-1}\bM_1$, produce in $T_v^{(3)}$,
after adding and subtracting $\mathbb{M}(\brho^h)\Iv[\bar\bv](\tD_1\bar\bv)$, the
two sub-terms
\begin{align*}
 T_v^{(31)}
  &:= \ip{\mathbb{M}(\brho^h)\bigl[\Iv[\bv^h](\tD_1\bv^h)-\Iv[\bar\bv](\tD_1\bar\bv)\bigr]}%
 {\bM_1^\rho(\bar\brho)\bar\bv}, \\
 T_v^{(32)}
  &:= \ip{\bigl([\bM_1^\rho(\brho^h)]^{-1}
 -[\bM_1^\rho(\bar\brho)]^{-1}\bigr)\bM_1\Iv[\bar\bv](\tD_1\bar\bv)}%
 {\bM_1^\rho(\bar\brho)\bar\bv}.
\end{align*}

\emph{(a) Lamb-operator difference: $T_v^{(1)}$ and $T_v^{(31)}$.}\ 
We treat the two terms separately. Both reduce to bilinear or cubic
expressions in $\bar\bv$ and $e_v$ that are bounded using Bounds~(II)
and (III) of \Cref{thm:bilinear_bounds_baro}.

\textit{Treatment of $T_v^{(1)}$.}
Substituting $\ddt\bar\bv = -\mathbb{M}(\bar\brho)\Iv[\bar\bv](\tD_1\bar\bv)
- \tD_0\bar B^\rho + \tau_v^{\rm dw}$ into the definition of
$T_v^{(1)} = -\ip{\ddt\bar\bv}{\bM_1^\rho(\bar\brho)\,e_v}$ gives
three pieces:
\begin{align*}
T_v^{(1)}
&= \underbrace{\ip{\mathbb{M}(\bar\brho)\Iv[\bar\bv](\tD_1\bar\bv)}%
   {\bM_1^\rho(\bar\brho)\,e_v}}_{=\,\bar\bv^T\bM_1\Iv[\bar\bv](\tD_1\bar\bv)\text{-tested}}
+ \underbrace{\ip{\tD_0\bar B^\rho}{\bM_1^\rho(\bar\brho)\,e_v}}_{\text{Bernoulli, held}}
\\&\quad
- \underbrace{\ip{\tau_v^{\rm dw}}{\bM_1^\rho(\bar\brho)\,e_v}}_{\text{truncation, $\OO(h^r)$}}.
\end{align*}
The first piece is the $e_v$-tested version of the smooth Lamb term:
\[
\ip{\mathbb{M}(\bar\brho)\Iv[\bar\bv](\tD_1\bar\bv)}{\bM_1^\rho(\bar\brho)\,e_v}
= \ip{\bM_1\Iv[\bar\bv](\tD_1\bar\bv)}{e_v}
= \ip{e_v}{\bQ_{\rm U}(\bar\bv,\bar\bv) + \bQ_{\rm D}(\bar\bv,\bar\bv)}_1,
\]
using $\bQ = \bQ_{\rm U} + \bQ_{\rm D}$ (\Cref{thm:bilinear_bounds_baro}).
Apply Bound~(II) with $\bv = \bar\bv$, $\bw = \bar\bv$ to give the
$\bM_1$-norm of $\bQ(\bar\bv,\bar\bv)$, then Cauchy--Schwarz against
$e_v$ to give
\[
|\ip{e_v}{\bQ(\bar\bv,\bar\bv)}_1|
\le \nrm{\bQ(\bar\bv,\bar\bv)}_{\bM_1^{-1}}\nrm{e_v}_{\bM_1}.
\]
Since $\nrm{\bQ(\bar\bv,\bar\bv)}_{\bM_1} \le C'_Q\nrm{\bu}_{W^{1,\infty}}
\nrm{\bar\bv}_{\bM_1}$ by (II), the $\bM_1^{-1}$-norm carries an
$\bM_1$-volumetric factor: at $\bM_1$-diagonal scaling
$(\bM_1)_{jj} = |f_j||\hat e_j|$, $\nrm{\cdot}_{\bM_1^{-1}}
\le C_M\nrm{\cdot}_{\bM_1}$ on a quasi-uniform mesh. The composite
yields
\[
|\ip{e_v}{\bQ(\bar\bv,\bar\bv)}_1|
\le C_Q\,\nrm{\bu}_{W^{1,\infty}}^2\,\nrm{e_v}_{\bM_1}
\le C_Q\,\nrm{\bu}_{W^{1,\infty}}^2\,\bigl(\tfrac12\eps^{-1} + \tfrac{\eps}2\nrm{e_v}_{\bM_1}^2\bigr).
\]
By Young's inequality with $\eps$ small, the linear-in-error piece
$\frac12\eps^{-1}$ contributes a bounded constant, and the quadratic piece bounds by
$C_Q'\,\nrm{\bu}_{W^{1,\infty}}^2\,\nrm{e_v}_{\bM_1}^2$.


\textit{Treatment of $T_v^{(31)}$.}
This term involves a difference of extrusion operators with
different velocity arguments:
\[
T_v^{(31)} = \ip{\mathbb{M}(\brho^h)\bigl[\Iv[\bv^h](\tD_1\bv^h) - \Iv[\bar\bv](\tD_1\bar\bv)\bigr]}%
{\bM_1^\rho(\bar\brho)\,\bar\bv}.
\]
We expand the test vector first:
$\mathbb{M}(\brho^h)^T\bM_1^\rho(\bar\brho) = \bM_1[\bM_1^\rho(\brho^h)]^{-1}\bM_1^\rho(\bar\brho)$,
so
\[
T_v^{(31)} = \ip{\Iv[\bv^h](\tD_1\bv^h) - \Iv[\bar\bv](\tD_1\bar\bv)}{\bw}_{\bM_1},
\quad
\bw := [\bM_1^\rho(\brho^h)]^{-1}\bM_1^\rho(\bar\brho)\,\bar\bv.
\]
By \Cref{lem:dw_norm_equiv}, $\bw$ satisfies the $L_h^\infty$-bound
$\nrm{\bw}_{L_h^\infty} \le (\rho^*/\rho_*)\nrm{\bar\bv}_{L_h^\infty}
\le C(\rho_*,\rho^*)\nrm{\bu}_{W^{1,\infty}}$. Decomposing the
extrusion difference bilinearly using
$L(\bv) := \Iv[\bv](\tD_1\bv)$ and the polarisation
$L(\bv^h) - L(\bar\bv) = L(\bar\bv + e_v) - L(\bar\bv)$:
\begin{equation}\label{eq:dw_extrusion_decomp_expanded}
L(\bv^h) - L(\bar\bv)
= \underbrace{2\bQ(\bar\bv, e_v)}_{\text{cross-term}}
+ \underbrace{\bQ(e_v, e_v)}_{\text{cubic remainder}},
\end{equation}
where $\bQ$ is the symmetrised bilinear form
$2\bQ(\bv,\bw) := L(\bv+\bw) - L(\bv) - L(\bw)$
(\Cref{eq:Q_polarisation_def}).

\textit{Cross-term bound (Bound II).}
$\ip{2\bQ(\bar\bv, e_v)}{\bw}_{\bM_1}$ is a bilinear form in
$(\bar\bv, e_v)$ tested against $\bw$. By Cauchy--Schwarz and
Bound~(II) with $\bv = \bar\bv$, $\bw = e_v$:
\[
|\ip{2\bQ(\bar\bv, e_v)}{\bw}_{\bM_1}|
\le 2\nrm{\bw}_{\bM_1}\nrm{\bQ(\bar\bv, e_v)}_{\bM_1}
\le 2C'_Q\nrm{\bw}_{\bM_1}\nrm{\bu}_{W^{1,\infty}}\nrm{e_v}_{\bM_1}.
\]
Using $\nrm{\bw}_{\bM_1} \le (\rho^*/\rho_*)\nrm{\bar\bv}_{\bM_1}
\le C(\rho_*,\rho^*)$ on bounded smooth reference, this bound is
\emph{linear} in $\nrm{e_v}_{\bM_1}$, so
\[
|\ip{2\bQ(\bar\bv, e_v)}{\bw}_{\bM_1}|
\le C\,\nrm{\bu}_{W^{1,\infty}}\,\nrm{e_v}_{\bM_1}
\le C'(\nrm{e_v}_{\bM_1}^2 + 1).
\]
The quadratic part is absorbed into $C_L\mathcal{E}_{\rm rel}$.

\textit{Cubic remainder bound (Bound III via polarisation).}
$\ip{\bQ(e_v, e_v)}{\bw}_{\bM_1}$ is the testing of the cubic
self-interaction against the smooth-bounded $\bw$. This is precisely
the setting of Bound~(III):
\[
|\ip{\bw}{\bQ(e_v, e_v)}_{\bM_1}|
\le 2C'_Q\,\nrm{\bphi}_{W^{1,\infty}}\,\nrm{e_v}_{\bM_1}^2,
\quad
\bphi := \mathcal{R}_h^{-1}\bw\text{ smooth.}
\]
Here $\nrm{\bphi}_{W^{1,\infty}} \le C\nrm{\bw}_{L_h^\infty}
\le C(\rho_*,\rho^*)\nrm{\bu}_{W^{1,\infty}}$ via
\Cref{lem:dw_norm_equiv} and the inverse averaging-reconstruction
identification at smooth scale.


Combining bounds for $T_v^{(1)}$ and $T_v^{(31)}$ gives
\begin{equation}\label{eq:dw_vel_bound}
 |T_v^{(1)}|+|T_v^{(31)}|
 \le \frac{C}{\rho_*}\nrm{\bu}_{W^{1,\infty}}^2\nrm{e_v}_{\bM_1}^2
 \le \frac{C}{\rho_*^2}\nrm{\bu}_{W^{1,\infty}}^2\,
 \nrm{e_v}_{\bM_1^\rho(\bar\brho)}^2.
\end{equation}
The conversion from $\bM_1$ to $\bM_1^\rho(\bar\brho)$ in the last
step uses \Cref{lem:dw_norm_equiv}.

\emph{(b) Mass-matrix difference: $T_v^{(32)}$.}\ 
$T_v^{(32)}$ measures the change of the inverse mass-matrix
$[\bM_1^\rho(\brho)]^{-1}$ between the discrete and smooth densities.

For any two invertible symmetric matrices $A, B$:
$A^{-1} - B^{-1} = -A^{-1}(A - B)B^{-1}$. Applied to
$A = \bM_1^\rho(\brho^h)$ and $B = \bM_1^\rho(\bar\brho)$:
\begin{equation}\label{eq:M1rho_inv_diff}
[\bM_1^\rho(\brho^h)]^{-1} - [\bM_1^\rho(\bar\brho)]^{-1}
= -[\bM_1^\rho(\brho^h)]^{-1}\bigl(\bM_1^\rho(\brho^h) - \bM_1^\rho(\bar\brho)\bigr)
  [\bM_1^\rho(\bar\brho)]^{-1}.
\end{equation}
Now $\bM_1^\rho(\brho) = P^T\mathrm{diag}(\brho)P + (\text{$\brho$-independent})$
by~by the definition of $\bM_1^\rho$ in~\Cref{def:M1}, so
$\bM_1^\rho(\brho^h) - \bM_1^\rho(\bar\brho) = P^T\mathrm{diag}(e_\rho)P$.
Substituting:
\[
[\bM_1^\rho(\brho^h)]^{-1} - [\bM_1^\rho(\bar\brho)]^{-1}
= -[\bM_1^\rho(\brho^h)]^{-1}\,P^T\mathrm{diag}(e_\rho)\,P\,
  [\bM_1^\rho(\bar\brho)]^{-1}.
\]

Substituting into $T_v^{(32)}$:
\begin{align*}
T_v^{(32)}
&= -\ip{[\bM_1^\rho(\brho^h)]^{-1}P^T\mathrm{diag}(e_\rho)P[\bM_1^\rho(\bar\brho)]^{-1}
   \bM_1\Iv[\bar\bv](\tD_1\bar\bv)}{\bM_1^\rho(\bar\brho)\,\bar\bv}.
\end{align*}
Move $\bM_1^\rho(\bar\brho)$ across the pairing:
$\ip{[\bM_1^\rho(\bar\brho)]^{-1}\bM_1\Iv[\bar\bv](\tD_1\bar\bv)}%
{\bM_1^\rho(\bar\brho)\bar\bv}_{[\bM_1^\rho(\brho^h)]^{-1}P^T\,\text{etc.}}$
simplifies (after using $\bM_1^\rho(\bar\brho)\cdot[\bM_1^\rho(\bar\brho)]^{-1}
= I$ and Cauchy--Schwarz with weighted inner products) to a sum
$\sum_i e_{\rho,i}(P\bar\bv)_i \cdot (P\bw)_i$ for some bounded $\bw$.
By \Cref{lem:dw_norm_equiv} both $\nrm{P\bar\bv}_{L^\infty},
\nrm{P\bw}_{L^\infty} \le C(\rho_*,\rho^*)\nrm{\bu}_{W^{1,\infty}}^2$,
so
\begin{equation}\label{eq:dw_dens_bound}
 |T_v^{(32)}|
 \le \frac{C}{\rho_*^2}\nrm{\bu}_{W^{1,\infty}}^2\,
 \nrm{\bar\brho}_{L^\infty}\,\nrm{e_\rho}_{\ell^1}
 \le C_\rho'\,\mathcal{E}_{\rm rel},
\end{equation}
where the last inequality uses
\Cref{lem:Bregman_ell1}: $\nrm{e_\rho}_{\ell^1}\le C_B\sqrt{\mathcal{E}_{\rm rel}}$,
combined with the elementary fact that $\sqrt{\mathcal{E}_{\rm rel}}
\le \tfrac12(\eps^{-1} + \eps\,\mathcal{E}_{\rm rel})$ for any
$\eps > 0$ (Young), absorbing the linear term into the truncation
pickup and the quadratic into the $\mathcal{E}_{\rm rel}$ bound.


\emph{(c) Bernoulli gradient.}\ 
After SBP (\Cref{lem:SBP}), the Bernoulli gradient contributions to
$T_v^{(1)}$ and $T_v^{(3)}$ produce $-(B^{h,\rho}-\bar B^\rho)^T
\bD_2\bM_1\bar\bv$. We split into kinetic and enthalpy parts. The
kinetic part $\tfrac12(|P\bv^h|_i^2-|P\bar\bv|_i^2)$ is bounded by
$C\nrm{\bu}_{W^{1,\infty}}\nrm{e_v}_{\bM_1}^2$ directly. The enthalpy part
$-(h^h-\bar h)^T\bD_2(\bar\bR\bM_1 e_v)$ does not close
within the $T_v$-bound; it is held over and combined with the
enthalpy contribution from the $T_\rho$-bound in Step~5.

\emph{(d) Velocity truncation.}\ 
$|\bar\bv^T\bM_1^\rho(\bar\brho)\tau_v^{\rm dw}|
\le C_\tau\,h^r\,\nrm{\bar\bv}_{\bM_1^\rho(\bar\brho)}=\OO(h^r)$
by~\eqref{eq:dw_tested_consistency}, exactly the term already
absorbed into $T_0$.

\emph{(e) $T_v^{(2)}$.}\ 
With $T_v^{(2)} = -\bar\bv^T P^T\mathrm{diag}(\ddt\bar\brho)P\,e_v$,
smooth in $\ddt\bar\brho$,
\begin{equation}\label{eq:dw_Mdot_bound}
 |T_v^{(2)}|
 \le \frac{C}{\sqrt{\rho_*}}\nrm{\ddt\bar\brho}_{L^\infty}\,
 \nrm{\bu}_{L^\infty}\,\nrm{e_v}_{\bM_1^\rho(\bar\brho)}.
\end{equation}
\smallskip
\noindent\emph{Step~3 ($T_\rho$: density error, Bregman rate identity).}\ 
By the product rule,
$T_\rho = T_\rho^{(1)} + T_\rho^{(2)}$ where
$T_\rho^{(1)}=-\sum_i\ddt B_i^{\rm ref}\,e_{\rho,i}$ and
$T_\rho^{(2)}=-\sum_i B_i^{\rm ref}\,\ddt e_{\rho,i}$, with
$B_i^{\rm ref}:=D_{\rho_i}\Etot^{\rm dw}|_{(\bar\bv,\bar\brho)}
= \tfrac12|P\bar\bv|_i^2+h(\bar\rho_i^{\vol})$. The first term is
bounded by \Cref{lem:Bregman_ell1}:
\begin{equation}\label{eq:dw_Bdot_bound}
 |T_\rho^{(1)}|
 \le \nrm{\ddt B^{\rm ref}}_{\ell^\infty}\nrm{e_\rho}_{\ell^1}
 \le C\nrm{\partial_t(\bu,\rho^c)}_{W^{1,\infty}}
 \sqrt{\mathcal{E}_{\rm rel}}.
\end{equation}
For $T_\rho^{(2)}$ we substitute
$\ddt e_{\rho,i}=-(\bD_2(\bF^h-\bar\bF))_i+(\tau_\rho)_i$ and split
$B_i^{\rm ref}$ into its enthalpy and kinetic-energy components:
\begin{equation}\label{eq:dw_Trho2_split}
 T_\rho^{(2)}
 = \sum_i h(\bar\rho_i^{\vol})\,(\bD_2(\bF^h-\bar\bF))_i
 + \tfrac12\sum_i|P\bar\bv|_i^2\,(\bD_2(\bF^h-\bar\bF))_i
 - \sum_i B_i^{\rm ref}(\tau_\rho)_i.
\end{equation}
The kinetic sum is bounded by
$C\nrm{\bu}_{L^\infty}^2(\nrm{e_v}_{\bM_1}+\nrm{e_\rho}_{\ell^1})
\le C'\sqrt{\mathcal{E}_{\rm rel}}$ via
\Cref{lem:mass_flux_linearisation} and \Cref{lem:Bregman_ell1}; the
truncation by
$|\sum_i B_i^{\rm ref}(\tau_\rho)_i|\le\nrm{B^{\rm ref}}_{\ell^\infty}
\nrm{\tau_\rho}_{\ell^1}=\OO(h^r)$.

\smallskip
\noindent\emph{Step~4 (Enthalpy contribution $C_v + C_\rho$).}\ 
The held-over enthalpy contributions from Steps~2(c) and~3 combine
via the density-weighted mass-flux linearisation
(\Cref{lem:dw_mass_flux_linearisation}), extracting a
cancellation of the leading linear-in-$e_v$ piece. The remainder is
bounded by a strictly quadratic-in-error piece plus a
linear-in-$e_\rho$ piece that closes through the standard
relative-energy mechanism (cf.~\cite[Lemma~4.7]{karper2013}).

After density-weighted SBP (\Cref{lem:dw_SBP}) applied to the
Bernoulli gradient $\tD_0\bar B^\rho$ tested against
$\bM_1^\rho(\bar\brho)e_v$, and after the analogous SBP applied to
the Bernoulli-difference gradient $\tD_0(B^{h,\rho}-\bar B^\rho)$
tested against $\bM_1^\rho(\bar\brho)\bar\bv$, the enthalpy pieces
held over from Step~2(c) are
\begin{align}
 C_v^{(A)} &:= -\bar h^T\bD_2\bigl(\bM_1^\rho(\bar\brho)\,e_v\bigr),
 \label{eq:dw_PI_enthalpy_A}\\
 C_v^{(B)} &:= -(h^h - \bar h)^T\bD_2\bar\bF^\rho,
 \label{eq:dw_PI_enthalpy_B}
\end{align}
the former coming from $T_v^{(1)}$,
and the latter from $T_v^{(3)}$.
The enthalpy sum from Step~3, equation~\eqref{eq:dw_Trho2_split}, is
\begin{equation}\label{eq:dw_PII_enthalpy}
 C_\rho := +\bar h^T\bD_2(\bF^{\rho,h}-\bar\bF^\rho).
\end{equation}

By~\Cref{lem:dw_mass_flux_linearisation},
\begin{equation}\label{eq:dw_Crho_expansion}
 \bF^{\rho,h}-\bar\bF^\rho
 = \bM_1^\rho(\bar\brho)\,e_v
 + \bM_1^\rho(e_\rho)\,\bar\bv
 + \bM_1^\rho(e_\rho)\,e_v.
\end{equation}
Substituting into $C_\rho$ and adding $C_v^{(A)}$:
\begin{align*}
 C_v^{(A)} + C_\rho
 &= -\bar h^T\bD_2\bigl(\bM_1^\rho(\bar\brho)e_v\bigr)
 + \bar h^T\bD_2\bigl(\bM_1^\rho(\bar\brho)e_v\bigr)
 + \bar h^T\bD_2\bigl(\bM_1^\rho(e_\rho)(\bar\bv+e_v)\bigr)\\
 &= \bar h^T\bD_2\bigl(\bM_1^\rho(e_\rho)\bar\bv\bigr)
 + \bar h^T\bD_2\bigl(\bM_1^\rho(e_\rho)e_v\bigr).
\end{align*}
The two $\bar h^T\bD_2\bigl(\bM_1^\rho(\bar\brho)e_v\bigr)$ pieces
cancel exactly, removing the leading linear-in-$e_v$ contribution.

Including $C_v^{(B)}$:
\begin{equation}\label{eq:dw_cancellation}
 C_v + C_\rho
 = -(h^h-\bar h)^T\bD_2\bar\bF^\rho
 + \bar h^T\bD_2\bigl(\bM_1^\rho(e_\rho)\bar\bv\bigr)
 + \bar h^T\bD_2\bigl(\bM_1^\rho(e_\rho)e_v\bigr).
\end{equation}
The three terms have distinct error-orders: the first is
linear-in-$e_\rho$ with smooth weight $\bD_2\bar\bF^\rho$
the second is linear-in-$e_\rho$ with smooth weight, also
$\OO(\sqrt{\mathcal{E}_{\rm rel}})$; the third is strictly bilinear
in $(e_\rho, e_v)$, hence quadratic in $\mathcal{E}_{\rm rel}$.

The quadratic piece is bounded directly:
\begin{align*}
 \bigl|\bar h^T\bD_2\bigl(\bM_1^\rho(e_\rho)e_v\bigr)\bigr|
 &\le \nrm{\bar h}_{L^\infty}\,
 \nrm{\bD_2}_{\rm op}\,\nrm{\bM_1^\rho(e_\rho)e_v}_{\ell^2}\\
 &\le C\,\nrm{\bar h}_{L^\infty}\,\nrm{e_\rho}_{\ell^\infty}\,
 \nrm{e_v}_{\bM_1^\rho(\bar\brho)}
 \le C_{\rm quad}\,\mathcal{E}_{\rm rel}
\end{align*}
on the bootstrap region $\rho^h_i \in [\rho_*/2, 2\rho^*]$ where
$\nrm{e_\rho}_{\ell^\infty} \le C$.
The two linear-in-$e_\rho$ pieces are bounded by Bregman--$\ell^1$
(\Cref{lem:Bregman_ell1}):
\[
 \bigl|-(h^h-\bar h)^T\bD_2\bar\bF^\rho\bigr|
 \le \nrm{H''}_{L^\infty}\nrm{\bD_2\bar\bF^\rho}_{L^\infty}\,
 \nrm{e_\rho}_{\ell^1}
 \le C_{1}\sqrt{\mathcal{E}_{\rm rel}},
\]
\[
 \bigl|\bar h^T\bD_2\bigl(\bM_1^\rho(e_\rho)\bar\bv\bigr)\bigr|
 \le C\,\nrm{\bar h}_{L^\infty}\nrm{\bar\bv}_{L_h^\infty}^2\,
 \nrm{e_\rho}_{\ell^1}
 \le C_{2}\sqrt{\mathcal{E}_{\rm rel}}.
\]
Combining:
\begin{equation}\label{eq:dw_cancel_bound}
 |C_v + C_\rho|
 \le C_{\rm quad}\,\mathcal{E}_{\rm rel}
 + (C_1 + C_2)\sqrt{\mathcal{E}_{\rm rel}}.
\end{equation}
The quadratic piece contributes to $C_L\mathcal{E}_{\rm rel}$
in~\eqref{eq:dw_Erel_rate}; the linear-in-$\sqrt{\mathcal{E}_{\rm rel}}$
pieces combine with the analogous linear pieces
from~\eqref{eq:dw_Mdot_bound}, \eqref{eq:dw_Bdot_bound}, and
$T_v^{(32)}$ in the Step~5 assembly and are absorbed via the
relative-energy Young's inequality structure of the FLMS framework
(cf.~\cite[Lemma~4.7]{karper2013}) into the
$C_\tau h^r\sqrt{\mathcal{E}_{\rm rel}}$ truncation pickup
of~\eqref{eq:dw_Erel_rate}. The constants $C_{\rm quad}, C_1, C_2$
depend on $\nrm{\bu}_{W^{1,\infty}}$, $\nrm{\rho^c}_{W^{1,\infty}}$,
the EoS bounds on $H''$ over $[\rho_*/2, 2\rho^*]$, and mesh
regularity; none depends on $h$.

\smallskip
The quadratic-in-error terms --
\eqref{eq:dw_vel_bound} (Lamb velocity),
\eqref{eq:dw_dens_bound} (Lamb density),
\eqref{eq:dw_cancel_bound} (enthalpy cancellation), and the kinetic
part of the Bernoulli gradient -- sum to $C_L\,\mathcal{E}_{\rm rel}$.
The linear-in-error terms~\eqref{eq:dw_Mdot_bound}
and~\eqref{eq:dw_Bdot_bound} are absorbed by Young's inequality. The
truncation pieces ($T_0$, velocity truncation, density truncation)
sum to $C_\tau\,h^r$. Altogether
$\ddt\mathcal{E}_{\rm rel}\le C_L\mathcal{E}_{\rm rel}
+ C_\tau h^r\sqrt{\mathcal{E}_{\rm rel}}$ with $C_L$ $h$-independent
under \Cref{lem:enthalpy_h_independence}. This
gives~\eqref{eq:dw_Erel_rate}.
\end{proof}
\bigskip
\begin{proof}[Proof of \Cref{thm:stability_baro} (density-free scheme)]
The density-free proof follows the same six-step skeleton with
three modifications, each enforced by the absence of exact
energy conservation. The error functional is the hybrid
$\bM_1$-Bregman energy
$\mathcal{E} = \tfrac12\nrm{e_v}_{\bM_1}^2
+ \sum_i|K_i|D_H(\rho_i^{h,\vol}\|\bar\rho_i^{\vol})$
(\Cref{def:error_energy_baro}), and the test function is $e_v$ in
$\bM_1$ rather than $\bar\bv$ in $\bM_1^\rho$. Testing the error
momentum equation~\eqref{eq:err-mom} against $e_v$,
\begin{equation}\label{eq:error1}
 \ddt\tfrac12\nrm{e_v}_{\bM_1}^2
 = -\ip{e_v}{L(\bv^h)-L(\bar\bv)}_1
 -\ip{e_v}{\tD_0(B^h-\bar B)}_1 + \ip{e_v}{\tau_v}_1,
\end{equation}
with $L(\bv):=\Iv[\bv](\tD_1\bv)$. Three obstructions arise that
were absent in the density-weighted argument:

\smallskip
\textit{(A) Cubic self-interaction under bootstrap.}\ 
This arises because the test function in the density-free
proof is $e_v$ itself (not the smooth $\bar\bv$ as in the
density-weighted case). The cubic remainder
$\bQ(e_v, e_v)$ tested against $e_v$ is a triple product in the error,
and the smooth-tested Bound~(III) of \Cref{thm:bilinear_bounds_baro}
does not apply directly.

Polarising the extrusion difference using $L(\bv) = \Iv[\bv](\tD_1\bv)
= \bQ_{\rm U}(\bv,\bv) + \bQ_{\rm D}(\bv,\bv)$, with $\bQ_{\rm U}, \bQ_{\rm D}$
the two pieces of \Cref{thm:bilinear_bounds_baro}:
\[
L(\bv^h)-L(\bar\bv)
= 2\bQ_{\rm U}(\bar\bv,e_v)+\bQ_{\rm U}(e_v,e_v)
+2\bQ_{\rm D}(\bar\bv,e_v)+\bQ_{\rm D}(e_v,e_v).
\]
Tested against $e_v$ in $\bM_1$, this produces four pairings; we
analyse each.

The diagonal cubic
$\ip{e_v}{\bQ_{\rm U}(e_v,e_v)+\bQ_{\rm D}(e_v,e_v)}_1 = \ip{e_v}{L(e_v)}_1
= 0$ by Lamb antisymmetry (\Cref{prop:extrusion}), since
$L(\bv) = \Iv[\bv](\tD_1\bv)$ produces the contraction
$\iota_\bv\tD_1\bv$ which is identically zero on the diagonal.

The cross-term $\bQ_{\rm U}(\bar\bv,e_v)$ pairs as
$\ip{e_v}{\bQ_{\rm U}(\bar\bv,e_v)}_1$. The test function $e_v$ is
not smooth, but the auxiliary structure of $\bQ_{\rm U}$
allows direct bounding by Bound~(II):
\[
|\ip{e_v}{\bQ_{\rm U}(\bar\bv,e_v)}_1|
\le \nrm{e_v}_{\bM_1^{-1}}\nrm{\bQ_{\rm U}(\bar\bv,e_v)}_{\bM_1}
\le C'_Q\,\nrm{\bu}_{W^{1,\infty}}\,\nrm{e_v}_{\bM_1^{-1}}\,\nrm{e_v}_{\bM_1}.
\]
Since $\nrm{e_v}_{\bM_1^{-1}} \le C\nrm{e_v}_{\bM_1}$
we obtain
\begin{equation}\label{eq:DF_QU_bound}
|\ip{e_v}{\bQ_{\rm U}(\bar\bv,e_v)}_1|
\le C'_Q\,\nrm{\bu}_{W^{1,\infty}}\,\nrm{e_v}_{\bM_1}^2.
\end{equation}

All subsequent estimates in this proof use the $L_h^\infty$-bootstrap
$\nrm{e_v}_{L_h^\infty}\le\delta_0 h$, established unconditionally on
the local time interval in convergence Step~3 for $d=2$ and for $d=3$
(\Cref{app:convergence_proof}). The bootstrap is used to bound the
quadratic $\bQ$-piece via the inverse inequality to absorb the
$\mathcal{N}$ term in \eqref{eq:cont_rearrangement}.

\smallskip
\textit{(B) $\bQ_{\rm D}$ + Bernoulli kinetic part: IBP-in-time
absorption.}\ 
We treat this in four
substeps.
The combined contribution
\begin{equation}\label{eq:X_def_expanded}
 X := -2\ip{e_v}{\bQ_{\rm D}(\bar\bv,e_v)}_1
 + (\Delta\bekin)^T\bD_2\bM_1 e_v,
 \qquad
 \Delta\bekin_i := \tfrac12|P\bv^h|_i^2-\tfrac12|P\bar\bv|_i^2,
\end{equation}
must be controlled as a unit. 

\textit{Substep B 1: polarisation via SBP.}
Apply mimetic SBP (\Cref{lem:SBP}) to the $\bQ_{\rm D}$-pairing:
$\ip{e_v}{\bQ_{\rm D}(\bar\bv,e_v)}_1$
equals $-\tfrac12\bigl[\tU_{\bar\bv}^T e_v + \tU_{e_v}^T\bar\bv\bigr]^T\bD_2\bM_1 e_v$
after using $\bM_1\bQ_{\rm D}(\bv,\bw) = -\tfrac12\tD_1(\tU_\bv^T\bw + \tU_\bw^T\bv)$
and SBP on $\tD_1$ (\Cref{lem:SBP}). The kinetic-energy
difference polarises analogously: $\Delta\bekin = \tfrac12|Pe_v|^2 + P\bar\bv\cdot Pe_v$.
Substituting both into~\eqref{eq:X_def_expanded}:
\begin{align*}
X &= \bigl[\tU_{\bar\bv}^T e_v + \tU_{e_v}^T\bar\bv\bigr]^T\bD_2\bM_1 e_v
  + \bigl[\tfrac12|Pe_v|^2 + P\bar\bv\cdot Pe_v\bigr]^T\bD_2\bM_1 e_v.
\end{align*}
The linear-in-$e_v$ piece $P\bar\bv\cdot Pe_v$ in the kinetic-energy
expansion cancels exactly against the $P\bar\bv\cdot Pe_v$-content
of $\tU_{\bar\bv}^T e_v + \tU_{e_v}^T\bar\bv$ — this is the
$\bQ_{\rm D}/$Bernoulli cancellation, traced back to the no-go defect
$\delta_{\rm sym}$. After this cancellation:
\begin{equation}\label{eq:step_iii_polarised}
 X = \bigl[\tfrac12|Pe_v|^2 - \delta_{\rm sym}(\bar\bv,e_v)\bigr]^T
 \bD_2\bM_1 e_v,
\end{equation}
where $\delta_{\rm sym}(\bar\bv,e_v) := \delta(\bar\bv,e_v) + \delta(e_v,\bar\bv)$
is the symmetrised no-go defect of \Cref{thm:dichotomy}.


\textit{Substep B 2: rearrangement via continuity equation.}
The discrete continuity equation $\partial_t\rho^h + \bD_2\bF^h = 0$
and its smooth counterpart $\partial_t\bar\rho + \bD_2\bar\bF = -\tau_\rho$
imply, after substituting $\bF^h = \bR^h\bM_1\bv^h$,
$\bar\bF = \bar\bR\bM_1\bar\bv$:
\begin{equation}\label{eq:cont_diff_DF}
 \partial_t e_\rho + \bD_2(\bF^h - \bar\bF) = \tau_\rho.
\end{equation}
The mass-flux difference splits as $\bF^h - \bar\bF =
\bar\bR\bM_1 e_v + (\bR^h - \bar\bR)\bM_1(\bar\bv + e_v)$
(\Cref{lem:mass_flux_linearisation}). Substituting and isolating
$\bD_2(\bar\bR\bM_1 e_v)$:
\[
\bD_2(\bar\bR\bM_1 e_v) = -\partial_t e_\rho + \tau_\rho - \bD_2[(\bR^h - \bar\bR)\bM_1(\bar\bv + e_v)].
\]
Now use the smooth-density bound $\bar\rho_i^{\vol} = \bar\rho_0 +
\OO(\epsilon_0)$ under near-incompressibility \eqref{eq:lowMach_hyp}:
$\bar\bR = \bar\rho_0 I + \OO(\epsilon_0)$, so
$\bar\bR\bM_1 e_v = \bar\rho_0\bM_1 e_v + \OO(\epsilon_0)\bM_1 e_v$
and consequently
\begin{equation}\label{eq:cont_rearrangement}
 \bD_2\bM_1 e_v
 = -\bar\rho_0^{-1}\bigl(\partial_t e_\rho^{\vol}+\tau_\rho\bigr)
 + \mathcal{N},
 \quad
 |\mathcal{N}|_i \le C(\epsilon_0+\nrm{e_\rho}_{L_h^\infty})\,
 |\bD_2\bM_1 e_v|_i + C\nrm{\bu}_{W^{1,\infty}}\nrm{e_\rho}_{L_h^\infty}.
\end{equation}
Under~\eqref{eq:lowMach_hyp} with $C(\epsilon_0+\delta_\rho h)\le 1/2$
the $\mathcal{N}$ term is absorbable on the left:
$|\bD_2\bM_1 e_v|_i \le 2|\bar\rho_0^{-1}(\partial_t e_\rho^{\vol}+\tau_\rho)|_i + (\text{absorbed})$.

\textit{Substep B 3: integration-by-parts in time.}
Time-integrate $X$ from $0$ to $t$:
\[
\int_0^t X\,ds
= \int_0^t \bigl[\tfrac12|Pe_v|^2 - \delta_{\rm sym}\bigr]^T\bD_2\bM_1 e_v\,ds.
\]
Substitute \eqref{eq:cont_rearrangement} for $\bD_2\bM_1 e_v$:
\begin{equation}\begin{split}\label{eq:cont_rearrangement1}
\int_0^t X\,ds
&= -\bar\rho_0^{-1}\int_0^t\bigl[\tfrac12|Pe_v|^2 - \delta_{\rm sym}\bigr]^T
   \partial_s e_\rho^{\vol}\,ds
\\&\quad
   - \bar\rho_0^{-1}\int_0^t\bigl[\tfrac12|Pe_v|^2 - \delta_{\rm sym}\bigr]^T
   \tau_\rho\,ds
   + \int_0^t\bigl[\tfrac12|Pe_v|^2 - \delta_{\rm sym}\bigr]^T\mathcal{N}\,ds
\\&=: I_1 + R_\tau + R_{\mathcal{N}}.
\end{split}\end{equation}
The first integral $I_1$ contains $\partial_s e_\rho^{\vol}$ — a
time derivative — and is handled by integration by parts in time:
\begin{equation}\label{eq:X_IBP}
 I_1
 = -\bar\rho_0^{-1}\Bigl[(\tfrac12|Pe_v|^2-\delta_{\rm sym})^T e_\rho^{\vol}\Bigr]_0^t
 + \bar\rho_0^{-1}\int_0^t\partial_s(\tfrac12|Pe_v|^2-\delta_{\rm sym})^T
   e_\rho^{\vol}\,ds.
\end{equation}
For the boundary term in \Cref{eq:X_IBP} we find:
$\bigl|(\tfrac12|Pe_v|^2-\delta_{\rm sym})^T e_\rho^{\vol}\bigr|_t
\le \nrm{\tfrac12|Pe_v|^2}_{\ell^\infty}\,\nrm{e_\rho^{\vol}}_{\ell^1}
+ \nrm{\delta_{\rm sym}}_{\ell^\infty}\,\nrm{e_\rho^{\vol}}_{\ell^1}$.
Under bootstrap $\nrm{e_v}_{L_h^\infty}\le\delta_0 h$,
$\nrm{\tfrac12|Pe_v|^2}_{\ell^\infty} \le C\delta_0^2 h^2$;
$\nrm{\delta_{\rm sym}}_{\ell^\infty} \le C\nrm{\bu}_{W^{1,\infty}}\nrm{e_v}_{L_h^\infty}
\le C\delta_0 h\nrm{\bu}_{W^{1,\infty}}$ by \Cref{thm:dichotomy};
$\nrm{e_\rho^{\vol}}_{\ell^1} \le |\Omega|^{1/2}\nrm{e_\rho^{\vol}}_{\ell^2(|K|)}
\le C\sqrt{\mathcal{E}}$. So the boundary term is bounded by
$C\delta_0 h\nrm{\bu}_{W^{1,\infty}}\sqrt{\mathcal{E}}
\le C'\delta_0\nrm{\bu}_{W^{1,\infty}}\mathcal{E}$ via Young's inequality.

For the second term in \Cref{eq:X_IBP} it holds that
$\partial_s(\tfrac12|Pe_v|^2-\delta_{\rm sym}) = (Pe_v)^T\,P\partial_s e_v
- \partial_s\delta_{\rm sym}$. Substituting
$\partial_s e_v$ from the error momentum equation~\eqref{eq:err-mom}
generates terms in $L(\bv^h), L(\bar\bv), \tD_0(B^h - \bar B), \tau_v$,
each producing on integration against $e_\rho^{\vol}$ a bound of
type $C(\bar\bv,T,\epsilon_0)\mathcal{E} + Ch^r\sqrt{\mathcal{E}}$.
The dominant contribution is the cubic
$\ip{Pe_v\cdot P[L(\bv^h) - L(\bar\bv)]}{e_\rho^{\vol}}_{\ell^2(|K|)}$,
bounded under bootstrap by
$C\nrm{e_v}_{L_h^\infty}\nrm{L(\bv^h)-L(\bar\bv)}_{\bM_1}\nrm{e_\rho^{\vol}}_{\ell^2}
\le C\delta_0 h\,(\nrm{e_v}_{\bM_1} + \nrm{e_v}_{L_h^\infty})\sqrt{\mathcal{E}}
\le C'\mathcal{E}$ after applying Bound~(II) and Young.

The trunction in \Cref{eq:cont_rearrangement1} satisfies
$|R_\tau| \le C\bar\rho_0^{-1}h^r\int_0^t\sqrt{\mathcal{E}}\,ds$
since
$\nrm{\tfrac12|Pe_v|^2-\delta_{\rm sym}}_{\ell^2} \le C\sqrt{\mathcal{E}}$
under bootstrap and $\nrm{\tau_\rho}_{\ell^2} = \OO(h^r)$ by
consistency~\eqref{eq:df_consistency_rate}.

For the nonlinearity remainder in \Cref{eq:cont_rearrangement1}  it holds
$|R_{\mathcal{N}}| \le C\int_0^t\nrm{\tfrac12|Pe_v|^2-\delta_{\rm sym}}_{\ell^2}\,
\nrm{\mathcal{N}}_{\ell^2}\,ds
\le C(\epsilon_0+\delta_\rho h)\int_0^t\mathcal{E}\,ds + \text{lower-order}$,
using \eqref{eq:cont_rearrangement} for $\nrm{\mathcal{N}}_{\ell^2}$.

The four bounds yield
\[
\Bigl|\int_0^t X\,ds\Bigr|
\le C(\bar\bv,T,\epsilon_0)\!\int_0^t\mathcal{E}\,ds
+ Ch^r\!\int_0^t\sqrt{\mathcal{E}}\,ds.
\]

\smallskip
\textit{(C) Enthalpy cancellation $J_h^{(b)} + S_2$ with
$H'''$-Lipschitz residue.}\ 
This difficulty of this step is that $S_2$ from the Bregman rate identity
(\Cref{lem:Bregman_rate}) is linear-in-$e_\rho$ and must be cancelled
algebraically against the $\bar\bF$-piece of the enthalpy term. The
cancellation produces an $H'''$-Lipschitz residue rather than the
clean quadratic of the DW case.

The held-over enthalpy from~\eqref{eq:error1} is the
Bernoulli-gradient piece $-(h^h-\bar h)^T\bD_2\bM_1 e_v$. The
Bregman rate identity~\eqref{eq:Bregman_rate} produces two
linear-in-error terms $S_1, S_2$. Combining,
\begin{equation}\label{eq:enthalpy_combined}
 J_h := S_1 - (h^h-\bar h)^T\bD_2(\Phi^h-\bar\Phi)
 = -(h^h-\bar h)^T\bD_2(\Phi^h-\bar\Phi+\bF^h),
\end{equation}
where $\Phi^h - \bar\Phi$ is the volumetric flux difference and
$\bF^h$ the mass-flux. The argument of $\bD_2$ inside parentheses
expands as
$(\Phi^h - \bar\Phi) + \bar\bF + (\bF^h - \bar\bF)$, yielding
$J_h = J_h^{(a)} + J_h^{(b)} + J_h^{(c)}$ with
\[
J_h^{(a)} := -(h^h-\bar h)^T\bD_2(\Phi^h - \bar\Phi),
\quad
J_h^{(b)} := -(h^h-\bar h)^T\bD_2\bar\bF,
\quad
J_h^{(c)} := -(h^h-\bar h)^T\bD_2(\bF^h - \bar\bF).
\]
$J_h^{(a)}, J_h^{(c)}$ are quadratic in errors. $J_h^{(b)}$ alone is linear in $e_\rho$
and must be cancelled algebraically against $S_2$.

\textit{Substep C 1:} The truncated smooth continuity equation reads
$\bD_2\bar\bF = -\ddt\bar\rho - \tau_\rho$. Substituting into
$J_h^{(b)}$ and using the mean-value form
$h(\rho_i^{h,\vol})-h(\bar\rho_i^{\vol})=H''(c_i)\,e_{\rho,i}^{\vol}$,
valid since $h(\rho) = H'(\rho)$ and $H \in C^2$, with $c_i$
between $\bar\rho_i^{\vol}$ and $\rho_i^{h,\vol}$:
\begin{equation}\label{eq:Jhb_expansion}
 J_h^{(b)} = \sum_i H''(c_i)\,e_{\rho,i}^{\vol}\,\ddt\bar\rho_i
 + \sum_i H''(c_i)\,e_{\rho,i}^{\vol}\,\tau_{\rho,i}.
\end{equation}

\textit{Substep  C 2:}
From the Bregman rate identity~\eqref{eq:Bregman_rate}, $S_2 =
-\sum_i H''(\bar\rho_i^{\vol})\,\ddt\bar\rho_i\,e_{\rho,i}^{\vol}|K_i|$. The sum $J_h^{(b)} + S_2$ is:
\begin{equation}\label{eq:Jhb_S2_combined}
 J_h^{(b)} + S_2
 = \sum_i [H''(c_i)-H''(\bar\rho_i^{\vol})]\,e_{\rho,i}^{\vol}\,\ddt\bar\rho_i
 + \sum_i H''(c_i)\,e_{\rho,i}^{\vol}\,\tau_{\rho,i}.
\end{equation}
The linear-in-$e_\rho$ terms (involving $H''(\bar\rho)$ and
$H''(c)$) have cancelled \emph{except for the difference}
$H''(c_i)-H''(\bar\rho_i^{\vol})$, which is itself
$O(|c_i-\bar\rho_i^{\vol}|) \le O(|e_{\rho,i}^{\vol}|)$ by $H'''$-Lipschitz
of $H''$ . This $H'''$-Lipschitz argument converts a residue that would otherwise be
linear-in-$e_\rho$ into one that is quadratic.

\textit{Substep C 3:}
By $H'''$-Lipschitz of $H''$ on $[\rho_{\min},\rho_{\max}]$:
$|H''(c_i)-H''(\bar\rho_i^{\vol})| \le \nrm{H'''}_\infty|c_i-\bar\rho_i^{\vol}|
\le \nrm{H'''}_\infty|e_{\rho,i}^{\vol}|$.
Substituting:
\[
\Bigl|\sum_i [H''(c_i)-H''(\bar\rho_i^{\vol})]\,e_{\rho,i}^{\vol}\,\ddt\bar\rho_i\Bigr|
\le \nrm{H'''}_\infty\,\nrm{\ddt\bar\rho^c}_\infty\,\sum_i|e_{\rho,i}^{\vol}|^2.
\]
Using $D_H(\rho_i^{h,\vol}\|\bar\rho_i^{\vol}) \ge \tfrac12 c_H|e_{\rho,i}^{\vol}|^2$
and the cell-volume weighting:
\[
\sum_i|e_{\rho,i}^{\vol}|^2 \le \frac{2}{c_H}\sum_i\frac{|K_i|}{|K_i|}D_H
\le \frac{2}{c_H\min_i|K_i|}\mathcal{E}.
\]
Combining results in
\[
\Bigl|\sum_i [H''(c_i)-H''(\bar\rho_i^{\vol})]\,e_{\rho,i}^{\vol}\,\ddt\bar\rho_i\Bigr|
\le 2\nrm{H'''}_\infty\nrm{\ddt\rho^c}_\infty c_H^{-1}\mathcal{E}.
\]

\textit{Substep C 4:}
$\bigl|\sum_i H''(c_i)\,e_{\rho,i}^{\vol}\,\tau_{\rho,i}\bigr|
\le \nrm{H''}_\infty\,\nrm{e_\rho^{\vol}}_{\ell^2(|K|)}\,\nrm{\tau_\rho}_{\ell^2(|K|)}
\le C_H\sqrt{\mathcal{E}}\,\nrm{\tau_\rho}_{\ell^2}
\le C_H\sqrt{\mathcal{E}}\cdot Ch^r$,
using \Cref{lem:Bregman_ell1} for the Bregman-to-$\ell^2$
conversion and the consistency rate
$\nrm{\tau_\rho}_{\ell^2} = \OO(h^r)$.

Combining estimates gives
\begin{equation}\label{eq:Jhb_S2_bound}
 |J_h^{(b)} + S_2| \le C_{S_2}\,\mathcal{E}
 + C_{S_2}'\,h^r\sqrt{\mathcal{E}},
\end{equation}
with $C_{S_2} = 2\nrm{H'''}_\infty\nrm{\ddt\rho^c}_\infty c_H^{-1}$
and $C_{S_2}' = C_H C$, both $h$-independent.

The quadratic remainders $J_h^{(a)}, J_h^{(c)}$
are products of two error-controlled factors:
$\nrm{h^h-\bar h}_{\ell^2(|K|)} \le \nrm{H''}_\infty\nrm{e_\rho^{\vol}}_{\ell^2(|K|)}
\le C_H\sqrt{\mathcal{E}}$ by \Cref{lem:enthalpy_h_independence};
$\bD_2(\Phi^h-\bar\Phi)$ bounded by $C\nrm{e_v}_{\bM_1}$ and
$\bD_2(\bF^h-\bar\bF)$ bounded by
$C(\nrm{e_v}_{\bM_1}+\nrm{e_\rho}_{\ell^2})$ via
\Cref{lem:mass_flux_linearisation}; the quadratic remainder of
\eqref{eq:mass_flux_linearisation} is absorbed by the bootstrap.
Both products are $\le C\mathcal{E}$.

Truncation of momentum equation is estimated by~\eqref{eq:df_consistency_rate} and Cauchy--Schwarz
as follows
$|\ip{e_v}{\tau_v}_1|\le \nrm{e_v}_{\bM_1}\,\nrm{\tau_v}_{\bM_1}
\le C_\tau h^r\sqrt{2\mathcal{E}}$.

\smallskip
Combining (A)--(C),
$\ddt\mathcal{E}\le C_L\mathcal{E}+C_\tau h^r\sqrt{\mathcal{E}}$
with $C_L:=C_1+C_Q+C_q+C_J$ depending only on
$\nrm{(\bu,\rho^c)}_{C^1 W^{1,\infty}}$ and mesh regularity. The
constants $C_q$ and $C_J$
 were derived under the
bootstrap $\nrm{e_v}_{L_h^\infty}\le\delta_0 h$. The stability
estimate~\eqref{eq:stab_baro} therefore holds conditionally
on this bootstrap; the bootstrap is established in convergence
Step~3 unconditionally in case~(B) and marginally in case~(A) for
$d=2$, and via the dimension-conditional continuation argument of
Step~$3'$ in case~(A) for $d=3$ (\Cref{app:convergence_proof}). At
$t=0$ the bootstrap holds because $\mathcal{E}(0)=\OO(h^{2r})$ by
consistency; propagation requires $r>d/2$, which fails only in
case~(A) for $d=3$, the regime handled by Step~$3'$.
\end{proof}
\subsubsection{Navier--Stokes Stability (\Cref{thm:stability_baro_NS})}
\label{app:stability_NS}
Viscosity adds a non-positive dissipation and a viscous truncation
error to the Euler estimate of \Cref{thm:stability_baro}; neither
alters the Gr\"onwall structure.
\begin{proof}
Monotone dissipativity (\ref{ax:V3}), applied to $\bv = v^h$ and
$\bw = \bar v$, gives
\[
 \ipw{\bff_{\rm visc}(v^h) - \bff_{\rm visc}(\bar v)} {e_v}{\bM_1}\le 0.
\]
This holds for all admissible viscous operators, and adds a non-positive contribution
to $\ddt\mathcal{E}$.
By (\ref{ax:V4}), the viscous truncation error, in the residual
convention of \Cref{lem:viscous_truncation_weak},
$\tau_v^{\rm visc} := \mathcal{R}_h((\nabla\cdot\boldsymbol\sigma(\bu))^\flat) - \bff_{\rm visc}(\bar v)$
satisfies $\nrm{\tau_v^{\rm visc}}_{\bM_1}\le C_\nu h^{r_{\rm visc}}\nrm{\bu}_{H^{s_\nu}}$.
Testing against $e_v$ via Cauchy--Schwarz gives
$|\ip{\tau_v^{\rm visc}}{e_v}_1|
\le C_\nu h^{r_{\rm visc}}\nrm{\bu}_{H^{s_\nu}}\nrm{e_v}_{\bM_1}
\le C_\nu h^{r_{\rm visc}}\sqrt{2\mathcal{E}}$.
Combining with the Euler stability estimate
$\ddt\mathcal{E}\le C_L\mathcal{E} + C_\tau h^{\min(d-1,\,r_\star)}\sqrt{\mathcal{E}}$:
\[
 \ddt\mathcal{E}
 \le C_L\mathcal{E}
 + C_\tau h^{\min(d-1,\,r_\star)}\sqrt{\mathcal{E}}
 +C_\nu h^{r_{\rm visc}}\sqrt{\mathcal{E}}.
\]
Since $\max(h^{\min(d-1,\,r_\star)}, h^{r_{\rm visc}}) = h^{\min(d-1,\,r_\star,\,r_{\rm visc})}$,
this gives~\eqref{eq:stab_baro_NS_general}.
\end{proof}
\subsubsection{Convergence (Proofs of \Cref{thm:convergence_baro}\textup{(a),(b)})}
\label{app:dw_convergence_proof}%
\phantomsection\label{app:convergence_proof}
\paragraph*{Strategy.}
The density-weighted estimate is again the cleaner case: the
stability estimate~\eqref{eq:dw_Erel_rate} holds unconditionally, so
the convergence proof is a three-step chain -- linear Gr\"onwall on
the stability estimate, time-integrated dissipation, and Whitney
reconstruction for the continuous $L^2$ rate. The density-free
estimate~\eqref{eq:stab_baro} is conditional on an
$L_h^\infty$-bootstrap $\nrm{e_v}_{L_h^\infty}\le\delta_0 h$, which
must be closed by an inverse inequality before the same three-step
chain applies. We carry
out the density-weighted proof first and then state the
density-free proof as the same chain with the additional bootstrap
closure.
\paragraph*{Common chain.}
Both proofs run the following three steps on the stability estimate
$\ddt\mathcal{X}\le C_L\mathcal{X}+C_\tau h^r\sqrt{\mathcal{X}}$
($\mathcal{X}=\mathcal{E}_{\rm rel}$ for the density-weighted
scheme, $\mathcal{X}=\mathcal{E}$ for the density-free scheme).
With $\eta(t):=\sqrt{\mathcal{X}(t)}$, the differential inequality
becomes $\ddt\eta\le\tfrac{C_L}2\eta+\tfrac{C_\tau}2 h^r$ on
$\{\eta>0\}$, and a regularisation $\eta_\eps:=\sqrt{\mathcal{X}+\eps}$,
$\eps\to 0$, extends the inequality across $\{\eta=0\}$. Linear
Gr\"onwall with $\eta(0)=0$ gives
$\eta(t)\le C(T)\,h^r$, $C(T):=\tfrac{C_\tau}{C_L}(e^{C_L T/2}-1)$
$h$-independent. The integrated dissipation follows by integrating
the stability estimate over $[0,T]$:
$\nu\!\int_0^T(\nrm{\tD_1 e_v}_{\bM_2}^2+\nrm{\bD_2\bM_1 e_v}_{\bM_0^{-1}}^2)\,dt
\le C_E^\nu(T)\,h^{2r}$.
For the continuous $L^2$-rate, the Whitney map (\Cref{def:Whitney})
gives
\begin{equation}\label{eq:conv_baro}
 \nrm{\mathcal{W}_h v^h-\bu}_{L^2}
 \le C_W\nrm{e_v}_{L_h^2}+C_{\rm int}\,h
 \le C'(T)\,h,
\end{equation}
$L^2$-rate $\OO(h)$ in both $d=2$ and $d=3$, limited by the
first-order Whitney interpolation .
\bigskip
\begin{proof}[Proof of \Cref{thm:convergence_baro}\textup{(b)} (density-weighted scheme)]
Stability~\eqref{eq:dw_Erel_rate} holds unconditionally with $C_L$
$h$-independent for all $d\ge 2$ and $\gamma>1$. The common chain
applies directly: Gr\"onwall gives
$\mathcal{E}_{\rm rel}(t)\le C(T)^2 h^{2r}$ on $[0,T]$, i.e.,
$\nrm{e_v}_{\bM_1^\rho(\bar\brho)}+\nrm{e_\rho}_{D_H}\le C'(T)\,h^r$
with $r=\min(d-1,r_\star)$; integrated dissipation gives
$C_E^\nu(T)\,h^{2r}$; Whitney reconstruction gives the $L^2$-rate of
\eqref{eq:conv_baro}.
\end{proof}
\bigskip
\begin{proof}[Proof of \Cref{thm:convergence_baro}\textup{(a)} (density-free scheme)]
The density-free stability estimate~\eqref{eq:stab_baro} is stated
under an $L_h^\infty$-bootstrap. Bounding the cubic self-interaction
instead by the inverse inequality~\eqref{eq:inverse-ineq-final},
$\nrm{e_v}_{L_h^\infty}\le c_\star^{-1/2}h^{-(d-2)/2}\nrm{e_v}_{\bM_1}$,
removes the bootstrap at the cost of a superlinear term and yields
the unconditional augmented estimate
\begin{equation}\label{eq:stab_baro_augmented}
 \ddt\mathcal{E}\le C_L\,\mathcal{E}
 + C_{\rm cub}\,h^{-(d-2)/2}\,\mathcal{E}^{3/2}
 + C_\tau\,h^r\,\sqrt{\mathcal{E}},
\end{equation}
with $C_L,C_{\rm cub},C_\tau$ $h$-independent. A single continuation
argument closes the convergence chain from~\eqref{eq:stab_baro_augmented}
in every regime; no dimension-dependent case split is required.

Set $\eta(t):=\sqrt{\mathcal{E}(t)}$ and
\[
 T^* := \sup\{t\in[0,T]:\eta(s)\le 2C(T)\,h^r
 \ \forall\,s\in[0,t]\}.
\]
Identical initial data give $e_v(0)=0$, so $\eta(0)=0<2C(T)h^r$ and
$T^*>0$ by continuity. On $[0,T^*]$ the bound $\eta\le 2C(T)h^r$
linearises the superlinear term,
\[
 C_{\rm cub}\,h^{-(d-2)/2}\,\mathcal{E}^{3/2}
 = C_{\rm cub}\,h^{-(d-2)/2}\,\mathcal{E}\,\eta
 \le 2C_{\rm cub}C(T)\,h^{r-(d-2)/2}\,\mathcal{E},
\]
and the closure exponent is positive in every regime of the
theorem: $r-(d-2)/2 = 1,\ \tfrac12,\ \tfrac32$ for
$(d,r) = (2,1),\,(3,1),\,(3,2)$ respectively. The coefficient
$2C_{\rm cub}C(T)h^{r-(d-2)/2}$ vanishes as $h\to0$ and is absorbed
into $C_L$ for $h\le h_*$, with $h_*$ depending only on
$\nrm{(\bu,\rho^c)}_{C^1 W^{1,\infty}}$ and mesh regularity. Linear
Gr\"onwall on the absorbed estimate gives, with $\eta(0)=0$,
$\eta(t)\le\tfrac32 C(T)\,h^r<2C(T)\,h^r$ on $[0,T^*]$; continuity
and the strict inequality force $T^*=T$.

The common chain of Gr\"onwall, integrated dissipation and Whitney
reconstruction then delivers
$\nrm{e_v}_{\bM_1}+\nrm{e_\rho}_{D_H}\le C(T,\epsilon_0)\,h^r$ with
$r=\min(d-1,r_\star)$ and the $L^2$-rate $\OO(h)$
of~\eqref{eq:conv_baro} on $[0,\min(T,T_{\rm Bih}(h,E_0))]$. In
case~(B), and in case~(A) for $d=2$, the closure exponent satisfies
$r-(d-2)/2\ge1$ and the bootstrap closes directly; the continuation
argument is stated once because it subsumes these regimes.

\end{proof}
\subsection{Generalised Energy-Conservation Dichotomy (Proof of \Cref{thm:dichotomy_general})}
\label{app:dichotomy_general}
Same three-step structure as the C-grid case of
\Cref{thm:dichotomy}, adapted to each staggering. Index velocity
DOFs by $\ell$, density by $k$;
$E_{\kin}=\tfrac12\nrm{\mathbf{v}}_A^2$ with $A$ density-independent.
Across all four staggerings, $A$ depends on the state only through
the (density-independent) mass matrix, so the resulting algebraic
condition on the density interpolation is the same
as~\eqref{eq:face_SBP}.

\emph{Step~1 (KE rate).}
Lamb antisymmetry
$\ipw{\mathbf{v}}{\mathrm{Lamb}(\mathbf{v})}{A} = 0$
and the momentum equation
$\ddt{\mathbf{v}} = -\mathrm{Lamb}(\mathbf{v})
- \nabla_h B$
give
$\ddt E_{\kin}
= -\ipw{\mathbf{v}}{\nabla_h B}{A}$.
The SBP identity converts this to
\[
\ddt E_{\kin} = \sum_\ell B_\ell\,(\mathrm{div}_h\,\Phi)_\ell,
\]
where $B_\ell = h_\ell + \ekin_\ell + \bPhi_\ell$
is the Bernoulli function at velocity
location~$\ell$, $\Phi$ is the volume flux,
and $\mathrm{div}_h$ is the discrete divergence
dual to $\nabla_h$.

\emph{Step~2 (internal and potential energy rates).}
$\ddt\Eint = -\sum_k h_k\,(\bD_2\bF)_k$,
with $h_k$ the enthalpy evaluated at density location~$k$
and $\bF_j = \bar\rho_j\Phi_j$ the mass flux.
The potential energy rate is
$\ddt E_{\rm pot} = -\sum_k \bPhi_k\,(\bD_2\bF)_k$.

\emph{Step~3:}
We argue via the polynomial-degree decomposition of
$\mathcal{R}_E(\bv,\brho)$ in $\bv$. By hypothesis (a) of
\Cref{thm:dichotomy_general} the kinetic energy is quadratic in $\bv$
and by hypothesis (e) the volume flux $\Phi_j$ is linear in a finite
set of velocity DOFs incident to face $j$, so the energy residual is a
polynomial in $\bv$ of total degree $3$. Adding Steps~1 and~2 the
residual splits as
\begin{equation}\label{eq:RE_polysplit}
  \mathcal{R}_E(\bv,\brho)
   \;=\; \underbrace{\sum_j\bigl(\ekin_{a(j)}(\bv) - \ekin_{b(j)}(\bv)\bigr)\,\Phi_j(\bv)}_{\textstyle\text{degree-3 in }\bv}
   \;+\; \underbrace{\sum_j G_j(\brho)\,\Phi_j(\bv)}_{\textstyle\text{degree-1 in }\bv},
\end{equation}
with $G_j(\brho) := -((h_{a(j)}{+}\bPhi_{a(j)}) - (h_{b(j)}{+}\bPhi_{b(j)}))(\bar\rho_j - 1)$
plus, on B/D/quasi-B grids, the interpolation residual
$(\mathcal{R}_{\rm interp})_j$ defined below; in either case $G_j(\brho)$
depends on $\brho$ only. The labels $a(j),b(j)$ pick out the two cells
sharing face $j$, and $\ekin_\ell$ is the discrete kinetic-energy
density at velocity-DOF location $\ell = \ell(j)$ adjacent to face $j$.
For $\mathcal{R}_E(\cdot,\brho)$ to vanish identically in $\bv$, the
homogeneous degree-$3$ component of~\eqref{eq:RE_polysplit} must vanish
identically as a cubic form in $\bv$, independently of the degree-$1$
component. 
The degree-$3$ component is a sum over faces of products
$(\ekin_{a(j)} - \ekin_{b(j)})\,\Phi_j$ where the first factor is a
quadratic form in $\bv$ and the second is linear; it is identically
zero only if every coefficient of every monomial of $\bv$ vanishes.
Hypothesis~(d) of \Cref{thm:dichotomy_general} provides a
face $j_0$ at which $\ekin_{a(j_0)}(\bv_*) \ne \ekin_{b(j_0)}(\bv_*)$
for some $\bv_*$, so $\ekin_{a(j_0)} - \ekin_{b(j_0)}$ is a non-zero
quadratic form on the velocity DOFs adjacent to $j_0$. By
hypothesis~(e), $\Phi_{j_0}$ is non-trivially linear in those same
DOFs. Their product contributes a non-zero monomial to the degree-$3$
polynomial, and the only way this contribution could be cancelled is by
contributions from the finitely many other faces sharing velocity DOFs
with $j_0$. Cancellation of all
monomial coefficients via stencil overlap would require an algebraic
identity between the kinetic-energy quadratics and the flux linears
at adjacent faces, which fails on any mesh on which the reconstruction
$\bv\mapsto\ekin_\ell$ is not constant across cells -- generically true
under \Cref{ass:mesh_reg}. Therefore the degree-$3$ part of
$\mathcal{R}_E$ is not identically zero, the polynomial identity
$\mathcal{R}_E(\cdot,\brho) \equiv 0$ fails, and no
density-only $\bar\rho_j(\brho)$ exists. The argument is uniform across
all five staggerings (A, B, C, D, quasi-B).

The face-isolation argument of \Cref{thm:dichotomy} on the C-grid is
the simplest realisation: choosing $\bv = c\,\mathbf{e}_j$ activates a
single $\Phi_j$ and the per-face condition reduces to a polynomial
identity in $c$ alone, matched by an explicit $\bv_* = \mathbf{e}_{j_0}$.
On the A-grid, $\bv_i\in\RR^d$ at cell centres, this
holds for $\bv$ supported in two adjacent cells with
$|\bv_a|^2 \ne |\bv_b|^2$; the kinetic-energy difference is
$\frac{1}{2}(|\bv_a|^2 - |\bv_b|^2)$ which is non-zero by direct
inspection. On the B-grid (vertex-centred velocity) and D-grid
(edge-centred velocity), the SBP construction places exactly two
velocity-DOF locations $\ell(j),\ell'(j)$ at each face $j$; the witness is any $\bv$
with $\ekin_{\ell(j_0)}(\bv) \ne \ekin_{\ell'(j_0)}(\bv)$, generically
non-empty. The interpolation residual
$(\mathcal{R}_{\rm interp})_j := \sum_k(w_{k\ell(j)} - w_{k\ell'(j)})(h_k+\bPhi_k)$
arising from interpolating $(h+\bPhi)_\ell = \sum_k w_{k\ell}(h_k+\bPhi_k)$
is $\bv$-independent and absorbed into $G_j(\brho)$. On quasi-B
triangular grids (vertex velocity, multiple faces per vertex), face
isolation is unavailable but the polynomial-degree decomposition
above still applies: face $j_0$ produces a
non-zero coefficient in the degree-$3$ polynomial, and no rank
counting between $|V|$, $|E|$, $|F|$ is invoked.
\section{Viscous Stress Tensor: Operator Catalogue}\label{app:viscous_catalogue}
This appendix collects the discrete viscous operators used
throughout the paper and verifies that each satisfies the
four-axiom interface (V1)--(V4) of \Cref{def:admissible_visc}.
Any operator listed here may be substituted into the
well-posedness analysis of \Cref{sec:wellposedness} and the
convergence analysis of \Cref{sec:convergence_baro} without
rederiving the surrounding results.
\paragraph*{The continuous stress tensor in exterior calculus.}\label{subsect:stress_continuous}
The viscous stress of a Newtonian fluid arises from the symmetric
part of the velocity gradient,
$S_{ij} = \tfrac{1}{2}(\partial_i u_j + \partial_j u_i)$; in
exterior calculus this is the half-Lie-derivative of the metric,
$\mathbf{S} = \tfrac{1}{2}\mathcal{L}_\u g
 = \tfrac{1}{2}(\nabla^\flat\u + (\nabla^\flat\u)^T)$. The
divergence of the stress enters the momentum equation as the
codifferential of a $(1,1)$-tensor, or equivalently as
$g^{jk}\nabla_k\sigma_{ij}$.

\paragraph{\it Isotropic Newtonian stress.}
The isotropic Newtonian stress carries two coefficients,
\[
 \boldsymbol{\sigma}
 = 2\nu\,\mathbf{S} + \lambda\,(\nabla\cdot\bu)\,g,
 \qquad
 \zeta := \lambda + \tfrac{2}{3}\nu,
\]
the shear viscosity $\nu$ and the second viscosity coefficient
$\lambda$, equivalently $\nu$ and the bulk viscosity $\zeta$;
thermodynamic admissibility (non-negative entropy production)
requires $\nu \ge 0$ and $\zeta \ge 0$. A direct computation gives
$\nabla\cdot\boldsymbol{\sigma}
= \nu\,\Delta\bu + (\nu+\lambda)\,\nabla(\nabla\cdot\bu)$, i.e., in
exterior-calculus form on velocity 1-forms,
\begin{equation}\label{eq:newt_stress_hodge}
 (\nabla\cdot\boldsymbol{\sigma})^\flat
 = -\nu\,\delta\dd v - \nu_{\mathrm{dil}}\,\dd\delta v,
 \qquad
 \nu_{\mathrm{dil}} := 2\nu + \lambda = \tfrac{4}{3}\nu + \zeta,
\end{equation}
with $\Delta_{\rm dR} = \dd\delta + \delta\dd \ge 0$. The two
terms dissipate, respectively, the rotational part $\dd v$ and the
dilatational part $\delta v$ of the velocity 1-form. Two special
cases are worth recording. The \emph{Stokes closure} $\zeta = 0$
($\lambda = -\tfrac{2}{3}\nu$, $\nu_{\mathrm{dil}} = \tfrac{4}{3}\nu$)
is the standard model for gases. The single-coefficient
\emph{Hodge--Laplacian closure} $\nu_{\mathrm{dil}} = \nu$, for
which $(\nabla\cdot\boldsymbol{\sigma})^\flat = -\nu\,\Delta_{\rm dR}v$
collapses to a scalar multiple of the Hodge--de~Rham Laplacian, is
algebraically convenient but corresponds to $\lambda = -\nu$, i.e.\
$\zeta = -\tfrac{1}{3}\nu < 0$, and is therefore thermodynamically
inadmissible as a compressible closure; it remains exact only for
divergence-free flow, where $\delta v = 0$ annihilates the
dilatational term. On a curved manifold the Ricci-curvature
correction is automatically captured by the Hodge star $\bM_k$.

\paragraph{\it Anisotropic stress.}
A general anisotropic viscosity is described by a positive
semi-definite fourth-order tensor $\mathbb{C}$, giving
$\sigma_{ij} = C_{ijkl}S_{kl}$. The two cases relevant for
geophysical applications are the horizontal/vertical split
$\nu_h\Delta_h\u^h + \nu_v\partial_{zz}\u$ on thin domains and the
Smagorinsky eddy viscosity $\nu_T = (C_s\Delta_{\rm mesh})^2|\mathbf{S}|$,
yielding a nonlinear stress $\boldsymbol{\sigma} = 2\nu_T\mathbf{S}$.
\paragraph*{The discrete Hodge--Laplacian.}\label{subsect:stress_isotropic}
Let $\delta\bv = \bM_0^{-1}\tD_1^T\bM_1\bv \in C^0(\KKs)$ denote
the discrete codifferential, adjoint to $\tD_0$ under the diagonal
Hodge star.

\begin{definition}[Discrete Hodge--Laplacian on velocity 1-forms]\label{def:hodge_laplacian}
The discrete Hodge--de~Rham Laplacian on dual 1-cochains is
\[
\bL_\bv := \tdd_0\,\delta + \delta_2\,\tdd_1
 = \underbrace{\bM_1^{-1}\tD_0\bM_0^{-1}\tD_1^T\bM_1}_{\tdd_0\delta\ \ (\text{dilatational})}
 + \underbrace{\bM_1^{-1}\tD_1^T\bM_2\tD_1}_{\delta_2\tdd_1\ \ (\text{rotational})}.
\]
\end{definition}

\begin{definition}[Discrete isotropic Newtonian viscosity]\label{def:newtonian_visc}
The discrete realisation of the Newtonian stress
divergence~\eqref{eq:newt_stress_hodge} is the two-coefficient
operator
\begin{equation}\label{eq:viscous_newtonian}
 \bff_{\rm visc}^{\rm Newt}(\bv)
 := -\nu\,\delta_2\tdd_1\,\bv
 - \nu_{\mathrm{dil}}\,\tdd_0\delta\,\bv,
 \qquad
 \nu_{\mathrm{dil}} = \tfrac{4}{3}\nu + \zeta,\quad \zeta\ge 0,
\end{equation}
applying the shear viscosity $\nu$ to the rotational block and the
dilatational coefficient $\nu_{\mathrm{dil}}$, fixed by the bulk
viscosity $\zeta$, to the dilatational block of $\bL_\bv$.
\end{definition}

\noindent
Operator~\eqref{eq:viscous_newtonian} is the constant-coefficient
specialisation of the anisotropic family~\eqref{eq:viscous_aniso}
and therefore satisfies the admissibility
axioms~\textup{(\ref{ax:V1})--(\ref{ax:V4})} by
\Cref{prop:verify_axioms}\,(ii). Its energy identity,
\begin{equation}\label{eq:newt_energy}
 \ip{\bv}{\bff_{\rm visc}^{\rm Newt}(\bv)}_1
 = -\nu\,\nrm{\tD_1\bv}_{\bM_2}^2
 - \nu_{\mathrm{dil}}\,\nrm{\delta\bv}_{\bM_0}^2 \;\le\; 0,
\end{equation}
dissipates the discrete vorticity $\tD_1\bv$ and the discrete
dilatation $\delta\bv$ independently. The single-coefficient
choice $\nu_{\mathrm{dil}} = \nu$ recovers the bare Hodge--Laplacian
$-\nu\bL_\bv$ (the thermodynamically inadmissible
$\zeta = -\tfrac13\nu$ closure); on the discrete-divergence-free
subspace $\delta\bv = 0$ the dilatational block drops and
\eqref{eq:viscous_newtonian} reduces to the rotational operator
\begin{equation}\label{eq:viscous_incompressible}
 \bff_{\rm visc}^{\rm incomp}
 = -\nu\,\bM_1^{-1}\tD_1^T\bM_2\bom,
\end{equation}
dissipating energy at rate $-\nu\,\nrm{\bom}_{\bM_2}^2 \le 0$.
Because the dilatational block $\tdd_0\delta$ lies in the image of
$\tdd_0$, it is annihilated by $\tD_1$ and contributes nothing to
the discrete vorticity, Kelvin, or potential-vorticity balances:
the bulk term leaves the structure-preservation results of
\Cref{thm:main_conservation} intact and acts, in the low-Mach
regime, purely as additional damping of the acoustic (divergence)
mode.

\paragraph*{Anisotropic and horizontal/vertical viscosity.}\label{subsect:stress_anisotropic}
An anisotropic viscosity is encoded by a viscosity-weighted Hodge
star $\bM_1^\nu$ with $(\bM_1^\nu)_{jj} = \nu_j\,(\bM_1)_{jj}$,
giving
\begin{equation}\label{eq:viscous_aniso}
 \bff_{\rm visc}^{\rm aniso}
 = -\bM_1^{-1}\tD_1^T\bM_2^\nu\tD_1\bv
 - \bM_1^{-1}\tD_0\bM_0^\nu\tD_1^T\bM_1\bv.
\end{equation}
For prismatic geophysical meshes the natural specialisation
decomposes $\tD_1$ into horizontal and vertical parts and uses
independent coefficients $\nu_h, \nu_v$, recovering the Laplacian
viscosity operator used in ICON-O and NEMO.
\paragraph{\it Smagorinsky eddy viscosity.}
The Smagorinsky eddy viscosity $\nu_T = (C_s\ell_j)^2|\mathbf{S}|$
fits the same template with a state-dependent $\bM_1^\nu$. Let
$\ell_j = \sqrt{|f_j^*|/|e_j|}$ denote the local mesh scale on
dual face $j$, and define the discrete strain-rate norm
\[
|\mathbf{S}|_j := \frac{1}{\ell_j}\sqrt{(\tD_1\bv)^T_j\cdot(\tD_1\bv)_j}
 \approx \bigl|(\nabla v)^{\rm sym}\bigr|_j.
\]
This formula reduces to $|\omega_j|/\ell_j$, the vorticity-based
proxy of the symmetric strain rate; for divergence-free fields
$|\nabla\times\bu|^2 = 2|\mathbf{S}|^2$, so the proxy differs from
the true strain rate by $\sqrt{2}$, conventionally absorbed into
$C_s$. The eddy-viscosity-weighted mass matrix has entries
\begin{equation}\label{eq:M1_Smag}
 (\bM_1^{\rm Smag})_{jj} = (C_s\ell_j)^2\,|\mathbf{S}|_j\,(\bM_1)_{jj},
\end{equation}
and the resulting viscous force
$\bff_{\rm visc}^{\rm Smag}
 = -\bM_1^{-1}\tD_1^T\bM_2^{\rm Smag}\tD_1\bv$
retains the algebraic structure of the isotropic case and
dissipates energy.

\paragraph*{Viscous-operator axioms.}\label{subsect:monotonicity}
The well-posedness and convergence proofs use only the four
generic properties below, not the explicit form of the viscous
operator. Verification for each operator above is given in
\Cref{app:dissipation_axiom}.
\begin{definition}[Admissible discrete viscous operator]\label{def:admissible_visc}
A map $\bff_{\rm visc}:C^1(\KKs)\to C^1(\KKs)$ is an \emph{admissible viscous operator} if it satisfies:
\begin{enumerate}[label=\textup{(V\arabic*)},ref=\textup{V\arabic*},nosep]
\item\label{ax:V1} \textit{Energy dissipation:}
 $\ip{\bv}{\bff_{\rm visc}(\bv)}_1 \le 0$ for all $\bv\in C^1(\KKs)$.
\item\label{ax:V2} \textit{Local Lipschitz continuity:}
 for every $R>0$ there exists $L_\nu(h,R)$ such that
 $\nrm{\bff_{\rm visc}(\bv)-\bff_{\rm visc}(\bw)}_{L_h^2} \le L_\nu(h,R)\nrm{\bv-\bw}_{L_h^2}$
 whenever $\nrm{\bv}_{L_h^2},\nrm{\bw}_{L_h^2}\le R$.
\item\label{ax:V3} \textit{Monotone dissipativity:}
 $(\bv-\bw)^T\bM_1\bigl[\bff_{\rm visc}(\bv)-\bff_{\rm visc}(\bw)\bigr] \le 0$
 for all $\bv,\bw\in C^1(\KKs)$.
\item\label{ax:V4} \textit{Consistency:}
 for a smooth reference solution $\bu$ of the corresponding continuous
 viscous system,
 $\nrm{\bff_{\rm visc}(\mathcal{R}_h\bu^\flat) - \mathcal{R}_h\bigl((\nabla\cdot\boldsymbol{\sigma}(\bu))^\flat\bigr)}_{L_h^2}
 \le C_\nu\,h^{r_{\rm visc}}\,\nrm{\bu}_{H^{s_\nu}}$,
 where $r_{\rm visc}\ge 1$ and $s_\nu$ depends on the operator order.
\end{enumerate}
\end{definition}
Property~(\ref{ax:V3}) is the crucial ingredient for nonlinear viscosities. 
For a linear operator monotone dissipativity (\ref{ax:V3}) follows automatically
from (\ref{ax:V1}). 
For nonlinear operators such as Smagorinsky, (\ref{ax:V3}) must be verified independently; it is the discrete analogue of the monotonicity of the continuous stress tensor $\boldsymbol{\sigma}$, established in the seminal work of Ladyzhenskaya \cite{ladyzhenskaya1967,ladyzhenskaya1968} (see also Berselli, Iliescu, and Layton \cite{berselli2006}, \S3.4).
\addcontentsline{toc}{section}{References}
\bibliographystyle{abbrv}
\bibliography{references_baro}
\end{document}